\documentclass[11pt]{article}
\usepackage{amssymb,amsthm,amsmath,mathrsfs}
\usepackage{prettyref,cite}
\usepackage[colorlinks,linkcolor=blue,citecolor=magenta]{hyperref}
\usepackage{graphicx,tikz,tkz-berge}
\usepgflibrary{shapes.misc}
\usepackage{algorithmic,algorithm}
\usepackage [a4paper, footskip=1cm, headheight = 16pt, top=2cm, bottom=2.5cm,  right=2cm,  left=2cm ]{geometry}
\newcommand\pref[1]{\prettyref{#1}}
\newrefformat{chp}{\chaptername~\ref{#1}}
\newrefformat{sec}{Section~\ref{#1}}
\newrefformat{sub}{Subsection~\ref{#1}}
\newrefformat{app}{Appendix~\ref{#1}}
\newrefformat{sub}{Subsection~\ref{#1}}
\newrefformat{thm}{Theorem~\ref{#1}}
\newrefformat{qst}{Question\ref{#1}}
\newrefformat{lem}{Lemma~\ref{#1}}
\newrefformat{cor}{Corollary~\ref{#1}}
\newrefformat{con}{Conjecture~\ref{#1}}
\newrefformat{rem}{Remark~\ref{#1}}
\newrefformat{fig}{\figurename~\ref{#1}}
\newrefformat{tbl}{\tablename~\ref{#1}}
\newrefformat{eq}{$($\ref{#1}$)$}
\newrefformat{fac}{Fact~\ref{#1}}

\newtheorem{pro}{Proposition}[section]
\newtheorem{cor}[pro]{Corollary}
\newtheorem{lem}[pro]{Lemma}
\newtheorem{thm}[pro]{Theorem}
\newtheorem{con}[pro]{Conjecture}
\newtheorem{qst}[pro]{Question}
\newtheorem{prerem}[pro]{Remark}
\newenvironment{rem}{\begin{prerem}\rm}{\hfill $\blacktriangle$\end{prerem}}
\DeclareMathOperator{\cc}{cc}

\newcommand{\vsp}{\vspace{2mm}}
\begin{document}
\title{
	\vspace{3cm}
	\textbf{Edge Clique Cover of Claw-free Graphs}}
\author{Ramin Javadi 
\thanks{Corresponding author}
 \thanks{Department of Mathematical Sciences,
Isfahan University of Technology,
P.O. Box: 84156-83111, Isfahan, Iran. School of Mathematics, Institute for Research in Fundamental Sciences (IPM), P.O. Box: 19395-5746,
Tehran, Iran. This research was in part supported by a grant from IPM (No. 94050079). Email Address: \href{mailto:rjavadi@cc.iut.ac.ir}{rjavadi@cc.iut.ac.ir}.}
 \and Sepehr Hajebi\thanks{Department of Mathematical Sciences,
Isfahan University of Technology,
P.O. Box: 84156-83111, Isfahan, Iran. Email Address: \href{mailto:s.hajebi@math.iut.ac.ir}{s.hajebi@math.iut.ac.ir}.}} 
\date{}
\maketitle
%
\begin{abstract}
The smallest number of cliques, covering all edges of a graph $ G $, is called the (edge) clique cover number of $ G $ and is denoted by $ \cc(G) $. It is an easy observation that if $ G $ is a line graph on $ n $ vertices, then $\cc(G)\leq n $. G. Chen et al. [Discrete Math. 219 (2000), no. 1--3, 17--26; \href{http://www.ams.org/mathscinet-getitem?mr=1761707}{MR1761707}] extended this observation to all quasi-line graphs and questioned if the same assertion holds for all claw-free graphs. In this paper, using the celebrated structure theorem of claw-free graphs due to Chudnovsky and Seymour, we give an affirmative answer to this question for all claw-free graphs with independence number at least three. In particular, we prove that if $ G $ is a connected claw-free graph on $ n $ vertices with three pairwise nonadjacent vertices, then $ \cc(G)\leq n $ and the equality holds if and only if $ G $ is either the graph of icosahedron, or the complement of a graph on $10$ vertices called ``twister'' or the $p^{th}$ power of the cycle $ C_n $, for some positive integer $p \leq \lfloor (n-1)/3\rfloor $. \\[2mm]
\end{abstract}
\newpage
\tableofcontents
\newpage
\section{Introduction}\label{sec:intro}
Throughout the paper all graphs are finite without parallel edges and loops, unless it is explicitly mentioned. By a \textit{clique} of $ G $ we mean a subset of pairwise adjacent vertices. The size of the largest clique of $ G $ is called the \textit{clique number} of $ G $ and is denoted by $ \omega(G) $. Also, a subset of pairwise nonadjacent vertices is called a  \textit{stable set} of $ G $ and the size of the largest stable set of $ G $ is called the \textit{independence number} of $ G $ and is denoted by $ \alpha(G) $.  
A clique of size three and a stable set of size three are called a \textit{triangle} and a \textit{triad}, respectively. The graph $ K_{1,3} $ (a star graph with three edges) is called a \textit{claw}. We also use the term ``a claw in $ G $'' for a subset of vertices $ X\subseteq V(G) $ where the induced subgraph of $ G $ on $X$ is isomorphic to $ K_{1,3} $. 
Given a graph $ H $, a graph $ G $ is called $ H$-free, if $ G $ does not have the graph $ H $ as an induced subgraph. 
Thus, \textit{claw-free} graphs are $ K_{1,3} $-free graphs. Also, \textit{triangle-free} graphs and \textit{triad-free} graphs are graphs with no triangle and no triad, respectively. A graph $ G $ is called a \textit{line graph} if $ G $ is the line graph of a graph $ H $ (i.e. the vertices of $ G $ correspond to the edges of $ H $ and two vertices are adjacent in $ G $ if their corresponding edges in $ H $ have a common endpoint). Also, a graph is called a \textit{quasi-line graph} if the neighbourhood of every vertex can be expressed as the union of two cliques. Line graphs, quasi-line graphs and triad-free graphs are examples of claw-free graphs. In a series of elegant papers \cite{seymour1,seymour2,seymour3,seymour4,seymour5}, Seymour and Chudnovsky gave a complete description that explicitly explains the structure of all claw-free graphs. Since then, their structure theorem has been used to prove several conjectures for claw-free graphs. In this paper, we are going to apply this machinery to investigate the edge clique covering of claw-free graphs.
 
Given a graph $ G $, let us say that an edge of $ G $ is covered by a clique $ C $ of $ G $ if both of its endpoints belong to $ C $. An (\textit{edge}) \textit{clique covering} for $ G $  is a collection $ \mathscr{C} $ of cliques of $ G $ such that every edge of $ G $ is covered by a clique in $ \mathscr{C} $. The (\textit{edge}) \textit{clique cover number} (or the \textit{intersection number}, see \cite{erdos}) of $ G $, denoted by $ \cc(G) $, is the least number $ k $ such that $ G $ admits a clique covering of size $ k $. The clique cover number has been studied widely in the literature (see e.g. \cite{sur1,sur2,sur3,erdos}).
In \cite{Chen}, Chen et al.  proved that the clique cover number of a quasi-line graph on $ n $ vertices is at most $ n $. Claw-free graphs being generalization of quasi-line graphs, they questioned if the same result holds for all claw-free graphs.  In order to answer the question of Chen et al., it is enough to answer it for connected claw-free graphs. Regarding this question, in \cite{javadi}, it is proved that if $ G $ is a connected claw-free graph on $ n $ vertices with the minimum degree $ \delta $, then $ \cc(G)\leq n+c\, \delta^{4/3}\log^{1/3}\delta $, for some constant $ c $. Moreover, assuming $ \Delta $ as the maximum degree of $ G $,
 it is proved there that $ G $ admits a clique covering $ \mathscr{C} $ such that each vertex is in at most $ c\Delta /\log \Delta $ cliques of $ \mathscr{C} $, where $ c $ is a constant (though they did not use the structure theorem of claw-free graphs).  
In this paper, we apply the structure theorem to give an affirmative answer to the question of Chen et al. for all claw-free graphs $ G $ with $\alpha(G)\geq  3 $. We also characterize the equality cases, i.e. all such graphs whose clique cover number is equal to their number of vertices. Let us describe the main results of this paper in more detail. 

First, we need a couple of definitions. For a graph $G$ and a nonnegative integer $p$, the \textit{$p^{th}$ power of} $G$, denoted by $G^p$,  is defined as the graph on the same vertices such that two distinct vertices $u,v$ are adjacent in $G^p$ if and only if there is a path in $ G $ between $u$ and $ v $ with at most $ p $ edges  (thus, $ G^1=G $). For instance, when $n\geq 2p+1$, the $p^{th}$ power of the cycle $C_n$ (also called a \textit{circulant graph}), denoted by $C_n^p$, is the graph with vertices $ v_0,v_1,\ldots,v_{n-1} $, such that $ v_i $ is adjacent to $ v_j $ if and only if either $ i-j $ or $ j-i $ modulo $ n $ is in $ \{1,\ldots,p\} $. All graphs $C_n^p$ are clearly claw-free (they are in fact quasi-line graphs).  The graph of \textit{icosahedron} and \textit{the complement of a twister} (depicted in \pref{fig:icosa}) are two more examples of claw-free graphs. The main result of this paper is as follows.
\begin{thm} \label{thm:main1}
Let $ G $ be a connected claw-free graph on $ n $ vertices with $ \alpha(G)\geq 3 $.  Then, $ \cc(G)\leq n$ and equality holds if and only if $ G $ is either the graph of icosahedron, or the complement of a twister, or the $p^{th}$ power of the cycle $C_n$, for some positive integer $p\leq \lfloor (n-1)/3\rfloor$.
\end{thm}
According to the above theorem, in order to answer the question of Chen et al. completely, it suffices to prove the upper bound $ n $ for the clique cover number of $ n$-vertex triad-free graphs (i.e. graphs $ G $ with $ \alpha(G)\leq 2 $). Since there is no structure theorem for triangle-free graphs (and so triad-free graphs), dealing with the question for these graphs seems to require a completely different approach. In \cite{javadi2}, it is proved that for every triad-free graph $ G $ on $ n $ vertices, $\cc(G)\leq 2(1-o(1)) n $. Combining this result with \pref{thm:main1} gives the following corollary.
\begin{cor}
For every connected claw-free graph $ G $ on $ n $ vertices, we have $ \cc(G)\leq  2(1-o(1)) n $. 
\end{cor}
However, improving the upper bound to $ n $ for all triad-free graphs remains open. 
\begin{con}\label{con:tf}
If $ G $ is a graph on $ n $ vertices with $ \alpha(G)\leq 2 $, then $ \cc({G})\leq n $.
\end{con}
In this paper, we also prove the upper bound $ n+1 $ for the clique cover number of a special class of triad-free graphs. Pursuing \cite{seymour6}, a graph $G$ is said to be \textit{tame} if $G$ is an induced subgraph of a connected claw-free graph $H$ such that $H$ has a triad. For instance, every triad-free graph whose vertex set is the union of three cliques is tame. To see this, let $ G $ be a triad-free graph, where $V(G)$ is the union of three cliques $A,B$ and $C$. Now, define $H$ to be the graph obtained from $G$ by adding three new vertices $a,b,c$, such that $\{a,b,c\}$ is a triad, $a$ is adjacent to every vertex in  $A\cup B$, $b$ is adjacent to every vertex in $B\cup C$ and $c$ is adjacent to every vertex in $C\cup A$. Then, $H$ is a claw-free graph with a triad and $G$ is an induced subgraph of $ H $. Thus, $G$ is tame.  In fact, we prove the following theorem. 
\begin{thm} \label{thm:main2}
Let $G$ be a connected graph on $n$ vertices which is tame. If $V(G)$ is not the union of three cliques, then $\cc(G)\leq n$ and otherwise, $\cc(G)\leq n+1$.
\end{thm}
Finally, it is worth noting that investigating the clique cover number of $ K_{1,t}$-free graphs ($ t\geq 4 $), as a generalization of claw-free graphs, seems to be an interesting and challenging problem. 
For a graph $ G $, let us define the \textit{local independence number} of $ G $, denoted by $ \alpha_l(G) $, as the maximum number $ t $ such that $ G $ has a stable set of size $ t $ within the neighbourhood of a vertex. Thus, claw-free graphs are exactly the graphs with $ \alpha_l(G)\leq 2 $. Towards the generalization of \pref{thm:main1}, the following question is naturally raised in the same line of thought, relating the clique cover number and the local independence number.
\begin{qst} \label{qst:k1tfree}
Is it true that for every graph $ G $ on $ n $ vertices, $ 2\cc(G)\leq n\, \alpha_l(G) $?
\end{qst} 
Moreover, a relaxed version of the above question can be asked about the independence number as follows (note that this is a generalization of \pref{con:tf}).
\begin{qst} \label{qst:ktfree}
Is it true that for every graph $ G $ on $ n $ vertices, $ 2\cc(G)\leq n\, \alpha(G) $?
\end{qst}

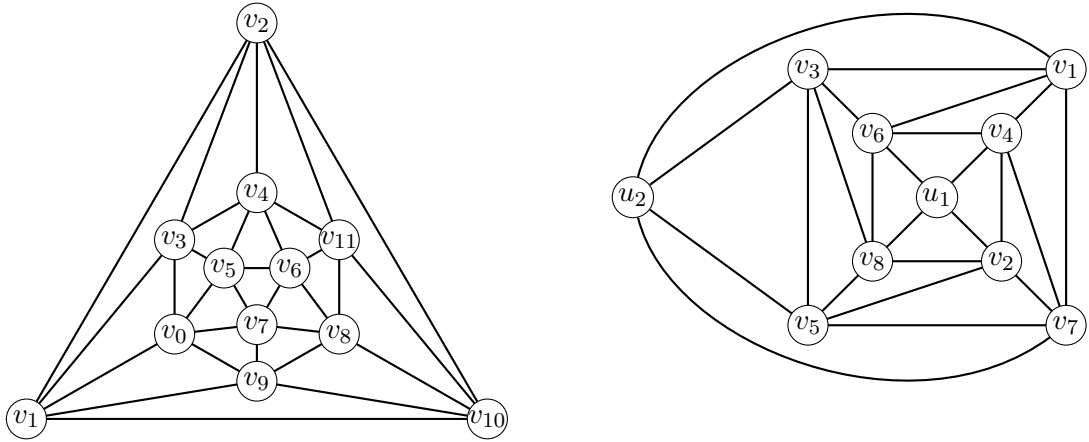
\begin{figure}
  \begin{center}
  \begin{tikzpicture}
\SetVertexNoLabel
\tikzset{VertexStyle/.style ={shape=circle,minimum size=15pt,draw}}
 \begin{scope}[scale=.5]
  \begin{scope}[rotate=-30]
 \grCycle[RA=7,prefix=v]{3}
  \end{scope}
\begin{scope}[rotate=90]
 \grCycle[RA=2.5,prefix=u]{6}
 \end{scope}
\begin{scope}[rotate=30]
 \grCycle[RA=1,prefix=w]{3}
 \end{scope}
 \EdgeFromOneToSeq{v}{u}{0}{3}{5}
 \EdgeFromOneToSeq{v}{u}{2}{1}{3}
 \EdgeFromOneToSel{v}{u}{1}{0,1,5}
  \EdgeFromOneToSeq{w}{u}{1}{0}{2}
  \EdgeFromOneToSeq{w}{u}{2}{2}{4}
  \EdgeFromOneToSel{w}{u}{0}{0,4,5}  
  \AssignVertexLabel{v}{$ v_{{10}} $,$ v_{2} $,$ v_{1} $}
  \AssignVertexLabel{u}{$v_{4}$,$ v_{3} $,$ v_{0} $,$v_{9}$,$ v_{8} $,$ v_{{11}} $}
  \AssignVertexLabel{w}{$v_{6}$,$ v_{5} $,$ v_{7} $}
  \end{scope}
  \end{tikzpicture}
  \hspace{1cm}
    \begin{tikzpicture}
    \begin{scope}[scale=.5]
  \SetVertexNoLabel
  \tikzset{VertexStyle/.style ={shape=circle,inner sep=1mm,minimum size=15pt,draw}}
   \begin{scope}[scale=.6]
    \begin{scope}[rotate=-45]
   \grCycle[RA=8,prefix=v]{4}
    \end{scope}
    \end{scope}
   \begin{scope}[scale=.3]
    \begin{scope}[rotate=-45]
   \grCycle[RA=8,prefix=w]{4}
    \end{scope}
    \end{scope}
  \node at (0,0) [circle,draw,inner sep=1pt,minimum size=15pt] (t) {$u_1$};
  \node at (-8,0) [circle,draw,inner sep=1pt] (s) {$u_2$};
  \draw [-,thick] (s) to (v2);
  \draw [-,thick] (s) to (v3);
  \draw [-,thick] (s) .. controls +(1,4) and +(-4,3)..  (v1);
  \draw [-,thick] (s)  .. controls +(1,-4) and +(-4,-3)..  (v0);
  \draw [-,thick] (w0) to (v0);
   \draw [-,thick] (w0) to (v3);
  \draw [-,thick] (w1) to (v0);
   \draw [-,thick] (w3) to (v3);
   \draw [-,thick] (w1) to (v1);
   \draw [-,thick] (w3) to (v2);
   \draw [-,thick] (w2) to (v1);
   \draw [-,thick] (w2) to (v2);
   \EdgeFromOneToSeq{t}{w}{}{0}{3}
    \AssignVertexLabel{v}{$ v_7 $,$ v_1 $,$ v_3 $,$ v_5 $}
    \AssignVertexLabel{w}{$v_2$,$v_4 $,$ v_6 $,$v_8$}
    \end{scope}
      \end{tikzpicture}
  \end{center}
  \vspace{-10pt}
  \caption{The graph of icosahedron $ G_ 0$ (left) and the complement of a twister (right).} \label{fig:icosa} \label{fig:twister}
\end{figure}

In order to prove \pref{thm:main1}, we deploy the structure theorem in \cite{seymour5}. It basically asserts that for every connected claw-free graph $ G $, either $ V(G) $ is the union of three cliques, or $ G $ is some kind of ``generalized line graph'' which admits a certain structure called \textit{non-trivial strip-structure}, or $ G $ is in one of the three basic classes, namely, \textit{graphs from the icosahedron}, \textit{fuzzy long circular interval graphs} and \textit{fuzzy antiprismatic graphs} (see \pref{sec:structure} for the definitions and the exact statement of the theorem). We will apply this theorem and more intricate structure of the three-cliqued claw-free graphs as well as the antiprismatic graphs to prove \pref{thm:main1}. The most difficult and lengthiest part of the proof (like most of the other applications of the structure theorem) is dealing with antiprismatic graphs.  

The organisation of forthcoming sections is as follows. In \pref{sec:structure}, we recall  from \cite{seymour5} the necessary definitions and the structure theorem. In \pref{sec:equality}, we discuss the claw-free graphs whose clique cover number is equal to the number of their vertices.   In \pref{sec:main}, we give a sketch of the proof of the main results (Theorems \ref{thm:main1} and \ref{thm:main2}). In \pref{sec:strip}, we study the clique covering of claw-free graphs which admit a non-trivial strip-structure. \pref{sec:fci} is devoted to the clique covering of the fuzzy long circular interval graphs. In \pref{sec:3cliques}, using the structure of three-cliqued claw-free graphs from \cite{seymour5}, we investigate their clique cover number. Finally, in Sections~\ref{sec:oap} and \ref{sec:noap}, we go through the clique covering of the antiprismatic graphs.
\subsection{Notation and terminology} \label{sub:notation}
Here, we collect some notation and terminology that is used throughout the paper. For two graphs $ G,H $, we write $ G=H $ if $ G $ is isomorphic to $ H $. In all definitions, we often omit the subscript $G$ (if exists) whenever there is no ambiguity. The graph  $ G $ is called \textit{non-null} if $ V(G) $ is nonempty. The \textit{complement} of $ G $, denoted by $ \overline{G} $, is a graph on the same vertices such that two distinct vertices are adjacent in $ \overline{G} $ if and only if they are nonadjacent in $G$. For a vertex $u\in V(G)$, let $N_G(u)$ stand for the set of all neighbours of $u$ in $G$ and define $N_G[u]=N_G(u)\cup\{u\}$ as the closed neighbourhood of $u$ in $G$. 
For a graph $ G $ and a subset of vertices $ X\subseteq V(G) $, the induced subgraph of $ G $ on $ X $ is denoted by $ G[X] $.
For two sets $X,Y\subseteq V(G)$, $ E_G(X,Y) $ stands for the set of edges in $ G $ whose one end is in $ X $ and another is in $ Y $.  Also, $ E_G(X) $ is a shorthand for $ E_G(X,X) $. For a vertex $ u\in V(G) $ and a set $ X\subseteq V(G)  $, define $N_G(u,X)=N_G(u)\cap X$ as the set of all neighbours of $u$ in $X$, and $N_G[u,X]=N_G(u,X)\cup\{u\}$. For two sets $ X,Y\subseteq V(G) $, define $ \mathscr{N}_G[X;Y]=\{N_G[x,Y]: x\in X\} $. If $F$ is a set of unordered pairs of $V(G)$, then the graph $G\setminus F$ (resp. $ G+F $) is obtained from $G$ by deleting (resp. adding) all edges in $F\cap E(G)$ (resp. $ F\cap E(\overline{G}) $). Also, when $ F $ is singleton, we often drop the brackets. 
Let $ \mathscr{C} $ be a collection of cliques of $ G $ and $ C_1,\ldots, C_k \in \mathscr{C} $. If $C= C_1\cup \ldots \cup C_k $ is a clique of $ G $, then by \textit{merging} $ (C_1,\ldots, C_k) $ in $ \mathscr{C} $ we mean removing the cliques $ C_1,\ldots, C_k $ from $ \mathscr{C} $ and adding the clique $ C $.

Furthermore, we follow some definitions from \cite{seymour5,seymour1}. A vertex $ u\in V(G) $ is called \textit{simplicial} if $ N_G(u) $ is a clique of $ G $. For a vertex $u\in V(G)$ and a set $X\subseteq V(G)\setminus \{u\}$, we say that $u$ is complete to $X$ if $u$ is adjacent to every vertex in $X$ and that $u$ is anticomplete to $X$ if $u$ has no neighbour in $X$. For two disjoint sets $X,Y\subseteq V(G)$, we say that $X$ is complete (resp. anticomplete) to $Y$ if every vertex in $X$ is complete (resp. anticomplete) to $Y$. Also, we say that $ X,Y $ are \textit{matched}, if $ |X|=|Y| $ and every vertex in $ X $ has a unique neighbour in $ Y $ and vice versa. Also, we say that $ X,Y $ are \textit{anti-matched}, if $ X,Y $ are matched in $ \overline{G} $. 
For a claw-free graph $ G $, the set of all vertices of $ G $ which are in at least one triad of $ G $ is called the \textit{core} of $ G $ and is denoted by $ W(G) $. The set $ V(G)\setminus W(G) $ is called the \textit{non-core} of $ G $ and is denoted by $ \tilde{W}(G) $.

\section{Global structure of claw-free graphs} \label{sec:structure}
In this section, we recall from \cite{seymour5,seymour6} the exact statement of the theorem that explicitly describes the structure of claw-free graphs. It will be later applied to prove the main results stated in \pref{sec:intro}. For it, we have to recall a couple of definitions (for more details, see \cite{seymour1,seymour2,seymour3,seymour4,seymour5}).
Given a graph $ G $, a set $ F $ of unordered pairs of $V(G)$ is called a \textit{valid set} for $ G $, if every vertex of $ G $ belongs to at most one member of $F$.
For a graph $ G $ and a valid set $ F $, the graph $G'$ is called a \textit{thickening} of $(G,F)$, if its vertex set $V(G')$ is the disjoint union of some sets $(X_v)_{v\in V(G)}$ such that
\begin{itemize}
\item[(T1)] for each $ v\in V(G) $, $ X_v \subseteq V(G')$ is a clique of $ G' $,
\item[(T2)] if $ u,v\in V(G) $ are adjacent in $ G $ and $ \{u,v\}\not\in F $, then $ X_u $ is complete to $ X_v $ in $ G' $,
\item[(T3)] if $ u,v\in V(G) $ are nonadjacent in $ G $ and $  \{u,v\}\not\in F$, then $ X_u $ is anticomplete to $ X_v $ in $ G' $,
\item[(T4)] if $ \{u,v\}\in F $, then $ X_u $ is neither complete nor anticomplete to $ X_v $ in $ G' $.
\end{itemize}
Here are three classes of claw-free graphs which are needed for the statement of the structure theorem.
\begin{itemize}
\item \textbf{$\mathcal{T}_1$: Graphs from the icosahedron.} The graph of \textit{icosahedron} is the unique planar five regular graph $G_0$ with twelve vertices. In particular, $G_0$ has vertices $v_0,v_1,\ldots, v_{11}$, where for $1\leq i\leq 10$, $v_i$ is adjacent to $v_{i+1},v_{i+2}$ (reading subscripts modulo 10) and $v_0$ is adjacent to $v_1,v_3,v_5,v_7,v_9$ and $v_{11}$ is adjacent to $v_2,v_4,v_6,v_8,v_{10}$ (see \pref{fig:icosa}). Let $G_1$ be obtained from $G_0$ by deleting $v_{11}$ and let $G_2$ be obtained from $G_1$ by deleting $v_{10}$. Furthermore, let $F'=\{\{v_1,v_4\},\{v_6,v_9\}\}$. 
The class $\mathcal{T}_1$ is the set of all graphs $G$ where $G$ is a thickening of either $(G_0,\emptyset)$, or $(G_1,\emptyset)$, or $(G_2,F)$, for some $F\subseteq F'$.
\item \textbf{$\mathcal{T}_2$: Fuzzy long circular interval graphs.} Let $\Sigma$ be a circle and let $\mathcal{I}=\{I_1,\ldots, I_k\}$ be a collection of subsets of $\Sigma$, such that each $I_i$ is homeomorphic to the interval $[0,1]$, no two of $I_1,\ldots, I_k$ share an endpoint, and no three of them have union $\Sigma$. Let $H$ be the graph whose vertex set is a finite subset $V\subseteq \Sigma$ and distinct vertices $u,v\in V$ are adjacent precisely if $u,v\in I_i$ for some $i$. The graph $H$ is called a \textit{long circular interval graph}. The powers of cycles defined in \pref{sec:intro} are examples of long circular interval graphs.  Furthermore, let $F'$ be the set of all pairs $\{u, v\}$ such that $u, v \in V$ are distinct endpoints of $I_i$ for some $i$ and there exists no $j\neq i$ for which $u,v\in I_j$. Also, let $ F \subseteq F'$. Then, for some such $ H $ and $ F $, any thickening $G$ of $ (F,H) $ is called a \textit{fuzzy long circular interval graph}.
The class of all fuzzy long circular interval graphs is denoted by $\mathcal{T}_2$.
\item \textbf{$\mathcal{T}_3$: Fuzzy antiprismatic graphs.}
A graph $H$ is called \textit{antiprismatic} if for every triad $\tau$ and every vertex $v \in V (H) \setminus \tau$, $v$ has exactly two neighbours in $\tau$. Let $u, v$ be two vertices of an antiprismatic graph $H$. We say that the pair $\{u, v\}$ is \textit{changeable} if $u$ is nonadjacent to $v$, and $ G+uv$ is also antiprismatic. 
Let $H$ be an antiprismatic graph and let $F$ be a valid set of changeable pairs of $ H $. Then, for some such $ H$ and $F $, every thickening of pair $ (H,F) $ is called a \textit{fuzzy antiprismatic graph}.
The class of all fuzzy antiprismatic graphs is denoted by $\mathcal{T}_3$.
\end{itemize}
Now, we recall the structure theorem of claw-free graphs from \cite{seymour6}. See \pref{sec:strip} for the definition of a non-trivial strip-structure.
\begin{thm} {\rm \cite{seymour5,seymour6}} \label{thm:seymour1}
Let $G$ be a connected claw-free graph. Then either
\begin{itemize}
\item $G$ admits  a non-trivial strip-structure, or
\item $G$ is in $\mathcal{T}_1\cup \mathcal{T}_2\cup \mathcal{T}_3$, or
\item $V(G)$ is the union of three cliques of $G$.
\end{itemize}
\end{thm}
Finally, we conclude this section with two useful lemmas showing how one can extend a clique covering of a graph to a clique covering of its thickening.
\begin{lem} \label{lem:thickening}
Let $ H $ be a graph and $ F $ be a valid set for $ H $. Also, let $ G $ be a thickening of $ (H,F) $. Assume that no isolated vertex of $ H\setminus F$ belongs to a member of $ F $ and $ H\setminus F $ admits a clique covering of size at most $ |V(H)|-t $, for some number $t$. Then $ G $ admits a clique covering of size at most $ |V(G)|-t $.
\end{lem}
\begin{proof}
Let $ (X_v)_{v\in V(H)} $ be as in the definition of thickening. 
Suppose that $ \mathscr{C} $ is a clique covering for $ H\setminus F$ of size at most $ |V(H)|-t $.  
For each $ C\in \mathscr{C} $, define $ X_C=\cup_{u\in C} X_u $ which is a clique of $ G $. Note that if $ F=\emptyset $, then $ \{X_C: C\in \mathscr{C} \} $ is a clique covering for $ G $ of size at most $ |V(G)|-t $. Now, let $ F=\{\{u_1,v_1\}, \ldots, \{u_\ell,v_\ell\} \} $ and w.l.o.g. assume that for every $ i \in \{1,\ldots, \ell\}$, $ |X_{u_i}|\leq |X_{v_i}| $. Note that by (T4), $ |X_{v_i}|\geq 2 $. Also, let $ I $ be the set of all isolated vertices of $ H\setminus F $. 
Now, the collection of cliques
\[(\bigcup_{i=1}^l \mathscr{N}_H[X_{u_i};X_{v_i}])\cup \{X_v: v\in I, |X_v|\geq 2 \} \cup \{X_C: C\in \mathscr{C}\},\]
is a clique covering for $G$ of size at most
\begin{align*}
\sum_{i=1}^\ell |X_{u_i}|+\sum_{v\in I}(|X_v|-1)+|V(H)|-t &\leq 
\sum_{i=1}^{\ell} (|X_{u_i}|+|X_{v_i}|)-2\ell +\sum_{v\in I}|X_v|-|I|+|V(H)|-t\\ &\leq |V(G)|-t.
\end{align*}
\end{proof}
The following lemma is the counterpart of \pref{lem:thickening} for the thickening of an antiprismatic graph.
\begin{lem} \label{lem:thickening_antiprismatic}
Let $ H $ be an antiprismatic graph and $ F $ be a valid set of changeable pairs of  $ H $. Also, let $ G $ be a thickening of $ (H,F) $ containing a triad. Assume that $ H $ admits a clique covering of size at most $ |V(H)|-t $, for some number $t\leq 1$. Then, $ G $ admits a clique covering of size at most $ |V(G)|-t$.
\end{lem}
\begin{proof}
First, note that the members of $F$ are non-edges of $H$. Thus, $H\setminus F=H$ and since $G$ contains a triad, $H$ contains a triad, too. If no isolated vertex of $ H$ belongs to a member of $ F $, then by \pref{lem:thickening}, we are done. Now, assume that $ F $ contains a pair $ \{u_1,v_1\} $ such that $ u_1 $ is an isolated vertex of $H$. Since $ \{u_1,v_1\} $ is a changeable pair of $ H $,  $V(H)\setminus \{u_1,v_1\}$  is a clique. 
Also,  since $H$ is antiprismatic and contains a triad, $ v_1 $ has exactly one non-neighbour, say $v'_1$, in $V(H)\setminus \{u_1,v_1\}$. Thus, since $ F $ is valid, we have $ F=\{\{u_1,v_1\}\} $.
Assume that $\{a,b\}=\{u_1,v_1\}$ such that $|X_{a}|\leq |X_b|$ and so $|X_b|\geq 2$. Then, consider the collection $\mathscr{N}[X_a;X_b]\cup \{X_{u_1}, \cup_{v\in V(H)\setminus \{u_1,v'_1\}}X_v \} $ of cliques of $G$ and if $ |\cup_{v\in V(H)\setminus \{u_1,v_1\}}X_v|\geq 2$, then add the clique $\cup_{v\in V(H)\setminus \{u_1,v_1\}}X_v$ to obtain a clique covering for $G$ of size at most $|X_a|+2+|\cup_{v\in V(H)\setminus \{u_1,v_1\}}X_v|-1\leq |V(G)|-|X_b|+1\leq |V(G)|-t$.
\end{proof}

\section{Equality cases}\label{sec:equality}
In this section, we investigate the claw-free graphs whose clique cover number is equal to the number of their vertices. \pref{thm:main1} asserts that these graphs are merely the graph of icosahedron, the complement of a twister, and some powers of cycles.
First, we need the following simple lemma which gives a lower bound for the clique cover number of a graph.
\begin{lem}\label{lem:lowerbound}
Let $G$ be a graph and $(\chi_v: v\in V(G))$ be some positive integers such that for every vertex $v\in V(G)$, the chromatic number of $\overline{G}[N_G(v)]$ is at least $\chi_v$. Then, $\cc(G)\geq \lceil(\sum_{v\in V(G)} \chi_v)/\omega(G)\rceil$, where $\omega(G)$ denotes the clique number of $ G $.
\end{lem}
\begin{proof}
Let $\mathscr{C}$ be a clique covering for $G$ of size $\cc(G)$. For every vertex $v\in V(G)$, let $r_v$ be the number of cliques in $\mathscr{C}$ containing $v$. Then, $N_G(v)$ can be partitioned into at most $r_v$ cliques of $G$ and so $r_v\geq \chi_v$, for all $v\in V(G)$. Hence, 
\[\sum_{v\in V(G)} \chi_v\leq \sum_{v\in V(G)} r_v =\sum_{C\in \mathscr{C}} |C|\leq |\mathscr{C}|\, \omega(G), \]
as desired. 
\end{proof}
Now, we prove the assertion of \pref{thm:main1} for the graphs in $\mathcal{T}_1$, in the following lemma.
\begin{lem}\label{lem:icosa}
Let $G$ be a graph in $\mathcal{T}_1$ on $n$ vertices. Then, $\cc(G)\leq n$ and equality holds if and only if $G$ is isomorphic to the graph of icosahedron.
\end{lem}
\begin{proof}
Let $G_0,G_1,G_2$ be as in the definition of $\mathcal{T}_1$ in \pref{sec:structure}, where $|V(G_0)|=12$, $|V(G_1)|=11$ and $|V(G_2)|=10$ ($ G_0 $ is depicted in \pref{fig:icosa} and $ G_1 $ and $ G_2 $ are obtained from $ G_0 $ by deleting $ v_{11} $ and $ v_{10}, v_{11} $, respectively). First, note that the family of cliques 
\begin{align*}
\mathscr{C}_0=&\{\{v_0,v_1,v_9\},\{v_0,v_1,v_3\}, \{v_2,v_3,v_4\}, \{v_3,v_4,v_5\}, \{v_0,v_5,v_7\}, \{v_1,v_2,v_{10}\}, \\ &\{v_8,v_9,v_{10}\}, \{v_7,v_8,v_9\}, \{v_4,v_6,v_{11}\}, \{v_2,v_{10},v_{11}\}, \{v_6,v_8,v_{11}\}, \{v_5,v_6,v_7\}\},
\end{align*}
is a clique covering for $G_0$ and thus, $\cc(G_0)\leq 12$. 
On the other hand, note that $\omega(G_0)=3 $ and for every vertex $ v\in V(G_0) $, $ \overline{G_0} [N_{G_0}(v)] $ is a cycle on five vertices whose chromatic number is equal to three. Thus, by \pref{lem:lowerbound}, $\cc(G_0)\geq (12\times 3)/3=12$ and so, $\cc(G_0)=|V(G_0)|=12$. Also, if $ G\neq G_0 $ is a thickening of $ (G_0,\emptyset) $, then by replacing each clique $ C\in \mathscr{C}_0 $ with the clique $ \cup_{v\in C} X_v $, one may obtain a clique covering for $ G $ of size $ 12\leq n-1 $. 

In order to obtain a clique covering $ \mathscr{C}_1 $ for $ G_1 $, in $ \mathscr{C}_0 $, replace the cliques $\{v_4,v_6,v_{11}\}$, $\{v_2,v_{10},v_{11}\}$, $\{v_6,v_8,v_{11}\}$ and $\{v_5,v_6,v_7\}$ with the cliques $\{v_4,v_5,v_6\}$ and $\{v_6,v_7,v_8\}$. Also, to obtain a clique covering for $G_2$, in $ \mathscr{C}_1 $, replace the cliques $\{v_1,v_2,v_{10}\}$ and $\{v_8,v_9,v_{10}\}$ with the clique $\{v_1,v_2\}$. Thus, $\cc(G_1)\leq 10$ and $\cc(G_2)\leq 9$. 
Let  $F'=\{\{v_1,v_4\},\{v_6,v_9\}\}$ and $F\subseteq F'$. Since $v_1v_4$ and $v_6v_9$ are non-edges of $G_2$, by \pref{lem:thickening}, if $ G $ is a thickening of either $ (G_1,\emptyset) $ or $(G_2,F)$, then $ \cc(G)\leq n-1 $.
\end{proof}

The following lemma  proves the assertion of \pref{thm:main1} for the powers of cycles.

\begin{lem} \label{lem:lemcirc}
Let $p,n$ be two positive integers such that $n\geq 2p+1$. Then $\cc(C_n^p)\leq n$ and equality holds if and only if $n\geq 3p+1$.
\end{lem}
\begin{proof}
Let $V=\{v_0,\ldots,v_{n-1}\}$ be the vertex set of $C_n^p$ and read all subscripts modulo $n$. Define  $S=\{\{v_k,v_{k+p}\}: 0\leq k\leq n-1\}$ as a subset of $E(C_n^p)$ of size $n$. We claim that if $n\geq 3p+1$, then no pair of edges in $S$ can be covered by a single clique of $G$. To see this, on the contrary and w.l.o.g. assume that the edges $\{v_0,v_p\}, \{v_k,v_{k+p}\}$ are covered by a clique, for some $k\in\{1,\ldots, n-1\}$. Since $v_k$ is adjacent to $v_0$, either $1\leq k\leq p$ or $n-p\leq k\leq n-1$. In the former case, since $p+1\leq k+p\leq 2p\leq n-p-1$, $v_0$ is not adjacent to $v_{k+p}$, and in the latter case, since $2p+1\leq n-p\leq  k\leq n-1$, $v_k$ is not adjacent to $v_p$, a contradiction. This proves the claim, and thus for $n\geq 3p+1$,  $\cc(G)\geq n$. Also the family of cliques $\{\{v_{i},v_{i+1},\ldots, v_{i+p}\}: 0\leq i\leq n-1\}$ is a clique covering for $G$ of size $n$. Thus, when $n\geq 3p+1$, we have $\cc(G)=n$. Also, if $n=2p+1$, then $G$ is isomorphic to the complete graph $K_n$ and thus, $\cc(G)=1\leq n-2$. Now, assume that $2p+2\leq n\leq 3p$. We provide a clique covering for $G$ of size $n-1$. Note that the cliques $C_{i}=\{v_{2i},v_{2i+1},\ldots, v_{2i+p}\}$, $i\in\{0,1,\ldots, \lceil n/2\rceil-1\}$, cover all edges of $G$ except the edges in $S'=\{\{v_{2j+1},v_{2j+1+p}\}: 0\leq j \leq \lfloor n/2\rfloor -1\}$. It can be easily seen that $K=\{v_1,v_{p+1},v_{2p+1},v_{3p+1}\}$ is a clique of $G$ covering two distinct edges $\{v_1,v_{p+1}\}$ and $\{v_{2p+1},v_{3p+1}\}$ in $S'$. Hence, $\{C_i: 0\leq i \leq \lceil n/2\rceil-1\}\cup \{K\}\cup S'\setminus \{\{v_1,v_{p+1}\},\{v_{2p+1},v_{3p+1}\}\}$ is a clique covering for $G$ of size $\lceil n/2\rceil+|S'|-1=n-1$.
This proves \pref{lem:lemcirc}.
\end{proof}
Finally, the following lemma provides the appropriate clique covering for the complement of a twister (an antiprismatic graph depicted in \pref{fig:twister}) and its thickening. 
\begin{lem}\label{lem:twister}
If $G$ is the complement of a twister, then $\cc(G)=|V(G)|=10$. Also, if $G'$ is a thickening of $(G,F)$, for a valid set $F$ of changeable pairs of $ G $, and $G'\neq G$, then $\cc(G')\leq |V(G')|-1$.
\end{lem}
\begin{proof}
Let $V(G)=\{u_1,u_2,v_1,\ldots, v_8\}$ as in \pref{fig:icosa}. First, note that $\omega(G)=3$. For every vertex $u\in \{u_1,u_2\}$, $\overline{G}[N_G(u)]$ is a matching on four vertices and thus has the chromatic number equal to two. Also, for every vertex $v\in \{v_1,\ldots, v_8\}$, $\overline{G}[N_G(v)]$ is a cycle on five vertices and so has the chromatic number equal to three. Hence, by \pref{lem:lowerbound}, 
$\cc(G)\geq \lceil{(8\times 3+2\times 2)}/{3}\rceil=10$.
Moreover, the following collection is a clique covering for $G$ and so, $\cc(G)=10$.
\begin{align*}
\mathscr{C}=\{& C_1=\{u_2,v_1,v_7\},   C_2=\{u_2,v_3,v_5\},   C_3=\{u_1,v_4,v_6\},   C_4=\{u_1,v_2,v_8\}, C_5=\{v_1,v_3,v_6\}, \\
& C_6=\{v_2,v_5,v_7\},\   C_7=\{v_3,v_6,v_8\},   C_8=\{v_3,v_5,v_8\},\   C_9=\{v_1,v_4,v_7\},  C_{10}=\{v_2,v_4,v_7\}\}.
\end{align*}
Now, let $G'\neq G$ be a thickening of $(G,F)$. If $F=\emptyset$, then it is clear that $\cc(G')\leq 10\leq |V(G')|-1$. Thus, assume that $F\neq \emptyset$. It is easy to check that the only changeable pairs of $G$ are $v_1v_5$, $v_2v_6$, $v_3v_7$ and $v_4v_8$. Let $ (X_v)_{v\in V(G)}$  be as in the definition of thickening. By symmetry, assume that $\{v_1,v_5\}\in F$ and $|X_{v_1}|\geq 2$. Then, the collection of cliques,
\begin{align*}
&\{\cup_{v\in C_i} X_v: 5\leq i\leq 10\}\cup \mathscr{N}[X_{v_5};X_{u_2}\cup X_{v_1}\cup X_{v_3}]\cup \mathscr{N}[X_{v_6};X_{u_1}\cup X_{v_2} \cup X_{v_4}]\cup\\ & \mathscr{N}[X_{v_7};X_{u_2}\cup X_{v_1}\cup X_{v_3}]\cup \mathscr{N}[X_{v_8};X_{u_1}\cup X_{v_2}\cup X_{v_4}]
\end{align*}
is a clique covering for $G'$ of size $|X_{v_5}|+|X_{v_6}|+|X_{v_7}|+|X_{v_8}|+6\leq |V(G')|-1$, where the last inequality holds since $|X_{v_1}|\geq 2$.
\end{proof}
\section{Outline of the proofs}\label{sec:main}
In this section, we give an outline of our approach towards the proofs of the main theorems of this paper, i.e. Theorems~\ref{thm:main1} and \ref{thm:main2}.
The main tool is \pref{thm:seymour1} which asserts that for every connected claw-free graph $ G $ either $ G $ admits  a non-trivial strip-structure, or $ V(G) $ is the union of three cliques, or $ G $ is in $\mathcal{T}_1\cup \mathcal{T}_2\cup \mathcal{T}_3$. We prove our results for each of these classes separately. First, in the following theorem, we deal with those claw-free graphs which admit a non-trivial strip-structure. The proof of this theorem is given in \pref{sec:strip}.
\begin{thm} \label{thm:stripA}
Let $G$ be a connected claw-free graph on $n$ vertices such that $V(G)$ is not the union of three cliques and $G$ is not in $\mathcal{T}_1\cup \mathcal{T}_2\cup \mathcal{T}_3$. Then, $ \cc(G)\leq n-1$.
\end{thm}
Then, in \pref{sec:fci}, we prove the following theorem regarding the clique covering of the graphs in $ \mathcal{T}_2 $.
\begin{thm}\label{thm:circA}
Let $H$ be a long circular interval graph and $G$ be a connected graph on $n$ vertices which is a thickening of $ (H,F) $, for some $ F $ as in the definition of $\mathcal{T}_2$. Then $\cc(G)\leq n$ and equality holds if and only if $G=H$ and both $ G $ and $ H $ are  isomorphic to the $p^{th}$ power of the cycle $C_n$, for some positive integer $p\leq \lfloor(n-1)/3\rfloor$. 
\end{thm}
The following two theorems, proved in \pref{sec:3cliques}, consider the clique covering of those claw-free graphs whose vertex set is the union of three cliques. 
\begin{thm} \label{thm:tcB}
 Let $G$ be a claw-free graph on $n$ vertices such that  $ V(G) $ is the union of three cliques of $ G $. Then, $\cc(G)\leq n+1$.
\end{thm}
\begin{thm} \label{thm:tcA}
	Let $G$ be a claw-free graph on $n$ vertices which contains at least one triad and $ V(G) $ is the union of three cliques of $ G $. Then $ \cc(G)\leq n $, and equality holds if and only if $ n=3p+3 $ for some positive integer $ p $ and $G$ is isomorphic to the $p^{th}$ power of the cycle $C_n$.
\end{thm}

Dealing with the case of antiprismatic graphs is the most difficult part of the proof. In \cite{seymour1,seymour2}, the study of antiprismatic graphs is divided into two different parts (depending on whether the triads of the graph admit a certain kind of orientation or not) called \textit{orientable and non-orientable} antiprismatic graphs (to see the definitions of these graphs, see \cite{seymour1,seymour2}). In \pref{sec:oap}, we go through the clique covering of orientable antiprismatic graphs and prove the following.
\begin{thm}\label{thm:oapA}
Let $G$ be an orientable antiprismatic graph on $n$ vertices which contains at least one triad. Then $\cc(G)\leq n$ and equality holds if and only if $n=3p+3$, for some positive integer $p$ and $G$ is isomorphic to the $p^{th}$ power of the cycle $C_n$.
\end{thm}
In \pref{sec:noap}, we look into non-orientable antiprismatic graphs and prove the following.
\begin{thm}\label{thm:noapA}
Let $G$ be a non-orientable antiprismatic graph on $n$ vertices. Then $\cc(G)\leq n$ and equality holds if and only if $\overline{G}$ is isomorphic to a twister.
\end{thm}
The two preceding theorems enable us to prove the following which handles the case of graphs in $ \mathcal{T}_3 $.
\begin{thm} \label{thm:antiprismaticA}
Let $G$ be a connected fuzzy antiprismatic graph on $n$ vertices which contains at least one triad. Then $\cc(G)\leq n$ and equality holds if and only if either $n=3p+3$, for some positive integer $p$ and $ G $ is isomorphic to the $p^{th}$ power of the cycle $C_n$, or $ \overline{G} $ is isomorphic to a twister.
\end{thm}
\begin{proof}
Let $H$ be an antiprismatic graph and $F$ be a valid set of changeable pairs of $H$, where $G$ is a thickening of $(H,F)$.
Note that since the pairs in $F$ are non-edges of $H$ and $G$ contains a triad, $H$ contains a triad as well. 
On the other hand, since $H$ contains a triad, by Theorems~\ref{thm:oapA} and \ref{thm:noapA}, either $\cc(H)\leq |V(H)|-1$, or $H$ is isomorphic to the complement of a twister, or the graph $C_{3p+3}^p$, for some positive integer $p$. In the former case, the result follows from \pref{lem:thickening_antiprismatic}. If $H$ is isomorphic to the complement of a twister, then \pref{lem:twister} implies the assertion. Finally, assume that $H$ is isomorphic to  the graph $C_{3p+3}^p$. Let $V(H)=\{v_0,v_1,\ldots, v_{3p+2}\}$, where adjacency is as in the definition (see \pref{sec:intro}). It is evident that $H$ contains exactly $p+1$ disjoint triads and $F$ is a subset of $\{\{v_i,v_{i+p+1}\}: 0\leq i \leq 3p+2\} $ (reading subscripts modulo $3p+3$). If $ G=H $, then by \pref{lem:lemcirc}, we are done. Now, assume that $ G\neq H $ and so $ n>3p+3 $. Let $H'$ be the graph obtained from $H$ by adding edges between all the pairs of vertices in $F$. Thus, $H'$ is a long circular interval graph and $F$ can be considered as a set of pairs of distinct endpoints of some intervals. Therefore, $G$ is a thickening of $ (H',F) $ and $ G\neq H' $. Hence, by \pref{thm:circA}, $ \cc(G)\leq n-1 $.
\end{proof}
With all these theorems in hand, now we can prove \pref{thm:main1}. 
\begin{proof}[{\rm \textbf{Proof of \pref{thm:main1}.}}]
By \pref{thm:seymour1}, either $G$ admits a non-trivial strip-structure, or $V(G)$ is the union of three cliques,  or $G\in\mathcal{T}_1\cup \mathcal{T}_2\cup \mathcal{T}_3$. Hence, \pref{thm:main1} follows immediately from Theorems~\ref{lem:icosa}, \ref{thm:stripA}, \ref{thm:circA}, \ref{thm:tcA} and \ref{thm:antiprismaticA}.
\end{proof}

%
For the purpose of proving \pref{thm:main2}, we need the following lemma from \cite{seymour6}.
\begin{lem}{\rm \cite{seymour6}}\label{lem:tame2}
Let $G$ be a claw-free graph and $X, Y$ be disjoint subsets of $V (G)$ with $X\neq \emptyset$. Assume that
for every two nonadjacent vertices in $Y$, every vertex in $X$ is adjacent to exactly one of them. Then
$Y$ is the union of two cliques.
\end{lem}
\begin{proof}[{\rm \textbf{Proof of \pref{thm:main2}.}}] 
Let $G$ be a connected graph on $n$ vertices which is tame. If $G$ contains a triad, then \pref{thm:main1} implies the desired bound. Also, if $V(G)$ is the union of three cliques, then the result follows from \pref{thm:tcB}. Thus, assume that $G$ is triad-free and $V(G)$ is not the union of three cliques. In this case, we prove that $G$ is obtained from a fuzzy antiprismatic graph containing a triad by deleting a vertex.
Let $ H $ be a connected claw-free graph containing a triad, where $G$ is an induced subgraph of $ H $. One can obtain a sequence of induced subgraphs of $H$, say $G_0,G_1,\ldots, G_s$, such that $G_0=G$, $G_s=H$ and $G_i$ is obtained from $G_{i-1}$ by adding a vertex $v_i$ such that $v_i$ is not anticomplete to $V(G_{i-1})$. Suppose that $G_t$ is the first graph in the sequence which contains a triad. Let $ X,Y $ be the sets of neighbours and non-neighbours of $ v_t $ in $ V(G_{t-1}) $, respectively. \vsp

(1) \textit{There exists two nonadjacent vertices in $ Y $ and for every two nonadjacent vertices $ y,y'\in Y $ and every vertex $ x\in X  $, $ x $ is adjacent to exactly one of $ y $ and $ y' $. Moreover, $X$ is the disjoint union of two cliques $N_H(y,X)$ and $N_H(y',X)$.} \vsp

Since $G_{t-1}$ is triad-free and $G_t$ contains a triad, there exist two nonadjacent vertices in $Y$. 
Now, if $x$ is adjacent to both $ y$ and $y' $, then $\{x,v_t,y,y'\}$ is a claw in $ H $, a contradiction. Also, if $ x $ is nonadjacent to both $ y$ and $y' $, then $ \{x,y,y' \} $ is a triad in $ G_{t-1} $, a contradiction. Therefore, $ x $ is adjacent to exactly one of $ y $ and $ y' $. Hence, since $ G_{t-1} $ is triad-free, it turns out that $X$ is partitioned into two disjoint cliques $N_H(y,X)$ and $N_H(y',X)$. This proves (1). \vsp

By (1) and \pref{lem:tame2}, $Y$ is the union of two cliques $Y_1$ and $Y_2$. Also, let $ \tilde{Y}\subseteq Y $ be the set of all vertices in $ Y $ which have no non-neighbour in $ Y $.  Now, note that $ \overline{G_t}$ induces a bipartite graph on  $Y\setminus \tilde{Y}$ with bipartition $(Y_1\setminus \tilde{Y},Y_2\setminus \tilde{Y})$ and let $G'_1,\ldots, G'_k$ be the connected components of $\overline{G_t}[Y\setminus \tilde{Y}]$. Also, let $A_i= V(G'_i)\cap Y_1$ and $B_i= V(G'_i)\cap Y_2$, $1\leq i\leq k$, where both are nonempty. \vsp

(2) \textit{For $ i=1,\ldots, k $, every vertex in $ X $ is either complete to $ A_i $ and anticomplete to $ B_i $, or vice versa.} \vsp

If $a,a'\in A_i$, then there exist vertices $u_1,\ldots, u_{2l+1}$ in $ V(G_i') $, such that $u_1=a$, $u_{2l+1}=a'$ and $u_j$ is nonadjacent to $u_{j+1}$ in $ G_t $, for every $j\in\{1,\ldots, 2l\}$. Thus, by (1), $N_H(u_1,X)=N_H(u_3,X)=\cdots=N_H(u_{2l+1},X)$. Therefore,  every vertex in $ X $ is either complete or anticomplete to $ A_i $ and the same holds for $ B_i $ similarly. 
Now, since $ A_i $ is not complete to $ B_i $, (2) follows from (1). \vsp

Now, let $G'$ be the graph obtained from $G_t$ by deleting all vertices in $Y\setminus \tilde{Y}$ and adding the new vertices $Y'=\{a_1,b_1,\ldots, a_k, b_k\}$ such that both $ \{a_1,\ldots, a_k\} $ and $ \{b_1,\ldots, b_k\} $ are cliques in $ G' $, $a_i$ is adjacent to $b_j$ in $ G' $ if and only if $i\neq j$, $Y'$ is complete to $\tilde{Y}$ and anticomplete to $ v_t $ and every vertex vertex $x\in X$ is adjacent to the vertex $a_i\in Y'$ (resp. $b_i\in Y'$) if and only if $x$ is complete to $A_i$ (resp. $B_i$). Adjacency of the other vertices is the same as in $G_t$. Let $F$ be the set of all the pairs $\{a_i,b_i\}$ where $A_i$ is not anticomplete to $B_i$ in $ G_t $. It is clear that $G_t$ is a thickening of $(G',F)$. Also, the only triads of $G'$ are $\{v_t,a_1,b_1\}, \ldots, \{v_t,a_k,b_k\}$. Every vertex in  $Y'\cup \tilde{Y}\setminus \{a_i,b_i\}$ is adjacent to both $a_i$ and $b_i$ and nonadjacent to $v_t$ and every vertex in $X$ is adjacent to $v_t$ and, by (2), exactly one of $a_i$ and $b_i$. Hence, $G'$ is an antiprismatic graph (in fact, $ G' $ is a non-2-substantial graph, see \pref{sec:oap} for the definition) and $F$ is a valid set of changeable pairs of $G'$. Therefore, $G_t$ is a fuzzy antiprismatic graph. Let $H'$ be the induced subgraph of $H$ on $V(G)\cup\{v_t\}$. Since, $H'$ is an induced subgraph of $G_t$, $H'$ is also a fuzzy antiprismatic graph on $n+1$ vertices. If $V(G)\cap Y$ is a clique, then by (1), $V(G)$ is the union of three cliques, a contradiction. Hence, $V(G)\cap Y $ contains two nonadjacent vertices and so $H'$ contains a triad. Hence, by \pref{thm:antiprismaticA}, $\cc(H')\leq |V(H')|=n+1$ and equality holds, if and only if  $ H' $ is isomorphic to either the graph $ C_{3p+3}^p $, for some positive integer $ p $, or the complement of a twister. On the other hand, removing the vertex $v_t$ from $H'$ yields the triad-free graph $G$. However, both $ C_{3p+3}^p $ and the complement of a twister contain two disjoint triads. Consequently, $H'$ is  isomorphic to neither $ C_{3p+3}^p $ nor the complement of a twister and thus, $\cc(G)\leq \cc(H')\leq |V(H')|-1=n$. This proves \pref{thm:main2}.
\end{proof}

\begin{rem}
The proof of \pref{thm:main2} is actually explaining the structure of tame graphs which is interesting in its own right. In fact, it is proved there that if $ G $ is connected and tame, then either $ G $ contains a triad, or $ V(G) $ is the union of three cliques, or $ G $ is obtained from a fuzzy antiprismatic graph whose all triads are on a vertex $ v $, by deleting the vertex $ v $. 
\end{rem}

\section{Claw-free graphs with non-trivial strip-structure} \label{sec:strip}
In this section, we are going to prove \pref{thm:stripA}. First, we recall some definition from \cite{seymour5}. 
Given a graph $ G $, a \textit{non-trivial strip-structure} of $ G $ is a pair $ (H,\eta) $, where $ H $ is a graph with no isolated vertex (which is allowed to have loops and parallel edges) with $ |E(H)|\geq 2 $ and a function $ \eta $ mapping each $ f\in E(H) $ to a subset $ \eta(f) $ of $ V(G) $, and mapping each pair $ (f,h) $ where $ f\in E(H) $ and $ h $ is an endpoint of $ f $, to a subset $ \eta(f,h) $ of $ \eta(f) $, satisfying the following conditions.
\begin{itemize}
\item[(S1)] The sets $ \eta(f) (f\in E(H)) $ are nonempty and partition $ V(G) $.
\item[(S2)] For each $ h\in V(H) $, the union of the sets $ \eta(f,h) $ for all $ f\in E(H) $ incident with $ h $ is a clique of  $ G $.
\item[(S3)] For all distinct $ f_1,f_2\in E(H) $, if $ v_1\in \eta(f_1) $ and $ v_2\in \eta(f_2) $ are adjacent, then there exists $ h\in V(H) $ incident with $ f_1,f_2 $, such that $ v_i\in \eta(f_i,h) $, for $ i=1,2 $.
\item[(S4)] For every $ f\in E(H) $, $ h\in V(H) $ incident with $ f $ and $ v\in \eta(f,h) $, $ N_G(v,\eta(f)\setminus \eta(f,h)) $ is a clique of $ G $.
\item[(S5)] For every edge $ f\in E(H) $ with distinct endpoints $h_1$ and $h_2$, either $ \eta(f,h_1)\cap \eta(f,h_2)=\emptyset $ or $ \eta(f,h_1)= \eta(f,h_2)=\eta(f) $.   
\end{itemize}
Also, by a \textit{stripe}, we mean a pair $ (G,Z) $ where $G$ is a claw-free graph such that $ |V(G)|\geq 2 $, and $ Z\subseteq V(G) $ is a stable set of simplicial vertices of $G$ such that every vertex of $G$ has at most one neighbour in $ Z $.
Moreover, for a non-trivial strip-structure $(H,\eta)$ of a graph $G$ and each edge $f\in E(H)$, let us denote the set of all vertices of $H$ incident with $f$ by $\overline{f}$. Also, define $J_f$ to be a graph obtained from $G[\eta(f)]$ by adding a new vertex $z_h$ for every $h\in \overline{f}$, such that $z_h$ is complete to $\eta(f,h)$ and anticomplete to $\eta(f)\setminus \eta(f,h)$. Then, assuming $Z_f=\{z_h: h\in \overline{f}\}$, $(J_f,Z_f)$ is called \textit{the strip of $(H,\eta)$ at $f$}.
The following theorem is a restatement of the structure theorem (\pref{thm:seymour1}) which is proved in \cite{seymour5}.
\begin{thm} {\rm \cite{seymour5}} \label{thm:seymour2}
Let $G$ be a connected claw-free graph. If $V(G)$ is not the union of three cliques and $G$ is not in $\mathcal{T}_1\cup \mathcal{T}_2\cup \mathcal{T}_3$, then $G$ admits a non-trivial strip-structure such that for each strip $(J,Z)$, we have $1\leq |Z|\leq 2$, and either
\begin{itemize}
\item $|V(J)|=3$ and $|Z|=2$, or
\item $(J,Z)$ is a stripe. 
\end{itemize}
\end{thm}

In order to deal with the claw-free graphs with non-trivial strip-structure, first we have to provide appropriate clique coverings for stripes (or more generally, claw-free graphs with simplicial vertices), which is done in the following theorem. For a graph $G$, by $Z(G)$ we denote the set of all simplicial vertices of $G$. Also, a clique covering for $G$ is said to be a \textit{simplicial clique covering}, if it contains the clique $N[z]$, for every $z\in Z(G)$.      

\begin{thm} \label{thm:stripe}
Let $G$ be a connected claw-free graph on $n$ vertices such that $Z(G)\neq \emptyset$. Then, $G$ admits a simplicial clique covering of size at most $ n-|Z(G)|$, unless $G$ is isomorphic to the line graph of a tree on $n+1$ vertices, and in this case,  $G$ admits a simplicial clique covering of size at most $ n-|Z(G)|+1$.
\end{thm}
\begin{proof}
We begin with setting up some notation and convention. Pick a vertex $z_0\in Z(G)$ and due to the connectedness of $G$, let $V_0,\ldots,V_d$, $d\geq 1$ be the partition of $V(G)$, where $V_i$ is the set of the vertices of $G$ at distance $i$ from $z_0$, $0\leq i\leq d$ (thus, $V_0=\{z_0\}$). 
For every $i\in \{1,\ldots,d\}$ and every nonempty set $X\subseteq V_i$, let $N^-(X)$ denote the set of vertices in $V_{i-1}$ with a neighbour in $X$ (which is clearly nonempty). It is clear that if $i\geq 2$, then $N^-(X)\cap Z(G)=\emptyset$. Also, for every $X\subseteq V(G)$, let $\mathcal{K}(X)$ denote the family of connected components of $G[X]$. Now, note that every connected component of $G[Z(G)]$ is a clique, and let $\mathcal{K}(Z(G))=\{Z_0,\ldots,Z_k\}$, where $k\geq 0$ and $z_0\in Z_0$. Also, for every $i\in \{0,\ldots,k\}$, $V(G)\setminus Z(G)$ can be partitioned into the sets $N_i$ and $V(G)\setminus (Z(G)\cup N_i)$, such that $Z_i$ is complete to $N_i$ and anticomplete to $V(G)\setminus (Z(G)\cup N_i)$. In fact, $N[z_0]=Z_0\cup N_0$ and for every $i\in \{1,\ldots,k\}$, there exists a unique $j\in \{2,\ldots ,d\}$ such that $Z_i\subseteq V_j$ and $N_i\setminus N^-(Z_i)\subseteq V_j$.

Finally, for every $i\in \{0,\ldots,d\}$ and every vertex $v\in V_i$ (assuming $V_{-1}=V_{d+1}=\emptyset$), let $N^{\circ}(v)$ be the set of all neighbours of $v$ in $V_i$ which have no common neighbour with $v$ in $V_{i-1}$, and $N^+(v)$ be the set of all neighbours of $v$ in $V_{i+1}$. Also, let $C(v)=N^{\circ}(v)\cup N^+(v)$ and $C[v]=C(v)\cup \{v\}$.
We observe that,\vsp

(1) \textit{The following hold}.
\begin{itemize}
\item[{\rm(i)}] \textit{For every vertex $v\in V(G)$, $C[v]$ is a clique of $G$. Also, $C[z_0]=V_0\cup V_1$ and for every $z\in Z(G)\setminus \{z_0\}$, $C[z]=\{z\}$}.
\item[{\rm(ii)}] \textit{Let $i\in \{1,\ldots, d\}$ and $K\in \mathcal{K}(V_i)$. Then for every vertex $u\in N^-(K)$, $C[u]\subseteq N^-(K)\cup K$.}
\end{itemize}

The second and third assertion of (i) is obvious. To see the first assertion of (i), note that for every $v\in V(G)\setminus \{z_0\}$, if $C(v)$ contains two nonadjacent vertices $v_1$ and $v_2$, then $\{v,v^-,v_1,v_2\}$ is a claw for a vertex $v^-\in N^-(\{v\})$. This proves (i). Also, to see (ii), note that for every vertex $u\in N^-(K)$, since $C[u]$ is a clique, we have $N^{\circ}(u)\subseteq N^-(K)$ and $N^+(u)\subseteq K$, and thus, $C[u]\subseteq N^-(K)\cup K$. This proves (1).\vsp

We continue with a couple of definitions. Let us say that $G$ is \textit{irreducible}, if
\begin{itemize}
\item[(R1)] for every vertex $v\in V(G)\setminus Z(G)$, $N^+(v)\neq \emptyset$, and
\item[(R2)] for every nonempty set $U\subseteq V(G)\setminus Z(G)$, if $ \cup_{u\in U} C[u] $ is a clique of $ G $, then $|U|=1$. 
\end{itemize}
Otherwise, we say that $G$ is \textit{reducible}. For every $i\in \{1,\ldots ,d\}$, let us say that $V_i$ is a \textit{nest-hotel}, if each member of $\mathcal{K}(V_i)$ is a clique, and for every clique $K\in \mathcal{K}(V_i)$, there exists a vertex $x_K\in V_{i-1}$ such that $N^+(x_K)=K$ and $V_{i-1}\setminus \{x_K\}$ is anticomplete to $K$.\vsp

(2) \textit{If $G$ is irreducible, then for every $i\in \{1,\ldots,d\}$, $V_i$ is a nest-hotel, and consequently, $G$ is isomorphic to the line graph of a tree on $n+1$ vertices.}\vsp

The assertion is trivial for $i=1$. First, we prove the assertion for $i=d\geq 2$. Since $G$ is irreducible, by (R1), we have $V_d\subseteq Z(G)$ and thus, $\mathcal{K}(V_d)\subseteq \mathcal{K}(Z(G))$. 
Therefore, by (1)-(ii),  for every clique $Z_i\in \mathcal{K}(V_d)$ and every $u\in N^-(Z_i)$,  $C[u]\subseteq N^-(Z_i)\cup Z_i$.
Also, $N^-(Z_i)\cup Z_i$ is a clique and $N^-(Z_i)\subseteq V(G)\setminus Z(G)$. Hence, since $G$ is irreducible, by (R2), $|N^-(Z_i)|=1$. Also, by (1)-(i), for every two distinct cliques $Z_i,Z_j\in \mathcal{K}(V_d)$, we have $N^-(Z_i)\cap N^-(Z_j)=\emptyset$. Hence, defining $x_{Z_i}$ as the single member of $N^-(Z_i)$, for every $Z_i\in \mathcal{K}(V_d)$, it turns out that $V_d$ is a nest-hotel.

Now, contrary to (2), let $i_0\geq 2$ be the maximal $i$ such that $V_{i}$ is not a nest-hotel. 
By the above argument, we have $2\leq i_0\leq d-1$. 
First, assume that $ V_{i_0} $ contains vertices $ u,v,w $ such that $ v $ is adjacent to $ u,w $ and $ u,w $ are nonadjacent. Then, $v\in V(G)\setminus Z(G)$ and by (R1), we have $N^+(v)\neq \emptyset$. Since $V_{i_0+1}$ is a nest-hotel, 
$ N^+(u),  N^+(v)  $ and $ N^+(w) $ are disjoint and thus, for every vertex $v^+\in N^+(v)$, $\{v,v^+,u,w\}$ is a claw, a contradiction. Therefore, $G[V_{i_0}]$ contains no induced path of length two, and thus each of its connected components is a clique.
Now, by (1)-(ii), for every $K\in \mathcal{K}(V_{i_0})$ and every $u\in N^-(K)$, $N^+(u)\subseteq K$. We claim that $N^+(u)=K$. For if there exists $v\in K\setminus N^+(u)$, then for every $u^+\in N^+(u)$, we have $u^+\in V(G)\setminus Z(G)$, and thus by (R1), $N^+(u^+)\neq \emptyset$. Consequently, since $V_{i_0+1}$ is a nest-hotel, for every vertex $u^{++}\in N^+(u^+)$, $\{u^+,u^{++},u,v\}$ is a claw, a contradiction. This proves that $N^+(u)=K$, i.e. $N^-(K)$ is complete to $K$. Now, if for some $K\in \mathcal{K}(V_{i_0})$, $N^-(K)$ contains two nonadjacent vertices $u_1,u_2$, then since $N^-(K)$ is complete to $K$, we have $K\cap Z(G)=\emptyset$. Thus by (R1), for every vertex $v\in K$, $N^+(v)\neq \emptyset$, and so for every $v^+\in N^+(v)$, $\{v,v^+,u_1,u_2\}$ is a claw, a contradiction. Therefore, $N^-(K)$ and so $N^-(K)\cup K$ is a clique. Now, by (1)-(ii), for every $u\in N^-(K)$, we have $C[u]\subseteq N^-(K)\cup K$, and since $N^-(K)\subseteq V(G)\setminus Z(G)$, by (R2), we have $|N^-(K)|=1$.  Also, by (1)-(i), for every two distinct cliques $K,K'\in \mathcal{K}(V_{i_0})$, we have $N^-(K)\cap N^-(K')=\emptyset$. Now, defining $x_{K}$ as the single member of $N^-(K)$, it turns out that $V_{i_0}$ is a nest-hotel, a contradiction. This proves (2).\vsp

Now, we claim that $\{C[v]: v\in V(G)\}$ is a clique covering for $G$ of size $n$. To see this, note that for every edge $uv\in E(G)$, either $u\in V_i$ and $ v\in V_{i-1} $, or $ u,v\in V_i$, for some $i\in \{1,\ldots,d\}$. In the former case, $uv$ is covered by $C[v]$. In the latter case, if $u$ and $v$ have no common neighbour in $ V_{i-1}$, then $uv$ is covered by both $C[u]$ and $C[v]$. Also, if $u$ and $v$ have a common neighbour in $ V_{i-1} $, say $ w $, then $uv$ is covered by $C[w]$. This proves the claim. Moreover, by (1)-(i), $C[z]=\{z\}$, for every $z\in Z(G)\setminus \{z_0\}$. Also, for every $i\in \{1,\ldots,k\}$ and every $u\in N^-(Z_i)$, since $Z_i\subseteq C[u]$, $C[u]\subseteq N_i\cup Z_i$. Now, let $V'=(Z(G)\setminus \{z_0\}) \bigcup (N^-(Z_i):i\in \{1,\ldots,k\})$ and $V''=V(G)\setminus V'$. Therefore, the family of cliques $\mathscr{C}=\{C[v]: v\in V''\}\cup \{N_i\cup Z_i: i\in \{1,\ldots,k\}\}$ is a clique covering for $G$ of size $n-|Z(G)|-\sum_{i=1}^k|N^-(Z_i)|+k+1\leq n-|Z(G)|+1$. Also, since $N_G[z]=N_i\cup Z_i$, for every $z\in Z_i$, $i\in \{1,\ldots,k\}$, $\mathscr{C}$ is a simplicial clique covering for $G$.
In addition, for every vertex $v\in V(G)$, there exists a clique in $\mathscr{C}$ containing $C[v]$. 

Now, suppose that $G$ is not isomorphic to the line graph of a tree on $n+1$ vertices. Then, by (2),  $G$ is reducible. Contrary to (R1), assume that $N^+(v)=\emptyset$, for some vertex $v\in V(G)\setminus Z(G)$, say $v\in V_i$ for some $i\in \{1,\ldots,d\}$. Then evidently $v\in V''$ (i.e. $C[v]\in \mathscr{C}$) and $C(v)=N^{\circ}(v)$. We claim that $\mathscr{C}\setminus \{C[v]\}$ is a simplicial clique covering of size at most $n-|Z(G)|$. To see this, let $xy$ be an edge with $x,y\in C[v] $, where $y\in N^{\circ}(v)$. Now, if $x\in N^{\circ}(y)$, then $xy$ is covered by a clique in  $\mathscr{C}\setminus \{C[v]\}$  containing $C[y]$. Otherwise, $x$ and $y$ has a common neighbour $w$ in $V_{i-1}$, and $xy$ is covered  by a clique in  $\mathscr{C}\setminus \{C[v]\}$  containing $C[w]$. This proves the claim.

Next, assume that (R1) holds and contrary to (R2), assume that there exist distinct vertices $u,v\in V(G)\setminus Z(G)$ such that  $K=C[u]\cup C[v]$ is a clique of $ G $. If $u,v\in V''$, then $(\mathscr{C}\cup \{K\})\setminus \{C[u],C[v]\}$ is a simplicial clique covering for $G$ of size at most $n-|Z(G)|$. Otherwise, if say $u\in N^-(Z_i)$, where $Z_i\subseteq V_j$, $i\in \{1,\ldots ,k\}, j\in \{1,\ldots ,d\}$, then $Z_i\subseteq C[u]\subseteq K$, and since $C[v]\subseteq K$, it turns out that $v\in N_i$. Now, if $v\notin N^-(Z_i)$, then since $N^+(v)\neq \emptyset$ (due to (R1)), we have $K\cap V_{j+1}\neq \emptyset$. This contradicts the fact that $K$ is a clique containing $u\in V_{j-1}$. Thus, $v\in N^-(Z_i)$, and so $|N^-(Z_i)|\geq 2$, which implies that $|\mathscr{C}|\leq n-|Z(G)|$. This completes the proof of \pref{thm:stripe}.
\end{proof}

Orlin  in \cite{orlin} (see also \cite{mcguinness}) proved that for every graph $H\neq K_3$, the clique cover number of the line graph of $H$ is equal to the number of vertices of $H$ of degree at least two minus the number of its \textit{wings} (triangles with two vertices of degree two).  If $G$ is the line graph of a tree $T$ on $n+1$ vertices, then $Z(G)$ is equal to the set of pendant edges of $T$ and thus, $\cc(G)=n+1-|Z(G)|$. This, together with \pref{thm:stripe}, implies the following which is not needed for us, however, it seems to be of interest in its own right.
\begin{cor}
Let $G$ be a connected claw-free graph on $n$ vertices such that $Z(G)\neq \emptyset$. Then $\cc(G)\leq n-|Z(G)|+1$ and equality holds if and only if $G$ is isomorphic to the line graph of a tree on $n+1$ vertices.
\end{cor}

Now, we are ready to prove \pref{thm:stripA} which we restate here.
\begin{thm} \label{thm:strip}
Let $ G $ be a connected claw-free graph on $n$ vertices, such that $V(G)$ is not the union of three cliques and $G$ is not in $\mathcal{T}_1\cup \mathcal{T}_2\cup \mathcal{T}_3$. Then, $ \cc(G)\leq n-1$.
\end{thm}
\begin{proof}
By \pref{thm:seymour2}, $G$ admits  a non-trivial strip-structure such that for each strip $(J,Z)$, we have $1\leq |Z|\leq 2$, and either $|V(J)|=3$ and $|Z|=2$, or $(J,Z)$ is a stripe. Let $(H,\eta)$ be such a strip-structure for $G$ with $|V(H)|=\nu$ and $ |E(H)|=\epsilon\geq 2$. Also, let $\lambda\geq 0$ be the number of loops in $E(H)$. For every $h\in V(H)$, define $\eta(h)$ to be the union of the sets $ \eta(f,h) $ for all $ f\in E(H) $ incident with $ h $ (i.e. $ h\in \overline{f} $).  By (S2), $\eta(h)$ is a clique of $ G $. Also, for every $f\in E(H)$, let $(J_f,Z_f)$ be the strip of $(H,\eta)$ at $f$, where  $1\leq |Z_f|\leq 2$ and define $n_f=|\eta(f)|= |V(J_f)|-|Z_f|$.   

Note that since $H$ has no isolated vertex, the connectedness of $G$ along with (S3) implies that $H$ is also connected and for every $f\in E(H)$, the set $ Z_f $ intersects each connected component of $J_f$. In particular, if $|Z_f|=1$, then $J_f$ is connected and if $|Z_f|=2$, then either $J_f$ is connected or $J_f$ has exactly two connected components, each of which containing exactly one member of $Z_f$. 
With this observation, the following statement is directly deduced from \pref{thm:stripe}. \vsp

(1) \textit{For every $f\in E(H)$, $J_f$ admits a simplicial clique covering $\mathscr{C}_f$ such that $|\mathscr{C}_f|=n_f+1$, when $J_f$ is a  path on $n_f+2$ vertices and $|Z_f|=2$, and $|\mathscr{C}_f|\leq n_f$, otherwise. In particular, if $f$ is a loop, then $|\mathscr{C}_f|\leq n_f$.} \vsp

Henceforth, for each $f\in E(H)$, let $\mathscr{C}_f$ be as in (1). Now, remove all vertices in $Z_f$ from the cliques in $\mathscr{C}_f$ to obtain a clique covering $\mathscr{C}'_f$ for $G[\eta(f)]$. Thus, since $\mathscr{C}_f$ is a simplicial clique covering, we have $\eta(f,h)\in \mathscr{C}'_f$, for every $h\in \overline{f}$. Now, consider the collection $\bigcup_{f\in E(H)} \mathscr{C}'_f$ of cliques of $G$, remove the clique $\eta(f,h)$ for every $f\in E(H)$ and every $h\in \overline{f}$ and add the clique $\eta(h)$ for every $h\in V(H)$. Note that the resulting family, called $\mathscr{C}$, is a clique covering for $G$, because by (S3), for every distinct edges $f_1,f_2\in E(H)$, the edges in $E_G(\eta(f_1),\eta(f_2))$ are covered by the cliques $\eta(h), h\in \overline{f}_1\cap \overline{f}_2$. Hence, \vsp

(2) We have,
\[ 
\cc(G)\leq |\mathscr{C}|\leq \nu +\sum_{f\in E(H)} (|\mathscr{C}_f|-|\overline{f}|)=\nu+\lambda-2\epsilon+\sum_{f\in E(H)} |\mathscr{C}_f| \leq \nu-\epsilon+\sum_{f\in E(H)}n_f,
\]
where the last inequality is by (1) (applying $|\mathscr{C}_f|\leq n_f$, if $f$ is a loop, and $|\mathscr{C}_f|\leq n_f+1$, otherwise). \vsp

If $\epsilon \geq \nu+1$, then we are done. On the other hand, since $H$ is connected, we have $\epsilon \geq \nu+\lambda-1$. Hence, it remains to discuss two possibilities that either $\epsilon= \nu-1$ and $\lambda=0$ or $\epsilon=\nu $ and $\lambda\leq 1$.
First, consider the former case. In this case, $H$ is a tree and the right hand side of (2) is equal to $n+1$.
Let $h$ be a \textit{leaf} (a vertex of degree one in $V(H)$) and $f\in E(H)$ be the edge incident with $ h $. If $J_{f}$ is not a path, then by (1), $|\mathscr{C}_{f}|\leq n_{f}$. Also, if $J_{f}$ is a path, then $|\eta(h)|=|\eta(f,h)|=1$, and thus, we can remove the singleton clique  $\eta(h)$ from $\mathscr{C}$. Hence, for each leaf in $V(H)$, we can subtract the right hand side of (2) by one and since $H$ has at least two leaves, we are done. 

Finally, consider the latter case $\epsilon=\nu$ and $\lambda\leq 1$. In this case, the right hand side of (2) is equal to $n$. If $H$ has a leaf, then by the above argument, we can subtract the right hand side of (2) by one and $\cc(G)\leq n-1$. Now, assume that $H$ has no leaf. If $\lambda=1$, then the graph obtained from $H$ by removing its loop is a tree and thus $H$ has a leaf. Therefore, $\lambda=0$ and $H$ is a unicyclic graph with no leaf and no loop, i.e. a cycle (we consider two parallel edges as a cycle).  If for some $f\in E(H)$, either $J_f$ is not a path on $n_f+2$ vertices or $|Z_f|=1$, then $|\mathscr{C}_f|\leq n_f$ and by (2), $\cc(G)\leq n-1$. Therefore, we can assume that for every $f\in E(H)$, $J_f$ is a path and $Z_f$ is the set of its both endpoints. Consequently, $G$ is obtained from $H$ by replacing each edge $f\in E(H)$ by a path on $n_f$ vertices. Hence, $G$ is a cycle. But then $G\in \mathcal{T}_2$, a contradiction. This proves \pref{thm:strip}. 
\end{proof}
\section{Fuzzy long circular interval graphs} \label{sec:fci}
In this section, we are going to prove \pref{thm:circA}. Note that in a long circular interval graph, the vertices included in each interval form a clique and these cliques easily comprise a clique covering, although there are some difficulties in the case of thickening, where \pref{lem:thickening} cannot be applied. We restate \pref{thm:circA} as follows.

\begin{thm} \label{thm:circ}
Let $H$ be a long circular interval graph and $G$ be a connected graph on $n$ vertices which is a thickening of $ (H,F) $, for some $ F $ as in the definition of $\mathcal{T}_2$. Then $\cc(G)\leq n$ and equality holds if and only if $G=H$ and both $ G $ and $ H $ are isomorphic to the $p^{th}$ power of the cycle $C_n$, for some positive integer $p\leq \lfloor(n-1)/3\rfloor$. 
\end{thm}
\begin{proof}
Assume that $|V(H)|=m$ and $\Sigma$ and $\mathcal{I}=\{I_1,\ldots, I_k\}$ are as in the definition. Throughout the proof, read all subscripts modulo $k$. For simplicity, we may assume that the number of intervals $k$ is minimal, in the sense that there is no long circular interval graph isomorphic to $H$ whose number of intervals is less than $k$. For every two distinct points $x,y\in \Sigma$, let $[x,y]$ denote the arc in $\Sigma$ joining $x$ to $y$ clockwise. Also, let $[x,y)$ be obtained from $[x,y]$ by removing $y$, and $(x,y]$ and $(x,y)$ are defined similarly. Considering a clockwise orientation for $\Sigma$, for every $i\in \{1,\ldots,k\}$, we denote the opening and closing points of $I_i$ by $o_i,c_i\in \Sigma$, respectively, such that $I_i=[o_i,c_i]$. Furthermore, assume that $o_1,o_2,\ldots, o_k$  are placed on $\Sigma$ exactly in this order (i.e. $(o_i,o_{i+1})\cap \{o_1,o_2,\ldots, o_k\}=\emptyset$, for every $i\in \{1,\ldots, k\}$). First, note that no interval in $\mathcal{I}$ contains another interval in $\mathcal{I}$ (otherwise one might reduce the number of intervals $k$ by deleting the smaller interval without changing adjacency). Thus, the closing points $c_1,c_2,\ldots, c_k$ are placed on $\Sigma$ exactly in this order, and for each $i\in\{1,\ldots, k\}$, $c_i\in (o_i,c_{i+1})$. We need the following.\vsp

(1) \textit{If for each $i\in \{1,\ldots,k\}$, $|[o_i,o_{i+1})\cap V(H)|=1$, then $H$ is isomorphic to the $p^{th}$ power of the cycle $C_m$, for some positive integer $p$}.\vsp

For every $i\in \{1,\ldots,k\}$, assume that $|I_i\cap V(H)|=t_i$. We can observe that there exists a positive integer $t$ such that $t_i=t$ for all $i\in \{1,\ldots,k\}$, and so $H$ is isomorphic to $C_m^{t-1}$. For the contrary assume that there exists $i\in \{1,\ldots,k\}$ such that $t_i>t_{i+1}$. This implies that $I_{i+1}\cap V(H) \subseteq I_i$ and we can remove $I_{i+1}$ from $\mathcal{I}$ without changing adjacency, a contradiction with the minimality of $k$. This proves (1).\vsp

Note that w.l.o.g. we may assume that for every $i\in\{1,\ldots, k\}$, the set $[o_i,o_{i+1})\cap V(H)$ is nonempty (otherwise one could delete the interval $I_i$ without changing adjacency). Hence, \vsp

(2) We have,
\[
k\leq \sum _{i=1}^k |[o_i,o_{i+1})\cap V(H)|=m.
\]
Now, if there exists some $i\in\{1,\ldots, k\}$, for which $|[o_i,o_{i+1})\cap V(H)|\geq 2$, then by (2), we have $k\leq m-1$, and thus the collection of cliques  $\{I_i\cap V(H): 1\leq i\leq k\}$ is a clique covering for $H$ of size at most $m-1$. Also, if for each $i\in \{1,\ldots,k\}$, $|[o_i,o_{i+1})\cap V(H)|=1$, then (1) implies that $H$ is isomorphic to $C^p_m$, for some positive integer $p$. This argument together with \pref{lem:lemcirc}, implies that\vsp

(3) \textit{We have $\cc(H)\leq m$ and equality holds if and only if $H$ is isomorphic to $C^p_m$, for $p\leq \lfloor(m-1)/3\rfloor$}.\vsp

Now, let $G$ be a connected graph on $n$ vertices which is a thickening of $ (H,F) $, for some $F\subseteq F'$, where $F'$ is the set of pairs $\{u,v\}$ such that $u,v\in V(H)$ are distinct endpoints of $I_i$ for some $i\in \{1,\ldots,k\}$.  Also, let $(X_v)_{v\in V(H)}$ be the subsets of $V(G)$ as in the definition of thickening. If $G=H$, then by (3) we are done. Thus, suppose that $G\neq H$, which implies that $m\leq n-1$. For each interval $I_i\in \mathcal{I}$, let ${C_i}=\cup_{v\in I_i\cap V(H)} X_v$. If $F=\emptyset$, then trivially $\{C_1,\ldots,C_k\}$ is a clique covering for $G$ of size $k$, and thus by (2), $\cc(G)\leq  n-1$. Now, assume that $F\neq \emptyset$, say $F=\{\{u_1,v_1\},\ldots, \{u_{\ell},v_{\ell}\}\}$, where for each $i\in\{1,\ldots, \ell\}$, $u_i,v_i\in V(H)$ are the distinct endpoints of the interval $I_{p_i}$, $p_i\in \{1,\ldots ,k\}$. Note that w.l.o.g. we may assume that $ |X_{v_1}|\geq 2 $. Also, since $G$ is connected, for each $i\in \{1,\ldots, \ell\}$, there exists $q_i\in \{1,\ldots ,k\}$ such that either $u_i$ or $v_i$ belongs to the interior of $I_{q_i}$, and by symmetry, we assume that for each $i$, $u_i$ is in the interior of $I_{q_i}$. Also, for every $i\in\{1,\ldots, \ell\}$, let $Y_i=C_{p_i}\setminus X_{u_i}$.

In the sequel, we prove that $G$ admits a clique covering of size at most $n-1$ which completes the proof. First, suppose that there exists some $i\in\{1,\ldots, k\}$, for which $|[o_i,o_{i+1})\cap V(H)|\geq 2$. Thus, by (2), $k\leq m-1$. 
For every $ i\in\{1,\ldots,\ell\} $, let $\mathscr{C}_i=\{Y_i\}$, when $|X_{v_i}|\geq 2$,  and let $\mathscr{C}_i=\emptyset$, when $ |X_{v_i}|=1$. Thus, $|\mathscr{C}_i\cup \mathscr{N}[X_{u_i};Y_i]|\leq |X_{u_i}|+ |X_{v_i}|-1$. Now, the family of cliques $(\bigcup_{i=1}^\ell (\mathscr{C}_{i}\cup \mathscr{N}[X_{u_i};Y_i])) \cup \{{C_i}: i\in \{1,\ldots,k\}\setminus \{j_1,\ldots, j_\ell\}\}$ is a clique covering for $G$ of size at most $\sum_{i=1}^\ell (|X_{u_i}|+|X_{v_i}|-1)+k-\ell\leq \sum_{i=1}^\ell (|X_{u_i}|+|X_{v_i}|)+m-1-2\ell \leq n-1$, as required (note that the edges in $E(X_{u_i})$ are covered since $u_i$ is in the interior of $I_{q_i}$). 

Next, assume that $|[o_i,o_{i+1})\cap V(H)|=1$ for all $i\in\{1,\ldots, k\}$ (and thus by (1), $H$ is isomorphic to the $p^{th}$ power of the cycle $C_m$, for some positive integer $p$). Note that w.l.o.g. we may assume $ u_i=o_{p_i} $ and $v_i=c_{p_i}$. For each $i\in\{1,\ldots, \ell\}$, if $p_i+1\not\in \{p_1,\ldots, p_\ell\}$, then the edges in $E_G(Y_i)$ are covered by the clique ${C}_{p_i+1}$. Otherwise, if $p_i+1=p_j$, for some $j\in\{1,\ldots, \ell\}$, then these edges are covered by the cliques in $\mathscr{N}[X_{u_i};Y_i]\cup \mathscr{N}[X_{u_j};Y_j]$. Hence, $(\bigcup_{i=1}^\ell \mathscr{N}[X_{u_i};Y_i]) \cup \{C_i: i\in \{1,\ldots,k\}\setminus \{p_1,\ldots, p_\ell\}\}$ is a clique covering for $H$ of size $\sum_{i=1}^\ell |X_{u_i}|+k-\ell\leq \sum_{i=1}^\ell (|X_{u_i}|+|X_{v_i}|)+m-2\ell-1 \leq n-1$, where the first inequality is due to (2) and the fact that $|X_{v_1}|\geq 2$ (note that for each $i\in \{1,\ldots, \ell\}$, the edges in $E(X_{u_i})$ are covered, since $u_i$ is included in the interior of $I_{q_i}$).
\end{proof}%
\section{Three-cliqued claw-free graphs} \label{sec:3cliques}
The goal of this section is to prove Theorems~\ref{thm:tcB} and \ref{thm:tcA}. A three-cliqued graph $(G,A,B,C)$ consists of a graph $G$ and three cliques $A,B,C$ of $G$, pairwise disjoint and with the union $V(G)$. 
 If $G$ is also claw-free, then $(G,A,B,C)$ is called a three-cliqued claw-free graph.
First, as a warm-up, we prove \pref{thm:tcB}, restated as follows. Then, using more subtle structure of these graphs from \cite{seymour5}, we improve the upper bound $n+1$ to $n$ for the three-cliqued claw-free graphs containing a triad. 
\begin{thm}\label{thm:tc1}
Let $ (G,A,B,C) $ be a three-cliqued claw-free graph on $ n $ vertices. Then $ \cc(G)\leq n+1$.
\end{thm}
\begin{proof}
Let $C_1,C_2,C_3$ be three sets of vertices such that $\{A,B,C\}=\{C_1,C_2,C_3\}$ and $ |C_1|\leq |C_2|\leq |C_3| $. For each vertex $ x\in V(G) $ and $ i\in\{1,2, 3\} $, let $ N_i(x)=N(x,C_i) $ and $N_i[x]= N[x,C_i]$. Also, for every $ i,j \in \{1,2, 3\} $, define $\mathscr{N}^i_j=\mathscr{N}[C_i;C_j]$. Evidently, $\mathscr{N}^1_2\cup \mathscr{N}^1_3\cup \mathscr{N}^2_3\cup \{C_1,C_2,C_3\}$ is a clique covering for $G$ of size $2|C_1|+|C_2|+3=n+|C_1|-|C_3|+3$. Thus, whenever $|C_3|-|C_1|\geq 2$, we are done. Now, assume that $|C_3|\leq |C_1|+1$. We consider two cases.
 
First, assume that there exists some $ i\in\{1,2,3\}$, for which some edges in $E(C_i)$ are not covered by the cliques in $\cup_{j\in\{1,2,3\}\setminus \{i\}} \mathscr{N}^j_i$. Let $i_0$ be the smallest such $i$ and also let $\{i_0,j_0,k_0\}=\{1,2,3\}$, where $j_0<k_0$. Therefore, there exist two vertices $x,y \in C_{i_0}$ such that for every $j\in\{1,2,3\}\setminus \{i_0\}$, $N_j(x)$ and $N_j(y)$ are disjoint. Since $G$ is claw-free, $N_{j_0}(x)$ is complete to $N_{k_0}(x)$ and $N_{j_0}(y)$ is complete to $N_{k_0}(y)$. Now, in the clique covering  $\mathscr{N}^{i_0}_{j_0}\cup \mathscr{N}^{i_0}_{k_0}\cup \mathscr{N}^{j_0}_{k_0}\cup\{C_1,C_2,C_3\}$, merge the pairs $(N_{j_0}[x],N_{k_0}[x])$ and $(N_{j_0}[y],N_{k_0}[y])$ to obtain the clique covering $\mathscr{C}$. If $ i_0\leq 2 $, then $|\mathscr{C}|\leq 2|C_{i_0}|+|C_{j_0}|-2+3\leq |C_{i_0}|+|C_{j_0}|+|C_{k_0}|+1= n+1$. Also, if $i_0=3$, then $j_0=1, k_0=2$ and, by minimality of $i_0$, the cliques in $\mathscr{N}^1_2\cup \mathscr{N}^3_2$ cover all edges in $E(C_2)$. Thus, removing the clique $C_2$ from $\mathscr{C}$ yields a clique covering for $G$ of size $2|C_{3}|+|C_{1}|-2+2\leq |C_3|+|C_1|+1+|C_2|= n+1$.

Now, assume that for each  $i\in\{1,2,3\}$,  all edges in $E(C_i)$ are covered by the cliques in $\cup_{j\in\{1,2,3\}\setminus \{i\}} \mathscr{N}^j_i$. If all edges in $E(C_2)$ are covered by the cliques in $\mathscr{N}^3_2$ and all edges in $E(C_1)$ are covered by the cliques in $\mathscr{N}^2_1$, then the family $\mathscr{N}^1_3\cup \mathscr{N}^2_1\cup \mathscr{N}^3_2\cup \{C_3\}$ is a clique covering of size $n+1$. Hence, suppose that either some edges in $E(C_2)$ are not covered by the cliques in $\mathscr{N}^3_2$, or some edges in $E(C_1)$ are not covered by the cliques in $\mathscr{N}^2_1$, and w.l.o.g. assume that the former occurs. Thus, there exist two vertices $x,y\in C_2$ for which $N_3(x)\cap N_3(y)=\emptyset$. Since $G$ is claw-free, $N_3(x)$ is complete to $N_1(x)\setminus N_1(y)$ and $N_3(y)$ is complete to $N_1(y) \setminus N_1(x)$. Now, in the clique covering $\mathscr{N}^1_3\cup\mathscr{N}^2_1\cup\mathscr{N}^2_3\cup\{C_1,C_2\}$, replace the cliques $N_1[x]$, $N_3[x]$, $N_1[y]$ and $N_3[y]$ with the cliques $(N_1(x)\cap N_1(y))\cup \{x,y\}$, $(N_1[x]\cup N_3[x])\setminus N_1(y)$ and $(N_1[y]\cup N_3[y])\setminus N_1(x)$, thereby obtaining a clique covering for $G$ of size $|C_1|+2|C_2|+2-1\leq |C_1|+|C_2|+|C_3|+1=n+1$. This proves \pref{thm:tc1}.
 \end{proof}
 In order to improve the bound in \pref{thm:tc1} for three-cliqued claw-free graphs containing a triad, we need to know their more detailed structure. It is proved in \cite{seymour5} that every three-cliqued claw-free graph can be obtained from some special graphs using a certain construction called ``worn hex-chain''. Let us recall the definition of this construction from \cite{seymour5}.
 
Let $k\geq 1$, and for every $1 \leq i \leq k$, let $(G_i,A_i,B_i,C_i )$ be a three-cliqued graph, where
$V (G_1),\ldots, V (G_k)$ are all nonempty and pairwise disjoint. Let $A = \cup_{i=1}^kA_i$, $B = \cup_{i=1}^kB_i$ and $C = \cup_{i=1}^kC_i$, and let $G$ be the graph with vertex set
$V(G)=\cup_{i=1}^kV(G_i)$ and with adjacency as follows,
\begin{itemize}
\item[(W1)]  for every $1 \leq  i \leq  k$, $G[V (G_i )] = G_i$,
\item[(W2)] for every $1 \leq i < j \leq k$, $A_i$ is complete to $V (G_j ) \setminus B_j$; $B_i$ is complete to $V (G_j ) \setminus C_j$;  and $C_i$ is complete to $V (G_j ) \setminus A_j$, and
\item[(W3)] for every $1 \leq i < j \leq k$, if $u \in A_i$ and $v \in B_j$ are adjacent, then $u, v$ are both in no triads; and the same holds if $u \in B_i$ and $v \in C_j$, and if $u \in C_i$ and $v \in A_j$.
\end{itemize}
Thus, $A,B,C$ are cliques of $G$, and so $(G,A,B,C)$ is a three-cliqued graph. We call the sequence $(G_i,A_i,B_i,C_i) (i = 1,\ldots , k)$ a \textit{worn hex-chain} for $(G,A,B,C)$. Note that every triad of $G$ is a triad of one of $G_1, \ldots , G_k$, and if each term $G_i$ is claw-free, then so is $G$.
  
Let $(G,A,B,C)$ be a three-cliqued graph and $F$ be a valid set for $ G $ such that for every pair $\{u,v\}\in  F$, $ u,v $ are not both in the same set $ A $, $ B $ or $ C $. Then, every thickening $G'$ of $(G,F)$ is also a three-cliqued graph $(G',A',B',C')$, where $A'=\cup_{v\in A} X_v$ and $B',C'$ are defined similarly.
If $(G,A,B,C)$ is a three-cliqued graph and $\{\hat{A},\hat{B},\hat{C}\}=\{A,B,C\}$, then $(G,\hat{A},\hat{B},\hat{C})$
is also a three-cliqued graph, called a \textit{permutation} of $(G,A,B,C)$. 
The following theorem from \cite{seymour5} states that every three-cliqued claw-free graph admits a worn hex-chain whose all terms are obtained from graphs in five basic classes $\mathcal{TC}_1,\ldots,\mathcal{TC}_5$ (defined in \pref{app:3clique}).
\begin{thm} {\rm \cite{seymour5}}\label{thm:tcseymour}
Every three-cliqued claw-free graph admits a worn hex-chain into terms each of which is a thickening of a permutation of a member of one of $\mathcal{TC}_1,\ldots,\mathcal{TC}_5$ with respect to their corresponding sets $F$ $ ( $see \pref{app:3clique} for the definition of $ \mathcal{TC}_1,\ldots,\mathcal{TC}_5) $.
\end{thm}
We will apply \pref{thm:tcseymour} to prove \pref{thm:tcA}. First, we need a couple of lemmas. We begin with a simple observation showing that one can change the ordering in a worn hex-chain.
\begin{lem}\label{lem:worn_ordering}
Assume that $G$ admits a worn hex-chain  $(G_i,A_i,B_i,C_i) (i = 1,\ldots , k)$ and let $i_0,j_0\in \{1, \ldots, k\}$. Then,
\begin{itemize}
\item[\rm (i)] $G$ admits a worn hex-chain   $(G_i',A_i',B_i',C_i') (i = 1,\ldots , k)$ such that for every $i\in\{1,\ldots,k\}$, $(G_i',A_i',B_i',C_i')$ is a permutation of $(G_{i+i_0-j_0},A_{i+i_0-j_0},B_{i+i_0-j_0},C_{i+i_0-j_0})$ $($reading subscripts modulo $k)$. Also, $A_{j_0}'={A}_{i_0}$, $B_{j_0}'={B}_{i_0}$ and  $C_{j_0}'={C}_{i_0}$.
\item[\rm (ii)] if $(G_{i_0},\hat{A}_{i_0},\hat{B}_{i_0},\hat{C}_{i_0})$ is a permutation of $(G_{i_0},A_{i_0},B_{i_0},C_{i_0})$, then $G$ admits a worn hex-chain   $(G_i',A_i',B_i',C_i') (i = 1,\ldots , k)$ such that $G_{i_0}'=G_{i_0}$, $A_{i_0}'=\hat{A}_{i_0}$, $B_{i_0}'=\hat{B}_{i_0}$ and  ${C}_{i_0}'=\hat{C}_{i_0}$. 
\end{itemize} 
\end{lem}
\begin{proof}
(i) Let $k_0=i_0-j_0$ and for each $i\in\{1,\ldots, k\}$, if $1\leq i+k_0\leq k$, then  define $G_i'=G_{i+k_0}$, $A_i'=A_{i+k_0}$, $B_i'=B_{i+k_0}$ and $C_i'=C_{i+k_0}$, if  $i+k_0\leq 0$, then define $G_i'=G_{i+k_0+k}$, $A_i'=C_{i+k_0+k}$, $B_i'=A_{i+k_0+k}$ and $C_i'=B_{i+k_0+k}$ and if  $i+k_0\geq k+1$, then define $G_i'=G_{i+k_0-k}$, $A_i'=B_{i+k_0-k}$, $B_i'=C_{i+k_0-k}$ and $C_i'=A_{i+k_0-k}$.  It is easy to check that  $(G_i',A_i',B_i',C_i') (i = 1,\ldots , k)$ is a worn hex-chain for $G$.

(ii) Suppose that $\hat{A}_{i_0}=A_{i_0}$, $\hat{B}_{i_0}=C_{i_0}$ and $\hat{C}_{i_0}=B_{i_0}$ (the other cases are similar). 
For every $i\in\{1,\ldots, k\}$, if $1\leq 2i_0-i\leq  k$, then define $G_i'=G_{2i_0-i}$, $A_i'=A_{2i_0-i}$, $B_i'=C_{2i_0-i}$ and $C_i'=B_{2i_0-i}$, if $2i_0-i\leq 0$, then define $G_i'=G_{2i_0-i+k}$, $A_i'=C_{2i_0-i+k}$, $B_i'=B_{2i_0-i+k}$ and $C_i'=A_{2i_0-i+k}$ and finally if $2i_0-i\geq k+1$, then define $G_i'=G_{2i_0-i-k}$, $A_i'=B_{2i_0-i-k}$, $B_i'=A_{2i_0-i-k}$ and $C_i'=C_{2i_0-i-k}$. Again, it is easy to check that $(G_i',A_i',B_i',C_i') (i = 1,\ldots , k)$ is a worn hex-chain for $G$.
\end{proof}
In the following lemma, we elaborate how one can extend a (special) clique covering of one of the terms of a worn hex-chain to a clique covering of the whole graph. Let us say that a clique covering $\mathscr{C}$ of a three-cliqued graph $(G,A,B,C)$ is \textit{splitting} if
\begin{itemize}
\item[(SP1)] we have $A,B,C\in \mathscr{C}$, and
\item[(SP2)] for every $u\in A\cap \tilde{W}(G)$, there exists a clique $C_u$ in $\mathscr{C}\setminus \{A\}$ such that  $C_u\cap A=\{u\}$. The same holds for every $u\in B\cap \tilde{W}(G)$ and every $u\in C\cap \tilde{W}(G)$, where $ \tilde{W}(G) $ is the non-core of $ G $, defined in \pref{sub:notation}.
\end{itemize}
\begin{lem} \label{lem:worn}
Let $G$ be a three-cliqued graph which admits a worn hex-chain $(G_i,A_i,B_i,C_i)$ $(i=1,\ldots,k)$. Assume that there exists $i_0 \in \{1,\ldots, k\}$ such that $G_{i_0}$ has a splitting clique covering $\mathscr{C}$ of size at most $|V(G_{i_0})|-t$, for some number $t$. Then $G$ admits a clique covering of size at most $|V(G)|-t$.
\end{lem}
\begin{proof}
Due to \pref{lem:worn_ordering}(i), w.l.o.g. we may assume that $i_0=1$. 
Define ${A}=\cup_{i=1}^{k} A_i$, ${B}=\cup_{i=1}^{k} B_i$ and ${C}=\cup_{i=1}^{k} C_i$.
By (W3), $A_{1}\cap W(G_1)$ is anticomplete to ${B}\setminus B_{1}$, $B_1\cap W(G_1)$ is anticomplete to ${C}\setminus C_1$ and $C_{1}\cap W(G_1)$ is anticomplete to ${A}\setminus A_1$. For every $u\in \tilde{W}(G_1)$, let $C_u\in \mathscr{C}$ be as in (SP2). For every $u\in A_1\cap \tilde{W}(G_1)$, define $\tilde{C}_u=N[u,B\setminus B_1]\cup C_u$ which is a clique of $G$. Similarly, for every $u\in B_{1}\cap \tilde{W}(G_1)$, define $\tilde{C}_u=N[u,C\setminus C_1]\cup C_u$ and for every $u\in C_{1}\cap \tilde{W}(G_1)$, define $\tilde{C}_u=N[u,A\setminus A_1]\cup C_u$ which are all cliques of $G$.
Moreover, by (SP1), we have $A_{1}, B_{1}, C_{1}\in \mathscr{C}$. Now, let $\mathscr{C}'$ be the collection of cliques of $G$ obtained from $\mathscr{C}$ by replacing the cliques $A_{1}, B_{1}$ and $C_{1}$ with the cliques ${A}$, ${B}$, and  ${C}$, and also replacing the cliques  $C_u$, $u\in \tilde{W}(G_{1})$, with the cliques $\tilde{C}_u$, $u\in \tilde{W}(G_{1})$. It can be easily seen that the collection  $\mathscr{C}'\cup \mathscr{N}[A\setminus A_{1};B]\cup \mathscr{N}[B\setminus B_{1};C]\cup \mathscr{N}[C\setminus C_{1};A]$ is a clique covering for $G$ of size at most $|V(G_1)|-t+|V(G)|-|V(G_1)|=|V(G)-t$. This proves \pref{lem:worn}.
\end{proof}
In the following, as a modification of \pref{lem:thickening}, we prove how one can derive a splitting clique covering for a thickening of a three-cliqued graph.
\begin{lem} \label{lem:thickening_splitting}
Let $(H,A,B,C) $ be a three-cliqued graph and $ F $ be a valid set for $ H $  such that for every pair $\{u,v\}\in F$, $ u,v $ are not both in the same set $ A $, $ B $ or $ C $. Also, let $ G $ be a thickening of $ (H,F) $. Suppose that $\tilde{W}(H\setminus F)=\emptyset$ and $H\setminus F$ admits a splitting clique covering $\mathscr{C}$ of size at most $|V(H)|-t$, for some number $t$. Then $(G,A',B',C')$ admits a splitting clique covering of size at most $|V(G)|-t$, where $A'=\cup_{v\in A} X_v$ and $B',C'$ are defined similarly.
\end{lem}
\begin{proof}
The proof is similar to the proof of \pref{lem:thickening}. Let $ (X_v)_{v\in V(H)} $ be as in the definition of thickening. For each clique $K \in \mathscr{C} $, define $ X_K=\cup_{u\in K} X_u$ which is a clique of $G$ and let $ \mathscr{C}'=\{X_K\ :\ K\in \mathscr{C}\} $. By (SP1), we have $ A,B,C\in \mathscr{C} $ and thus $ A',B',C'\in\mathscr{C}' $. If $F=\emptyset$, then since $\tilde{W}(H\setminus F)=\emptyset$, we have $ \tilde{W}(G)=\emptyset $ and thus $ \mathscr{C}' $ is a splitting clique covering for $G$ of size at most $|V(H)|-t\leq |V(G)|-t$. Now, assume that $F\neq \emptyset$, say $F=\{\{u_1,v_1\}, \ldots, \{u_{\ell},v_{\ell}\} \}$. For every $ i\in\{1,\ldots, \ell \} $, let $ \tilde{X}_{u_i}\subseteq X_{u_i} $ (resp.  $ \tilde{X}_{v_i}\subseteq X_{v_i} $) be the set of vertices in $ X_{u_i} $ (resp. $ X_{v_i} $) which are complete to $ X_{v_i} $ (resp. $ X_{u_i} $). 
Note that by (T4) in the definition of thickening, $X_{u_i}\setminus \tilde{X}_{u_i}$ and $X_{v_i}\setminus \tilde{X}_{v_i}$ are both nonempty. 
Assume w.l.o.g. that for every $i \in \{1,\ldots, \ell\}$, $1\leq |X_{u_i}\setminus \tilde{X}_{u_i}|\leq |X_{v_i}\setminus \tilde{X}_{v_i}| $ and for each $ x\in \tilde{X}_{u_i} $, define $ C_x=X_{v_i}\cup \{x\} $ and for each $ x\in \tilde{X}_{v_i} $, define $ C_x=X_{u_i}\cup \{x\} $. Now, if $|X_{v_i}\setminus \tilde{X}_{v_i}|=1$, then let $ \mathscr{C}_i=\{C_x: x\in \tilde{X}_{u_i}\cup \tilde{X}_{v_i}\}$, and if $|X_{v_i}\setminus \tilde{X}_{v_i}|\geq 2$, then let $ \mathscr{C}_i=\{C_x: x\in \tilde{X}_{u_i}\cup \tilde{X}_{v_i}\}\cup \mathscr{N}[X_{u_i}\setminus \tilde{X}_{u_i};X_{v_i}]$. Note that $|\mathscr{C}_i|\leq |X_{u_i}|+|X_{v_i}|-2$ and the cliques in $\mathscr{C}_i$ cover all the edges in $E(X_{u_i},X_{v_i})$.
Now, the collection of cliques $\mathscr{C}''=\mathscr{C}'\cup (\cup_{i=1}^{\ell}\mathscr{C}_i)$ is a clique covering for $G$ of size at most
\[|V(H)|-t+\sum_{i=1}^{\ell}(|X_{u_i}|+|X_{v_i}|-2) \leq |V(H)|-t-2\ell +\sum_{i=1}^{\ell} (|X_{u_i}|+|X_{v_i}|)\leq |V(G)|-t.\]
It remains to prove that $\mathscr{C}''$ is splitting.  First, since $A',B',C'\in \mathscr{C}''$, (SP1) is satisfied. In addition, since $ \tilde{W}(H\setminus F)=\emptyset$, we have $\tilde{W}(G)\subseteq \bigcup_{i=1}^{\ell} (\tilde{X}_{u_i}\cup \tilde{X}_{v_i})$, and since no pair in $F$ is included in the same set $A, B$ or $C$, for every $x\in \bigcup_{i=1}^{\ell} (\tilde{X}_{u_i}\cup \tilde{X}_{v_i})$, the clique $C_x$ defined above fulfills the conditions in (SP2). This proves \pref{lem:thickening_splitting}.
\end{proof} 
Now, we are ready to prove \pref{thm:tcA}, which is restated below. It should be noted that in the proof, we presume the truth of \pref{thm:antiprismaticA}. In fact, within the proof, we need appropriate clique coverings for three-cliqued antiprismatic graphs which will be given in \pref{sub:3coap}.
\begin{thm} \label{thm:tc2}
Let $G$ be a three-cliqued claw-free graph on $n$ vertices which contains at least one triad. Then $\cc(G)\leq n$, and equality holds if and only if  $n=3p+3$, for some positive integer $p$, and $G$ is isomorphic to the $p^{th}$ power of the cycle $C_n$. 
\end{thm}
\begin{proof}
By \pref{thm:tcseymour}, $G$ admits a worn hex-chain $(G_i,A_i,B_i,C_i) (i=1,\ldots,k)$, where for each $i\in \{1,\ldots,k\}$, $(G_i,A_i,B_i,C_i)$ is a thickening of a permutation of a member of one of $\mathcal{TC}_1,\ldots,\mathcal{TC}_5$.  First, assume that for each $i$, $G_i$ is a thickening of a member of $\mathcal{TC}_4$. Since every triad of $G$ is a triad of one of $G_1,\ldots, G_k$, $G$ is also a thickening of a member of $\mathcal{TC}_4$, i.e. $ G $ is a fuzzy antiprismatic graph. Now, since $G$ contains a triad, by \pref{thm:antiprismaticA}, $\cc(G)\leq n$ and equality holds if and only if  $G$ is isomorphic to $C_{3p+3}^p$,  for some positive integer $p$ (note that the complement of a twister is not three-cliqued). Now, assume that there exists some $i_0\in \{1,\ldots,k\}$ such that $G_{i_0}$ is not a thickening of a member of $\mathcal{TC}_4$, i.e. $G_{i_0}$ is a thickening of a permutation of a graph $G'\in (\mathcal{TC}_1\cup \mathcal{TC}_2\cup \mathcal{TC}_3\cup \mathcal{TC}_5)\setminus \mathcal {TC}_4$, with respect to the corresponding valid set $F$. Due to \pref{lem:worn_ordering}(ii), we may assume that $G_{i_0}$ is a thickening of $(G',F)$.

Now, we claim that $G'\setminus F$ admits a splitting clique covering of size at most $|V(G')|-1$ (note that by the definitions in \pref{app:3clique}, we have $\tilde{W}(G'\setminus F)=\emptyset$ and thus, it is enough to check (SP1) in the definition of splitting clique covering). 
This along with \pref{lem:thickening_splitting} implies that $G_{i_0}$ admits a splitting clique covering of size at most $|V(G_{i_0})|-1$. Thus, by \pref{lem:worn}, $G$ admits a clique covering of size at most $n-1$ and the proof is complete. It just remains to prove the claim. 
For this, we consider the following four cases.
\begin{itemize}
\item $G'\in \mathcal{TC}_1$.
Let $ (G',A,B,C)$, $H$, $F'$ and $F$ be as in the definition of $\mathcal{TC}_1$. For every vertex $v\in V(H)$ of degree at least three, let $K_v$ be the set of all edges of $H$ incident with $v$ and note that $K_v$ is a clique of $G'\setminus F$. These cliques together with all pairs in $F'\setminus F$ comprise a clique covering $\mathscr{C}$ for $G'\setminus F$. Also, $A=K_{v_1},B=K_{v_2}$ and $C=K_{v_3}$ are all in $\mathscr{C}$. Thus, $\mathscr{C}$ satisfies (SP1). Moreover, note that if $H$ has $t$ vertices of degree at least three, then $|\mathscr{C}|\leq t+|F'|$. On the other hand, $|F'|$ is equal to the number of vertices of $H$ with degree two, and thus, 
\[|V(G')|=|E(H)|= \frac{1}{2} \sum_{u\in V(H)} \deg(u) \geq \frac{3}{2} t+ |F'|\geq \frac{1}{2} t+|\mathscr{C}|.\]
Therefore, since $t\geq 3$, we have $|\mathscr{C}|\leq |V(G')|-2$.
\item $G'\in \mathcal{TC}_2\setminus \mathcal{TC}_4$.
Let $(G',A,B,C)$ be a graph in $ \mathcal{TC}_2$ with $\Sigma,\mathcal{I}=\{I_1,\ldots, I_k\},L_1,L_2,L_3$ and the valid set $F$ as in the definition, where $A=V(G')\cap L_1$, $B=V(G')\cap L_2$ and $C=V(G')\cap L_3$. Then, $G'\setminus F$ is also a long circular interval graph with some intervals, say $\mathcal{J}=\{J_1,\ldots, J_l\}$ (to obtain these intervals, it is enough to exchange each interval $I_i\in \mathcal{I}$ whose endpoints are in $F$ with two other intervals, each of which is obtained from $I_i$ by a slight moving of one of its endpoints). We may also assume that for every $i\in\{1,2,3\}$, $L_i$ is contained in an interval in $\mathcal{J}$. 
Moreover, since for every $ \{u,v\}\in F $, $ u,v $ do not belong to the same set $ A,B,$ or $C $, $ G'\setminus F $ is three-cliqued. 

First, we observe that $G'\setminus F$ is not isomorphic to $C_m^p$, for any positive integers $m,p$. 
For, on the contrary, if $m\geq 3p+4$, then $G'\setminus F$ is not three-cliqued. Otherwise, if $m\leq 3p+3$, then $G'\setminus F$ is in $\mathcal{TC}_4$ and since $G'$ is claw-free, we have $G'\in \mathcal{TC}_4$, a contradiction. Therefore, $G'\setminus F$ is not isomorphic to $C_m^p$, for any positive integers $m,p$. 
This observation, besides an argument similar to the proof \pref{thm:circ} (Statements (1) and (2)), allows us to assume that the number of intervals $l$ is at most $ |V(G')|-1$ and so the collection of cliques $\mathscr{C}=\{J_i\cap V(G'): 1\leq i\leq l\}$ is a clique covering for $ G'\setminus F $ of size at most $|V(G')|-1$. Moreover, let $J_j$ be an interval in $\mathcal{J}$ containing $L_1$. Thus, $J_j\cap V(G'\setminus F)$ contains $A$. Also, every vertex $u\in J_j\cap (B\cup C)$ is complete to either $A\cup B\setminus \{u\}$ or $A\cup C\setminus \{u\}$, and so $u$ is in no triad, a contradiction with $\tilde{W}(G'\setminus F)=\emptyset$. Thus, $J_j\cap (B\cup C)=\emptyset $ and $J_j\cap V(G'\setminus F)=A$ belongs to $\mathscr{C}$. By a similar argument, $B,C\in \mathscr{C}$. Hence, $ \mathscr{C}$ satisfies (SP1).

\item $G'\in \mathcal{TC}_3$.
 Let $m\geq 2$, $(G',A,B,C)$, $X$ and $F$ be as in the definition of $\mathcal{TC}_3$ and $|V(G')|=3m+2-|X|$. For every edge $e=\{a_i,b_i\}\in E(G'\setminus F)$, define $K_e=(C\cup \{a_i,b_i\})\setminus \{c_i\}$, which is a clique of $G'\setminus F$. Let $A'$ (resp. $B'$) be the set of vertices in $A\setminus \{a_0\}$ (resp. $B\setminus \{b_0\}$) which are anticomplete to $B$ (resp. $A$) in $G'\setminus F$. Now,  the family $\{K_e: e\in E_{G'\setminus F}(A,B)\}\cup\mathscr{N}[A'\cup B';C]\cup \{A,B,C\}$ is a splitting clique covering for $G'\setminus F$ of size at most
\[|E_{G'\setminus F}(A,B)|+|A'|+|B'|+3= |A|+|B|-|E_{G'\setminus F}(A,B)|+1\leq |V(G')|-1, \]
where the last inequality is due to the fact that $|C|\geq 2$.
\item $G'\in \mathcal{TC}_5$.
Let $(G',A,B,C)$, $X$ and $F$ be as in the first construction in the definition of $\mathcal{TC}_5$. First, note that $\mathscr{C}_1=\{A, B, C, \{v_1,v_6,v_7\}, \{v_2,v_3,v_4\}\setminus X,\{v_3,v_4,v_5\}\setminus X\}$ is a splitting clique covering for $G'\setminus F=G'$ of size $6$. If $|X|\leq 1$, then $|V(G')|\geq 7$ and we are done. Otherwise, if $X=\{v_3,v_4\}$, then $|V(G')|=6$ and we can remove the cliques $\{v_2\}$ and $\{v_5\}$ from $\mathscr{C}_1$ to obtain a clique covering for $G'$ of size $4$, as desired.

Next, let $(G',A,B,C)$, $X$ and $F$ be as in the second construction in the definition of $\mathcal{TC}_5$. Consider the collection of cliques $\{A,B,C,\{v_1,v_8,v_9\},\{v_2,v_3\}\setminus X,\{v_6,v_7\}\setminus X\}$ of size $6$ and if $v_2 v_4\in E(G'\setminus F)$, then replace $\{v_2,v_3\}\setminus X$ with $\{v_2,v_3,v_4\}\setminus X$. Also, if $v_5 v_7\in E(G'\setminus F)$, then replace $\{v_6,v_7\}\setminus X$ with $\{v_5,v_6,v_7\}\setminus X$. The resulting family called $\mathscr{C}_2$ is a splitting clique covering for $G'\setminus F$. If $|X|\leq 2$, then $|V(G')|\geq 7$ and we are done. If $|X|=3$, then either $\{v_3,v_4\}\subseteq X$ or $\{v_5,v_6\}\subseteq X$ and consequently either $\{v_2\}$ or $\{v_7\}$ is in $\mathscr{C}_2$, which can be removed. Finally, if $|X|=4$, i.e. $X=\{v_3,v_4,v_5,v_6\}$, then both $\{v_2\}$ and $\{v_7\}$ are in $\mathscr{C}_2$, which can be removed.
\end{itemize}
\end{proof}
\section{Orientable antiprismatic graphs}\label{sec:oap}
The main goal of this and next section is to prove the upper bound $ n $ for the clique cover number of $ n$-vertex antiprismatic graphs which contain a triad (see \pref{thm:antiprismaticA}). This is the most cumbersome part of the proof of \pref{thm:main1} which will occupy the whole rest of this paper. To warm-up, we first prove a weaker version of \pref{thm:antiprismaticA} which has a short independent proof.
\begin{thm} \label{thm:antiprismatic1}
	If $G$ is a connected fuzzy antiprismatic graph on $n$ vertices which contains at least one triad, then $\cc(G)\leq n+3(1+o(1))\log n$.
\end{thm}
\begin{proof}
	Let $G$ be a thickening of $(H,F)$, where $ H $ is an antiprismatic graph on $m$ vertices and $ F $ is a valid set of changeable pairs of $ H $. Since  the pairs in $F$ are non-edges of $H$ and $G$ contains a triad, $H$ contains a triad, say $\tau=\{x,y,z\}$, as well. First, we claim that $\cc(H)\leq m+3(1+o(1))\log m$. To see this, note that every vertex in $V(H)\setminus \tau$ is nonadjacent to exactly one of the vertices $x,y$ or $z$. Let  $X, Y,Z$ be the set of non-neighbours of $x,y,z $ in $H$, respectively. Therefore, $V(H)$ is the disjoint union of $X\setminus \tau, Y\setminus \tau, Z\setminus \tau$ and $\tau$. Since $H$ is antiprismatic, we observe that,
	\begin{itemize}
		\item For every vertex $u$ in $V(H)\setminus X$, the set $N(u,X)$ is a clique of $H$. The same holds for $Y,Z$. 
		\item The induced subgraph $\overline{H}[X]$ is the union of an induced matching and some isolated vertices. The same holds for $Y,Z$.
	\end{itemize}
	(The proofs are straightforward and left to the reader.) 
	Gregory and Pullman in \cite{pullman} proved that the edges of the complement of a perfect matching on $2t$ vertices can be covered by at most $(1+o(1)) \log t$ cliques. Thus, all the edges in $E(X)$ can be covered by at most $(1+o(1))\log m$ cliques and the same holds for $Y$ and $Z$. On the other hand, the collection of cliques $ \mathscr{N}[X\setminus \{y,z\};Y] \cup \mathscr{N}[Y\setminus \{x,z\};Z]\cup \mathscr{N}[Z\setminus \{x,y\};X]$ covers all the remaining edges. Hence, $H$ admits a clique covering of size at most $m+3(1+o(1))\log m$. Now, by \pref{lem:thickening_antiprismatic}, we have $\cc(G)\leq n+3(1+o(1))\log m\leq  n+3(1+o(1))\log n.$
	This proves \pref{thm:antiprismatic1}.
\end{proof}
In this section, we focus particularly on a special class of antiprismatic graphs called orientable antiprismatic graphs (see \pref{sec:main}) and prove \pref{thm:oapA}, restated as follows.
\begin{thm}\label{thm:oap}
	Let $G$ be an orientable antiprismatic graph on $n$ vertices which contains at least one triad. Then $\cc(G)\leq n$ and equality holds if and only if $n=3p+3$, for some positive integer $p$ and $G$ is isomorphic to the $p^{th}$ power of the cycle $C_n$.
\end{thm}
In order to prove this theorem, we apply the structure theorem of orientable antiprismatic graphs from \cite{seymour1} as stated in the following. Note that the description of ``prismatic'' graphs (the graphs whose complement are antiprismatic) is given in \cite{seymour1}, so we have reformulated the definitions and the statements in terms of the complements. For every positive integer $k$, an antiprismatic graph $G$ is called to be \textit{$k$-substantial} if for every $S\subseteq V(G)$ of size at most $k-1$, there exists a triad $\tau$ in $G$ such that $\tau\cap S=\emptyset$.  
\begin{thm} {\rm \cite{seymour1}} \label{thm:orient-anti}
	Every 3-substantial orientable antiprismatic graph is either three-cliqued, or the complement of a cycle of triangles graph, or the complement of a ring of five, or the complement of a mantled $L(K_{3,3})$ $($see the definitions of these graphs in \pref{app:oap}$)$.
\end{thm}
In \pref{sub:3coap}, we deal with three-cliqued antiprismatic graphs. In \pref{sub:substantial}, using \pref{thm:orient-anti}, we complete the proof for 3-substantial graphs and finally in \pref{sub:degenerate} we tackle the case of non-3-substantial graphs, thereby establishing  \pref{thm:oap}.
\subsection{Three-cliqued antiprismatic graphs}\label{sub:3coap}
Firstly, we consider the case of three-cliqued antiprismatic graphs and prove the following.
\begin{thm}\label{thm:3cap}
	If $G$ is a three-cliqued antiprismatic graph on $n$ vertices which contains at least one triad, then $\cc(G)\leq n$ and equality holds if and only if $n=3p+3$, for some positive integer $p$ and $G$ is isomorphic to the $ p^{th} $ power of the cycle $C_n$.
\end{thm}
To prove \pref{thm:3cap}, we apply the following structure theorem from \cite{seymour1} (reformulated in terms of the complements) as well as \pref{lem:worn}. Note that the line graph of $K_{3,3}$ is self-complementary.  
\begin{thm}{\rm \cite{seymour1}} \label{thm:anti3col}
	Every three-cliqued antiprismatic graph admits a worn hex-chain whose all terms are either triad-free, or isomorphic to the line graph of $K_{3,3}$, or the complement of a canonically-coloured path of triangles graph  $($see the definitions of these graphs in \pref{app:oap}$)$.
\end{thm}
Let $G$ be a graph whose vertex set is the union of three disjoint cliques $X_1,X_2, X_3$ such that $X_2$ contains a  subset $ \hat{X}_2 $ with $ |\hat{X}_2|=1 $ which is anticomplete to $X_1, X_3$. Also, $1\leq |X_1|=|X_3|\leq 2$, $X_1, X_3$ are anti-matched and every vertex in $X_2\setminus \hat{X}_2$ is either complete to $X_1$ and anticomplete to $X_3$, or anticomplete to $X_1$ and complete to $X_3$. Let us call the three-cliqued graph $(G,X_1,X_2,X_3)$ a \textit{tripod}.  A tripod is an example of the complement of a canonically-coloured path of triangles graph (with setting $m=1$ in the definition) which should be treated separately. In the following, we give a splitting clique covering for the complement of a canonically-coloured path of triangles graph which is not a tripod (see the definition of a splitting clique covering in \pref{sec:3cliques}). 
\begin{lem} \label{lem:path}
	Let $G$ be the complement of a canonically-coloured path of triangles graph on $n$ vertices which is not a tripod. Then $G$ admits a splitting clique covering of size at most $n-1$.
\end{lem}
\begin{proof}
	Let $X_1,\ldots, X_{2m+1}$ be as in the definition. Also, for every $i\in\{1,\ldots, m\}$, define $\tilde{X}_{2i}=X_{2i}\setminus \hat{X}_{2i}$ (for every $i\not\in \{1,\ldots, 2m+1\}$, all sets $X_i, \tilde{X}_i,L_i,M_i,R_i$ are supposed to be empty). Let $A=X_1\cup X_4\cup X_7\cup\ldots, B=X_2\cup X_5\cup X_8\cup \ldots$, $C=X_3\cup X_6\cup X_9\cup \ldots$, and
	for every $i\in\{-1,\ldots, m-1\}$, define 
	\[C_i= \ldots \cup X_{2i-8}\cup X_{2i-5} \cup X_{2i-2} \cup R_{2i-1}\cup X_{2i+2} \cup  X_{2i+5}\cup X_{2i+8}\cup \ldots,\]
	(note that $C_{-1}=C$, $C_0=B$ and $C_1=A$).
	Moreover, for every $x\in R_{2m-1}$, define
	\[C'_{x}=N[x,L_{2m+1}]\cup X_{2m-2}\cup X_{2m-5}\cup X_{2m-8}\cup \ldots. \]
	Note that, by (P2)-(2) and (P5), all above sets are cliques of $G$.
	Now, for every $i\in \{1,\ldots, m-1\}$ and every $x\in M_{2i+1}$, define $C''_{i,x}$ as follows,
	\begin{itemize}
		\item if $|\hat{X}_{2i}|=1$, then $C''_{i,x}= M_{2i-1}\cup R_{2i-1}\cup N[x, M_{2i+3}]\cup L_{2i+3}$,
		\item if $|\hat{X}_{2i}|>1$, then $C''_{i,x}=N[x, M_{2i-1}\cup X_{2i}]\cup M_{2i+3}\cup L_{2i+3}$.
	\end{itemize}
	Note that the former is a clique due to (P2)-(2), (P6)-(1) and (P7)-(1). Also, note that if $|\hat{X}_{2i}|>1$, then by (P1),  $1<i<m$ and  $|\hat{X}_{2i-2}|=|\hat{X}_{2i+2}|=1$ and thus by (P6)-(4) and (P7)-(2), $M_{2i-1}, M_{2i+1}$ are matched and both are anti-matched to $\hat{X}_{2i}$. This along with (P2)-(2) and (P5) implies that the latter is also a clique. 
	Also, by (P1) and (P6)-(4),  for every $i\in\{1,\ldots, m-1\}$, $M_{2i+1}$ is nonempty.
	Finally, for every $i\in\{1,\ldots, m\}$, define 
	\[Y_i=\ldots\cup \tilde{X}_{2i-8}\cup \tilde{X}_{2i-2} \cup \tilde{X}_{2i+2}\cup \tilde{X}_{2i+8}\cup \ldots,\]
	and for every  $i\in\{1,\ldots, m\}$ and every $x\in \tilde{X}_{2i}$, define $\tilde{C}_{i,x}$ as follows,
	\begin{itemize}
		\item if $|\hat{X}_{2i}|=1$, then $\tilde{C}_{i,x}=N[x, Y_i\cup M_{2i-1} \cup R_{2i-1}\cup M_{2i+1}\cup L_{2i+1}]$,
		\item if $|\hat{X}_{2i}|>1$, then $\tilde{C}_{i,x}=N[x, Y_i]$,
	\end{itemize}
	which is a clique due to (P2)-(2) and (P5). Define the collection of cliques $\mathscr{C}$ as follows.
	\begin{align*}
		\mathscr{C}=& \{C_i: -1\leq i\leq m-1\} \cup \{N[x,C_i]: 1\leq i\leq m-1, x\in L_{2i+1}\}
		\cup\{C'_x: x\in R_{2m-1}\} \\\quad & \cup \{C''_{i,x}: 1\leq i\leq m-1, x\in M_{2i+1}\} \cup \{\tilde{C}_{i,x}: 1\leq i\leq m, x\in \tilde{X}_{2i}\}.
	\end{align*}
	Note that by (P6)-(1) and (P7)-(1), $|R_{2i-1}|=|L_{2i+1}|$, for every $i\in\{1,\ldots, m\}$. This together with (P3) implies that, 
	\begin{align*}
		|\mathscr{C}|&= m+1+\sum_{i=1}^{m-1} (|L_{2i+1}|+ |M_{2i+1}|)+|R_{2m-1}|+\sum_{i=1}^m|\tilde{X}_{2i}| \\
		&= m+1+\sum_{i=1}^{m} (|L_{2i+1}|+ |M_{2i+1}|+|\tilde{X}_{2i}|)\\
		&= n+1-\sum_{i=1}^{m} (|R_{2i-1}|+|\hat{X}_{2i}|-1).
	\end{align*}
	First, suppose that $ m=1 $. Then, $ B,C\in \mathscr{C} $ and by (P4), $ R_1\neq \emptyset $. Thus, adding the clique $A= X_1 $ to $ \mathscr{C} $ yields a splitting clique covering for $ G $ of size $ n+2-|R_1| $. If $ |R_1|\geq 3 $, then we are done. If $ |R_1|\leq 2 $, then since $ G $ is not a tripod, we have $ |R_1|=2 $ and at least one of the two edges in $ E(R_1,L_3) $ are covered by the cliques $ \tilde{C}_{1,x} $, $ x\in \tilde{X}_{2} $. Thus, we may remove one of the cliques $ C'_x, x\in R_1 $, from $ \mathscr{C} $ to obtain a splitting clique covering of size $ n-1 $.  
	
	Now, suppose that $ m\geq 2 $ and let $I=\{i: 1< i< m, |\hat{X}_{2i}|>1, R_{2i+1}\neq \emptyset\}$. We leave the reader to check that the only edges of $G$ which are not covered by $\mathscr{C}$ are the edges in $E(R_{2i+1},L_{2i-1})$, $i\in I$. If $ |\hat{X}_{2i}|>1$, then by (P5) and (P7)-(2), $R_{2i+1}$ is complete to $L_{2i-1}$. Therefore, the collection $\mathscr{C}'=\mathscr{C}\cup \{R_{2i+1}\cup L_{2i-1}: i\in I\}$ is a clique covering for $G$. Also, $A,B,C\in \mathscr{C}$ and $\tilde{W}(G)=\cup_{i=1}^m \tilde{X}_{2i}$. Thus, the cliques $\tilde{C}_{i,x}$, $x\in \tilde{X}_{2i}$, satisfy (SP2) in the definition of splitting clique covering and $\mathscr{C}'$ is a splitting clique covering. It remains to calculate the cardinality of $\mathscr{C}'$.
	
	If at least two of the sets $R_{2i-1}$, $i\in\{1,\ldots, m\}$, are nonempty, then since $|I|\leq \sum_{i=1}^{m} (|\hat{X}_{2i}|-1)$, we have $|\mathscr{C}'|\leq n-1$. Now, assume that at most one of the sets  $R_{2i-1}$, $i\in\{1,\ldots, m\}$, is nonempty and thus $|I|\leq 1$. If $R_1\neq \emptyset$, then $I=R_{2m-1}=\emptyset$  and thus by (P4), $|\hat{X}_{2m-2}|>1$.  Therefore, $|\mathscr{C}'|=|\mathscr{C}|\leq n-1$. Finally, if $R_1=\emptyset$, then by (P4), $|\hat{X}_4|>1$. Now, if $m=2$, then $I=\emptyset $ and by (P4), we have $R_3\neq \emptyset$ and thus $|\mathscr{C}'|=|\mathscr{C}|\leq n-1$.  Also, if $m\geq 3$, then for every $x\in M_3$, $C''_{1,x}=N[x,M_5]\subseteq \cup_{y\in M_5} C''_{2,y}$ and we may remove the cliques $C''_{1,x}$, $x\in M_3$, from $\mathscr{C}'$ to obtain a splitting clique covering of size at most $n-1$. 
\end{proof}
We also need the following lemma.
\begin{lem} \label{lem:worn_triadfree}
	Let $G$ be a three-cliqued graph on $n$ vertices which admits a worn hex-chain $(G_i,A_i,B_i,C_i)$ $(i=1,\ldots, k)$. If for some $ i_0\in\{1,\ldots, k\} $, $ G_{i_0} $ is triad-free and $ G_{i_0+1} $ is a triad  $($reading $i_0+1$ modulo $ k )$, then $ \cc(G)\leq n-1 $.
\end{lem}
\begin{proof}
	Due to \pref{lem:worn_ordering}(i), w.l.o.g. we may assume that $ G_{1} $ is a triad and $ G_k $ is triad-free. Let  ${A}=\cup_{i=1}^{k} A_i$, ${B}=\cup_{i=1}^{k} B_i$ and ${C}=\cup_{i=1}^{k} C_i$, ${A}'=\cup_{i=2}^{k-1} A_i$, ${B}'=\cup_{i=2}^{k-1} B_i$ and ${C}'=\cup_{i=2}^{k-1} C_i$. Note that $A_1$ is anticomplete to $B$, $B_1$ is anticomplete to $C$ and $C_1$ is anticomplete to $A$. Define $\mathscr{C}=\mathscr{N}[A';B]\cup\mathscr{N}[B';C] \cup\mathscr{N}[C';A]$. 
	First, assume that there are two of the sets $A_k,B_k,C_k$, say $ A_k $ and $ B_k $, such that some vertex in one of them is complete to the other (note that it includes the case that at least one of $ A_k, B_k,C_k $ is empty). If there is a vertex in $ A_k $ complete to $ B_k $, then the collection of cliques $\mathscr{C}\cup \mathscr{N}[A_k;B]\cup \mathscr{N}[B_k;C]\cup \mathscr{N}[C_k;A]\cup \{A,C \}$ is a clique covering for $G$ of size $|A'|+|B'|+|C'|+|A_k|+|B_k|+|C_k|+2=n-1$. Also, if there is a vertex in $B_k$ complete to $A_k$, then the collection $\mathscr{C}\cup \mathscr{N}[A_k;A'\cup  A_1\cup C_k]\cup \mathscr{N}[B_k;B'\cup B_1\cup A_k]\cup \mathscr{N}[C_k;C'\cup C_1\cup B_k]\cup \{C'\cup C_1\cup B_k,A'\cup A_1\cup C_k\}$ is a clique covering for $ G $ of size $ n-1 $. 
	Therefore, we may assume that every vertex in one of the sets $A_k,B_k,C_k$ has a non-neighbour in each of two others. In addition, we claim that every vertex in one of the sets $A_k,B_k,C_k$ has a neighbour in each of two others. For if say there exists a vertex $x\in A_k$ which is anticomplete to $B_k$, then by the assumption, $x$ has a non-neighbour $y$ in $C_k$ and also $y$ has a non-neighbour $z$ in $B_k$, and so $\{x,y,z\}$ is a triad in $ G_k $, a contradiction. This proves the claim. Moreover, assume that $|A_k|\leq |B_k|\leq |C_k|$ (the other cases are similar).
	
	Now, if every two distinct vertices in $B_k$ have a common neighbour in  $A_k\cup C_k$, then $\mathscr{C}\cup \mathscr{N}[A_k;B]\cup \mathscr{N}[C_k;B_k\cup C_1\cup C']\cup \mathscr{N}[A_k;C_k\cup A_1\cup A']\cup \{A,C\}$ is a clique covering for $G$ of size $|\mathscr{C}|+2|A_k|+|C_k|+2\leq |\mathscr{C}|+|A_k|+|B_k|+|C_k|+2=n-1$ (note that, by the above claim, the edges in $E(B)$ are covered by the cliques in $\mathscr{N}[A_k;B]\cup \mathscr{N}[C_k;B_k\cup C_1\cup C']$).
	Finally, if there exist two distinct vertices $u,v\in B_k$ with no common neighbour in $A_k\cup C_k$, then since $G_k$ is triad free, $N[u,A_k\cup C_k]$ and $N[v,A_k\cup C_k]$ are both cliques. Hence, the collection of cliques $\mathscr{C}\cup \mathscr{N}[B_k\setminus\{u,v\};A_k] \cup  \mathscr{N}[B_k\setminus\{u,v\};C_k]\cup  \mathscr{N}[A_k;C_k\cup A_1\cup A']\cup\{N[u,A_k\cup C_k], N[v,A_k\cup C_k],A_k\cup B_1\cup B', B_k\cup C_1\cup C', B,C   \}$ is a clique covering for $G$ of size $|\mathscr{C}|+2(|B_k|-2)+|A_k|+6\leq |\mathscr{C}|+|A_k|+|B_k|+|C_k|+2=n-1$ (note that the edges in $ E(A) $ are covered by the cliques in $ \mathscr{N}[A_k;C_k\cup A_1\cup A']\cup \{A_k\cup B_1\cup B'\} $). This proves \pref{lem:worn_triadfree}.
\end{proof}
Now, we are ready to prove \pref{thm:3cap}. 
\begin{proof}[{\rm \textbf{Proof of \pref{thm:3cap}.}}]
	By \pref{thm:anti3col}, $G$ admits a worn hex-chain  $(G_i,A_i,B_i,C_i)$ $(i=1,\ldots, k)$, where each $G_i$ is either triad-free or isomorphic to the line graph of $K_{3,3}$, or the complement of a canonically-coloured path of triangles graph. 
	
	Firstly, suppose that there exists $ i_0\in\{1,\ldots,k\} $ such that $G_{i_0}$ is isomorphic to the line graph of $K_{3,3}$ on the vertex set $\{a^i_j\ :\ 1\leq i,j\leq 3\}$, where $a^i_j$ is adjacent to $a^{i'}_{j'}$ if and only if $i=i'$ or $j=j'$. Then, $A_{i_0},B_{i_0},C_{i_0}$ are three disjoint triangles of $G_{i_0}$ and $\{\{a^i_1,a^i_2,a^i_3\}, \{a^1_i,a^2_i, a^3_i\}\ :\  1\leq i\leq 3\}$ is a splitting clique covering for $G_{i_0}$ of size six. Hence, \pref{lem:worn} implies that $G$ admits a clique covering of size at most $n-3$, as desired. 
	
	Secondly, suppose that  there exists $ i_0\in\{1,\ldots,k\} $ such that $G_{i_0}$ is the complement of a canonically-coloured path of triangles graph which is not a tripod. Then, by \pref{lem:path}, $G_{i_0}$ admits a splitting clique covering of size at most $|V(G_{i_0})|-1$, and thus by virtue of \pref{lem:worn}, we obtain a clique covering for $G$ of size at most $n-1$. 
	
	Thirdly, assume that there exists $i_0\in\{1,\ldots, k\}$ such that $G_{i_0}$ is a tripod with the sets $X_1,X_2, \hat{X}_2,X_3$ as in the definition, where  $X_2\setminus \hat{X}_2\neq \emptyset$.  In the light of \pref{lem:worn_ordering}(i), assume that $ i_0=k $. 
	Let $\tilde{X_1}$ (resp. $\tilde{X}_3$) be the set of vertices in $ X_2 \setminus \hat{X}_2$ which are complete to $ X_1 $ (resp. $ X_3 $). Then, since $X_2\setminus \hat{X}_2\neq \emptyset$, either $\tilde{X_1}$ or $\tilde{X}_3$, say $\tilde{X}_1$, is nonempty and by \pref{lem:worn_ordering}(ii), we may suppose that $A_k=X_1, B_k=X_2  $ and $ C_k=X_3 $.
	Let $(H,A',B',C')$ be a three-cliqued graph which admits a worn hex-chain $ (H_i,A_i',B_i',C_i')$ $(i=1,\ldots, k+2) $ such that for every $ i\in\{1,\ldots, k-1\}   $, $ H_i=G_i, A_i'=A_i, B_i'=B_i$ and $C_i'=C_i$. Also, $ H_{k} $ and $H_{k+2}$ are both cliques and $ H_{k+1} $ is a triad on the vertex set $\{a,b,c\}$, where $ A_{k}'=\emptyset $, $ B_k'=\tilde{X}_1 $, $ C_k'=\emptyset $, $ A'_{k+1}=\{a\} $, $ B'_{k+1}=\{b\} $, $ C'_{k+1}=\{c\}$,  $ A'_{k+2}=\emptyset $, $ B'_{k+2}=\tilde{X}_3 $ and $ C'_{k+2}=\emptyset $. Since $H_k$ is triad-free and $H_{k+1}$ is a triad, by \pref{lem:worn_triadfree}, $H$ admits a clique covering of size at most $|V(H)|-1$. Now, if $|X_1|=|X_3|=1$, then $G$ is isomorphic to $H$ (by an isomorphism mapping the single vertices of $G$ in $X_1,\hat{X}_2, X_3$  to the vertices $a,b,c$ in $H$, respectively)  and if $|X_1|=|X_3|=2$, then $G$ is a thickening of $(H,F)$, for $F=\{\{a,c\}\}$ (we leave the reader to check that $\{a,c\}$ is a changeable pair of $H$). Thus, by \pref{lem:thickening_antiprismatic}, $\cc(G)\leq n-1$. 
	
	Therefore, we may assume that for all $i\in\{1,\ldots,k\}$, $G_i$ is either triad-free or a tripod with $ X_2=\hat{X}_2 $. Hence, each $ G_i $ is either triad-free, or a thickening of a triad. 
	First, assume that there is some $ i $ such that $ G_i $ is triad-free. 
	If all terms $ G_i $ are triad-free, then so is $ G $, a contradiction. Therefore, there is some $ i_0\in\{1,\ldots, k\} $ such that $ G_{i_0} $ is triad-free and $G_{i_0+1}$ is a thickening of a triad (reading $ i_0+1 $ modulo $ k $). Hence, Lemmas~\ref{lem:worn_triadfree} and \ref{lem:thickening_antiprismatic} yield that $ \cc(G)\leq n-1 $.
	Finally, assume that each term $ G_i $ is a thickening of a triad. Note that if all terms $G_i$ are triads, then $G$ is isomorphic to the graph $C_{3k}^{k-1}$ (by the isomorphism mapping the single vertices of $G$ in $A_1,\ldots,A_k$, $B_1,\ldots,B_k$, $C_1,\ldots,C_k$ to the vertices $v_{0},\ldots,v_{3k-1}$ of $C_{3k}^{k-1}$, respectively), and if each $ G_i $ is a thickening of a triad, then it is easy to check that $ G $ is a fuzzy long circular interval graph. Thus, in this case, by \pref{thm:circ}, $\cc(G)\leq n$ and equality holds if and only if $ G $ is isomorphic to $C_{n}^{p}$, for some positive integer $ p\leq \lfloor (n-1)/3\rfloor $. However, if $ n>3p+3 $, then $ C_n^p $ is not three-cliqued and if $ n<3p+3 $, then $ C_n^p $ is triad-free. Therefore, $\cc(G)\leq n$ and equality holds if and only if $ n=3p+3 $ and $ G $ is isomorphic to $C_{n}^{p}$. This proves \pref{thm:3cap}.
\end{proof}
\subsection{Other 3-substantial antiprismatic graphs}\label{sub:substantial}
According to Theorems~\ref{thm:orient-anti} and \ref{thm:3cap}, in order to complete the proof of \pref{thm:oap} for 3-substantial orientable antiprismatic graphs, we need to prove the assertion for the complement of a cycle of triangles graph, the complement of a mantled $L(K_{3,3})$ and the complement of a ring of five (see \pref{app:oap} for the definitions). 
\begin{lem} \label{lem:cycle}
	If $G$ is the complement of a cycle of triangles graph on $n$ vertices, then $\cc(G)\leq n-1$.
\end{lem}
\begin{proof}
	Let $m\geq 5$ and  $m=3k+2$ for some integer $k$ and also let $X_1,\ldots, X_{2m}$ be as in the definition and define $\tilde{X}_{2i}=X_{2i}\setminus \hat{X}_{2i}$. Throughout the proof, read subscripts modulo $2m$. For every $i\in\{1,\ldots, m\}$, define 
	\[C_i= L_{2i-1}\cup X_{2i}\cup X_{2i+3}\cup X_{2i+6}\cup \ldots \cup X_{2i+6k},\]
	which by (C2) and (C4) is a clique of $G$. 
	Note that by (C1) and (C5)-(4), $M_{2i+1}$ is nonempty for every $i\in\{1,\ldots, m\}$. Now, for every $i\in \{1,\ldots, m\}$ and every $x\in M_{2i+1}$, define the clique $C'_{i,x}$ as follows,
	\begin{itemize}
		\item if $|\hat{X}_{2i}|=1$, then $C'_{i,x}= M_{2i-1}\cup R_{2i-1}\cup N[x, M_{2i+3}]\cup L_{2i+3}$,
		\item if $|\hat{X}_{2i}|>1$, then $C'_{i,x}=N[x, M_{2i-1}\cup X_{2i}]\cup M_{2i+3}\cup L_{2i+3}$.
	\end{itemize}
Finally, for every $i\in\{1,\ldots, m\}$, define
\[Y_i=\tilde{X}_{2i+2}\cup \tilde{X}_{2i+8}\cup \ldots \cup \tilde{X}_{2i+2+6k},\]
and for every $x\in \tilde{X}_{2i}$, define $\tilde{C}_{i,x}$ as follows,
	\begin{itemize}
		\item if $|\hat{X}_{2i}|=1$, then $\tilde{C}_{i,x}= N[x, M_{2i-1}\cup R_{2i-1}\cup M_{2i+1}\cup L_{2i+1} \cup Y_i]$,
		\item if $|\hat{X}_{2i}|>1$, then $\tilde{C}_{i,x}=N[x, Y_i]$.
	\end{itemize}
	Now, define $\mathscr{C}=\{C_i: 1\leq i\leq m\}\cup \{N[x,C_i]: 1\leq i\leq m, x\in R_{2i-3}\}\cup  \{C'_{i,x}: 1\leq i\leq m, x\in M_{2i+1}\}  \cup \{\tilde{C}_{i,x}: 1\leq i\leq m, x\in \tilde{X}_{2i}\} $. Thus,
	\[|\mathscr{C}|=m+\sum_{i=1}^m ( |M_{2i-1}|+|R_{2i-1}|+|\tilde{X}_{2i}|)= n-\sum_{i=1}^m (|L_{2i-1}|+|\hat{X}_{2i}|-1).\]
	Also, let $I=\{i: 1\leq i\leq m, |\hat{X}_{2i}|>1, L_{2i-1}\neq \emptyset\}$. It is easy to check that the only edges which are not covered by $ \mathscr{C} $ are the edges in $ E(L_{2i-1},R_{2i+1}), i\in I $.
	Therefore, $ \mathscr{C}'=\mathscr{C}\cup \{L_{2i-1}\cup R_{2i+1} : i\in I\} $ is a clique covering for $ G $. If either $L_{2i-1}$ is nonempty or $|\hat{X}_{2i}|>1$ for some $i$, then  $|\mathscr{C}'|\leq n-1$. Now, assume that $L_{2i-1}=R_{2i-1}=\emptyset $ and $|\hat{X}_{2i}|=1$ for all $i\in\{1,\ldots, m\}$. In this case, remove the cliques $C'_{i,x}$, for all  $1\leq i\leq m$ and $x\in M_{2i+1}$ from $\mathscr{C}$ and add the following cliques
	\[\forall 1\leq i\leq \left\lfloor \frac{m+1}{2} \right\rfloor,\  C'_i=M_{4i-3}\cup M_{4i-1}\cup M_{4i+1}\cup M_{4i+3}. \]
	Thus, in this case, $\cc(G)\leq m+\lfloor (m+1)/2\rfloor+\sum_{i=1}^m |\tilde{X}_{2i}|= \lfloor (m+1)/2\rfloor+ \sum_{i=1}^m |X_{2i}|\leq n-1$.
\end{proof}
The following two lemmas handle the last two cases of the 3-substantial antiprismatic graphs.
\begin{lem} \label{lem:mantled}
	If $G$ is the complement of a mantled $L(K_{3,3})$ on $n$ vertices, then $\cc(G)\leq n-1$.
\end{lem}
\begin{proof}
	Let $a_j^i, V^i, V_j$, $i,j\in\{1,2,3\}$, be as in the definition and read all subscripts and superscripts modulo $3$. 
	For every $i,j\in\{1,2,3\}$, $i\neq j$, define
	\[
	A^i_j= V^i\cup \{a^{j}_1,a^{j}_2,a^{j}_3\}, \quad
	B_j^i= V_i\cup \{a_{j}^1,a_{j}^2,a_{j}^3\}.\]
	First, assume that at least one of $V_1,V_2,V_3$ and at least one of $V^1,V^2,V^3$ are empty, say $V_3=V^3=\emptyset$.  In this case, the family of cliques $ \mathscr{N}[V^1;V^2\cup V_1]\cup \mathscr{N}[V_1;V_2\cup V^2]\cup \mathscr{N}[V^2;V_2]\cup \mathscr{N}[V_2; V^1] \cup \{ A^i_j, B_j^i: 1\leq i\leq 2, 1\leq j\leq 3, j\neq i\}$ is a clique covering for $G$ of size $n-1$. Now, assume that either all $V^1,V^2,V^3$, or all $V_1,V_2,V_3$ are nonempty. Assume w.l.o.g. that the former case occurs. 
	Since $G[V^1\cup V^2\cup V^3]$ contains no triad, there exists $i_0\in\{1,2,3\}$ such that every vertex in $V^{i_0+1}$ has a neighbour in $V^{i_0}$. For every $i\in\{1,2,3\}$, define $U^i= V^{i+1}\cup \left\{a^{i+2}_i,a^{i+2}_{i+1}, a^{i+2}_{i+2}\right\}$, if $V_{i+1}=\emptyset$, and define $U^i= V^{i+1}\cup V_{i+1} \cup \left\{a^{i+2}_i, a^{i+2}_{i+2}\right\}$, otherwise. Also, let $U_i= V_{i+1}\cup V^{i+1} \cup \left\{a_{i+2}^i, a^{i+2}_{i+2}\right\}$ and let
	\[\mathscr{C}=\bigcup_{i=1}^3 \left(\{A^i_{i+1}, B^i_{i+1}\} \cup \mathscr{N}[V^i;U^i]\cup \mathscr{N}[V_i;U_i]\right). \]
	Note that the edges in $E(V^{i_0+1},V_{i_0+1})$ are covered by the cliques in $\mathscr{N}[V^{i_0}; U^{i_0}]$. 
	For every $i\in\{i_0,i_0+2\}$, if $V_i\neq \emptyset$, then  the edges in $E(V^i,V_i)$ are possibly not covered by the cliques in $\mathscr{C}$. Also, for every $i\in\{1,2,3\}$, if $V_{i-1}=\emptyset$ and $V_{i+1}\neq \emptyset$, then the edges in $E(V^i,\left\{a_{i+1}^{i+2}\right\}) $ are not covered by the cliques in $\mathscr{C}$. In both of these cases, add the clique $V^i\cup V_i\cup \left\{a_{i+1}^{i+2}\right\}$ to $\mathscr{C}$. We leave the reader to check that this procedure adds at most two cliques to $\mathscr{C}$ and yields a clique covering for $G$ of size at most $|\mathscr{C}|+2=n-1$. 
\end{proof}
\begin{lem} \label{lem:ring}
	If $G$ is the complement of a ring of five on $n$ vertices, then $\cc(G)\leq n-2$.
\end{lem}
\begin{proof}
	Let $V_i, a_i,b_i, 1\leq i\leq 5$, be as in the definition and consider the following cliques in $G$ (reading subscripts modulo 5).
	\begin{align*}
		&C_0=\{b_1,\ldots,b_5\},\\
		\forall i\in\{1,\ldots ,5\},\qquad &C_i=V_i\cup V_{i+2}\cup \{a_i,b_{i+1},a_{i+2}\}, \\
		\forall i\in\{1,2,3\},\qquad &C'_i=V_0\cup V_{i}\cup V_{i+2}\cup \{a_i,a_{i+2}\}, \\
		\forall i\in\{1,\ldots,5\}, \forall x\in V_i, \quad &C_{i,x}= N[x,V_{i+1}]\cup \{b_{i+2},b_{i+3}\}.
	\end{align*}
	The collection of cliques $\mathscr{C}=\{C_i: 0\leq i\leq 5\}\cup \{C'_i: 1\leq i\leq 3\} \cup \{C_{i,x}: 1\leq i\leq 5, x\in V_i \}$ forms a clique covering for $G$. If $V_0$ is nonempty, then $|\mathscr{C}|\leq n-2$ and if $V_0$ is empty, then one can remove the cliques $C'_1,C'_2,C'_3$ from $\mathscr{C}$ to obtain a clique covering of size $n-4$.
\end{proof}
\subsection{Non-3-substantial antiprismatic graphs}\label{sub:degenerate}
Finally, we encounter the case of non-3-substantial antiprismatic graphs, i.e. antiprismatic graphs whose all triads meet a fixed set of at most two vertices. They include several cases that should be dealt separately and this makes the proof lengthy. The following lemma would be fruitful which will be used in the proof of the next theorem as well as in \pref{sec:noap}. 
\begin{lem} \label{lem:2sub}
	Assume that $G$ is an antiprismatic graph on $n$ vertices which is not three-cliqued. Also, let $a_0$ be a vertex of $G$ which is contained in a triad, let $A=N_{\overline{G}}(a_0)$ and $\tilde{A}\subseteq A$ be the set of all vertices in $ A$ which have no non-neighbour in $A$. If for every $v\in A\setminus \tilde{A}$, $N_G(v,V(G)\setminus A)$ is a clique of $ G $, then $\cc(G)\leq n$ and equality holds if and only if $\overline{G}$ is isomorphic to a twister.
\end{lem}
\begin{proof}
	Let $B=N_G(a_0)=V(G)\setminus (A\cup\{a_0\})$. Since $ G $ is antiprismatic and $ a_0 $ is in a triad, we may assume that $A\setminus\tilde{A}=X\cup Y$, where $X=\{x_1,\ldots, x_k\}$ and $Y=\{y_1,\ldots, y_k\}$, such that the only non-edges whose both endpoints are in $A$ are $x_iy_i$, $i=1,\ldots, k$. Also, every vertex in $ B $ is adjacent to exactly one of $ x_i $ and $ y_i $, for each $ i\in\{1,\ldots, k\} $. Thus, by the assumption, for each $i\in\{1,\ldots,k\}$, $B$ is partitioned into two disjoint cliques $U_i=N(x_i,B)$ and $V_i=N(y_i,B)$. We may assume w.l.o.g. that $|V_i|\leq |U_i|$, for all $i\in\{1,\ldots, k\}$. First, we claim that \vsp
	
	(1) \textit{We have $k\geq 2$. Also, for every $i\in\{1,\ldots,k\}$, $U_i$ is not complete to $V_i$ and every vertex in $U_i$ has a neighbour in $V_i$ and vice versa.} \vsp
	
	For if $ k=1 $, then $V(G)$ is the union of three cliques $ A\setminus \{y_1\}$, $ V_1\cup \{y_1\}$ and $U_1\cup \{a_0\}$, a contradiction. Also, if $U_i$ is complete to $V_i$ for some $ i\in\{1,\ldots,k\} $, then $V(G)$ is the union of  three cliques  $X$, $A\setminus X$ and $U_i\cup V_i\cup \{a_0\}$, a contradiction.  Finally, assume that there exists a vertex $u\in U_i$ with no neighbour in $V_i$. If $ u $ is nonadjacent to some $ x_j \in X $, then $ u \in V_j $ and since $V_j $ is a clique, we have $V_i\cap V_j=\emptyset $ and thus $ x_j $ is complete to $ V_i $. Similarly, if $ u $ is nonadjacent to some $ y_j \in Y $, then $ y_j $ is complete to $ V_i $.  Therefore, $V_i$ is complete to $ N_{\overline{G}}(u,X\cup Y)$ and consequently $V$ is the union of three cliques $U_i\cup \{a_0\}$, $V_i\cup N_{\overline{G}}(u,X\cup Y) $ and $A\setminus N_{\overline{G}}(u,X\cup Y)$, a contradiction. This proves (1). \vsp
	
	By (1), for each $i\in\{1,\ldots,k\}$, there exist $u_i\in U_i$ and $v_i\in V_i$, such that $u_i$ and $v_i$ are nonadjacent. Then, for each $i\in\{1,\ldots,k\}$, $X_i=N(u_i,X\cup Y)$ and $Y_i=N(v_i,X\cup Y)$ are two disjoint cliques partitioning $X\cup Y$ and it is evident that $|X_i|=|Y_i|=k$.
	Now, for every $i\in\{1,\ldots,k\}$, define
	\[\mathscr{C}^1_i= \mathscr{N}[B; A]\cup \mathscr{N}[X_i;Y_i\cup \tilde{A}]\cup \mathscr{N}[V_i;U_i\cup \{a_0\}]\cup \{U_i\cup \{a_0\}, V_i\cup \{a_0\}\}, \]
	and
	\[\mathscr{C}^2_i=\mathscr{N}[A\setminus \tilde{A};B]\cup \mathscr{N}[\tilde{A};U_i]\cup \mathscr{N}[\tilde{A};V_i]\cup \mathscr{N}[X;Y\cup \tilde{A}]\cup \mathscr{N}[V_i;U_i\cup \{a_0\}]\cup \{X\cup \tilde{A}\}.\]
	First, note that due to (1), the edges in $ E(\{a_0\}, U_i\cup V_i) $ are covered by the cliques in $ \mathscr{N}[V_i;U_i\cup \{a_0\}] $. Also, by (1), we have $k\geq 2$. Thus,  for each $i\in \{1,\ldots k\}$, the edges in $E(A)$ are covered by the cliques in  $\mathscr{N}[B; A]\cup \mathscr{N}[X_i;Y_i\cup \tilde{A}]\subseteq \mathscr{C}^1_i$, and so $\mathscr{C}^1_i$ is a clique covering for $G$ of size $|U_i|+2|V_i|+k+2$.
	
	Moreover, note that for each $i\in\{1,\ldots,k\}$, the edges in $E(U_i)$ and $E(V_i)$ are covered by the cliques in $ \mathscr{N}[A\setminus \tilde{A};B]\subseteq \mathscr{C}^2_i$. Also, if $k\geq 3$, then the edges in $E(Y)$ are covered by the cliques in $\mathscr{N}[X;Y\cup \tilde{A}]\subseteq \mathscr{C}^2_i$. Hence, if $k\geq 3$, then for every $i\in \{1,\ldots ,k\}$, $\mathscr{C}^2_i$ is also a clique covering for $G$ of size $|A|+|\tilde{A}|+k+|V_i|+1$. Now, we prove the lemma in two cases of $ k\geq 3 $ and $ k=2 $. 
	
	First, assume that $ k\geq 3 $. If there exists some $i\in\{1,\ldots,k\}$, such that either $|V_i|\leq |\tilde{A}|+k-2$, or $|\tilde{A}|+k\leq |U_i|-1$,  then either $\mathscr{C}^1_i$ or $\mathscr{C}^2_i$ is a clique covering for $G$ of size at most $n-1$. Thus, suppose that for every $i\in\{1,\ldots,k\}$, 
	\[|\tilde{A}|+k-1 \leq |V_i|\leq |U_i|\leq  |\tilde{A}|+k.\]
	We provide a clique covering for $ G $ of size at most $ n-1 $ in the following three possibilities. \vsp
	
	(2) \textit{Assume that $ k\geq 3 $ and for every vertex $x_i\in X$ and $v\in A\setminus\{y_i\}$, $x_i$ and $v$ have a common neighbour in $B$. Then $ \cc(G)\leq n-1 $. } \vsp
	
	In this case, the edges whose one end is in $ X $ and the other end is in $ A $ are covered by the cliques in $ \mathscr{N}[B; A] $. Now, in the clique covering $\mathscr{C}^1_1$, replace the cliques in $\mathscr{N}[X_1;Y_1\cup \tilde{A}]$ with the clique $Y\cup \tilde{A}$ to obtain a clique covering of size $|U_1|+2|V_1|+3\leq |U_1|+|V_1|+|\tilde{A}|+2k\leq n-1$. This proves (2).\vsp
	
	(3)  \textit{Assume that $ k\geq 3 $ and for every vertex $x_i\in X$ and $v\in X\cup Y\setminus \{y_i\}$, $x_i$ and $v$ have a common neighbour in $B$, however, there exists a vertex $x_{i_0}\in X$ and a vertex $v_0\in \tilde{A}$ such that $x_{i_0}$ and $v_0$ have no common neighbour in $B$. Then $ \cc(G)\leq n-1 $.} \vsp
	
	In this case, $N(v_0,B)\subseteq V_{i_0}$. Now, if $|U_{i_0}|=|\tilde{A}|+k$, then remove the clique $ N[v_0, U_{i_0}] $ from $\mathscr{C}^2_{i_0}$ to obtain a clique covering of size $n-1$. Otherwise, if $|U_{i_0}|=|V_{i_0}|=|\tilde{A}|+k-1$, then in the collection $ \mathscr{C}^1_{i_0} $, replace the cliques in $\mathscr{N}[X_{i_0};Y_{i_0}\cup \tilde{A}]$ with the cliques $ X\cup \tilde{A} $ and $Y\cup \tilde{A}$ to obtain a clique covering of size $|U_{i_0}|+2|V_{i_0}|+4\leq |U_{i_0}|+|V_{i_0}|+|\tilde{A}|+2k\leq n-1$ (note that due to the assumption, the edges in $ E(X,Y) $ are covered by the cliques in $ \mathscr{N}[B;A] $). This proves (3).\vsp
	
	(4)  \textit{Assume that $ k\geq 3 $ and there exist vertices $x_{i_0}\in X$ and $v_0\in X\cup Y\setminus \{y_{i_0}\}$ which have no common neighbour in $B$. Then, $\cc(G)\leq n-1$.} \vsp
	
	First, assume that $v_0=y_{j_0}\in Y$, for some $ j_0\neq i_0 $. Thus, $U_{i_0}\cap V_{j_0}=\emptyset $ and so $U_{i_0}\subseteq U_{j_0}$ and $V_{j_0}\subseteq V_{i_0}$. Therefore, since $ |U_{j_0}|-|V_{j_0}|\leq 1 $, we have $U_{i_0}=U_{j_0}$ and $V_{i_0}=V_{j_0}$. In this case, in  $\mathscr{C}^2_{i_0}$, replace the cliques $N[x,B], x\in\{x_{i_0},y_{i_0},x_{j_0},y_{j_0}\}$, with the cliques $ U_{i_0}\cup \{x_{i_0},x_{j_0}\}$ and $V_{i_0}\cup \{y_{i_0},y_{j_0}\}$ to obtain a clique covering of size at most $n-1$. Now, assume that $v_0=x_{j_0}\in X$, for some $ j_0\neq i_0 $. Thus, $U_{i_0}\cap U_{j_0}=\emptyset$ and then $U_{j_0}\subseteq V_{i_0}$ and $U_{i_0}\subseteq V_{j_0}$. Since $ |V_i|\leq |U_i|$ for all $ i $, we have  $U_{j_0}=V_{i_0}$ and $U_{i_0}=V_{j_0}$ and in $\mathscr{C}^2_{i_0}$, replacing the cliques $N[x,B], x\in\{x_{i_0},y_{i_0},x_{j_0},y_{j_0}\}$ with the cliques $U_{i_0}\cup \{x_{i_0},y_{j_0}\}$ and $V_{i_0}\cup \{y_{i_0},x_{j_0}\}$ yields a clique covering of size at most $n-1$. This proves (4). \vsp 
	
	Hence, if $ k\geq 3 $, then by (2), (3) and (4), we are done. Now, assume that $ k=2 $. 
	Note that for every $i\in\{1,2\}$, $\mathscr{C}^1_i$ is a clique covering for $G$. Also, adding the clique $Y$ to $\mathscr{C}^2_i$  yields a clique covering for $G$. Thus, if there exists some $i\in\{1,2\}$, such that either $|V_i|\leq |\tilde{A}|$, or $|\tilde{A}|\leq |U_i|-4$,  then $\cc(G)\leq n-1$. Hence, suppose that for every $i\in\{1,2\}$,
	\[|\tilde{A}|+1\leq  |V_i|\leq |U_i|\leq  |\tilde{A}|+3.\]
	Now, let $B_1=U_1\cap V_2$, $B_2=U_1\cap U_2$, $B_3=V_1\cap U_2$ and $B_4=V_1\cap V_2$. 
	Also, for every $i\in\{1,2,3,4\}$, let $\hat{B}_i\subseteq B_i$ be the set of vertices in $B_i$ which are anticomplete to $B_{i+2}$ (reading $ i+2 $ modulo $4$). In the sequel, we need the following three facts. \vsp
	
	(5) \textit{For every $ i\in\{1,2\} $, $ B_i $ is not complete to $ B_{i+2} $ and the edges $x_1x_2, x_1y_2, y_1x_2$ and $y_1y_2$ are covered by the cliques in $\mathscr{N}[B;A]$. Also, for every $ {v}\in \tilde{A} $, if $ N({v},B) $ is not a clique, then all edges incident with $ {v} $ are covered by the cliques in $\mathscr{N}[B;A]$.} \vsp
	
	To see the first claim, note that if say $ B_1 $ is complete to $ B_3 $,  then $ V(G) $ is the union of three cliques $ B_1\cup B_2\cup B_3\cup \{a_0\} $, $ B_4\cup Y $ and $ \tilde{A}\cup X $, a contradiction. Therefore,  for every $ i\in\{1,2\} $, $ B_i $ is not complete to $ B_{i+2} $ and thus, all sets $B_i$, $1\leq i\leq 4$, are nonempty. Hence, the four edges $x_1x_2, x_1y_2, y_1x_2$ and $y_1y_2$ are covered by the cliques in $\mathscr{N}[B;A]$.
	Now, assume that the vertex $v\in \tilde{A}$ is adjacent to vertices $b,b'\in B$, where $ b $ and $ b' $ are nonadjacent. Let $ v'\in A $ be a vertex distinct from $ v $. Since $G$ is claw-free, $v'$ is adjacent to at least one of $b,b'$ and thus the edge $vv'$ is covered by the clique $N[b,A]$ or $N[b',A]$. This proves (5). \vsp
	
	(6) \textit{For every $ v\in \tilde{A} $, either $ v $ is anticomplete to $ B $, or there is some $ i\in\{1,2,3,4\} $ such that $ v $ is complete to $ \hat{B}_i\cup \hat{B}_{i+1} $ $($reading $ i+1 $ modulo $4)$.}\vsp
	
	Assume that $ v\in \tilde{A} $ and for every $ i\in\{1,2,3,4\} $, $ v $ is not complete to $ \hat{B}_i\cup \hat{B}_{i+1} $. Thus, there exists some $ i\in\{1,2\} $ such that $ v $ has two non-neighbours in $ \hat{B}_i $ and $ \hat{B}_{i+2} $, say  $ b_1\in \hat{B}_1 $ and $ b_3\in \hat{B}_3 $. Therefore, $ v $ is anticomplete to $ B_2\cup B_4 $ (since if say $ v $ has a neighbour $ b $ in $ B_2\cup B_4 $, then $ \{b,b_1,b_3,v\} $ would be a claw). Also, since by (5), $ B_2 $ is not complete to $ B_4 $, by a similar argument, $ v $ is anticomplete to $ B_1\cup B_3 $ and so $ v $ is anticomplete to $ B $.  This proves (6). \vsp
	
	Now, let $\mathscr{C}'$ be the collection of cliques obtained from $\mathscr{C}^2_1$ by removing all the cliques in $ \mathscr{N}[A\setminus \tilde{A};B] \cup \mathscr{N}[X;Y\cup \tilde{A}]$  and adding the cliques $B_1\cup \{x_1,y_2\}$, $B_2\cup \{x_1,x_2\}$, $B_3\cup \{x_2,y_1\}$, $B_4\cup \{y_1,y_2\}$, $ Y\cup \tilde{A} $, $ U_1 $ and $V_1$. Note that $ \mathscr{C}' $ is a clique covering for $ G $ of size $ 2|\tilde{A}|+|V_1|+8 $.
	Also, we have \vsp
	
	(7) \textit{If there exists $ v_0\in \tilde{A} $ which is complete to $ \hat{B}_1\cup \hat{B}_2 $, then $ \mathscr{C}'\setminus \{U_1\} $ is a clique covering for $ G $.} \vsp
	
	Since $ U_1=B_1\cup B_2 $, it is enough to show that edges in $ E(B_1,B_2) $ are covered by the cliques in $ \mathscr{C}'\setminus \{U_1\}  $. To see this, note that the cliques in $\mathscr{N}[V_1;U_1\cup \{a_0\}]$ cover all edges between $B_1$ and $B_2$ except the edges between $\hat{B}_1$ and $\hat{B}_2$, while these edges are covered by the clique $N[v_0,U_1]$. This proves (7).\vsp
	
	Now, we complete the proof for $k=2$ through the following four possibilities. \vsp
	
	(8) \textit{Assume that $k=2$ and for every vertex $v\in \tilde{A}$, $N(v,B)$ is not a clique of $G$. Then $ \cc(G)\leq n-1 $.} \vsp
	
	First, note that by (5), all edges whose both endpoints are in $A$ are covered by  the cliques in $\mathscr{N}[B; A]$. Hence, removing the cliques in $\mathscr{N}[X_i;Y_i\cup \tilde{A}]$ from $\mathscr{C}^1_i$ leads to a clique covering for $G$ of size $|U_i|+2|V_i|+2$. Therefore, if there exists some $i\in\{1,2\}$ such that $|V_i|\leq |\tilde{A}|+2$, then $\cc(G)\leq n-1$. Now, assume that for every $i\in\{1,2\}$, $|U_i|=|V_i|=|\tilde{A}|+3$. Thus, $ |B_1|=|B_3| $ and $ |B_2|=|B_4| $. If $ \tilde{A} $ is empty, then remove the cliques $ X\cup \tilde{A} $ and $ Y\cup \tilde{A} $ from $ \mathscr{C}' $ to get a clique covering  of size $ |V_1|+6= n-2 $. Now, assume that $ \tilde{A} $ is nonempty and let $ v_0\in \tilde{A} $. Due to (6), assume w.l.o.g. that $ v_0 $ is complete to $ \hat{B}_1\cup \hat{B}_2 $.  Thus, by (7), $ \mathscr{C}'\setminus \{U_1\} $ is a clique covering of size $ 2|\tilde{A}|+|V_1|+7=n-1 $. This proves (8). \vsp
	
	(9) \textit{Assume that $k=2$ and there are at least two vertices $v,v'\in \tilde{A}$ such that  $N(v,B)$ and $N(v',B)$  are both cliques of $G$. Then $ \cc(G)\leq n-1 $.} \vsp
	
	First, assume that $ N(v,B) $ is empty. In this case, in the clique covering $\mathscr{C}^2_i$, one can replace the cliques $ N[v,U_i],N[v,V_i], N[v',U_i]$ and $N[v',V_i] $ with the cliques $ N[v',B]$ and $Y$ to obtain a clique covering of size $2|\tilde{A}|+|{V}_i|+5\leq |\tilde{A}|+|{U}_i|+|V_i|+4 = n-1$. Now, assume that $ N(v,B) $ is nonempty. If there exists some $i\in\{1,2\}$ such that $|U_i|\geq |\tilde{A}|+2$, then adding the clique $ N[v,B] $ to the above collection yields a clique covering of size at most $ n-1 $. Finally, assume that for every $i\in\{1,2\}$, $|U_i|=|V_i|=|\tilde{A}|+1$. By (6), w.l.o.g. we may assume that $ v $ is complete to $ \hat{B}_1\cup \hat{B}_2 $. Thus, by (7), the collection obtained from  $\mathscr{C}'$ by removing the clique $ U_1 $ and merging the pairs $ (N[v,U_1],N[v,V_1])$ and $(N[v',U_1],N[v',V_1]) $  is a clique covering for $G$ of size $2|\tilde{A}|+|V_1|+5=n-1$. This proves (9).\vsp
	
	(10) \textit{Assume that $k=2$ and there exists a vertex  $v_0\in \tilde{A}$ such that $N(v_0,B)$ is a nonempty clique of $G$ and for every vertex $v \in\tilde{A}\setminus \{v_0\}$, $N(v,B)$ is not a clique of $G$. Then $ \cc(G)\leq n-1 $.}\vsp 
	
	Assume w.l.o.g. that $v_0$ has a neighbour in $B_1$.  First, note that in the collection $ \mathscr{C}^2_i $, merging the pair $(N[v_0,U_i],N[v_0,V_i])$  and adding the clique $Y$ yield a clique covering for $ G $ of size $2|\tilde{A}|+|{V}_i|+7$. Thus, if there exists some $i\in\{1,2\}$, such that $|U_i|=|\tilde{A}|+3$, then $\cc(G)\leq n-1$. Moreover, by (5),  the cliques in $\mathscr{N}[B;A]$ cover all the edges in $ E(A)\setminus \{v_0x_2,v_0y_1\} $  (note that since $v_0$ has a neighbour in $B_1$, the edges $v_0x_1$ and $v_0y_2$ are covered). Thus, replacing the cliques in $\mathscr{N}[X_i; Y_i\cup \tilde{A}]$ with the clique $\tilde{A}\cup \{x_2,y_1\}$ in $\mathscr{C}^1_i$ leads to a clique covering for $G$ of size $|U_i|+2|V_i|+3$. Thus, if there exists some $i\in\{1,2\}$, such that $|V_i|=|\tilde{A}|+1$, then $\cc(G)\leq n-1$. Consequently, assume that for every $i\in\{1,2\}$, $|U_i|=|V_i|=|\tilde{A}|+2$. Thus, $|B_1|=|B_3|$ and $|B_2|=|B_4|$. By (6), w.l.o.g. we may assume that $ v_0 $ is complete to $ \hat{B}_1\cup \hat{B}_2 $. Thus, by (7), the collection obtained from  $\mathscr{C}'$ by removing $ U_1 $ and merging the pair $(N[v_0,U_1], N[v_0,V_1]) $ is a clique covering for $G$ of size $2|\tilde{A}|+|V_1|+6=n-1$. This proves (10). \vsp
	
	(11) \textit{Assume that $k=2$ and there exists a vertex $v_0\in \tilde{A}$, such that $N(v_0,B)$ is empty and for every vertex $v \in\tilde{A}\setminus \{v_0\}$, $N(v,B)$ is not a clique of $G$. Then $\cc(G)\leq n$ and equality holds if and only if $\overline{G}$ is isomorphic to a twister.} \vsp
	
	Note that removing the cliques $N[v_0,U_i]$ and $N[v_0,V_i]$ from $\mathscr{C}^2_i$ and adding the clique $Y$ yield a clique covering for $ G $ of size $2|\tilde{A}|+|{V}_i|+6$. Thus, if there exists some $i\in\{1,2\}$, such that $|\tilde{A}|+2\leq |U_i|$, then $\cc(G)\leq n-1$. Hence, suppose that for every $i\in \{1,2\}$, $|U_i|=|V_i|=|\tilde{A}|+1$. First, assume that $|\tilde{A}|\geq 2$ and let $v\in \tilde{A}\setminus \{v_0\}$. By (6), assume w.l.o.g. that $v$ is complete to $\hat{B}_1\cup \hat{B}_2$. Thus, by (7), the collection of cliques obtained from $\mathscr{C}'$ by removing the cliques $ U_1 $, $ N[v_0,U_1] $ and $N[v_0,V_1]$ is a clique covering for $G$ of size $2|\tilde{A}|+|V_1|+5=n-1$. Finally, assume that $|\tilde{A}|=1$. Then, $|B_1|=|B_2|=|B_3|=|B_4|=1$, say $ B_i=\{b_i\} $, $ i=1,2,3,4 $, and so $|V(G)|=10$. Therefore, by (5), for every $i\in\{1,2\}$, $B_i$ is anticomplete to $B_{i+2}$.  Hence, $G$ is isomorphic to the complement of a twister (by mapping the vertices $ a_0,v_0,b_1,b_2,b_3, b_4, x_1,x_2,y_1,y_2 $ of $ G $ to the vertices $ u_1,u_2,v_4,v_6,v_8,v_2,v_1,v_3,v_5,v_7 $ of the complement of a twister in \pref{fig:twister}, respectively). Thus, by \pref{lem:twister}, $\cc(G)=10$. This proves (11). \vsp
	
	Finally, (8), (9), (10) and (11) imply the assertion for $k=2$. This proves \pref{lem:2sub}.
\end{proof}
Now, we are ready to complete the proof of \pref{thm:oap} providing appropriate clique covering for non-$3$-substantial antiprismatic graphs.
\begin{thm} \label{thm:3sub}
Let $G$ be an antiprismatic graph on $n$ vertices which is neither $3$-substantial nor three-cliqued and  contains a triad. Then $\cc(G)\leq n$ and equality holds if and only if $\overline{G}$ is isomorphic to a twister. 
\end{thm}
\begin{proof} 
	If $G$ is not $2$-substantial, then there exists a vertex $u$ meeting all triads of $ G $. Let $ A $ be the set of non-neighbours of $ u $ and let $ \{u,v,w\} $ be a triad of $ G $. Then, $ N(v,V(G)\setminus A) $ is a clique (otherwise, $ v $ is adjacent to $ x,y \in V(G)\setminus A $, where $ x,y $ are nonadjacent and then $ \{x,y,w\}$ is a triad not containing $ u $, a contradiction). Therefore, the result follows from \pref{lem:2sub}.
	Now, assume that $G$ is $2$-substantial and not $3$-substantial. Then, 
	there exist two vertices $u,v$ meeting all triads. We prove the theorem in the following two cases. \vsp
	
	(1) \textit{If $u$ and $v$ are nonadjacent, then $ \cc(G)\leq n $ and equality holds if and only if $ \overline{G} $ is isomorphic to a twister.}\vsp
	
	If $u$ and $v$ are not contained in a common triad, then the result follows from \pref{lem:2sub}. 
	Thus, assume that $u$ and $v$ are contained in a triad, say $T=\{u,v,w\}$. Let $A,B,C$ be the sets of non-neighbours of $u,v,w$ in $V(G)\setminus T$, respectively, which partition $V(G)\setminus T$. By the assumption, $ C $ is a clique of $ G $. If $C=\emptyset$, then
	$ G $ is an extension of a member of $ \mathcal{F}_1 $ that will be handled in \pref{lem:F1} (for the definition of $\mathcal{F}_1$ and the extension, see \pref{app:noap} and \pref{sub:Inc-rotator-Exc-sf}, respectively). So, suppose that $C$ is nonempty. Also, let $A=X\cup Y\cup A'$ and $B=Z\cup W\cup B'$, where $X=\{x_1,\ldots,x_k\}$, $Y=\{y_1,\ldots, y_k\}$, $Z=\{z_1,\ldots, z_l\}$ and $W=\{w_1,\ldots, w_l\}$ and the only non-edges in $A$ are $x_iy_i$, $1\leq i\leq k$, and the only non-edges in $B$ are $z_iw_i$, $1\leq i\leq l$. Also, since $G$ is $2$-substantial, $k,l\geq 1$. Assume w.l.o.g. that $x_1$ is complete to $Z$  and anticomplete to $W$, $y_1$ is complete to $W$ and anticomplete to $Z$, $z_1$ is complete to $X$ and anticomplete to $Y$ and $w_1$ is complete to $Y$ and anticomplete to $X$. Moreover, let $X'= N(z_1,A')$, $Y'=N(w_1,A')$, $Z'=N(x_1,B')$, $W'=N(y_1,B')$, $Z''=N(x_1,C)$ and $W''=N(y_1,C)$. Thus, $A'=X'\cup Y'$, $B'=Z'\cup W'$ and $C=Z''\cup W''$. Since $x_i,y_i$ are contained in no triad than $\{x_i,y_i,u\}$, $N(x_i,B\cup C)$ and $N(y_i,B\cup C)$ are cliques of $ G $. Thus, $Z\cup Z'\cup Z''$ and $W\cup W'\cup W''$ are cliques. Similarly, $X\cup X'\cup Z''$ and $Y\cup Y'\cup W''$ are cliques. Moreover, assume w.l.o.g. that $k\leq l$ and $Z''$ is nonempty.
	Now, we claim that $X$ is complete to $Z$. Since $Z''$ is nonempty, let $z\in Z''$. If $x_i\in X$ is adjacent to some $w_j\in W$,  then $z$ is adjacent to $w_j$ (because $N(x_i,B\cup C)$ is a clique). Thus, $ G $ induces a claw on $ \{z,v,z_j,w_j\} $, a contradiction. Hence, $X$ is anticomplete to $W$  and thus complete to $Z$ and then $Y$ is anticomplete to $Z$ and complete to $W$. 
	Let $\tilde{X}\subseteq X'$ (resp. $\tilde{Y}\subseteq Y'$) be the set of all vertices in $X'$ (resp. $Y'$) which are complete to $Z$ (resp. $W$). Now, we claim that $X'\setminus \tilde{X}$ and $Y'\setminus \tilde{Y}$ are complete to $C$. For assume that $x\in X'\setminus \tilde{X}$ has a non-neighbour $ w''\in W'' $. By the definition, $ x $ also has a non-neighbour $ z_j\in Z $. Thus, $ \{x,w'',z_j\} $ is a triad disjoint from $ \{u,v\} $, a contradiction. Hence,  $X'\setminus \tilde{X}$ (and similarly  $Y'\setminus \tilde{Y}$) are complete to $ C $. 
	Now, define
	\begin{align*}
		\mathscr{C}=&\mathscr{N}[C;B\cup \{u\}]\cup \mathscr{N}[B';A\cup\{w\}] 
		\cup \mathscr{N}[(X'\setminus \tilde{X})\cup (Y'\setminus \tilde{Y});B\cup\{w\}]  \\
		& \cup \mathscr{N}[ \tilde{X};W''\cup Y\cup Y'\cup \{v\}] \cup \mathscr{N}[\tilde{Y};Z''\cup X\cup X'\cup \{v\}] \cup \mathscr{N}[X;Y] \cup  \mathscr{N}[Z;W\cup B'\cup \{w\}]  \\
		& \cup \{ C_1= X\cup X'\cup (Y'\setminus\tilde{Y}) \cup Z'' \cup \{v\}, 
		C_2= Y\cup Y'\cup (X'\setminus \tilde{X}) \cup W'' \cup \{v\}\} \\
		& \cup \{ C_3= X\cup \tilde{X}\cup Z \cup \{w\}, C_4= Y\cup \tilde{Y}\cup W\cup \{w\},
		C_5= C\cup\{v\}\}.
	\end{align*}
	
	If $W''=\emptyset$, then in $ \mathscr{C} $, replace $C_5$ with the clique $W\cup B'\cup \{u\}$. 
	If $ l\geq 2$, then it is easy to check that $\mathscr{C}$ is a clique covering for $G$ of size $|C|+|A'|+|B'|+k+l+5\leq n-1$, as desired. Now, assume that $k=l=1$ and thus, $\tilde{X}=X'$ and $\tilde{Y}=Y'$. In this case, let $ \mathscr{C}' $ be the collection obtained from $ \mathscr{C} $ by removing the degenerate clique in $ \mathscr{N}[X;Y] $. If $ W''= \emptyset $, then $ \mathscr{C}' $ is a clique covering of size $ n-1 $ and we are done. If $ W''\neq \emptyset $, then the only edges which are possibly not covered by $\mathscr{C}'  $ are the edges in $ E(W,Z') $. Thus, if $ Z' $ is empty, then again we are done. Finally, assume that $ W'' $ and $ Z' $ (and by symmetry all sets $ W', X' ,Y' $) are nonempty. In this case, replace four cliques $C_1,C_2,C_3,C_4,C_5$ in $\mathscr{C}'$ with five cliques $X\cup X'\cup Z\cup Z''$, $Y\cup Y'\cup W\cup W''$, $W\cup B'\cup \{w\}$, $A'\cup \{w\}$ and $ C\cup \{v\} $ to form a clique covering of size $n-1$. \vsp
	
	(2) \textit{If $u$ and $v$ are adjacent, then $ \cc(G)\leq n $ and equality holds if and only if $ \overline{G} $ is isomorphic to a twister.} \vsp
	
	If any triad on $u$ does not intersect any triad on $v$, then the result follows from \pref{lem:2sub}. Thus, assume that there exist two triads $T_1=\{u,w,x\}$ and $T_2=\{v,w,z\}$ and thus $ x $ is complete to $ \{z,v\} $ and $ z $ is complete to $ \{x,u\} $.
	Note that $V'=V(G)\setminus \{u,v,w,z,x\}$ is partitioned into five cliques $A= N_G(u,V')\cap N_{\overline{G}}(v,V')$, $B=   N_{\overline{G}} (u,V') \cap N_G(v,V') $, $C=N_{\overline{G}} (u,V')\cap N_{\overline{G}} (v,V')$, $D=N_{\overline{G}} (w,V')$ and $E=N_{G} (u,V')\cap N_{G} (v,V')\setminus D $. Thus, $u$ is complete to $A\cup D\cup E$, $v$ is complete to $B\cup D\cup E$, $w$ is complete to $A\cup B\cup C\cup E$, $x$ is complete to $B\cup C\cup D$ and $z$ is complete to $A\cup C\cup D$. Let $\tilde{A}\subseteq A$ and $\tilde{B}\subseteq B$ be the sets of vertices respectively in $A$ and $B$ which are complete to $C$. Also, let $\tilde{C}\subseteq C$ be the set of vertices in $C$ which are complete to $A\cup B$. It is evident that $\tilde{W}(G)=\tilde{A}\cup \tilde{B}\cup \tilde{C}\cup D\cup E$. We may assume w.l.o.g. that $\tilde{A}\neq A$, $\tilde{B}\neq B$ and $\tilde{C}\neq C$, otherwise if say $A=\tilde{A}$, then every triad of $G$ meets $\{u,w\}$ and $ u,w $ are nonadjacent and by (1), we are done.
	Moreover, note that since every triad of $G$ meets $\{u,v\}$, both $N_{\overline{G}}(z,V')= B\cup E$ and $N_{\overline{G}}(x,V')= A\cup E$ are cliques. Thus, since $G$ is antiprismatic, $E$ is anticomplete to $C\setminus \tilde{C}$. Moreover, since $G$ is antiprismatic and $N_{\overline{G}}(u,V')=B\cup C$, every vertex in $ B $ (resp. $ C $) has at most one non-neighbour in $ C $ (resp. $ B $) and for every vertex $y$ in $V'\setminus (B\cup C)$, $N(y, B\cup C)$ is a clique. Similarly, every vertex in $ A $ (resp. $ C $) has at most one non-neighbour in $ C $ (resp. $ A $) and  for every vertex $y$ in $V'\setminus (A\cup C)$, $N(y, A\cup C)$ is a clique. We prove (2) in two cases. 
	
	Firstly, assume that $|C\setminus \tilde{C}|\geq 2$. Also, assume w.l.o.g. that $|A\setminus \tilde{A}|\geq |B\setminus \tilde{B}|$. Thus, every vertex in $B$ has a non-neighbour in $A\setminus \tilde{A}$. Now, we claim that $E$ is complete to $D$. For contrary, suppose that $d\in D$ is not adjacent to $e\in E$. The facts that $E$ is anticomplete to $C\setminus \tilde{C}$ and the induced subgraph of $G$ on $C\cup D\cup E$ is triad-free yield that $d$ is complete to $C\setminus\tilde{C}$. Therefore, since $G$ is antiprismatic, $d$ is anticomplete to $B\setminus \tilde{B}$ and $A\setminus \tilde{A}$. Thus, there exists a triad in $A\cup B\cup D$, a contradiction. This shows that $E$ is complete to $D$.
	Now, let $D'\subseteq D$ be the set of vertices in $D$ which are complete to $B$ and define
	\begin{align*}
		\mathscr{D}=& \mathscr{N}[A;D\cup \{z,u\}] \cup \mathscr{N}[B;A\cup C \cup \{w\}] \cup
		\mathscr{N}[C;A\cup \{w\}] \\
		& \cup \mathscr{N}[E;B\cup C\cup \{w\}] \cup  \mathscr{N}[D\setminus D';B\cup C\cup \{x\}] \cup  \mathscr{N}[D';C\cup \{x,z\}] \\
		& \cup \{C\cup \{x,z\}, B\cup D'\cup \{x,v\}, D\cup E\cup \{u,v\}, A\cup E\cup \{w\}\}.
	\end{align*}
	Note that every vertex in $D\setminus D'$ has a neighbour in $A$, otherwise there would be a triad in $A\cup B\cup D$. Therefore, the edges in $ E(\{z\},D) $ are covered by the cliques in $ \mathscr{N}[A;D\cup \{z,u\}] \cup \mathscr{N}[D';C\cup \{x,z\}] $. This implies that $\mathscr{D}$ is a clique covering for $G$ of size $|A|+|B|+|C|+|D|+|E|+4=n-1$.
	
	Secondly, assume that $|C\setminus \tilde{C}|=1$ and $ C\setminus \tilde{C}=\{c_0\} $. Thus, $|B\setminus \tilde{B}|=|A\setminus \tilde{A}|=1$ and $A\setminus \tilde{A}$ is complete to $B\setminus \tilde{B}$. Assume w.l.o.g. that $|\tilde{B}|\leq |\tilde{A}|$. Also, let $D_1\subseteq D$ be the set of vertices in $D$ which are adjacent to $c_0$ and let $D_2=D\setminus D_1$. By above argument, $D_2$ is complete to $E\cup (B\setminus \tilde{B}) \cup (A\setminus \tilde{A})$ and $D_1$ is complete to $\tilde{A}\cup  \tilde{B}$ (if say $ d_1\in D_1 $ is nonadjacent to $ \tilde{a}\in \tilde{A} $, then $ \{\tilde{a}, d_1\} \cup B\setminus \tilde{B} $ is a triad). Now, define
	\begin{align*}
		\mathscr{D}'=& \mathscr{N}[A;D\cup \{z,u\}] \cup \mathscr{N}[B;A\cup C \cup \{w\}] \cup  \mathscr{N}[C;D\cup \{z,x\}]\\
		& \cup \mathscr{N}[E;B\cup C\cup \{w\}] \cup \mathscr{N}[D_1; E\cup D_2\cup \{u,v\}] \cup \mathscr{N}[D_2;B\cup C\cup \{x\}] \\
		& \cup \{C\cup \tilde{A}\cup \{w\},D_2\cup E\cup \{u,v\}, B\cup \{x,v\}, A\cup E\cup \{w\}\}.
	\end{align*}
	
	It is easy to verify that the cliques in $\mathscr{D}'$ cover all edges of $ G $ except the edges in $E(D_1,\tilde{B})$ (note that the edges in $ E(D_1\cup\{z,x\}) $ are covered by the cliques in $ \mathscr{N}[C;D\cup \{z,x\}] $ and the edges in $ E(A\setminus \tilde{A}, C) $ are covered by the cliques in $ \mathscr{N}[B;A\cup C\cup \{w\}] $). Thus, if $\tilde{B}=\emptyset$, then $\mathscr{D}'$ is a clique covering for $G$ of size $n-1$. 
	Now, assume that $ \tilde{B} \neq \emptyset $. Note that $ \tilde{C}\neq \emptyset $, because if $ \tilde{C}=\emptyset $, then $ V(G) $ is the union of three cliques $ A\cup \{z,u\} $, $ \tilde{B}\cup D_1\cup \{x,c_0\} $ and $ D_2\cup E\cup (B\setminus \tilde{B})\cup \{v\} $, a contradiction. Now, let $ \mathscr{D}'' $ be the collection obtained from $ \mathscr{D}' $ by replacing the cliques $N[c_0,D\cup \{z,x\}]=D_1\cup\{c_0,z,x\}$ and  $C\cup \tilde{A}\cup \{w\}$  with the cliques  $ D_1\cup \tilde{B} \cup\{c_0,x\} $ and $ C\cup \tilde{A}\cup \{z\}$. The facts that $\tilde{B}$ and thus $\tilde{A}$ are nonempty, $ D_1 $ is complete to $ \tilde{A} $ and $ D_2 $ is complete to $ A\setminus \tilde{A} $, imply that the edges in $ E(\{z\},D) $ are covered by the cliques in $\mathscr{N}[A;D\cup \{z,u\}]$ and the edges in $ E(\{w\},C) $ are covered by the cliques in $\mathscr{N}[B;A\cup C \cup \{w\}]$. Also, since $\tilde{C}\neq \emptyset$, the edge $zx$ is covered by the cliques in  $\mathscr{N}[\tilde{C};D\cup \{z,x\}]$. Hence, $ \mathscr{D}'' $ is a clique covering for $G$ of size $n-1$. 
\end{proof}
Now, we are ready to prove of \pref{thm:oap}.
\begin{proof}[{\rm\textbf{Proof of \pref{thm:oap}.}}]
	Let $G$ be an orientable antiprismatic graph on $n$ vertices containing a triad. If $ G $ is three-cliqued, then the result follows from \pref{thm:3cap}. If $ G $ is neither three-cliqued nor 3-substantial, then by \pref{thm:3sub}, $ \cc(G)\leq n-1 $ (note that the complement of a twister is not an orientable antiprismatic graph, see \cite{seymour1}). Now, assume that $ G $ is 3-substantial and not three-cliqued. Thus, by \pref{thm:orient-anti}, $ G $ is either the complement of a cycle of triangles graph, or the complement of a ring of five, or the complement of a mantled $L(K_{3,3})$. Hence, by Lemmas~\ref{lem:cycle}, \ref{lem:mantled} and \ref{lem:ring}, $ \cc(G)\leq n-1 $.
\end{proof}

\section{Non-orientable antiprismatic graphs} \label{sec:noap}
The main goal of this section is to prove \pref{thm:noapA}, which we restate as follows.
\begin{thm}\label{thm:noap}
	Let $G$ be a non-orientable antiprismatic graph on $n$ vertices. Then $\cc(G)\leq n$ and equality holds if and only if $\overline{G}$ is isomorphic to a twister.
\end{thm}
Chudnovsky and Seymour in \cite{seymour2} proved that every non-orientable antiprismatic graph can be obtained from graphs in some explicitly described classes, whose union is called the \textit{menagerie}, by applying a certain operation, what we call ``extension'' in the sequel. Unfortunately, despite of the other constructions like thickening and worn hex-chain, we are mostly unable to extend a clique covering of a graph to a clique covering of its extensions. On the other hand, the graphs in each of the classes of the menagerie require an exclusive method to handle their clique covering and the worse is that for most of these classes, the process is divided into several cases that should be treated separately. For these reasons, we find it best not to use the structure theorem of \cite{seymour2} directly and instead follow their approach in finding the structure theorem. More details will appear later on. 
Also, in \cite{seymour2}, everything is stated for prismatic graphs, so we have to reformulate the definitions and statements in terms of the complements and for the sake of convenience, we maintain the titles of the classes.  
Furthermore, note that all three-cliqued antiprismatic graphs are orientable (see 4.2 in \cite{seymour1}) which have already been handled in previous section. Therefore, for simplicity, the graphs are generally assumed not to be three-cliqued.

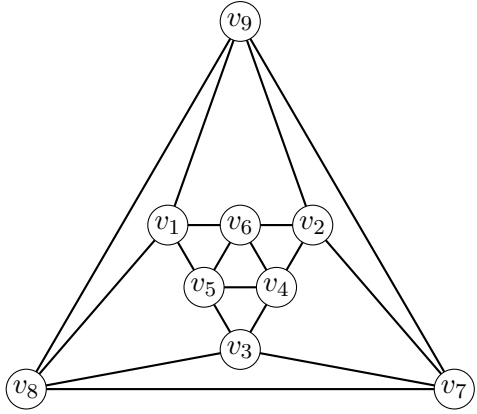
\begin{figure}[t]
	\begin{center}
		\begin{tikzpicture}
		\begin{scope}[scale=.5]
		\SetVertexNoLabel
		\tikzset{VertexStyle/.style ={shape=circle,inner sep=1mm,minimum size=15pt,draw,fill=white}}
		\begin{scope}[rotate=-30]
		\grCycle[RA=6.5,prefix=v]{3}
		\end{scope}
		\begin{scope}[rotate=30]
		\grCycle[RA=2.2,prefix=u]{3}
		\end{scope}
		\begin{scope}[rotate=-30]
		\grCycle[RA=1.1,prefix=w]{3}
		\end{scope}
		\AssignVertexLabel{v}{$v_7$,$v_9 $,$v_8$}
		\AssignVertexLabel{u}{$v_2 $,$v_1$,$v_3$}
		\AssignVertexLabel{w}{$v_4$,$v_6 $,$v_5$}
		\draw[-,thick] (u0) to (v0);
		\draw[-,thick] (u0) to (v1);
		\draw[-,thick] (u1) to (v1);
		\draw[-,thick] (u1) to (v2);
		\draw[-,thick] (u2) to (v2);
		\draw[-,thick] (u2) to (v0);
		\end{scope}
		\end{tikzpicture}
	\end{center}
	\vspace{-10pt}
	\caption{The complement of a rotator.}
	\label{fig:rotator}
\end{figure}

Let us begin with recalling a couple of definitions from \cite{seymour2}. Consider a graph with the vertex set $\{v_1,\cdots , v_9\}$ and edges as follows; $\{v_4, v_5, v_6\}$ and $\{v_7, v_8, v_9\}$ are both triangles, and for $i = 1, 2, 3$, $v_i$ is complete to $\{v_4,\ldots,v_9\}\setminus \{v_{i+3}, v_{i+6}\}$. This graph is the complement of a \textit{rotator}, defined in \cite{seymour2} (see \pref{fig:rotator}). Also, a pair $(u,v)$ of vertices of a graph $G$ is said to be a \textit{square-forcer} in $ G $, if $u,v$ are nonadjacent and the set of vertices nonadjacent to both $u,v$ is stable. As a first step to prove \pref{thm:noap}, we are going to show that the complement of a counterexample to \pref{thm:noap} contains an induced rotator and also excludes a square-forcer.

\begin{thm}\label{thm:Inc-rotator-Exc-sf}
	Let $G$ be a counterexample to \pref{thm:noap}. Then $\overline{G}$ contains an induced rotator and contains no square-forcer.
\end{thm}
The proof of \pref{thm:Inc-rotator-Exc-sf}, given in \pref{sub:Inc-rotator-Exc-sf}, uses a decomposition theorem in \cite{seymour2} for those non-orientable antiprismatic graphs whose complement either contain no induced  rotator or contain a square-forcer. For the case of non-orientable antiprismatic graphs whose complement contain an induced rotator and no square-forcer, there is another decomposition theorem in \cite{seymour2}. Nevertheless, we prefer not to use this theorem, and in Subsections~\ref{sub:Inc-rot}, \ref{sub:types}, \ref{sub:schlafli}, \ref{sub:max=3} and \ref{sub:2,2}, we tackle these graphs by a structural exploration similar to the one in \cite{seymour2}, thereby completing the proof of \pref{thm:noap}.

\subsection{Proof of \pref{thm:Inc-rotator-Exc-sf}}\label{sub:Inc-rotator-Exc-sf}
In order to prove \pref{thm:Inc-rotator-Exc-sf}, we need a result from \cite{seymour2}. We begin with a definition. Given a graph $G$ and a vertex $v\in V(G)$, by \textit{replicating} $ v $, we mean replacing the vertex $ v $ with a nonempty clique $ X_v $ such that every vertex in $ X_v $ is adjacent to a vertex $ u\in V(G) \setminus \{v\} $ if and only if $ v $ is adjacent to $ u $. Also, we say that the graph $ H $ is an \textit{extension} of $ G $, if $ H $ is obtained from $ G $ by replicating the vertices in $ \tilde{W}(G) $ and then adding edges between some pairs of nonadjacent vertices in $\cup_{v\in \tilde{W}(G)} X_v$. 
The following theorem is a direct consequence of 2.2 and 7.2 in \cite{seymour2} and determines the structure of non-orientable antiprismatic graphs whose complement excludes an induced rotator (see \pref{app:noap} for the definition of the classes $\mathcal{F}_0,\mathcal{F}_1,\mathcal{F}_2,\mathcal{F}_3$ and $\mathcal{F}_4$). 

\begin{thm}{\rm \cite{seymour2}}\label{thm:exc-rotator-seymour}
	Let $ G $ be a non-orientable antiprismatic graph such that $ \overline{G} $ contains no induced  rotator. Then, $ G $ is an extension of a member of $\mathcal{F}_0\cup \mathcal{F}_1\cup \mathcal{F}_2\cup \mathcal{F}_3\cup \mathcal{F}_4$.
\end{thm}

\begin{rem}\label{rem:7.2seymour}
	In fact, \pref{thm:exc-rotator-seymour} is the corrected version of 7.2 in \cite{seymour2}, whose proof is flawed. More precisely, in Statement~(6) in the proof, a number of graphs on between $12$ and $18$ vertices arise which are claimed to be isomorphic to a subgraph of the Schl\"{a}fli graph (see \pref{sub:Inc-rot} for the definition) by introducing an isomorphism map. Nevertheless, this map is actually not an isomorphism. The worse is that in some cases, these graphs are not even isomorphic to any subgraph of the Schl\"{a}fli graph. To correct the proof, we introduced the additional class $\mathcal{F}_0$ in \pref{app:noap}, including these graphs and consequently \pref{thm:exc-rotator-seymour} is valid in its present form.
\end{rem}

Now, \pref{thm:exc-rotator-seymour} together with 2.2, 5.1 and 8.2 in \cite{seymour2} implies the following theorem (for the definition of the graphs of parallel-square type and skew-square type, see \pref{app:noap}). Note that if $G$ is an extension of an antiprismatic graph $H$, where $\overline{G}$  contains an induced rotator and a square-forcer, then $\overline{H}$ also contains an induced rotator and a square-forcer (we leave the reader to verify it, also see the proof of 8.1 in \cite{seymour2}). 

\begin{thm}{\rm \cite{seymour2}} \label{thm:Seymour-Inc-rotator-Exc-sf-revised}
	Let $G$ be a non-orientable antiprismatic graph such that $ \overline{G} $ either contains no induced rotator or contains a square-forcer. Then $G$ is an extension of either a graph of parallel-square type, or a graph of skew-square type, or a member of  $\mathcal{F}_0\cup \mathcal{F}_1\cup \mathcal{F}_2\cup \mathcal{F}_3 \cup \mathcal{F}_4$.
\end{thm}

Through the following seven lemmas, we investigate the clique cover number of the extensions of the graphs in these classes, starting with the graphs of  parallel-square type. Indeed, in the following lemma, we prove something more which will be also used in the proof of \pref{lem:F2,3}.
\begin{lem} \label{lem:parallel}
	Let $ H_0 $ be a graph of parallel-square type with $V(H_0)=A\cup B\cup C\cup D\cup Z\cup \{u,v,x,y\}$ as in the definition. Assume that $H$ is obtained from $ H_0 $ by possibly deleting the vertex $u$. Also, if $u\notin V(H)$, then there exist vertices $a_{i_0}\in A$, $b_{j_0}\in B$, $c_{k_0}\in C$ and $d_{l_0}\in D$  such that $a_{j_0},a_{k_0}\notin A$, $a_{l_0}\notin A\setminus \{a_{i_0}\}$, $b_{i_0},b_{l_0}\notin B$,$b_{k_0}\notin B\setminus \{b_{j_0}\}$, $c_{i_0},c_{l_0}\notin C$, $c_{j_0}\notin C\setminus \{c_{k_0}\}$, $d_{j_0},d_{k_0}\notin D$ and $d_{i_0}\notin D\setminus \{d_{l_0}\}$. Also, suppose that $H$ is not three-cliqued. Then, for every extension $ G $ of $H$ on $n$ vertices, we have $\cc(G)\leq n-1$.
\end{lem}
\begin{proof}
	Define $U=V(H)\cap \{u\}$. Let $\tilde{A}\subseteq A, \tilde{B}\subseteq B, \tilde{C}\subseteq C$ and $\tilde{D}\subseteq D$ be the sets of vertices which are respectively complete to $B\cup D, A\cup C, B$ and $A$. It is easy to check that $\tilde{W}(H)\subseteq \tilde{A}\cup \tilde{B}\cup \tilde{C}\cup \tilde{D}$ (for instance, if say $u\in \tilde{W}(H)$, then $C\cup D\cup Z\cup \{x\}$ is a clique and consequently,  $V(H)$ is the union of three cliques $A\cup \{u,y\},B\cup \{v\}$ and $C\cup D\cup Z\cup \{x\}$, a contradiction). For every $v\in V(H)$, let $X_v\subseteq V(G)$ be as in the definition of the extension, if $v\in \tilde{W}(H)$, and let $X_v=\{v\}$, otherwise. Also, for every $S\subseteq V(H)$, let $X_S=\cup _{s\in S}X_s$ (note that $X_Z=Z$, $X_U=U$ and $v,x,y\in V(G)$). We need the following.\vsp
	
	(1) \textit{All the members of $\mathscr{N}_{G}[X_A;X_B\cup X_C]$, $\mathscr{N}_{G}[X_B;X_A\cup X_D]$, $\mathscr{N}_{G}[X_C;X_A\cup X_B]$, $\mathscr{N}_{G}[X_C;X_D\cup X_A]$, $\mathscr{N}_{G}[X_D;X_A\cup X_B]$ and $\mathscr{N}_{G}[X_D;X_B\cup X_C]$ are cliques of $G$.}\vsp
	
	For if say there exists $a\in X_A$ adjacent to $b\in X_B$ and $c\in X_C$, where $b,c$ are nonadjacent, then $c\in W(H)$, say $c=c_j$, and so $a=a_j\in W(H)$ and $b=b_j\in W(H)$. This implies that $a$ is nonadjacent to $b$, a contradiction. This proves (1). \vsp
	
	Now, we conclude\vsp
	
	(2) \textit{If $|U|=|Z|=1$, then $\cc(G)\leq n-1$.}\vsp
	
	By (1), the family $\mathscr{N}_{G}[X_A;X_B\cup X_C\cup Z]\cup \mathscr{N}_{G}[X_B;X_C\cup Z]\cup \mathscr{N}_{G}[X_C;X_D\cup Z]\cup \mathscr{N}_{G}[X_D;X_A\cup X_B\cup Z]\cup \{A\cup \{u,y\},B\cup \{u,v\},C\cup\{v,x\},D\cup\{x,y\}\}$ is a clique covering for $G$ of size $|X_A|+|X_B|+|X_C|+|X_D|+4=n-1$. This proves (2).\vsp
	
	Moreover, we have\vsp
	
	(3) \textit{If either $U$ or $Z$ is empty, then $E_H(A,C)$ and $E_H(B,D)$ are both nonempty.}\vsp
	
	For if say $E_H(A,C)=\emptyset$, then every vertex in $B$ is complete to either $A$ or $C$, and defining $B'\subseteq B$ as the set of vertices in $B$ that are complete to $A$, implies that $V(H)$ is the disjoint union of three cliques $A\cup B'\cup U\cup Z, (B\setminus B')\cup C\cup \{v\}$ and $D\cup \{x,y\}$, a contradiction. This proves (3).\vsp
	
	Now, by (2), we may suppose that either $U$ or $Z$ is empty. Thus, by (3), we may assume that $E_H(A,C)=\{a_{i_1}c_{i_1},\ldots, a_{i_{\ell}}c_{i_{\ell}}\}$ and $E_H(B,D)=\{b_{j_1}d_{j_1},\ldots, b_{j_{\ell '}}d_{j_{\ell '}}\}$ for some $\ell,\ell '\geq 1$. Henceforth, w.l.o.g. assume that $|X_A|\leq |X_B|$, if $|X_A|=|X_B|$, then $|X_{\{c_{i_1},\ldots,c_{i_{\ell}}\}}|\geq |X_{\{d_{j_1},\ldots,d_{j_{\ell'}}\}}|$, if also $|X_{\{c_{i_1},\ldots,c_{i_{\ell}}\}}|=|X_{\{d_{j_1},\ldots,d_{j_{\ell'}}\}}|$, then $|X_C\setminus X_{\{c_{i_1},\ldots,c_{i_{\ell}}\}}|\geq |X_D\setminus X_{\{d_{j_1},\ldots,d_{j_{\ell'}}\}}|$ and if also $|X_C\setminus X_{\{c_{i_1},\ldots,c_{i_{\ell}}\}}|= |X_D\setminus X_{\{d_{j_1},\ldots,d_{j_{\ell'}}\}}|$, then $|E(X_{b_{j_0}},X_C\setminus X_{\{c_{i_1},\ldots,c_{i_{\ell}}\}})|\geq |E(X_{a_{i_0}},X_D\setminus X_{\{d_{j_1},\ldots,d_{j_{\ell'}}\}})|$. Also, we use the following easy observation whose proof is left to the reader.\vsp
	
	(4) \textit{For every $i\in \{i_1,\ldots ,i_{\ell}\}$ $($resp. $j\in \{j_1,\ldots ,j_{\ell}\}$$)$ and every $a\in X_{a_{i}}$ and $c\in X_{c_{i}}$ $($resp. $b\in X_{b_{j}}$ and $d\in X_{d_{j}}$$)$, we have $N_{G}[a,X_B]=N_{G}[c,X_B]$ and $N_{G}[a,X_D]=N_{G}[c,X_D]$ $($resp. $N_{G}[b,X_A]=N_{G}[d,X_A]$ and $N_{G}[b,X_C]=N_{G}[d,X_C]$$)$.}\vsp
	
	Consider the following family,
	\begin{align*}
		\mathscr{C}=&\mathscr{N}_{G}[X_A;X_B\cup X_C\cup Z]\cup \mathscr{N}_{G}[X_C\setminus X_{\{c_{i_1},\ldots,c_{i_{\ell}}\}};X_B\cup \{v\}]\cup \mathscr{N}_{G}[X_D;X_B\cup X_C\cup Z]\cup \\
		&\mathscr{N}_{G}[X_A;X_D\cup \{y\}]\cup 
		\{X_A\cup U\cup \{y\}, X_B\cup U\cup \{v\},X_C\cup\{v,x\},X_D\cup\{x,y\}\}
	\end{align*}
	whose elements by (1) are cliques of $G$. 
	Note that by (4), the edges in $E_G(X_{\{c_{i_1},\ldots,c_{i_{\ell}}\}},X_B)$ are covered by the cliques in $\mathscr{N}_{G}[X_A;X_B\cup X_C\cup Z]$. 
	If $ Z=\emptyset $, then $ \mathscr{C} $ is a clique covering. Also, if $ Z $ is nonempty, then $ U=\emptyset $ and $ \mathscr{C} $ is also a clique covering (because $X_{a_{i_0}}$ is complete to $X_B$ and $X_{d_{l_0}}$ is complete to $X_C$ and thus the edges in $E_G(Z,X_B)$ and $E_G(Z,X_C)$ are covered by the cliques in $\mathscr{N}_{G}[X_A;X_B\cup X_C\cup Z]$ and $\mathscr{N}_{G}[X_D;X_B\cup X_C\cup Z]$). Also, $ |\mathscr{C}|=n-(|X_B|-|X_A|)-|X_{\{c_{i_1},\ldots,c_{i_{\ell}}\}}|-|U|-|Z|+1$. Hence, if either $ U\neq\emptyset $, or $ Z\neq\emptyset $, or $|X_B|-|X_A|\geq 1$ or $|X_{\{c_{i_1},\ldots,c_{i_{\ell}}\}}|\geq 2$, then we are done. So, we may assume that $ U=Z=\emptyset $, $|X_A|=|X_B|$ and $|X_{\{c_{i_1},\ldots,c_{i_{\ell}}\}}|=|X_{\{d_{j_1},\ldots,d_{j_{\ell'}}\}}|=1$ (i.e. $\ell=\ell'=1$). 
	
	In this case, let $ \mathscr{C}' $ be obtained from $ \mathscr{C} $ by removing the clique $X_B\cup U\cup \{v\}$. 
	Since $ X_{a_{i_0}} $ is complete to $ X_B $ and $X_{c_{k_0}}$ is complete to $X_B\setminus X_{b_{j_0}}$, the only edges of $G$ that are possibly not covered by the cliques in $\mathscr{C}'$ are the edges in $E_G(\{v\},X_{b_{j_0}})$. Now, if there exists a vertex $c_k\in C$, where $k\notin \{j_0,i_1\}$, then $X_{c_k}$ is complete to $X_{b_{j_0}}$, and thus $\mathscr{C}'$ is a clique covering for $G$ of size $n-(|X_B|-|X_A|)-|X_{c_{i_1}}|\leq n-1$. Therefore, we may assume that $j_0=k_0$ and $C=\{c_{i_1},c_{k_0}\}$. This implies that $b_{j_0},c_{k_0}\in W(H)$ and thus $|X_C\setminus X_{c_{i_1}}|=|X_{c_{k_0}}|=1$ and $E(X_{b_{j_0}},X_C\setminus X_{c_{i_1}})=\emptyset$. Therefore, $|X_D\setminus X_{d_{j_1}}|=1$ and so $E(X_{a_{i_0}},X_D\setminus X_{d_{j_1}})=\emptyset$. In particular,  $X_C=\{c_{i_1},c_{k_0}\}$, $X_D=\{d_{j_1},d_{l_0}\}$, $i_0=l_0$ and $j_0=k_0$.
	Now, if $|X_A|=|X_B|\geq 3$, then the family of cliques $\mathscr{N}_G[X_C\cup X_D;X_A\cup X_B]\cup \mathscr{N}_G[X_C;X_D]\cup \mathscr{N}_G[X_A\setminus \{a_{i_1}\};X_B]\cup \{X_A\cup \{y\},X_B\cup  \{v\},X_C\cup\{v,x\},X_D\cup\{x,y\}\}$ is a clique covering for $G$ of size $9+|X_A|=n-|X_B|+2\leq n-1$. Also, if $|X_A|=|X_B|=2$ (i.e. $X_A=\{a_{i_0},a_{i_1}\}$ and $X_B=\{b_{j_0},b_{j_1}\}$), then the family of cliques $\{N_G[a_{i_1},X_B\cup X_C],N_G[a_{i_1},X_C\cup X_D],N_G[b_{j_1},X_C\cup X_D],N_G[b_{j_1},X_D\cup X_A],\{a_{i_0},b_{j_0}\},\{c_{k_0},d_{l_0}\},X_A\cup \{y\},X_B\cup  \{v\},X_C\cup\{v,x\},X_D\cup\{x,y\}\}$ is a clique covering for $G$ of size $10=n-1$. This completes the proof of \pref{lem:parallel}.
\end{proof}

In the following, we examine the extensions of the graphs in the class $\mathcal{F}_0$.

\begin{lem} \label{lem:F0}
	Let $G$ be a graph in $\mathcal{F}_0$ and $G'$ be an extension of $G$ on $n$ vertices. Then $\cc(G')\leq n-1$.
\end{lem}
\begin{proof}
	Let $I_1,I_2$ and $I_3$ be as in the definition of the class $\mathcal{F}_0$. Let $R_2^1=V(G)\cap \{r_2^1\}$ and $R_1^2=V(G)\cap \{r_1^2\}$. This is evident that $\tilde{W}(G)=V(G)\cap \{s_1^3,s_3^3,t_3^2,t_3^3\}$. For every $v\in \tilde{W}(G)$, let $X_v\subseteq V(G')$ be as in the definition of extension, and for every $v\in \{s_1^3,s_3^3,t_3^2,t_3^3\}\setminus V(G)$, define $X_v=\emptyset$. Now, let $\mathscr{C}$ be the family containing the following $12$ cliques
	\[\begin{array}{lll}
	\{r_1^1,r_1^3,s_1^2\}\cup R_1^2,& \{r_1^1,r_1^3,s_2^3\}\cup X_{s_1^3}\cup X_{s_3^3}\cup R_1^2,& \{r_2^2,r_2^3,s_1^1\}\cup R_2^1,\\
	\{r_2^2,r_2^3,s_2^3\}\cup X_{s_1^3}\cup X_{s_3^3}\cup R_2^1, & \{r_3^1,r_3^2,s_1^1,s_1^2\}\cup X_{t_3^2}\cup X_{t_3^3},& \{s_1^1,s_1^2,t_1^2,t_2^2\}\cup X_{s_1^3}, \\
	\{s_2^3,t_3^1\}\cup X_{s_3^3}, &\{r_1^3,r_2^3,t_1^2,t_2^2\}\cup X_{s_3^3}, &\{r_2^2,r_3^2,t_1^2\}\cup R_1^2, \\
	\{t_2^2,r_1^1,r_3^1\}\cup R_2^1,&\{t_3^1,r_1^1,r_3^1\}\cup X_{t_3^2}\cup X_{t_3^3}\cup R_2^1, & \{t_3^1,r_2^2,r_3^2\}\cup X_{t_3^2}\cup X_{t_3^3}\cup R_1^2.
	\end{array}\]
	It can be easily seen that $\mathscr{C}\cup \mathscr{N}_{G'}[X_{t_3^2};X_{s_1^3}\cup X_{s_3^3}\cup \{t_1^2,t_2^2\}]\cup \mathscr{N}_{G'}[X_{t_3^3};X_{s_1^3}\cup X_{s_3^3}\cup \{s_2^3\}]$ is a clique covering for $G'$ of size $n-|R_2^1|-|R_1^2|-|X_{s_1^3}|-|X_{s_3^3}|$. Thus, if at least one of the sets $R_2^1,R_1^2,X_{s_1^3}$ or $X_{s_3^3}$ is nonempty, then we are done. Now, assume that $R_2^1=R_1^2=X_{s_1^3}=X_{s_3^3}=\emptyset$ and in $\mathscr{C}$, replace the cliques $\{s_1^1,s_1^2,t_1^2,t_2^2\}\cup X_{s_1^3}$ and $ \{s_2^3,t_3^1\}\cup X_{s_3^3} $ with the cliques $\{s_1^1,s_1^2,t_1^2,t_2^2\}\cup X_{t_3^2}$ and $ \{s_2^3,t_3^1\}\cup X_{t_3^3} $ respectively, thereby providing a clique covering for $G$ of size $n-|X_{t_3^2}|-|X_{t_3^3}|$. Hence, if either $X_{t_3^2}$ or $X_{t_3^3}$ is nonempty, then we are done. Finally, assume that $X_{t_3^2}=X_{t_3^3}=\emptyset$. In this case, $|V(G')|=12$ and the following $11$ cliques
	\[\begin{array}{llllll}
	\{s_1^2,t_2^2, r_1^1,r_3^1\}, & \{s_1^1, t^2_1,t^2_2, r_2^3\}, & \{s_1^1, s^2_1, r_3^1,r^2_3 \}, & \{s_1^1, t^2_1, r^2_2,r^2_3\}, & \{s^2_1, t^2_1, r^3_1\}, & \{t^2_2, r^3_1,r^3_2\}, \\
	\{s_2^3, r_1^1,r^3_1\}, & \{t^1_3, r_1^1,r^1_3\}, & \{s^3_2,r_2^2,r^3_2 \}, & \{ t^1_3,r^2_2, r^2_3 \}, & \{s^3_2, t^1_3\},
	\end{array}\]
	provide a clique covering for $G'$ of size $n-1$. This proves \pref{lem:F0}.
\end{proof}
In the following, we look into the extensions of the graphs in the class $\mathcal{F}_1$.
\begin{lem} \label{lem:F1}
	Let $G$ be a graph in $\mathcal{F}_1$ and $G'$ be an extension of $G$ on $n$ vertices  which is not three-cliqued. Then $\cc(G')\leq n$ and equality holds if and only if $\overline{G'}$ is isomorphic to a twister. 
\end{lem}
\begin{proof}
	Let $A,B,R,s,t$ be as in the definition of the class $\mathcal{F}_1$. Since $ G $ is antiprismatic, $ R $ is complete to $ A\cup B $. 
	Let $\tilde{A}\subseteq A$ be the set of all vertices in $A$ which have no non-neighbour in $A$ and let $\tilde{B}\subseteq B$ be defined analogously. 
	Note that from the definition, both $A$ and $B$ are the union of at most two disjoint cliques.
	Also, $ B\cup R$ is nonempty (since otherwise $ G $ and so $ G' $ is three-cliqued).
	Since $ G $ is antiprismatic and $ B\cup R$ is nonempty, $ t $ is complete to $ A $. Similarly, $ s $ is complete to $ B $. Now, if either $\tilde{A}=A $ or $\tilde{B}=B$, then $G$ and so $ G' $ is three-cliqued. This implies that $\tilde{A}\neq A $ and $\tilde{B}\neq B$. Hence, $\tilde{W}(G)=\tilde{A}\cup \tilde{B}$. 
	Now, let $A=\tilde{A}\cup X\cup Y$, where $X= \{x_1,\ldots, x_a\}$ and $Y=\{y_1,\ldots, y_a\}$ and $x_iy_i$, $1\leq i\leq a$, are the only non-edges with both ends in $A$. Also, let $B=\tilde{B} \cup U\cup V$, where $U= \{u_1,\ldots, u_b\}$ and $V=\{v_1,\ldots, v_b\}$ and $u_iv_i$, $1\leq i\leq b$, are the only non-edges with both ends in $B$.
	By abuse of notation, denote the set $\cup _{v\in \tilde{A}}X_v$ of vertices of $ G' $ by $\tilde{A}$ and the same assumption is applied to $\tilde{B}$ (note that  adjacency between $ \tilde{A} $ and $ \tilde{B} $ is arbitrary in $ G' $). If $ R $ is empty, then $ N_{\overline{G'}}(s)=A\cup \{t\}$ and for every vertex $ x\in A\cup \{t\} $,  $ N_{G'}(x,B\cup \{s\}) $ is a clique of $ G' $. Therefore, if $ R$ is empty, then the result follows from \pref{lem:2sub} (by setting $ u=s $). Hence, assume that $ R $ is nonempty. 
	Also, w.l.o.g. assume that $a\leq b$. Now, define
	\[\mathscr{C}=\mathscr{N}_{G'}[A;B\cup R]\cup \mathscr{N}_{G'}[X;(A\setminus X)\cup  \{t\}]\cup \mathscr{N}_{G'}[U;(B\setminus U)\cup \{s\} ]\cup \{X, (A\setminus X)\cup \{t\}\}. \]
	Note that since $\tilde{A}\neq A$, every vertex in $B$ has a neighbour in $A$, and so the edges in $E(R,V(G')\setminus R)$ are all covered by the cliques in $\mathscr{N}[A;B\cup R]$. Also, w.l.o.g. we may assume that $U\subseteq N(x_1,B)$ and $V\subseteq N(y_1,B)$. Therefore, the cliques in $\mathscr{C}$ cover all edges of $G'$ except possibly the edges in $E(V,\tilde{B}\cup \{s\})$. If $ b\geq 2 $, then these edges are covered by the cliques in $ \mathscr{N}_{G'}[U;(B\setminus U)\cup \{s\}] $ and thus $ \mathscr{C} $ is a clique covering for $ G' $ of size 
	$|\mathscr{C}|= |A|+a+b+2\leq |A|+|B|+2= n-1$.
	If $ b=1 $, then $ a=1 $ and in the collection $\mathscr{C}$, we may replace the clique $ X $ with the clique $\{v_1,s\}\cup \tilde{B}$ to obtain a clique covering of size $ |\mathscr{C}|\leq n-1  $. This proves \pref{lem:F1}.
\end{proof} 

Prior to examination of the classes $\mathcal{F}_2,\mathcal{F}_3$ and $\mathcal{F}_4$, we need to recall an operation defined in \cite{seymour2} which is used in the definition of these classes in \pref{app:noap}.

Let $\tau=\{a,b,c\}$ be a triad of an antiprismatic graph $H$. In \cite{seymour2}, $\tau$ is said to be a \textit{leaf triad} of $H$ at $c$, if both $a$ and $b$ belong to only one triad of $H$ (namely $\tau$). Let $\tau =\{a,b,c\}$ be a leaf triad of $H$ at $c$. 
Define subsets $D_1$, $D_2$, $D_3$ of the set of non-neighbours of $c$ as follows. If $v$ is nonadjacent to $c$ in $H$, let
\begin{itemize}
	\item $v\in D_1$ if $v$ belongs to a triad that does not contain $c$;
	\item $v\in D_2$ if $v\in W(H)\setminus \tau$, and every triad containing $v$ also contains $c$ (and hence is unique);
	\item $v\in D_3$ if $v\in \tilde{W}(H)$.
\end{itemize}
Thus, the four sets $D_1,D_2,D_3$ and  $\{a,b\}$ are pairwise disjoint sets whose union is the set of non-neighbours of $c$ in $H$. Suppose that $D_1,D_2$ are both cliques. Let $A,B,C$ be three pairwise disjoint sets of new vertices, and let $G$ be obtained from $H$ by deleting $a$, $b$ and adding the new vertices in $A\cup B\cup C$, with adjacency as follows:
\begin{itemize}
	\item $A$, $B$ and $C$ are cliques;
	\item every vertex in $A$ has at most one non-neighbour in $B$, and vice versa;
	\item every vertex in $V (H)\setminus \{a,b\}$ adjacent to $a$ (resp. $b$) in $H$ is complete to $A$ (resp. $B$) in $G$, and every vertex in $V (H)\setminus \{a,b\}$ nonadjacent to $a$ (resp. $b$) in $H$ is anticomplete to $A$ (resp. $B$);
	\item every vertex in $C$ is anticomplete to $D_1\cup D_3$, and complete to $V (H)\setminus (D_1\cup D_3\cup \{a,b\})$;
	\item every vertex in $C$ is nonadjacent to exactly one end of every non-edge between $A$ and $B$, and nonadjacent to every vertex in $A\cup B$ adjacent to all other vertices in $A\cup B$.
\end{itemize}
The graph $G$ is said to be obtained from $H$ by \textit{exponentiating} the leaf triad $\tau =\{a,b,c\}$. We leave the reader to check that $G$ is also antiprismatic and if $ H $ is three-cliqued, then so is $ G $.
Now, in the following, we give a method to extend a clique covering of $H$ (and its extensions) to a clique covering of $G$ (and its extensions).
\begin{lem} \label{lem:exponen}
	Assume that $H$ is an antiprismatic graph, $\tau=\{a,b,c\}$ is a leaf triad of $H$ at $c$ and $G$ is obtained from $H$ by exponentiating $\tau$. Also, suppose that for some fixed number $k$, every extension $H'$ of $H$ satisfies $\cc(H')\leq |V(H')|-k$ and if $ H+ab $ is not three-cliqued, then every extension $H''$ of $H+ab$ satisfies $\cc(H'')\leq |V(H'')|-k$. Let $ G' $  be an extension of $G$ which is not three-cliqued. Then, we have $\cc(G')\leq |V(G')|-k$.
\end{lem}
\begin{proof}
	Let $A,B, C, D_1,D_2$ and $D_3$ be subsets of $V(G)$ as in the definition, and also let $\tilde{A}$ (resp. $\tilde{B}$) be the set vertices in $A$ (resp. $B$) which are complete to $B$ (resp. $A$). Note that $\tilde{W}(G)=\tilde{W}(H)\cup \tilde{A}\cup \tilde{B}\cup C$. 
	Now, consider an extension $G'$ of $G$ and by abuse of notation, let $\tilde{A}=\bigcup _{v\in \tilde{A}}X_v$, $\tilde{B}=\bigcup _{v\in \tilde{B}}X_v$ and $C=\bigcup _{v\in C}X_v$ (note that adjacency between $ C $ and $ \tilde{A}\cup \tilde{B}\cup D_3 $ is arbitrary in $ G' $). We assume that $ G' $ is not three-cliqued and we are going to provide a clique covering for $G'$ of size at most $|V(G')|-k$. For every $z\in C$, define $C_z=D_2\cup N_{G'}[z,A\cup B\cup D_3]$. Note that since $H$ is antiprismatic, $A\cup B$ is complete to $D_2\cup D_3$, and since $ D_3 \subseteq \tilde{W}(H)\cap N_{\overline{H}}(c)  $ and $ D_2\subseteq N_{\overline{H}}(c) $, $D_2\cup D_3$ is a clique. Also, every vertex $z\in C$ is adjacent to exactly one end of every non-edge between $A$ and $B$. Therefore, $C_z$ is a clique of $G'$. 
	
	First, assume that $\tilde{A}=A$ and thus, $\tilde{B}=B$. In this case, it is easy to see that $G'$ is obtained from an extension $H''$ of $H+ab$ by adding the new set of vertices $C$, which is complete to $V(H'')\setminus (D_1\cup D_3\cup A\cup B)$ and anticomplete to $D_1$ and adjacency between $C$ and $A\cup B\cup D_3$ is arbitrary. Thus, $|V(G')|=|V(H'')|+|C|$. Also, since $ G' $ is not three-cliqued, neither is $ H+ab $ (because every vertex adjacent to $ c $ in $ H+ab $ is complete to $ C $ in $ G' $). Thus, by the assumption, there exists a clique covering $\mathscr{C}$ for $H''$ of size at most $|V(H'')|-k$. For every clique in $\mathscr{C}$ containing $c$, replace $c$ with $C\cup\{c\}$, and also add the clique $C_z$, for all $z\in C$, thereby obtaining a clique covering for $G'$ of size at most $|V(H'')|-k+|C|=|V(G')|-k$, as desired.
	
	Next, assume that $\tilde{A}\neq A$ and then $\tilde{B}\neq B$. It is clear that $G'$ is obtained from an extension $H'$ of $H$ by exponentiating the leaf triad $\tau$ and then adding some edges between $C\cup \tilde{A}\cup \tilde{B}$ and $\tilde{W}(H')$, and between $C$ and $\tilde{A}\cup \tilde{B}$ as well. Also, $|V(G')|=|V(H')|+|A|+|B|+|C|-2$. By the assumption, there exists a clique covering $\mathscr{C}$ for $H'$ of size at most $|V(H')|-k$. For every clique in $\mathscr{C}$ containing $a$, $b$ or $c$, replace $a$, $b$ and $c$ with $A$, $B$ and $C\cup \{c\}$, respectively to obtain a collection of cliques of $G'$, called $ \mathscr{C}' $, covering all edges of $G'$ except the edges between $A$ and $B$ and some edges between vertices in $\tilde{W}(G')$. Let $F$ be the subgraph of $G'$ induced on the set of these uncovered edges.  
	Since $N_{\overline{H}}(a,\tilde{W}(H))$ is a clique of $H$, for every $x\in \tilde{A}$, $N_F[x,\tilde{W}(H')]$ is a clique of $G'$. Now, for every $x\in \tilde{A}$, define $C_x=B\cup N_F[x,\tilde{W}(H')]$ which is a clique of $G'$ (because in $ H $ every non-neighbour of $a$ in $V(H)\setminus \{c\}$ is a neighbour of $b$). Also, for every $x\in A\setminus \tilde{A} $, define $C_x= N_{G'}[x, B\setminus \tilde{B}]$ and for every $y\in\tilde{B}$, define $C_y= A\cup N_F[y,\tilde{W}(H)]$, which are obviously cliques of $G'$. Now, the cliques in $\{C_x: x\in  A\cup \tilde{B}\cup C\}$ cover all the edges of $F$, and thus adding them to $\mathscr{C}'$ yields a clique covering for $G'$ of size at most $|V(H')|-k+|A|+|\tilde{B}|+|C|$. Now, if $|B\setminus \tilde{B}|\geq 2$, then we are done. Otherwise, if $|B\setminus \tilde{B}|=1$, then for $x\in A\setminus \tilde{A}$, we have $C_x=\{x\}$. Thus, removing $ C_x, x\in A\setminus \tilde{A}$, yields the desired clique covering of size at most $ |V(G')|-k $. This proves \pref{lem:exponen}.
\end{proof}

The following inquires into the extensions of the graphs in both classes $\mathcal{F}_2$ and $\mathcal{F}_3$.
\begin{lem}\label{lem:F2,3}
	Let $G$ be a graph in $\mathcal{F}_2\cup \mathcal{F}_3$ and $G'$ be an extension of $G$ on $n$ vertices  which is not three-cliqued. Then $\cc(G')\leq n-1$.
\end{lem}
\begin{proof}
	First, suppose that $G\in \mathcal{F}_2$ and let $H$ be the graph of parallel square type with $V(H)=A\cup B\cup C\cup  D\cup \{u,v,x,y\}$, where $G$ is obtained from $H$ by exponentiating the leaf triad $\{a_1,b_1,x\}$, as in the definition of $ \mathcal{F}_2 $. 
	Since $ G' $ is not three-cliqued, neither are $ G $ and $ H$. Thus, by \pref{lem:parallel}, every extension $ H' $ of $H$ admits a clique covering of size at most $|V(H')|-1$.
	On the other hand, note that $H+a_1b_1$ is also a graph of parallel-square type. To see this, let $\ell =1+\max\{i:a_i\in A\text{ or } b_i\in B \text{ or } c_i\in C \text{ or } d_i\in D \}$. Now, replacing the vertex $a_1$ with $a_{\ell}$ in the set $A$ yields a graph of parallel-square type which is isomorphic to  $H+a_1b_1$. 
	Hence, if $ H+a_1b_1 $ is not three-cliqued, then by \pref{lem:parallel}, every extension $ H'' $ of  $ H+a_1b_1$ admits a clique covering of size at most $|V(H'')|-1$. Consequently, by \pref{lem:exponen}, $\cc(G')\leq n-1$.
	
	Moreover, for the graph $G\in \mathcal{F}_3$, the result follows from a similar argument with the aid of Lemmas~\ref{lem:parallel} and \ref{lem:exponen} (we leave the details to the reader).
\end{proof}
The following investigates the extensions of the graphs in $ \mathcal{F}_4 $.
\begin{lem}\label{lem:F4}
	Let $G$ be a graph in $\mathcal{F}_4$ and $G'$ be an extension of $G$ on $n$ vertices which is not three-cliqued. Then $\cc(G')\leq n-1$.
\end{lem}
\begin{proof}
	Let $H$ be the subgraph of Schl\"{a}fli graph induced on $V(H)=Y\cup \{s_j^i: (i,j)\in I\}\cup \{t_1^1,t_2^2,t_3^3\}$, where $Y$ and $I$ are as in the definition of the class $\mathcal{F}_4$ and $G$ be obtained from $H$ by exponentiating the leaf triad $\{t_1^1,t_2^2,t_3^3\}$. Also, for every $ j\in \{1,2,3\} $, let $ R^3_j=V(H)\cap\{ r_j^3\}$, $ S_j=V(H)\cap \{s_j^1,s_j^2,s_j^3\} $ and $ S^j=V(H)\cap \{s^j_1,s^j_2,s^j_3\} $. Now the following cliques
	\[\begin{array}{lllllll}
	S_1\cup\{t^2_2\}, & S_1\cup \{t^3_3\}, & S_2\cup\{t^1_1\}, & S_2\cup\{t^3_3\}, & S_3\cup\{t^1_1\}, & S_3\cup\{t^2_2\},\\
	R^3_1\cup R^3_2\cup S^3, & R^3_2\cup R^3_3\cup S^1, & R^3_1\cup R^3_3\cup S^2, & Y\cup \{t^1_1\}, &  Y\cup \{t^2_2\},
	\end{array} \]
	form a clique covering $\mathscr{C}$ for $ H $ of size $11\leq |V(H)|-1$. Also, since $|I|\geq 8$, we have $\tilde{W}(H)\subseteq Y$, which is a clique of $H$. Thus, every extension $ H' $ of $H$ is in fact a thickening of $(H,\emptyset)$, and thus by \pref{lem:thickening}, $ H' $ admits a clique covering of size at most $|V(H')|-1$.
	
	Moreover, note that $\tilde{W}(H+t_1^1t_2^2)$ is contained in $\{t_1^1,t_2^2,t_3^3\}\cup Y$. Now, let $ H'' $ be an extension of $H+t_1^1t_2^2$, where $ X_v $, $ v\in \tilde{W}(H+t_1^1t_2^2) $, are as in the definition and for every $ v\in W(H+t_1^1t_2^2) $, set $ X_v=\{v\} $. For every clique $C$ of $H+t_1^1t_2^2$, let $X_C=\cup _{v\in C}X_v$ which is a clique of $ H'' $ and define $\mathscr{C}'=\{X_{C}:C\in \mathscr{C}\}$. Now, in $\mathscr{C}'$, merge the pairs $(X_{S_3\cup \{t_1^1\}},X_{S_3\cup \{t_2^2\}})$ and $(X_{Y\cup \{t^1_1\}},X_{Y\cup \{t^2_2\}})$ and add the cliques in $\mathscr{N}[X_{t_3^3}; X_{Y\cup \{t_1^1,t_2^2\}}]$, thereby obtaining a clique covering for $H''$ of size $9+|X_{t_3^3}|\leq |V(H'')|-2$. Now, \pref{lem:F4} follows immediately from \pref{lem:exponen}. 
\end{proof}

Eventually, we consider the extensions of the graphs of skew-square type.

\begin{lem}\label{lem:skew}
	Let $G$ be a non-orientable antiprismatic graph of skew-square type and $G'$ be an extension of $G$ on $n$ vertices. Then $\cc(G')\leq n$ and equality holds if and only if $ \overline{G'} $ is isomorphic to a twister.
\end{lem}
\begin{proof}
	Let $A,B,C,s,t,d_1,d_2$ and $d_3$ be as in the definition of the graphs of skew-square type.
	If $ \overline{ G }$ contains no induced rotator, then by \pref{thm:exc-rotator-seymour}, $ G $ and so $ G' $ is an extension of a non-orientable antiprismatic graph in $ \mathcal{F}_0\cup \mathcal{F}_1\cup\mathcal{F}_2\cup \mathcal{F}_3\cup \mathcal{F}_4 $ and the assertion follows from Lemmas~\ref{lem:F0}, \ref{lem:F1}, \ref{lem:F2,3} and \ref{lem:F4}.
	Thus, assume that $ G $ contains the complement of a rotator as an induced subgraph. Note that the only triads of $G$ are of the form $\{s,a_j,c_j\}$ or $\{t,b_j,c_j\}$, for some $j$, or $\{a_i,b_{i'},d_{i''}\}$, where $\{i,i',i''\}=\{1,2,3\}$. Therefore, since the complement of a rotator contains exactly three disjoint triads, (leaving the reader to check the details) one can see that the only rotators of $\overline{G}$ are induced on the subsets of $V(G)$ of the following four forms, 
	\[\begin{array}{ll}
	R_1=\{a_i,b_{i'},d_{i''}\}\cup \{a_{i'},b_{i''},d_i\}\cup \{b_i,c_i,t\},
	& R_2=\{a_i,b_{i'},d_{i''}\}\cup \{a_{i'},b_{i''},d_i\}\cup \{a_{i''},c_{i''},s\},  \\
	R_3=\{a_i,b_{i'},d_{i''}\}\cup \{a_{i'},c_{i'},s\}\cup \{b_j,c_j,t\},
	& R_4=\{a_i,b_{i'},d_{i''}\}\cup \{a_{j},c_{j},s\}\cup \{b_i,c_i,t\},
	\end{array}\]
	where $\{i,i',i''\}=\{1,2,3\}$ and $j\in \mathbb{N}\setminus \{i,i'\}$. Let $\tilde{A}\subseteq A$ and $\tilde{B}\subseteq B$ be the set of vertices in $A$ and $B$, respectively,  which are complete to $C$. Also, let $\tilde{C}\subseteq C$ be the set of vertices in $C$ which are complete to $A\cup B$. Note that $s,t\in W(G)$ (otherwise if say $s\in \tilde{W}(G)$, then $A$ is complete to $C$ and $V(G)$ would be the union  of three cliques $A\cup C, B$ and $\{s,t,d_1,d_2,d_3\}$, so  $G$ is orientable, a contradiction). Thus, $\tilde{W}(G)\subseteq \tilde{A}\cup \tilde{B}\cup \tilde{C}\cup \{d_1,d_2,d_3\}$. For $ \{i,i',i''\}=\{1,2,3\} $, note that $ a_i\in \tilde{W}(G) $ if and only if $ a_i\in \tilde{A} $ and  $ b_{i'}, b_{i''}\not\in B $ and the similar assertion holds for $  b_i\in \tilde{W}(G)  $.  Let $I=\{i\ :\ a_i\in A\text{ and } b_i\in B\}$. If $I=\emptyset$, then every vertex in $C$ is complete to either $A$ or $B$,  and if $C'$ is the set of vertices in $C$ which are complete to $A$ in $G$, then $V(G)$ is the union of three cliques $A\cup C', (C\setminus C')\cup B$ and $\{s,t,d_1,d_2,d_3\}$, a contradiction. Thus, $I\neq \emptyset$.
	
	Let $ G' $ be an extension of $G$ on $n$ vertices and $ X_v $, $ v\in \tilde{W}(G) $, be as in the definition. Also, for every $ v\in W(G) $, define $ X_v=\{v\} $ and for every $ v\not\in V(G) $, define $ X_v=\emptyset $. We are going to prove that $ \cc(G')\leq n-1 $. 
	Let $A'=\cup_{k\geq 1}X_{a_k}$, $B'=\cup_{k\geq 1}X_{b_k}$, $C'=\cup_{k\geq 1}X_{c_k}$ and $D'=X_{d_1}\cup X_{d_2}\cup X_{d_3}$. Also, define $I_1= \{k\ :\ 1\leq k\leq 3, a_k\in \tilde{W}(G)\}$, $I_2=\{k\ :\ 1\leq k\leq 3, b_k\in \tilde{W}(G)\}$ and $I_3=\{k\ :\ c_k\in \tilde{C}\}$ and for each $ l\in\{1,2,3\} $, let $K_l=X_{a_l}\cup X_{b_l}\cup (\cup_{k\in I_1} X_{a_k})\cup (\cup_{k\in I_2} X_{b_k})\cup (\cup_{k\in I_3} X_{c_k}) \cup (\cup _{k\in \{1,2,3\}\setminus \{l\}}X_{c_k})$. It is easy to verify that $ K_l $ is a clique of $ G' $ and also for every $a\in A'$, $N_{G'}[a,B'\cup C']$ is a clique of $G'$. Therefore, 
	\begin{align*}
		\mathscr{C}=&(\cup _{l=1}^3 \mathscr{N}_{G'}[X_{d_l};K_l])\cup \mathscr{N}_{G'}[A';B'\cup C']\cup \mathscr{N}_{G'}[\cup _{l\in \mathbb{N}\setminus I}X_{b_l};C']\\
		&\cup \{D'\cup \{s,t\},A'\cup \{t\},B'\cup \{s\},(\cup_{l\geq 4}X_{a_l})\cup D',(\cup_{l\geq 4}X_{b_l})\cup D'\}
	\end{align*}
	is a family of cliques of $G'$ and $|\mathscr{C}|=n-|C'|-|\cup _{l\in I}X_{b_l}|+3\leq n-|C|-|I|+3$. Note that for every $l\in I$, all the vertices in $X_{a_l}\cup X_{b_l}$ have the same set of neighbours in $C'$. Therefore, the cliques in $\mathscr{C}$ cover all edges of $G'$ except possibly the edges in $E(C')$. Now, we observe that,\vsp
	
	(1) \textit{If $|C|+|I|\geq 4$, then $ \cc(G')\leq n-1 $}.\vsp 
	
	Since $\overline{G}$ contains a rotator induced on a set of the form $R_k$ for some $k\in \{1,\ldots,4\}$, either $\{k: 1\leq k\leq 3, a_k\in A\}$ or  $\{k: 1\leq k\leq 3, b_k\in B\}$ has cardinality at least two. Now, if both of the sets $\{k: k\geq 4, a_k\in A\}$ and $\{k: k\geq 4, b_k\in B\}$ are nonempty, then either $A$ or $B$, say $A$, has cardinality at least three, and so the edges of $G'$ in $E(C')$ are covered by the cliques of $\mathscr{C}$ in $\mathscr{N}_{G'}[A';B'\cup C']$. Consequently, $\mathscr{C}$ is a clique covering for $G$ of size $|\mathscr{C}|\leq n-|C|-|I|+3\leq n-1$. If either $\{k: k\geq 4, a_k\in A\}$ or $\{k: k\geq 4, b_k\in B\}$ is empty, then in $ \mathscr{C} $, replace the clique $(\cup_{k\geq 4}X_{a_k})\cup D'$ or $(\cup_{k\geq 4}X_{b_k})\cup D'$, respectively, with the clique $C'$ to obtain a clique covering for $G'$ of size at most $n-1$. This proves (1).\vsp
	
	Note that $ C $ is nonempty (otherwise, $ G $ would be three-cliqued). If either $ |C|\geq 3 $ or $ |I|\geq 3 $, then since $ I,C\neq \emptyset $, by (1), we are done. Hence, either $ |C|=1 $ and $ 1\leq |I|\leq 2 $, or $ |C|=2 $ and $|I|=1$. First, suppose that $ |C|=1 $ and $ 1\leq |I|\leq 2 $. Now, $\overline{G}$ should contain a rotator induced on a subset of $V(G)$ of the form $R_1$ or $R_2$. Consequently, $C'=C=\{c_{i_0}\}$, for some $i_0\in\{1,2,3\}$, $|I|=2$ and at least one of the sets $\{k: 1\leq k \leq 3, a_k\in A\}$ or $\{k: 1\leq k \leq 3, b_k\in B\}$, say the second one, is equal to $\{1,2,3\}$. Now, if both of the sets $\{k: k\geq 4, a_k\in A\}$ and $\{k: k\geq 4, b_k\in B\}$ are nonempty, then $|B|\geq 4$, and it is enough to modify $\mathscr{C}$ by replacing the cliques in $\mathscr{N}_{G'}[\cup _{k\in \mathbb{N}\setminus I}X_{b_k};C']$ with the cliques in $\mathscr{N}_{G'}[C';B']$, thereby obtaining a clique covering for $G'$ of size $n-|B'|+3\leq n-|B|+3\leq n-1$. If either $\{k: k\geq 4, a_k\in A\}$ or $\{k: k\geq 4, b_k\in B\}$ is empty, then removing the clique $(\cup_{k\geq 4}X_{a_k})\cup D'$ or $(\cup_{k\geq 4}X_{b_k})\cup D'$ from $\mathscr{C}$ yields a clique covering for $G'$ of size $|\mathscr{C}|-1\leq n-|C|-|I|+2=n-1$. 
	
	Finally, suppose that $ |C|=2 $ and $|I|=1$. Thus, $\overline{G}$ should induce a rotator on a subset of $V(G)$ of the form $R_3$ or $R_4$. By symmetry, we may assume that $\overline{G}$ contains a rotator induced on a subset of $V(G)$ of the form $R_3$. Thus, $\{a_i,b_{i'},d_{i''}\}\cup \{a_{i'},c_{i'},t\}\cup \{b_j,c_j,s\}\subseteq V(G)$ and $C'=C=\{c_{i'},c_j\}$, where $\{i,i',i''\}=\{1,2,3\}$ and $j\in \mathbb{N}\setminus \{i,i'\}$.  Also, since $ |I|=1 $, $ b_i\not\in B $ and $ a_j\not\in A $. Now, if either $\{k: k\geq 4, a_k\in A\}$ or $\{k: k\geq 4, b_k\in B\}$ is empty, then again removing the clique $(\cup_{k\geq 4}X_{a_k})\cup D'$ or $(\cup_{k\geq 4}X_{b_k})\cup D'$ from $\mathscr{C}$ yields a clique covering for $G'$ of size $|\mathscr{C}|-1\leq n-|C|-|I|+2=n-1$ (note that $E(C')=\{c_{i'}c_j\}$ and the edge $c_{i'}c_j$ is covered by the clique $N_{G'}[a_i,B'\cup C']\in \mathscr{C}$). Thus, we may assume that both  $\{k: k\geq 4, a_k\in A\}$ and $\{k: k\geq 4, b_k\in B\}$ are nonempty and so $|A|\geq 3$.
	
	For every $k\in \{1,2,3\}$, define $L_k=X_{d_k}\cup X_{c_{k+1}}\cup X_{c_{k+2}}\cup (\cup_{l\in \mathbb{N}\setminus \{k+1,k+2\}}X_{a_{l}})$ and $M_k=X_{d_k} \cup X_{c_{k+1}}\cup X_{c_{k+2}}\cup (\cup_{l\in \mathbb{N}\setminus \{k+1,k+2\}}X_{b_{l}})$ (reading $k+1$ and $k+2$ modulo $3$). Consider the following family of cliques of $G'$ 
\begin{align*}
\mathscr{C}'=&\mathscr{N}_{G'}[B'\setminus (X_{b_{i'}}\cup X_{b_{i''}}\cup X_{b_{j}});A']\cup \{L_k,M_k\ :\ 1\leq k\leq 3\}\\ & \cup \{D'\cup \{s,t\},A'\cup \{t\},B'\cup \{s\},X_{a_{i'}}\cup X_{b_{i'}}\cup X_{c_{j}}\}.
\end{align*}	
If  $j=i''$, then the $\mathscr{C}'$ is a clique covering for $G'$ of size  $10+|\cup_{k\in \mathbb{N}\setminus \{i',i''\}} X_{b_k}|\leq n-|A|+1\leq n-2$.
Also, if $j\neq i''$ and $b_{i''}\in B$, then adding the cliques $(A'\cup X_{c_{i'}}\cup X_{c_j})\setminus X_{a_{i'}}$ and $(B'\cup X_{c_j})\setminus X_{b_j}$ to $\mathscr{C}'$ yields a clique covering of size  $12+|\cup_{k\in \mathbb{N}\setminus \{i',i'',j\}} X_{b_k}|\leq n-|A|+2\leq n-1$.
Finally, if $j\neq i''$ and $b_{i''}\not\in B$, then in $\mathscr{C}'$, add the cliques $(A'\cup X_{c_{i'}}\cup X_{c_j})\setminus X_{a_{i'}}$ and $(B'\cup X_{c_j})\setminus X_{b_j}$ and merge the pair $(M_i,M_{i''})$, thereby obtaining a clique covering for $G'$ of size $11+|\cup_{k\in \mathbb{N}\setminus \{i',j\}} X_{b_k}|\leq n-|A|+2\leq n-1$. This completes the proof of \pref{lem:skew}.
\end{proof}

Now, we are in the position to prove \pref{thm:Inc-rotator-Exc-sf}.
\begin{proof}[{\rm \textbf{Proof of \pref{thm:Inc-rotator-Exc-sf}.}}] 
	Let $ G $ be a counterexample to \pref{thm:noap}.
	On the contrary, assume that $\overline{G}$ either contains no induced rotator or contains a square-forcer. By \pref{thm:Seymour-Inc-rotator-Exc-sf-revised}, $G$ is an extension of a graph $ H $, where $ H $ is either a graph of parallel-square type, or a graph of skew-square type, or a member of  $\mathcal{F}_0\cup \mathcal{F}_1\cup \mathcal{F}_2\cup \mathcal{F}_3 \cup \mathcal{F}_4$. 
	Since $ G $ is non-orientable, so is $ H $ (and both are not three-cliqued). Hence, by Lemmas~\ref{lem:parallel}, \ref{lem:F0}, \ref{lem:F1}, \ref{lem:F2,3}, \ref{lem:F4} and \ref{lem:skew}, $\cc(G)\leq n$ and  equality holds if and only if $\overline{G}$ is isomorphic to a twister, a contradiction. This proves \pref{thm:Inc-rotator-Exc-sf}.
\end{proof}

\subsection{Including a rotator} \label{sub:Inc-rot}
Let $ G $ be a counterexample to \pref{thm:noap}. In view of \pref{thm:Inc-rotator-Exc-sf}, henceforth we may assume that $\overline{G}$ contains an induced rotator and no square-forcer. 
For some reasons, in contrast to the most arguments in this paper, we do not get involved with the detailed structure of such antiprismatic graphs. Instead, with arguments similar to \cite{seymour2}, we derive some structural properties of these graphs required for the construction of our clique coverings.
One of these graphs which play an important role in the characterization of non-orientable antiprismatic graphs is the \textit{Schl\"{a}fli graph}. Let us recall its definition from \cite{seymour2}. The Schl\"{a}fli graph is a graph $\Gamma$ on the set $\{r_j^i,s_j^i,t_j^i: 1\leq i, j\leq 3\}$ of $27$ vertices, with the following adjacency. Assume that $1\leq i,i',j,j'\leq 3$.
\begin{itemize}
	\item $s_j^i$ is adjacent to $s_{j'}^{i'}$, $t_j^i$ is adjacent to $t_{j'}^{i'}$ and $r_j^i$ is adjacent to $r_{j'}^{i'}$, if and only if either $i=i'$ or $j=j'$ (and not both).
	\item $s_j^i$ is adjacent to $t_{j'}^{i'}$, $t_j^i$ is adjacent to $r_{j'}^{i'}$ and $r_j^i$ is adjacent to $s_{j'}^{i'}$, if and only if $ j\neq i' $.
\end{itemize}
Note that, $ \overline{\Gamma} $ induces a rotator on $\{s_1^1,s_2^2,s_3^3;t_3^1,t_3^2,t_3^3;r_1^3,r_2^3,r_3^3\}$. In the sequel, we will see that the structure of $ G $, a counterexample to \pref{thm:noap}, is very similar to $ \Gamma $.
First of all, we set a couple of notations and definitions. Through the rest of the paper, read all subscripts and the superscripts modulo $3$. Suppose that $G$ contains the complement of a rotator as an induced subgraph called  $\rho=(s_1^1,s_2^2,s_3^3;t_3^1,t_3^2,t_3^3;r_1^3,r_2^3,r_3^3)$, where in $G$, $\{s_1^1,s_2^2,s_3^3\}$ is a triad, $T_3(\rho)=\{t_3^1,t_3^2,t_3^3\}$ and $R^3(\rho)=\{r_1^3,r_2^3,r_3^3\}$ are two triangles, $T_3(\rho)$ is anticomplete to $R^3(\rho)$ and for every $i,j\in\{1,2,3\}$, $s_i^i$ is adjacent to $t^j_3$ and $r_j^3$, if and only if $i\neq j$. The reason for this puzzling notation is that the subgraph of $G$ induced on the vertices in $\rho$ is isomorphic to the subgraph of the Schl\"{a}fli graph induced on the same vertices. 
Let $V_0(\rho)$ be the set of vertices of $G$ which are not in $\rho$. For every $i\in \{1,2,3\}$, define
\begin{itemize}
	\item $T^i(\rho)$ is the set of vertices in $V_0(\rho)$ which are nonadjacent to $s_i^i$, complete to $\{r_{i+1}^3,r_{i+2}^3\}$ and anticomplete to $\{t_3^{i+1},t_3^{i+2}\}$.
	\item $R_i(\rho)$ is the set of vertices in $V_0(\rho)$ which are nonadjacent to $s_i^i$, complete to $\{t_3^{i+1},t_3^{i+2}\}$ and anticomplete to $\{r_{i+1}^3,r_{i+2}^3\}$.
	\item $S_{i+2}^{i+1}(\rho)$ is the set of vertices in $V_0(\rho)$ which are nonadjacent to $s_i^i$, complete to $\{t_3^{i+1},r_{i+2}^3\}$ and anticomplete to $\{t_3^{i+2},r_{i+1}^3\}$.
	\item $S_{i+1}^{i+2}(\rho)$ is the set of vertices in $V_0(\rho)$ which are nonadjacent to $s_i^i$,  complete to $\{t_3^{i+2},r_{i+1}^3\}$ and anticomplete to $\{t_3^{i+1},r_{i+2}^3\}$.
\end{itemize}
Also, let $J=\{(i,j): 1\leq i,j\leq 3, i\neq j\}$. Since $G$ is antiprismatic, it is easy to verify that $\mathcal{D}(\rho)=(S_j^i(\rho), (i,j)\in J; T^i(\rho),R_i(\rho), i\in \{1,2,3\})$ is a partition of $V_0(\rho)$ into $12$ cliques (to see a proof for this fact, see 10.2 in \cite{seymour2}), which we call \textit{the decomposition corresponding to $\rho$} (see \pref{fig:decom_rotator}). Note that the decomposition depends not only on the rotator $\rho$ but also on the ordering of vertices in $\rho$. Henceforth, in all the definitions, we omit the term $\rho$ when there is no ambiguity.
For more convenience, define 
\[S=\bigcup_{(i,j)\in J} S_j^i, \ T=\bigcup_{i=1}^3T^i, \ R=\bigcup_{i=1}^3R_i,\  S'=S\cup \{s_1^1,s_2^2,s_3^3\}, \ T'=T\cup T_3,\ R'=R\cup R^3.\]
Also for every $i\in \{1,2,3\}$, define $S^i=S_{ i+1 }^i\cup S_{ i+2 }^i$ and $S_i=S_i^{ i+1 }\cup S_i^{ i+2 }$. Subsequently, let 
\[I_T=\{i: 1\leq i\leq 3 \text{ and }T^i\neq \emptyset\} \text{ and } I_R=\{j: 1\leq j\leq 3 \text{ and }R_j\neq \emptyset\}.\]

In order to provide appropriate clique covering for $G$,  we need to know more about adjacency between the above sets of vertices. Firstly, we need a definition. For every graph $H$, we define a triple $(x,y,U)$ to be a \textit{quasi-triad} of $H$, if $U$ is a subset of $V(H)$ (possibly empty) and $x,y$ are two distinct nonadjacent vertices in $V(H)\setminus U$ anticomplete to $U$. The following is a number of easy and useful observations about the structure of $G$ which are proved (directly or indirectly) in \cite{seymour2} and help us to give desired clique coverings. Here, we omit the proof.
\begin{lem}\label{lem:T,R Facts and Tools}
	Let $G$ be a counterexample to \pref{thm:noap}. Then the following hold (Statements {\rm (iii-vii)} also hold with $T,R$ exchanged).\\
	{\rm (i)} Let $(x,y,U)$ be a quasi-triad of $G$. Then $|U|\leq 1$, and every vertex in $V(G)\setminus (U\cup\{x,y\})$ which is complete to $\{x,y\}$, is anticomplete to $U$.\\
	{\rm (ii)}  For all $(i,j)\in J$, $|S_j^i|\leq 1$. Also, for every $k\in \{1,2,3\}$, $T^k$ and $R_k$ are both cliques of $G$. \\
	{\rm (iii)} For every $(i,j)\in J$, every vertex in $T^i$ has at most one non-neighbour in $T^j$.\\
	{\rm (iv)} For every $i\in \{1,2,3\}$ and every vertex $t\in T^i$, the sets $N_G(t,T\setminus T^i)$ and $N_{\overline{G}}(t, T\setminus T^i)$ are both cliques of $G$. Consequently, $G[T]$ contains no triad.\\
	{\rm (v)} For every vertex $t\in T$, the sets $N_G(t,R)$ and $N_{\overline{G}}(t,R)$ are both cliques of $G$.\\
	{\rm (vi)} For every pair of nonadjacent vertices $t^1,t^2\in T$, $(R^1,R^2)$ is a partition of $R$ into two cliques, where $R^1=N(t^1,R)$ and $R^2=N(t^2,R)$. Consequently, if $t^1$ and $t^2$ are two vertices in $T$ with a common non-neighbour $t^3\in T$, then $N(t^1,R)=N(t^2,R)=R\setminus N(t^3,R)$.\\
	{\rm (vii)} If $|I_T|=3$, then there exists $(i,j)\in J$ such that every vertex in $T^i$ has a neighbour in $T^j$.\\
	{\rm (viii)} For every $i\in \{1,2,3\}$, both $S^i\cup \{s_i^i\}$ and $S_i\cup \{s_i^i\}$ are cliques of $G$, and there are no more edges in $G[S']$.\\
	{\rm (ix)} For every $(i,j)\in J$, the sets $S_j\cup T^i$ and $S^i\cup R_j$ are both cliques of $ G $.
\end{lem}
According to \pref{lem:T,R Facts and Tools}~(ii), for every $(i,j)\in J$, $ |S_j^i|\leq 1 $. Let us denote the unique possible element of $S_j^i$ by $s_j^i$ (we often use $s_j^i\in S$ and $s_j^i\notin S$ instead of $S_j^i\neq \emptyset$ and $S_j^i=\emptyset$). 
\begin{figure}
	\begin{center}
		\begin{tikzpicture}
		\begin{scope}[scale=.5]
		\SetVertexNoLabel
		\tikzset{VertexStyle/.style ={shape=circle,inner sep=1mm,minimum size=15pt,draw}}
				\begin{scope}[xshift=-55]
		\node at (-2.5,2) [rectangle,draw,inner xsep=1.9cm,inner ysep=1.5cm] (S) {};
\node at (S.south west) [above left] {${S'}$};
		\node at (-.5,0) [inner sep=1pt] (s33) {$\bullet$};
		\node at (s33.north west) [xshift=-4pt] {$ s_3^3$};
		\node at (-.5,2) [circle,draw,inner sep=0pt] (s12) {$S_3^2$};
		\node at (-.5,4) [circle,draw,inner sep=0pt] (s13) {$S_3^1$};
		\node at (-2.5,0) [circle,draw,inner sep=0pt] (s21) {$S_2^3$};
        \node at (-2.5,2) [inner sep=1pt] (s22) {$ \bullet $};
		\node at (s22.north west) [xshift=-4pt] {$ s_2^2$};
        \node at (-2.5,4) [circle,draw,inner sep=0pt] (s23) {$S_2^1$};
		\node at (-4.5,0) [circle,draw,inner sep=0pt] (s21) {$S_1^3$};
		\node at (-4.5,2) [circle,draw,inner sep=0pt] (s22) {$S_1^2$};
		\node at (-4.5,4) [inner sep=1pt] (s11) {$\bullet$};
		\node at (s11.north west) [xshift=-4pt] {$ s_1^1$};
\node at (-.5,2.1) [rotate=-90,rounded rectangle,dashed,thick,red,draw,inner xsep=1.35cm,inner ysep=.5cm] (x1) {};
\node at (-2.5,0) [rounded rectangle,dashed,thick,blue,draw,inner xsep=1.5cm,inner ysep=.44cm] (z1) {};
		\end{scope}
		\node at (5,3.7) [rounded rectangle,draw,inner xsep=7mm,inner ysep=4pt] (t1) {$T^1$};
		\node at (5,1.9) [rounded rectangle,draw,inner xsep=7mm,inner ysep=4pt] (t2) {$T^2$};
		\node at (5,0.1) [rounded rectangle,draw,inner xsep=7mm,inner ysep=4pt] (t3) {$T^3$};
		\node at (8,3.7) [inner sep=1pt] (t31) {$\,\bullet\ t_3^1$};
		\node at (8,1.9) [inner sep=1pt] (t32) {$\,\bullet\ t_3^2$};
		\node at (8,0.1) [inner sep=1pt] (t33) {$\,\bullet\ t_3^3$};
		\node at (5.25,2) [rectangle,draw,inner xsep=1.85cm,inner ysep=1.5cm] (T) {};
		\node at (5,1.9) [rectangle,draw,dashed,inner xsep=1.1cm,inner ysep=1.3cm] (TT) {};	
		\node at (T.south west) [above left] {${T'}$};
		\node at (TT.west) [left] {$T$};
\node at (5.6,3.7) [rounded rectangle,dashed,red,thick,draw,inner xsep=1.6cm,inner ysep=.5cm] (x2) {};
\draw [bend left,red,thick,dashed] (x1.west) to node [below,xshift=-2mm] {$ X_3^1 $} (x2);
				\begin{scope}[xshift=30]
\node at (12,2.7) [rotate=-90,rounded rectangle,draw,inner xsep=7mm,inner ysep=4pt] (r1) {$R_1$};
\node at (14,2.7) [rotate=-90,rounded rectangle,draw,inner xsep=7mm,inner ysep=4pt] (r2) {$R_2$};
\node at (16,2.7) [rotate=-90,rounded rectangle,draw,inner xsep=7mm,inner ysep=4pt] (r3) {$R_3$};
\node at (12,.2) [inner sep=1pt] (r13) {$ \bullet$};
\node at (r13.south) [yshift=-5pt] { $ r^3_1$};
\node at (14,.2) [inner sep=1pt] (r23) {$\bullet$};
\node at (r23.south) [yshift=-5pt] {$ r^3_2$};
\node at (16,.2) [inner sep=1pt] (r33) {$\bullet$};
\node at (r33.south) [yshift=-5pt] {$ r^3_3$};
		\node at (13.8,2) [rectangle,draw,inner xsep=1.9cm,inner ysep=1.5cm] (R) {};
		\node at (14,2.7) [rectangle,dashed,draw,inner xsep=1.4cm,inner ysep=1.05cm] (RR) {};
		\node at (R.south west) [above left] {${R'}$};
		\node at (RR.west) [left] {$R$};
\node at (14,2) [rotate=-90,rounded rectangle,dashed,thick,blue,draw,inner xsep=1.45cm,inner ysep=.46cm] (z2) {};
\draw [bend left,blue,dashed,thick] (z2.east) [yshift=3cm] to node [above,xshift=-5mm,yshift=-8pt] {$ Z_2^3 $} (z1.south east);
		\end{scope}
		\end{scope}
		\end{tikzpicture}
	\end{center}
	\vspace{-35pt}
	\caption{The decomposition corresponding to a rotator. \label{fig:decom_rotator}}
\end{figure}
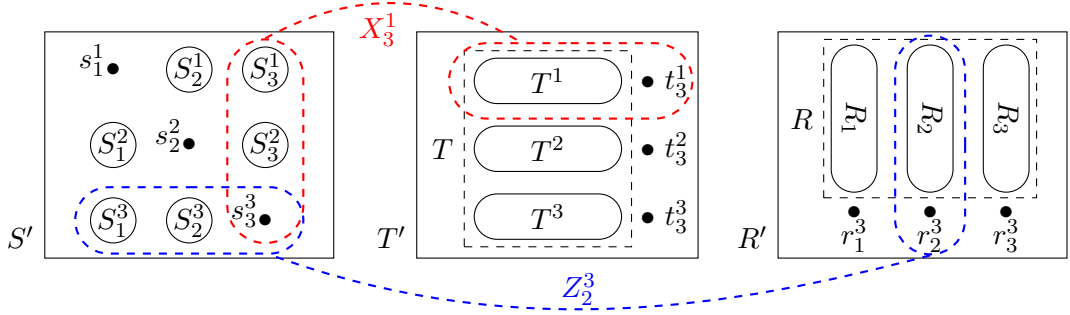

\pref{lem:T,R Facts and Tools} discloses the fact that the adjacency in $ G $ is very similar to the adjacency in the Schl\"{a}fli graph, except the adjacency between vertices in $ T\cup R $. For instance, in \pref{fig:decom_rotator}, each row and column in $ {S'},{T'} $ and $ {R'} $ is a clique and for $ 1\leq i,j\leq 3 $, column $ j $ in $ {S'} $ (resp. ${R'} $) is complete to row $ i $ in ${T'} $ (resp. ${S'} $) if and only if $ i\neq j $ (see \pref{lem:T,R Facts and Tools}~(viii,ix)). 
In fact, 
as it will be seen later, the major obstacle to proving \pref{thm:noap} is covering the edges of $G$ in $\mathcal{E}(G)=\bigcup_{(i,j)\in J}(E(T^i,T^j)\cup E(R_i,R_j))\cup E(T',R')\cup E(T_3)\cup E(R^3)$. 

Here, we introduce two collections of cliques of $G$ covering the edges in  $E(G)\setminus \mathcal{E}(G)$, and their various modified versions will be used to provide suitable clique coverings for $G$. The first one comes from a natural column-to-row clique covering of the Schl\"{a}fli graph and is defined as follows. For every $(i,j)\in J$, let $X_j^i=S_j \cup T^i\cup \{s_j^j,t_3^i\}$ and $Z_j^i=R_j \cup S^i\cup \{s_i^i,r_j^3\}$ 
which are cliques of $G$. Define $\mathscr{O}(G)=\{X_j^i,Z_j^i: (i,j)\in J\}$ and note that $|\mathscr{O}(G)|=12$.
The second collection is a bit more complicated, however enables us to conquer several sporadic cases. For every $(i,j)\in J$, define
\[\Delta_j^i=(S^i\setminus S_j^i)\cup \{s_i^i\}\cup T^j\cup (R^3\setminus \{r_i^3\}),\quad \nabla _j^i=(S_i\setminus S_i^j)\cup \{s_i^i\}\cup (T_3\setminus \{t_3^i\})\cup R_j.\]
Also, for every $i\in \{1,2,3\}$, let $\Omega ^i=S^i\cup T^i\cup \{t^i_3\}$ and $\Omega _i=S_i\cup R_i\cup \{r^3_i\}$. Note that by \pref{lem:T,R Facts and Tools}~(viii,ix) and the definitions, all the sets $\Delta _j^i, \nabla _j^i,\Omega ^i$ and $\Omega _i$ are cliques of $G$. Now, let $\mathscr{P}(G)$ be the family of cliques of $G$ constructed according to the following algorithm.
\begin{itemize}
	\item[\textbf{P1.}] Let $\mathscr{P}(G)=\emptyset$.
	\item[\textbf{P2.}] For every $i\in \{1,2,3\}$, if $T^i\neq \emptyset$, then add the cliques $\Delta _i^{ i+1 }$ and $\Delta _i^{ i+2 }$  to $\mathscr{P}(G)$.
	\item[\textbf{P3.}] For every $i\in \{1,2,3\}$, if either $T^i\neq \emptyset$ or $|S^i|= 2$, then add the clique $\Omega ^i$ to $\mathscr{P}(G)$.
	\item[\textbf{P4.}] For every $i\in \{1,2,3\}$, if $R_i\neq \emptyset$, then add the cliques $\nabla _i^{ i+1 }$ and $\nabla _i^{ i+2 }$  to $\mathscr{P}(G)$.
	\item[\textbf{P5.}] For every $i\in \{1,2,3\}$, if either $R_i\neq \emptyset$ or $|S_i|= 2$, then add the clique $\Omega _i$ to $\mathscr{P}(G)$.
	\item[\textbf{P6.}] For every $(i,j)\in J$ with $s_j^i\in {S}$, add both cliques $\Delta _{6-i-j}^i$ and $\nabla _{6-i-j}^j$ to $\mathscr{P}(G)$ (if they are not added before).
	\item[\textbf{P7.}] For every $i\in \{1,2,3\}$, add one of the cliques $\Delta _j^i$, $j\in \{1,2,3\}\setminus \{i\}$ and also one of the cliques $\nabla _j^i$, $j\in \{1,2,3\}\setminus \{i\}$  to $\mathscr{P}(G)$ (if no such cliques are added before).
\end{itemize}
Then, we have $|\mathscr{P}(G)|\leq 18$. Now, we state a simple lemma about the edges covered by the cliques in $\mathscr{O}(G)$ and $\mathscr{P}(G)$, which follows from previous arguments and we leave the proof to the reader.
\begin{lem}\label{lem:O&P} The following hold. \\
	{\rm (i)} The cliques in $\mathscr{O}(G)$ cover all the edges in $E(G)\setminus \mathcal{E}(G)$.\\
	{\rm (ii)} The cliques in $\mathscr{P}(G)$ cover all the edges in $E(G)\setminus \mathcal{E}(G)$, as well as the edges in $E(T,R^3)\cup E(T_3,R)\cup E(T_3)\cup E(R^3)$.
\end{lem} 
Note that, to avoid repetition, we refuse to refer to \pref{lem:O&P} and it is assumed that the reader will be considering its truth throughout the rest of the paper. Furthermore, we derive more information about $ G $ based on the fact that $ \overline{G} $ excludes square-forcer. 
\begin{lem}\label{lem:square-forcer}
	Let $G$ be a counterexample to \pref{thm:noap} and $\{i,j,k\}=\{1,2,3\}$. The following hold.\\
	\rm{(i)} If $G[T\setminus T^i]$ is a clique and  $k\notin I_R$, then either $s_i^j\in S$, or $s_l^i\in S$, for some $l\in I_T\setminus \{i\}$.\\
	\rm{(ii)} If $G[R\setminus R_i]$ is a clique and  $k\notin I_T$, then either $s_j^i\in S$, or $s_i^l\in S$, for some $l\in I_R\setminus \{i\}$.
\end{lem}
\begin{proof}
	If (i) does not hold, then $(s_i^i,r_k^3)$ is a square-forcer in $\overline{G}$, a contradiction to \pref{thm:Inc-rotator-Exc-sf}. The truth of (ii) follows similarly.  
\end{proof}

We conclude this subsection with three useful lemmas which will be used in the forthcoming subsections. For simplicity, we omit the term $ \rho $ within the proofs. 

\begin{lem}\label{lem:|I_T|+|I_R|>=4}
	Let $G$ be a counterexample to \pref{thm:noap}. Then for every rotator $\rho$ of $\overline{G}$, we have  $|I_T(\rho)|+|I_R(\rho)|\geq 4 $.
\end{lem}
\begin{proof}
	Suppse not, let $|I_T|+|I_R|\leq 3$. Then, we have\vsp
	
	(1) \textit{Both $T$ and $R$ are nonempty.}\vsp
	
	Suppose not, by symmetry, let $R=\emptyset$. Merge the pairs $(Z_2^1,Z_3^1),(Z_1^2,Z_3^2)$ and $(Z_1^3,Z_2^3)$ in $\mathscr{O}(G)$, and call the resulting collection $\mathscr{O}_1(G)$. First, assume that $|I_T|\leq 2$, say $T^3=\emptyset$. By \pref{lem:square-forcer}~(ii) (applying to $(i,j,k)=(1,2,3)$ and $(i,j,k)=(2,1,3)$), both $s_2^1,s_1^2\in S$. Thus, $|S|\geq 2$, and adding the cliques in $\mathscr{N}[T^1;T^2\cup R^3]\cup \{T^2\cup R^3,T_3\}$ to $\mathscr{O}_1(G)$, yields a clique covering $\mathscr{C}_1$ for $G$ of size $11+|T^1|=n-|S|-|T^2|+2\leq n-|T^2|$. Thus, $T^2=\emptyset$ and by \pref{lem:square-forcer}~(i) (applying to $(i,j,k)=(3,1,2)$), at least one of $s_3^1,s_1^3\in S$, and so $|S|\geq 3$, which implies that $|\mathscr{C}_1|\leq n-1$, a contradiction. Therefore, $|I_T|=3$.  By \pref{lem:T,R Facts and Tools}~(vii) and w.l.o.g. we may assume that for every $t\in T^3$, $N(t,T^2)\neq \emptyset$. Thus, the cliques in $\mathscr{N}[T^2;T^3\cup R^3]$ cover all the edges in $E(T^3,R^3)\cup E(R^3)$. Also, by \pref{lem:T,R Facts and Tools}~(iv), $\mathscr{N}[T^1;(T\setminus T^1)\cup R^3]$ is a set of cliques of $G$. Thus, $\mathscr{O}_1(G)\cup \mathscr{N}[T^1;(T\setminus T^1)\cup R^3]\cup \mathscr{N}[T^2;T^3\cup R^3]\cup \{T_3\}$ 
	is a clique covering for $G$ of size $n-|S|-|T^3|+1\leq n-|S|$. Hence, $S=\emptyset$. Now, the cliques in $(\mathscr{P}(G)\setminus \{\Delta _2^1,\Delta_3^1,\Delta_3^2,\Delta _1^3 \}) \cup (\bigcup_{i=1}^3\mathscr{N}[T^i;\Delta_{ i+1 }^{ i+2 }])$ cover all edges covered by $\mathscr{P}(G)$ as well as the edges in $E(T)$ and so this family is a clique covering for $G$ of size $8+\sum _{i=1}^3|T_i|=n-1$, a contradiction. This proves (1).\vsp
	
	(2) \textit{We have $|I_T|+|I_R|=3$.}\vsp
	
	Suppose not, let $|I_T|+|I_R|\leq 2$. By (1), $|I_T|=|I_R|=1$, say $I_T=\{p\}$ and $I_R=\{q\}$, for some $p,q \in \{1,2,3\}$. In $\mathscr{O}(G)$, merge the pairs $(X^{p+1}_{p},X^{p+2}_{p})$ and $(Z_{q+1}^{q},Z_{q+2}^{q})$ and add the cliques in $\mathscr{N}[T^p;R_q]\cup \{T^p\cup R^3,T_3\cup R_q\}$ to the resulting family, thereby obtaining a clique covering $\mathscr{C}$ for $G$ of size $12+|T^p|=n-|S|-|R_q|+3$. Thus $|S| \leq 2$. Now, if $p=q$, say $p=q=1$, then by \pref{lem:square-forcer}~(i) and (ii) (both applying to $(i,j,k)=(1,2,3)$ and $(i,j,k)=(1,3,2)$), all $s_2^1,s_3^1,s_1^2,s_1^3\in S$, a contradiction. Thus, $p\neq q$, say $p=1$ and $q=2$. Again by \pref{lem:square-forcer}~(i) (applying to $(i,j,k)=(1,2,3)$ and $(i,j,k)=(3,1,2)$), $s_1^2\in S$ and at least one of $s_1^3,s_3^2\in S$. Thus, $|S|=2$ (i.e. either $ S=\{s_1^2,s_1^3\}$ or $ S=\{s_1^2,s_3^2\} $). Now, in $\mathscr{C}$, remove the clique $X_3^1$, replace the cliques $X_2^1$ with the clique $X_2^1\cup S_3^2$ and replace the cliques in $\mathscr{N}[T^p;R_q]$ with the cliques in $\mathscr{N}[T^1;R_2\cup \{s_3^3,t_3^1\}]$. This yields a clique covering for $G$ of size $11+|T^1|=n-|R_2|\leq n-1$, a contradiction. This proves (2). \vsp
	
	By (1) and (2) and due to symmetry, we may assume that $I_T=\{1,2\}$ and $I_R=\{p\}$, for some $p \in \{1,2,3\}$. First, note that if $p \in \{1,2\}$, then by \pref{lem:square-forcer}~(ii) (applying to $(i,j,k)=(p,3-p,3)$), $s_{3-p }^p\in S$ (and so $|S|\geq 1$). Also, if $p=3$,  then by \pref{lem:square-forcer}~(ii) (applying to $(i,j,k)=(l,3-l,3)$), for every $l\in \{1,2\}$, $\{s_{3-l}^l,s_l^3\}\cap S\neq \emptyset$ (and so $|S|\geq 2$). Now, in $\mathscr{O}(G)$, merge the pair $(Z_{p+1}^{p},Z_{p +2}^p)$ and call the resulting collection $\mathscr{O}'(G)$. We prove the following claim.\vsp
	
	(3) \textit{There exists $i\in I_T$ such that every vertex in $T^{3-i}$ has a neighbour in $T^i$.}\vsp
	
	Suppose not, pick $t^1\in T^1$ with no neighbour in $T^2$. Thus $(t^1,t_3^3,T^2)$ is a quasi-triad of $G$ and by \pref{lem:T,R Facts and Tools}~(i), $|T^2|=1$, say $T^2=\{t^2\}$. Similarly, since $t^2$ has no neighbour in $T^1$, we have $T^1=\{t^1\}$. 
	Now, $\mathscr{C}=\mathscr{O}'(G)\cup \mathscr{N}[R_p;T]\cup  \{T^1\cup R^3,T^2\cup R^3,T_3\cup R_p\}$ is a clique covering for $G$ of size $14+|R_p|\leq n-|S|+3$. Thus, $ |S|\leq 3$. 
	
	Firstly, assume that $p=3$, and so $2\leq |S|\leq 3$. Now, if either $|S|=2$ or $S_3\neq \emptyset$, then $|\mathscr{P}(G)|\leq 11+|S|$. Thus, in $\mathscr{P}(G)$, replacing the cliques $\Omega ^ i$ with the clique $\Omega ^i\cup N(t^i,R)$ yields a clique covering for $G$ of size at most $11+|S|=n-|R_3|\leq n-1$, a contradiction. Also, if $|S|=3$ and  $S_3=\emptyset$, then $|\mathscr{P}(G)|\leq 12+|S|$. Thus, in $\mathscr{P}(G)$, replacing the cliques $\Omega ^ i$ with the clique $\Omega ^i\cup N(t^i,R)$ yields a clique covering for $G$ of size at most $12+|S|=n-|R_3|+1$. Hence, $|R_3|=1$,  and by \pref{lem:T,R Facts and Tools}~(vi), say $R_3$ is complete to $t^1$ and anticomplete to $t^2$. Now, remove the cliques $Z_3^1,Z_3^2$ as well as the single clique in $\mathscr{N}[R_3;T]$ from $\mathscr{C}$ and add the cliques $R_3\cup S_1^2\cup \{s_1^1,r_3^3\}$ and $R_3\cup S_2^1\cup \{s_2^2,t^1,r_3^3\}$, thereby obtaining a clique covering for $G$ of size $14=n-1$.
	
	Secondly, suppose that $ p\in\{1,2\} $, say $ p=1 $ and so $ s_2^1\in S $. 
	Now, if $|S|=|S_2^1|=1$, then in $\mathscr{P}(G)$, replace the clique $ \Omega_1 $ with the cliques in $ \mathscr{N}[R_1;T\cup \{r_1^3\}]$, thereby obtaining a clique covering of size $11+|R_1|=n-1$. If $|S|=2$, then in the case that $s_1^2\in S$, in the collection $ \mathscr{O}'(G)\cup \{T^1\cup R^3,T^2\cup R^3,T_3\cup R_1\} $, remove the clique $Z_1^3$, replace the cliques $Z_2^3$, $X_3^1$ and $X_3^2$ with the cliques $Z_2^3\cup \{r_1^3\}$, $X_3^1\cup N(t^1,R_1)$ and $X_3^2\cup N(t^2,R_1)$. Since by \pref{lem:T,R Facts and Tools}~(vi), every vertex in $R_1$ is either adjacent to $t^1$ or $t^2$, this yields to a clique covering for $G$ of size $13=n-|R_1|\leq n-1$. Also, in the case that $s_1^2\notin S$, we have $ |\mathscr{P}(G)|= 12+|S_3^1|$. Now, in $\mathscr{P}(G)$, and if $S_1=\emptyset$, then replace the clique $\Omega _1$ with the cliques in $\mathscr{N}[R_1;(\Omega _1\cup T)\setminus R_1]$, and if $S_1\neq \emptyset$ (i.e. $S_3^1=\emptyset$), then add the cliques in $\mathscr{N}[R_1;T]$, thereby obtaining a clique covering for $G$ of size at most $12+|R_1|=n-1$.
	
	Finally, in the case that $|S|=3$, if $s_3^1\notin S$, then removing the cliques in $\mathscr{N}[R_1;T]$ from $\mathscr{C}$ and replacing the clique $X_3^i$ with the cliques $X_3^i\cup N(t^i,R_1)$ for every $i\in \{1,2\}$, yields a clique covering for $G$ of size $14=n-|R_1|\leq n-1$. Thus, $s_3^1\in S$. Also, if $S^3=\emptyset$, then removing the cliques in $\mathscr{N}[R_1;T]$ from $\mathscr{C}$ and replacing the clique $Z_1^3$ with the cliques in  $\mathscr{N}[R_1;T\cup \{s_3^3,r_1^3\}]$ yields a clique covering for $G$ of size $13+|R_1|=n-1$. Thus, $S=\{s_2^1,s_3^1,s_j^3\}$, for some $j\in \{1,2\}$. Now, if $j=1$, then remove the clique $T^2\cup R^3$ from $\mathscr{C}$ and replace the cliques $Z_3^1,Z_3^2$ and $Z_2^3$ with the cliques $(Z_3^1\cup \{t^2\})\setminus S_2^1,Z_3^2\cup S_2^1$ and $Z_2^3\cup \{t^2,r_1^3\}$, and if $j=2$, then remove the clique $T^1\cup R^3$ from $\mathscr{C}$ and replace the cliques $Z_3^2$ and $Z_2^3$ with the cliques $Z_3^2\cup \{t^1,r_1^3\}$ and $Z_2^3\cup \{t^1\}$, thereby obtaining a clique covering for $G$ of size $13+|R_1|=n-1$. This proves (3).\vsp
	
	Now, (3) ensures that say every vertex in $ T^2 $ has a neighbour in $ T^1 $. Consequently, the collection $\mathscr{O}'(G)\cup \mathscr{N}[R_p;T]\cup \mathscr{N}[T^1;T^2\cup R^3]\cup \{T_3\cup R_p\}$ is a clique covering for $G$ of size $n-|S|-|T^2|+3$. Thus, we have $|S|+|T^2|\leq 3$. 
	
	First, suppose that $p=3$. Then $|S|=2$ and $|T^2|=1$,  say $T^2=\{t^2\}$. Let $X=N(t^2,T^1)$. Note that by the assumption, $|X|\geq 1$ and by \pref{lem:T,R Facts and Tools}~(iii), $|T^1\setminus X|\leq 1$. If $|T^1\setminus X|=1$, then in $\mathscr{P}(G)\cup \mathscr{N}[T^2;T^1]$, replace the clique $\Omega_3$ with the cliques in $\mathscr{N}[R_3;(\Omega _3\cup T)\setminus R_3]$, to obtain a clique covering  for $G$ of size at most $13+|R_3|=n-|X|\leq n-1$. If $T^1\setminus X=\emptyset$ and $|S^3|\leq 1$, then in $\mathscr{P}(G)$, replace the clique $\Omega_3$ with the cliques in $\mathscr{N}[R_3;(\Omega _3\cup T)\setminus R_3]$ and also choose $i\in \{1,2\}$ such that $s^3_i\notin S$ and replace the clique $\Delta _{3-i}^3$ with the clique $\Delta _{3-i}^3\cup T^i$, to obtain a clique covering for $G$ of size at most $12+|R_3|=n-|X|\leq n-1$. Also, if $T^1\setminus X=\emptyset$ and $|S^3|=2$, then in the family $\mathscr{O}'(G)\cup \mathscr{N}[R_3;T]\cup \{T_3\cup R_3\}$, in order to cover the edges in $E(T\cup R^3)$, replace the cliques $X_3^1,X_3^2,Z_2^1$ and $Z_1^2$ with the cliques $T\cup \{s_3^3,r_1^3,r_2^3\},\{s_3^3,t_3^1,t_3^2\},Z_2^1\cup T^2\cup \{r_3^3\}$ and $Z_1^2\cup T^1\cup \{r_3^3\}$, thereby having a clique covering for $G$ of size $12+|R_3|=n-|X|\leq n-1$. 
	
	Finally, suppose that $p \in \{1,2\}$, say $p =1$. Now, if $ |S|=|S_2^1|=1$, then in $\mathscr{P}(G)\cup \mathscr{N}[R_1;T]$, replace the cliques $\Delta_1^3$ and $\Delta_2^3$ with the cliques in $\mathscr{N}[T^1;\Delta_2^3] $ to obtain a clique covering for $G$ of size at most $10+|T^1|+|R_1|\leq n-1$. Also, if $|S|=2$, then $|T^2|=1$, say $T^2=\{t^2\}$. Again, let $X=N(t^2,T^1)$, where $|X|\geq 1$ and  $|T^1\setminus X|\leq 1$. Now, consider the collection $\mathscr{P}(G)$, in order to cover the edges in $E(T,R)$, if $S_1=\emptyset$, then replace the clique $\Omega _1$ with the cliques in  $\mathscr{N}[R_1;(\Omega _1\cup T)\setminus R_1]$ and if $S_1\neq \emptyset$, then add the cliques in $\mathscr{N}[R_1;T]$. Also, in order to cover the edges in $E(T^1,T^2)$, if $S^3=T\setminus X=\emptyset$, then merge the pair $(\Delta _1^3,\Delta _2^3)$, if $S^3=\emptyset$ and $T\setminus X\neq \emptyset$, then replace the clique $\Delta _2^3$ with the clique $\Delta _2^3\cup X$, if $S^3\neq \emptyset$ and $T\setminus X=\emptyset$, then choose $i\in \{1,2\}$ such that $s^3_i\notin S$, and replace the clique $\Delta _{3-i}^3$ with the clique $\Delta _{3-i}^3\cup T^i$, and if $S^3\neq \emptyset$ and $T\setminus X\neq \emptyset$, then add the clique $N[t^2;T^1]$. We leave the reader to check that this procedure yields a clique covering for $G$ of size at most $n-1$, a contradiction. This proves \pref{lem:|I_T|+|I_R|>=4}.
\end{proof}

\begin{lem}\label{lem:both clique}
	Let $G$ be a counterexample to \pref{thm:noap}. Then for every rotator $\rho$ of $\overline{G}$, either $T(\rho)$ or $R(\rho)$ is not a clique. 
\end{lem}
\begin{proof}
	On the contrary, assume that both $T$ and $R$ are cliques of $G$. Note that the cliques in  $\mathscr{C}=\mathscr{O}(G)\cup \{T\cup R^3,T_3\cup R\}$ cover all the edges in $E(G)\setminus E(T,R)$. \vsp
	
	(1) \textit{We have $|S|\geq 2$.}\vsp
	
	Suppose not, let $|S|\leq 1$. Let $\{k,k',k''\}=\{1,2,3\}$ and $S= S_{k'}^k$. Considering $\mathscr{P}(G)$, for every $i\in\{1,2,3\}\setminus \{k\}$ (resp. $i\in \{1,2,3\}\setminus \{k'\}$), if both $\Delta^i_{i+1}, \Delta^i_{i+2}\in \mathscr{P}(G)$ (resp. $\nabla^{i}_{i+1},\nabla_{i+2}^{i}\in \mathscr{P}(G)$), then merge the pair $(\Delta^i_{i+1}, \Delta^i_{i+2})$ (resp. $(\nabla^{i}_{i+1},\nabla_{i+2}^{i})$) and if $T^{k'}$ (resp. $R_{k}$) is nonempty, then replace the clique $\Delta^k_{k'}$ (resp. $\nabla_{k}^{k'}$) with the clique $\Delta^k_{k'}\cup T^{k''}$ (resp. $\nabla_{k}^{k'}\cup R_{k''}$), thereby covering all the edge in $E(T)\cup E(R)$. Also, in order to cover the edges in $E(T,R)$, for every $j\in I_T\setminus \{k\}$, replace the clique $\Omega^j$ with the cliques in $\mathscr{N}[T^j;(\Omega ^j \cup R)\setminus T^j]$ and for every $j\in I_R$, replace the clique $\Omega_j$ with the cliques in $\mathscr{N}[R_j;(\Omega _j \cup T^k)\setminus R_j]$. This provides a clique covering for $G$ of size at most $8+|T|+|R|\leq n-1$. This proves (1). \vsp
	
	(2)\textit{ We have $|I_T|, |I_R|\geq 2$, and thus $|T|,|R|\geq 2$. }\vsp
	
	Suppose not, w.l.o.g. let $|I_R|\leq 1$. By \pref{lem:|I_T|+|I_R|>=4}, $|I_T|=3$ and $I_R=\{p\}$, for some $p\in \{1,2,3\}$. Now, in $\mathscr{C}$, merge the pair $(Z_{p+1}^p, Z_{p+2}^p)$ and add the cliques in $\mathscr{N}[R_{p};T]$ to obtain a clique covering for $G$ of size $13+|R_p|=n-|S|-|T|+4$, which by (1) is at most $n-1$, a contradiction. This proves (2). \vsp
	
	(3) \textit{Let $\{k,k',k''\}=\{1,2,3\}$. If $k,k'\in I_T$, then  $S^{k''}\neq \emptyset$. Similarly,  if $k,k'\in I_R$, then $S_{k''}\neq \emptyset$.} \vsp
	
	Suppose not, by symmetry, assume that $1,2\in I_T$ and $S^3$ is empty. Considering $\mathscr{C}$, for every $j\in I_R\cap \{1,2\}$, replace the clique $Z_j^3$ with the cliques in $\mathscr{N}[R_j;T^1\cup T^2\cup \{s_3^3,r_j^3\}]$ and for every $j\in\{1,2\}$, replace the clique $X^j_{3-j}$ with the cliques in  $\mathscr{N}[T^{j};R_3\cup S_{3-j}\cup \{s_{3-j}^{3-j},t_3^{j}\}]$, thereby  covering all the edges in $E(T^1\cup T^2, R)$. Now, add the cliques in $\mathscr{N}[T^3,R]$ to the resulting family to obtain a clique covering $\mathscr{C}'$ for $G$. If $\{1,2\}\subseteq I_R$, then by (1), $|\mathscr{C}'|=n-|S|-|R_3|+1\leq n-1$. Otherwise, by (2), $I_R=\{j,3\}$, for some $j\in\{1,2\}$ and by (1), $|\mathscr{C}'|=n-|S|-|R_3|+2\leq n-1$, a contradiction. This proves (3). \vsp
	
	(4) \textit{We have $|S|=3$, $|I_T|=|T|=|I_R|=|R|=2$. Also, $T$ is neither complete nor anticomplete to $R$.}\vsp
	
	By symmetry, assume that $|T|\leq |R|$. Note that the collection $\mathscr{C}\cup \mathscr{N}[T;R]$ is a clique covering for $G$ of size $n-|S|-|R|+5$. Thus, $|S|+|R|\leq 5$. Consequently, by (1) and (2), we have $2\leq |S|\leq 3$. Now, if  either $|I_T|=3$ or $|I_R|=3$, then by (3), $|S|\geq 3$, and so $|S|+|R|\geq 6$, which is impossible. Hence, by (2), $|I_T|=|I_R|=2$, and w.l.o.g. let $I_T=\{1,2\}$ and $I_R=\{1,p\}$, for some $p\in\{2,3\}$. Note that by \pref{lem:square-forcer}~(i) (applying to $(i,j,k)=(1,p,5-p)$), $S_2^1\cup S_1^p\neq \emptyset$.  Now, in order to prove that $|S|=3$, note that if $p=2$, then by (3), both $S^3,S_3\neq \emptyset$ and so $|S|=3$. Also, if $p=3$ and $|S|=2$, then since by (3) both $S^3,S_2\neq \emptyset$, we have $S\subseteq S_2^1\cup S_1^3\cup S_2^3$. Now, considering $\mathscr{C}$, replace the cliques $X_3^1$ and $X_2^1$ with the clique $X_3^1\cup S^3_2$ and the cliques in $\mathscr{N}[T^1;R_3\cup S^1_2\cup \{s_2^2,t_3^1\}]$ to cover the edges in $E(T^1,R_3)$. Also, replace the cliques $Z_3^2$ and $Z_3^1$ with the clique $Z_3^2\cup S^1_2$ and the cliques in $\mathscr{N}[R_3;T^2\cup \{s^1_1,r_3^3\}]$ to cover the edges in $E(T^2,R_3)$. Moreover, replace the clique $X_3^2$ with the cliques in $\mathscr{N}[T^2;R_1\cup \{s_3^3,t_3^2\}]$ to cover the edges in $E(T^2,R_1)$ and replace the clique $Z_1^2$ with the cliques in $\mathscr{N}[R_1;T^1\cup \{s_2^2,r_1^3\}]$ to cover the edges in $E(T^1,R_1)$. This yields a clique covering for $G$ of size $10+|T|+|R|=n-1$, a contradiction. Therefore, $|S|=3$. This, together with (2) and inequalities $|T|\leq |R|$ and $|S|+|R|\leq 5$, implies that $|I_T|=|T|=|I_R|=|R|=2$. Finally, note that if $T$ is either complete or anticomplete to $R$, then either $\mathscr{C}\cup \{T\cup R\}$ or $\mathscr{C}$ would be a clique covering for $G$ of size at most $15=n-1$, a contradiction. This proves (4).\vsp
	
	Henceforth, by (4) and w.l.o.g. we may assume that $I_T=\{1,2\}$, $I_R=\{1,p\}$, for some $p\in\{1,2,3\}$, $T^i=\{t^i\}$ and $R_j=\{r_j\}$, for every $i\in \{1,2\}$ and $j\in \{1,p\}$. Now, by \pref{lem:square-forcer}, $ S^1_2\cup S^p_1\neq \emptyset $ and w.l.o.g. we may assume that $S_2^1\neq \emptyset$ (note that if $S_2^1=\emptyset$, then $S_1^p\neq \emptyset$ and we may consider the rotator $\rho' =(s_1^1,s_p^p,s_{5-p}^{5-p}; r_1^3,r_p^3,r_{5-p}^3; t_3^1,t_3^p,t_3^{5-p})$, where $S^1_2(\rho ')\neq \emptyset$). We observe that $r_1$ is either complete or anticomplete to $T$ (otherwise, if $t^ir_1$ is an edge and $t^{3-i}r_1$ is a non-edge, then $\{s_{3-i}^1,t^1,t^2,r_1\}$ induces a claw). Now, if $r_1$ is anticomplete to $T$, then so is $r_{p}$ (otherwise, if  $t^ir_{p}$  is an edge, for some $i\in\{1,2\}$, then $\{r_p,s_i^1,t^i,r_1\}$ induces a claw) and thus $T$ is anticomplete to $R$, which contradicts (4). Thus, $r_1$ is complete to $T$. Note that by (4), $r_p$ is nonadjacent to $t^{i_0}$, for some ${i_0}\in \{1,2\}$. Thus,  $s^p_{i_0}\notin S$  (otherwise,  $\{r_1, s_{i_0}^p,t^{i_0},r_p\}$ would be a claw). Now, add the clique $T\cup R_1$ to $\mathscr{C}$, and if $r_p$ is adjacent to $t^{3-{i_0}}$, then replace the clique $X_{{i_0}}^{3-{i_0}}$ with the clique $X^{3-{i_0}}_{i_0}\cup \{r_p\}$ to cover the edge $t^{3-i_0}r_p$, thereby obtaining a clique covering for $G$ of size $15=n-|S|+2$, which by (4) is equal to $n-1$, a contradiction. This proves \pref{lem:both clique}.
\end{proof}

\begin{lem}\label{lem:one is complete to R}
	Let $G$ be a counterexample to \pref{thm:noap} with  $|I_T(\rho)|,|I_R(\rho)|\geq 2$, for some rotator $\rho$ of $\overline{G}$,  say $\{i,j\}\subseteq I_T(\rho)$ for some $i\neq j$. Also, let $t^i\in T^i(\rho)$ is nonadjacent to $t^j\in T^j(\rho)$ and $R(\rho)$ is a clique. If one of the following holds, then exactly one of $t^i$ and $t^j$ is complete to $R(\rho)$ and the other one is anticomplete to $R(\rho)$.\\
	{\rm (i)} $i\in I_R(\rho)$ and $S_{j}^i(\rho)\neq \emptyset$.\\
	{\rm (ii)} There exists $l\in I_R(\rho)\setminus \{i,j\}$ such that for every $l'\in I_R(\rho)\cap \{i,j\}$, $S_{l'}^l(\rho)\neq \emptyset$.\\
	The same statement holds when $T$ is replaced with $R$ and the subscripts and the superscripts are interchanged.
\end{lem}
\begin{proof}
	First, assume that (i) holds. Note that if $t^i$ is nonadjacent to some $r_i\in R_i$, then $t^i$ is anticomplete to $R\setminus R_i$ (otherwise $\{r,s_i^i,t^i,r_i\}$ is a claw, for some $r\in R\setminus R_i$ adjacent to $t^i$). Consequently, $t^i$ is anticomplete to $R_i$ (otherwise by \pref{lem:T,R Facts and Tools}~(vi), $\{r,s_{j}^i,t^{j},r'_i\}$ would be a claw for every $r'_i\in R_i$ adjacent to $t^i$ and every $r\in R\setminus R_i$). Therefore, $t^i$ is anticomplete to $R$ and by \pref{lem:T,R Facts and Tools}~(vi), $t^j$ is complete to $R$, as desired. Now, if $t^i$ is complete to $R_i$, then so is to $R\setminus R_i$ (otherwise $\{r,s_{j}^i,t^{j},r_i\}$ would be a claw for some $r\in R\setminus R_i$ and every $r_i\in R_i$). Thus, $t^i$ is complete to $R$ and so by \pref{lem:T,R Facts and Tools}~(vi), $t^j$ is anticomplete to $R$.
	
	Next, suppose that (ii) holds, and w.l.o.g. assume that  $i\in I_R$.  Note that if $t^i$ is nonadjacent to some $r_i\in R_i$, then $t^i$ is anticomplete to $R\setminus R_i$ (otherwise $\{r,s_i^i,t^i,r_i\}$ is a claw, for some $r\in R\setminus R_i$). Consequently, $t^i$ is anticomplete to $R_i$ (otherwise $\{r'_i,s_{i}^l,t^{i},r\}$ would be a claw for every $r'_i\in R_i$ adjacent to $t^i$ and every $r\in R_l$), and thus anticomplete to $R$, as required. Now, if $t^i$ is complete to $R_i$, then $t^i$ is also  complete to $R_l$ (otherwise, $\{r_i,s_{i}^l,t^{i},r_l\}$ would be a claw for some $r_l\in R_l$ nonadjacent to $t^i$ and every $r_i\in R_i$) and so complete to $R_j$ (otherwise,  $\{r_j,s_{j}^l,t^{j},r_l\}$ would be a claw for some $r_j\in R_j$ nonadjacent to $t^i$ and every $r_l\in R_l$), as desired. This proves \pref{lem:one is complete to R}.
\end{proof}

\subsection{Types of $G[T]$ and $G[R]$}\label{sub:types}
In order to construct appropriate clique coverings for $G$, we need to know more about the structure of $G[T]$ and $G[R]$. Let us begin with a definition (all definitions and theorems here are stated in terms of $G[T]$ and the analogous ones are also valid for $G[R]$, where $T$ is replaced with $R$ and the subscripts and superscripts are interchanged).

Note that if $|I_T(\rho)|\leq 1$, for some rotator $\rho$ of $\overline{G}$, then $T(\rho)$ is a clique of $G$. Now, assume that $|I_T(\rho)|=2$, say $I_T(\rho)=\{q,q'\}$.
We say that $G[T(\rho)]$ is of \textit{type $1$}, if there exist nonadjacent vertices $t^q\in T^q(\rho)$ and $t^{q'}\in T^{q'}(\rho)$ such that $t^{q}$ is complete to $U_q(\rho)=T^{q'}(\rho)\setminus \{t^{q'}\}$ and  $t^{q'}$ is complete to $U_{q'}(\rho)=T^{q}(\rho)\setminus \{t^{q}\}$. If $U_q(\rho)$ (resp. $U_{q'}(\rho)$) is empty, then we say that $G[T(\rho)]$ is of \textit{type $1.1$ at $q$ (resp. $q'$)}. If both $U_q(\rho)$ and $U_{q'}(\rho)$ are nonempty, then we say that $G[T(\rho)]$ is of \textit{type $1.2$}. The following lemma more or less reveals the structure of $G[T(\rho)]$ in case of $|I_T(\rho)|=2$. The proof is obvious and left to the reader.

\begin{lem}\label{lem:G[T] structure |I_T|=2}
	Let $G$ be a counterexample to \pref{thm:noap} with $|I_T(\rho)|=2$, for some rotator $\rho$ of $\overline{G}$. Then, either $T(\rho)$ is a clique or $G[T(\rho)]$ is of type $1$.
\end{lem}

Next, suppose that $|I_T(\rho)|=3$. Let ${q}\in \{1,2,3\}$. We say that $G[T(\rho)]$ is of \textit{type $2$ at $q$}, if $T^{q'}(\rho)=\{t^{q'}\}$ for every $q'\in \{1,2,3\}\setminus \{q\}$ and there exists $t^{q}\in T^{q}(\rho)$ which is anticomplete to $ T(\rho)\setminus T^q(\rho) $ and $ T(\rho)\setminus\{t^q\} $ is a clique. 
Moreover, we say that $G[T(\rho)]$ is of \textit{type $3$ at $q$}, if $T^{q}(\rho)=\{t_1^{q},t_2^{q}\}$ and there exists $t_1^{q+1}\in T^{q +1}(\rho)$ and $t_2^{q+2}\in T^{q+2}(\rho)$ such that both $T_2^{q +1}(\rho)= T^{q +1}(\rho)\setminus \{t_1^{q +1}\}$ and $T_1^{q +2}(\rho)= T^{q+2}(\rho)\setminus \{t_2^{q +2}\}$ have cardinality at most one (this implies that $ 4\leq |T|\leq 6 $ and we denote the unique possible members of $ T_2^{q+1}(\rho) $ and $ T_1^{q+2}(\rho) $ by $ t_2^{q+1} $ and $ t_1^{q+2} $, respectively), the sets $T_1(\rho)=\{t_1^{q},t_1^{q +1}\}\cup T_1^{q +2}(\rho)$ and $T_2(\rho)=\{t_2^q,t_2^{q +2}\}\cup T_2^{q +1}(\rho)$ are both cliques of $G$ and there are no more edges in $E(T(\rho))$ except the edges in $E(T^1(\rho))\cup E(T^2(\rho))\cup E(T^3(\rho))$. Now, we state the following lemma, which discloses the structure of $G[T(\rho)]$ when $|I_T(\rho)|=3$.

\begin{lem}\label{lem:G[T] structure |I_T|=3}
	Let $G$ be a counterexample to \pref{thm:noap} with $|I_T(\rho)|=3$, for some rotator $\rho$ of $\overline{G}$. Then either $T(\rho)$ is a clique or $G[T(\rho)]$ is of type $\tau $ at $q$, for some $\tau \in \{2,3\}$ and some $q \in \{1,2,3\}$.
\end{lem}
\begin{proof}
	In the proof, we omit the term $\rho$. Suppose that $T$ is not a clique.
	First, assume that for every $q \in \{1,2,3\}$, there is $q '\in \{1,2,3\}\setminus \{q \}$ and some vertex $t^{q '}\in T^{q '}$ anticomplete to $T^{q }$. Thus, $(t^{q '},t^{q ''}_3,T^{q })$ is a quasi-triad, where $\{q ,q ',q ''\}=\{1,2,3\}$, and thus by \pref{lem:T,R Facts and Tools}~(i), for every $q \in \{1,2,3\}$, we have $|T^{q} |=1$, say $T^{q} =\{t^{q }\}$. On the other hand, by \pref{lem:T,R Facts and Tools}~(iv), the set of neighbours and non-neighbours of $t^{1 }$ in $T\setminus T^{1 }$ are both cliques, and also since $T$ is not a clique, $t^{1 }$ is not complete to $T\setminus T^{1 }$. Now, if $t^{1 }$ is nonadjacent to both $t^{2}$ and $t^{3}$, then by \pref{lem:T,R Facts and Tools}~(iv), $t^{2}$ and $t^{3}$ are adjacent and $G[T]$ is of type $2$ at $1$, and if $t^{1 }$ is adjacent to $t^{2}$ (resp. $t^{3}$) and nonadjacent to $t^{3}$ (resp. $t^{2}$), $t^2,t^3$ are nonadjacent and so $G[T]$ is of type $2$ at ${3}$ (resp. ${2}$).
	
	Next, assume that there exists $q\in \{1,2,3\}$, say $ q=1 $, such that every vertex in $T\setminus T^q$ has a neighbour in $T^q$. If every vertex in $T\setminus T^1$ is complete to $T^1$, then by \pref{lem:T,R Facts and Tools}~(iv), $T$ is a clique, a contradiction. Thus, we may assume that there exists a vertex $t^1\in T^1$, nonadjacent to a vertex in $T\setminus T^1$, say $t^2 \in T^2$. Define $A=N(t^1,T^2)$, $B=N(t^1,T^3)$, $C=N(t^2,T^1)$ and  $D=N(t^2,T^3)$, where by the assumption, $C\neq \emptyset$. Then,\vsp
	
	(1) \textit{The sets $ A\cup B\cup\{t^1\} $ and $ C\cup D\cup \{t^2\} $ are both cliques and there are no more edges in $E(T)$ except the edges in $ E(T^1)\cup E(T^2)\cup E(T^3) $. Also, $|A|,|B|,|D|\leq 1$.} \vsp
	
	By \pref{lem:T,R Facts and Tools}~(iv), it is evident that $ A\cup B\cup\{t^1\} $ and $ C\cup D\cup \{t^2\} $ are both cliques, $t^1$ is anticomplete to $D$, $t^2$ is anticomplete to $B$, and  $(B,D)$ is a partition of $T^3$. Also, since $(t^1,t^2_3,D)$ and  $(t^1_3,t^2,B)$ are quasi-triads of $G$, by \pref{lem:T,R Facts and Tools}~(i), $|B|, |D|\leq 1$, and since $t^1$ and $t^2_3$ are complete to $ A $, $A$ is anticomplete to $D$. Similarly, since  $t^1_3$ and $t^2$ are complete to $ C $, $C$ is anticomplete to $B$. 
	Now, either $ B $ or $ D $ is nonempty. If $B$ is nonempty, say $B=\{b\}$, then $(t^2_3, b, C)$ is a quasi-triad and $t^2_3$ and $b$ are complete to $A$. Thus, \pref{lem:T,R Facts and Tools}~(i) implies  that $A$ is anticomplete to $C$. Also, if $D$ is nonempty, say $D=\{d\}$, then $(t^1_3, d,A) $ is a quasi-triad and $t^1_3$ and $d$ are complete to $C$. Thus, again by \pref{lem:T,R Facts and Tools}~(i), $A$ is anticomplete to $C$. Finally, since for every $c\in C$, $(t^3_3, c,A)$ is a quasi-triad,  by \pref{lem:T,R Facts and Tools}~(i), $|A|\leq 1$. This proves (1). \vsp
	
	Now, by (1), if both $A$ and $B$ are empty, then $G[T]$ is of type $2$ at $1$. If either $A$ or $B$ is nonempty, say $B=\{b\}$, then $(t^2_3,b,C)$ is a quasi-triad and by \pref{lem:T,R Facts and Tools}~(i), $|C|=1$. Hence, $G[T]$ is of type $3$ at $1$. This proves \pref{lem:G[T] structure |I_T|=3}.
\end{proof}

We close this subsection with the following lemma, which includes the first application of \pref{lem:G[T] structure |I_T|=3}.

\begin{lem}\label{lem:|I_T|=3 |I_R|=1}
	Let $G$ be a counterexample to \pref{thm:noap}. Then for every rotator $\rho$ of $\overline{G}$, we have  $|I_T(\rho)|,|I_R(\rho)|\geq 2$.
\end{lem}
\begin{proof}
	On the contrary and w.l.o.g. assume that $|I_R|\leq 1$. By \pref{lem:|I_T|+|I_R|>=4}, we have $|I_T|=3$ and $|I_R|=1$, say $I_R=\{p \}$, for some $p \in \{1,2,3\}$. Now, merge the pair $(Z_{p +1}^{p },Z_{p +2}^{p })$ in $\mathscr{O}(G)$ and call the resulting collection $\mathscr{O}'(G)$.\vsp
	
	(1) \textit{If $G[T]$ is of type $2$ at $q$, for some ${q}\in \{1,2,3\}$, then $p\neq q$.}\vsp
	
	Suppose not, w.l.o.g. assume that $p=q=1$. Let $t^1, t^2$ and $t^3$ be as in the definition of type $2$ at $1$, and $X=T^1\setminus \{t^1\}$. Note that by \pref{lem:T,R Facts and Tools}~(vi), $R_1$ can be partitioned into $R^1=N(t^1,R_1)$ and $R^2=N(t^2,R_1)=N(t^3,R_1)$. Also, the cliques in $\mathscr{O}_1(G)=\mathscr{O}'(G)\cup \{ R^3\cup \{t^1\}, R^3\cup T\setminus \{t^1\},T_3\cup R_1\}$ cover all the edges in $E(G)\setminus E(T,R)$. We claim that $S^1\neq \emptyset$. For if $S^1=\emptyset$, then applying \pref{lem:square-forcer}~(i) to $(i,j,k)=(1,2,3)$ and $(i,j,k)=(1,3,2)$ implies that $|S_1|=2$. Now in $\mathscr{O}_1(G)$, replace the cliques $X_2^1,X_3^2$ and $X_2^3$ with the cliques $(X_2^1\cup R^1)\setminus X,X_3^2\cup R^2$ and $X_2^3\cup R^2$, and also add the cliques in $\mathscr{N}[X;R_1\cup S_2\cup \{s_2^2\}]$, thereby obtaining a clique covering for $G$ of size $14+|X|=n-|S|-|R|+2\leq n-1$, a contradiction. This proves the claim. Thus, assume that $s_k^1\in S$, for some $k\in \{2,3\}$. 
	Now,  in  $\mathscr{O}_1(G)\cup \mathscr{N}[R;T]$, if $S_1=\emptyset$, then remove the clique $R^3\cup \{t^1\}$ and replace the cliques $Z_3^2$ and $Z_2^3$ with the cliques $Z_3^2\cup \{t^1\}$ and $Z_2^3\cup \{t^1,r_1^3\}$ and call the resulting family $\mathscr{C}$, which is a clique covering for $G$. By the claim, if either $X\neq \emptyset$, or $|S|\geq 3$, or $|S|=2$ and $S_1=\emptyset$, then $|\mathscr{C}|\leq n-1$, a contradiction. Thus, $X=\emptyset$ and either $|S|=1$ or $|S|=2$ and $S_1\neq \emptyset$. Therefore, by the above claim, $S= S^1_k\cup S_1^{k'}$ for some $k'\in \{2,3\}$ and $|\mathscr{C}|=n$. Now, if $|R_1|=1$, say $R_1=\{r_1\}$, then remove the single clique in $\mathscr{N}[R;T]$ from $\mathscr{C}$ and replace the cliques $Z_1^2$ and $Z_1^3$ with the cliques  $Z_1^{k'}\cup N(r_1,T^{5-k'})$ and $Z_1^{5-k'}\cup N(r_1,T^1\cup T^{k'})$ in the resulting family, and if $|R_1|\geq 2$, then remove the cliques in $\mathscr{N}[R;T]$ from $\mathscr{C}$, replace the clique $X_{5-k}^1$ with the clique $X_{5-k}^1\cup R^1$ and add the clique $R^2\cup \{t^2,t^3\}$. This procedure yields a clique covering for $G$ of size at most $n-1$, a contradiction. This proves (1). \vsp

	(2) \textit{$G[T]$ is of type $3$.}\vsp
	
	Suppose not, by Lemmas~\ref{lem:both clique} and \ref{lem:G[T] structure |I_T|=3}, $G[T]$ is of type $2$ at $q$, for some $q\in \{1,2,3\}$. By (1), $p\neq q$, say $p=2$ and $q=1$. Let $t^1, t^2$ and $t^3$ be as in the definition of type $2$ at $1$, and $X=T^1\setminus \{t^1\}$. Note that by \pref{lem:T,R Facts and Tools}~(vi), $R_2$ can be partitioned into $R^1=N(t^1,R_2)$ and $R^2=N(t^2,R_2)=N(t^3,R_2)$. Also, the cliques in  $\mathscr{O}_2(G)=\mathscr{O}'(G)\cup \{R^3\cup \{t^1\}, R^3\cup T \setminus \{t^1\}, T_3\cup R_2\}$ cover all the edges in $E(G)\setminus E(T,R)$. We claim that either $S^1\neq \emptyset$ or $S=\{s_1^2\}$. On the contrary, assume that $S^1=\emptyset$ and $S\neq \{s_1^2\}$. Then, applying \pref{lem:square-forcer}~(i) to $(i,j,k)=(1,2,3)$ implies that $s_1^2\in S$ and thus $|S|\geq 2$. Now, in $\mathscr{O}_2(G)$, replace the cliques $X_2^1$ and $X_3^1$ with the cliques $(X_2^1\cup S_3^2)\setminus S_2^3,(X_3^1\cup S_2^3\cup R^1)\setminus (S_3^2\cup X)$, replace the clique $Z_2^1$  with the cliques in $\mathscr{N}[R_2;\{r^3_2,s_1^1,t^2,t^3\}]$ and also add the cliques in $\mathscr{N}[X;R_2\cup S_2^3\cup \{s_3^3\}]$, thereby obtaining a clique covering $\mathscr{C}$ for $G$ of size $13+|R|+|X|=n-|S|+1\leq n-1$, a contradiction. This proves the claim. Now, in  $\mathscr{O}_2(G)\cup \mathscr{N}[R;T]$, if either $S_1=\emptyset$ or $S=\{s_1^2\}$, then remove the clique  $ R^3\cup \{t^1\}$ and replace the cliques $Z_3^1,Z_3^2$ and $Z_1^3$ with the cliques $Z_3^1\cup S_1^2,(Z_3^2\cup \{t^1\})\setminus S_1^2$ and $Z_1^3\cup \{t^1,r_2^3\}$ and call the resulting family $\mathscr{C}$, which is a clique covering for $G$.  If either $X\neq \emptyset$, or $|S|\geq 3$, or $|S|=2$ and $S_1=\emptyset$, then $|\mathscr{C}|\leq n-1$, a contradiction. Thus, $X=\emptyset$ and either $|S|=1$ or $|S|=2$ and $S_1\neq \emptyset$. Therefore, by the above claim, $S=S_1\cup S^1$, and in particular, either $|S|=|S^1|=1$, or $|S|=|S^2_1|=1$, or $|S_1|=|S^1|=1$. Consequently, $|\mathscr{C}|\leq n$. 
	Then, remove the cliques in $\mathscr{N}[R;T]$ from $\mathscr{C}$ and in order to cover the edges in $E(T,R)$, replace the cliques $X_3^1$ and $X_3^2$ with the cliques $X_3^1\cup R^1$ and $X_3^2\cup R^2$. Also, if $s^2_1\notin S$, then replace the clique $X_1^3$ with the clique $X_1^3\cup R^2$ and if $s^2_1\in S$, then replace the cliques $X_1^3$ and $X_2^3$ with the cliques $(X_1^3\cup S_2^1\cup R^2)\setminus S_1^2$ and $(X_2^3\cup S_1^2)\setminus S_2^1$. This yields a clique covering of size at most $n-1$, a contradiction. This proves (2). \vsp

	By (2), $G[T]$ is of type $3$ at $q$, for some $q\in \{1,2,3\}$, say $q=1$. Let $T_1=\{t_1^1,t_1^2\}\cup T_1^3$ and $T_2=\{t_2^1,t_2^3\}\cup T_2^2$ be as in the definition of type $3$ at $1$. Note that by \pref{lem:T,R Facts and Tools}~(vi), $R_p$ can be partitioned into the sets $R^1$ and $R^2$ such that for every $i\in \{1,2\}$, $T_i$ is complete to $R^i$ and anticomplete to $R^{3-i}$. Now adding the cliques in $\mathscr{N}[R;T]$ to the family of cliques $\mathscr{O}_3(G)=\mathscr{O}'(G)\cup \{T_1\cup R^3,T_2\cup R^3,T_3\cup R_p\}$ provides a clique covering for $G$ of size $14+|R|=n-|S|-|T|+5$. Thus, $|S|+|T|\leq 5$. Also, if $|T|=5$, then $S=\emptyset$, and in this case, it is enough to remove the clique $\Omega _p$ from $\mathscr{P}(G)$ and add the cliques $T_1\cup R^1\cup \{r_p^3\}$ and $T_2\cup R^2\cup \{r_p^3\}$, thereby obtaining a clique covering for $G$ of size $14=n-|R|\leq n-1$, which is impossible. Thus, $|T|=4$ and $|S|\leq 1$. First, assume that either $S=\emptyset$ or $|S|=1$ and both $R^1,R^2\neq \emptyset$. In this case, considering  $\mathscr{P}(G)$, add the cliques in $\mathscr{N}[R_p;T\cup \{r_p^3\}]$, if $S=\emptyset$, then replace the cliques $\Delta _3^2$ and $\Delta _2^3$ with the cliques $\Delta _3^2\cup \{t_2^1\}$ and $\Delta _2^3\cup \{t_1^1\}$ and if $S_p=\emptyset$, then remove the clique $\Omega _p$, to obtain a clique covering for $G$ of size at most $n-1$, a contradiction. Therefore, we have $|S|=1$ and $R_p=R^{i_0}$, for some $i_0\in \{1,2\}$. Now, if $S_{p+2}^{p+1}\cup S_{p+1}^{p+2}=\emptyset$, then let $\{p,p',p''\}=\{1,2,3\}$ such that $S_p^{p'}=\emptyset$, and in $\mathscr{O}_3(G)$, replace the cliques $Z_p^{p'}$ and $Z_p^{p''}$ with the cliques $Z_p^{p'}\cup (T_{i_0}\cap T^{p''})$ and $Z_p^{p''}\cup (T_{i_0}\cap (T^p\cup T^{p'}))$. Also, if $S=S_{p+2}^{p+1}\cup S_{p+1}^{p+2}$, then in $\mathscr{O}_3(G)$, replace the cliques $Z_{p}^{p +1}$ and $Z_{p}^{p +2}$ with the clique $(Z_{p}^{p +1}\cup S_{p+1}^{p+2}\cup (T_i\setminus T^{p+1}))\setminus S_{p+2}^{p+1}$ and $(Z_{p}^{p +2}\cup S_{p+2}^{p +1}\cup (T_i\cap T^{p+1}))\setminus S_{p +1}^{p +2}$. This provides a clique covering for $G$ of size $14=n-|R_p|\leq n-1$, again a contradiction. This proves \pref{lem:|I_T|=3 |I_R|=1}.
\end{proof}

\subsection{Schl\"{a}fli-antiprismatic graphs}\label{sub:schlafli}
Many of the classes of graphs in the ``menagerie'' introduced in \cite{seymour2} are obtained by some modifications on some induced subgraphs of the Schl\"{a}fli graph $\Gamma$ (see \pref{sub:Inc-rot}). 
Our goal in this subsection is to provide appropriate clique coverings for some induced subgraphs of $\Gamma$, and obtain more properties of the graph $G$, the counterexample to \pref{thm:noap}. Let us begin with a definition. 
Let $G$ be an antiprismatic graph whose complement contains a rotator $ \rho $. Suppose that for every $i\in \{1,2,3\}$, $T^i(\rho)\subseteq \{t^i_1,t^i_2\}$ and $R_i(\rho)\subseteq\{r_i^1,r_i^2\}$ and adjacency in $G[T\cup R]$ is the same as in $\Gamma[T\cup R]$. Then, $G$ is an induced subgraph of $\Gamma$ and we say that the graph $G$ is \textit{Schl\"{a}fli-antiprismatic $($with respect to $\rho)$}. Also, $ G $ is called \textit{inflated Schl\"{a}fli-antiprismatic} (ISA for short) with respect to $ \rho $ if $ G $ is obtained from a Schl\"{a}fli-antiprismatic graph by replicating vertices in the non-core.  

The following is a technical lemma whose proof is lengthy and is postponed to \pref{app:schlafli}. 

\begin{lem}\label{lem:ccschlafli}
	Let $G$ be a counterexample to \pref{thm:noap} such that neither $ T(\rho)$ nor $ R(\rho) $ is a clique of $ G $, for some rotator $\rho$ of $\overline{G}$. Then $G$ is not ISA with respect to $\rho$.
\end{lem}

Now, using the above lemma, in the following lemma we deduce more properties of the counterexample to \pref{thm:noap} to be used in forthcoming subsections. We omit the term $ \rho $ within the proofs.

\begin{thm}\label{thm:IT+IR>=5}
Let $G$ be a counterexample to \pref{thm:noap} with $|I_T(\rho)|+|I_R(\rho)|\geq 5$, for some rotator $\rho$ of $\overline{G}$. Then, either $ T(\rho) $ or $ R(\rho) $ is a clique. 
\end{thm}
In order to prove \pref{thm:IT+IR>=5}, we need the following three lemmas.

\begin{lem}\label{lem:reduce-to-schl3,2-1}
	Let $G$ be a counterexample to \pref{thm:noap} with $|I_T(\rho)|=3$, $I_R(\rho)=\{1,2\}$, for some rotator $\rho$ of $\overline{G}$. If $G[T(\rho)]$ is of type $2$ at $ q $, for some $ q\in\{1,2\} $, then $ R(\rho) $ is a clique of $ G $.
\end{lem}
\begin{proof}
	By symmetry, assume that $ G[T] $ is of type $ 2 $ at $ q=1 $ with $t^1, t^{2}$ and $t^{3}$ as in the definition. Also, let $X=T^1\setminus \{t^1\}$. By the contrary, assume that $ R $ is not a clique and so by \pref{lem:G[T] structure |I_T|=2}, $G[R]$ is of type $1$ with $r_1\in R_1, r_2\in R_2$, $U^1\subseteq R_2$ and $U^2\subseteq R_1$ as in the definition.  Note that by \pref{lem:square-forcer}~(i)(applying to $(i,j,k)=(1,2,3)$), there is a vertex $s_{k'}^k\in S\cap \{s_1^2,s_2^1,s_3^1\}$. First, assume that $r_1$ is adjacent to $t^1$. Thus, by \pref{lem:T,R Facts and Tools}~(vi), $r_2$ is complete to $\{t^2,t^3\}$. For every $i\in \{1,2\}$ and every $r\in U^i$, let $A_i(r)=\{s_{k'}^k,t^{k'},r_i,r\}$ and $B_i(r)=\{s_i^i,t^i,r_i,r\}$. 
	Now, since $A_i(r)$, $i=1,2, r\in U^i$, is not a claw, $t^1$ is complete to $U^1$ and anticomplete to $U^2$ and  $\{t^2,t^3\}$ is anticomplete to $U^1$ and complete to $U^2$. Also, $r_2$ is complete to $X$ (otherwise, for $x\in X$ nonadjacent to $r_2$, if $(k,k')=(2,1)$, then $\{t^2,s_{1}^2,r_2,x\}$  would be a claw and if $(k,k')=(1,j)$, for some $j\in \{2,3\}$, then $\{x,s_{j}^1,t^{j},r_1\}$ would be a claw). Thus, $r_1$ is anticomplete to $X$.
	Moreover, since $(s_1^1,t^1,U^2)$ is a quasi-triad of $G$ and $U^1$ is complete $\{s_1^1,t^1\}$, by \pref{lem:T,R Facts and Tools}~(i), $U^1$ is anticomplete to $U^2$, and also since $(s_1^1,r_1,X)$ is a quasi-triad of $G$ and $U^1$ is complete $\{s_1^1,r_1\}$ (resp. $U^2$ is anticomplete to $s_1^1$), $U^1$ is anticomplete to $X$ (resp. $U^2$ is complete to $X$). Hence, $G$ is ISA, a contradiction with \pref{lem:ccschlafli}.
	
	Now, assume that $t^1$ is nonadjacent to $r_1$. By \pref{lem:T,R Facts and Tools}~(vi), $r_1$ is complete to $\{t^2,t^3\}$ and $r_2$ is adjacent to $t^1$ and nonadjacent to $ t^2,t^3 $. Now, since $B_i(r)$, $i=1,2, r\in U^i$, is not a claw, by  \pref{lem:T,R Facts and Tools}~(vi), $t^1$ is anticomplete to $U^1$ and complete to $U^2$ and $\{t^2,t^3\}$ is complete to $U^1$ and anticomplete to $U^2$. Also, $r_1$ is complete to $X$ and thus $r_2$ is anticomplete to $X$ (otherwise, for every $x\in X$ nonadjacent to $r_1$, $\{t^2,s_1^1,r_1,x\}$ would be a claw). 
	Moreover, since $(s_{k'}^k,t^{k'},U^{3-k})$ is a quasi-triad of $G$ and $U^k$ is complete to $\{s_{k'}^k,t^{k'}\}$, by \pref{lem:T,R Facts and Tools}~(i), $U^1$ is anticomplete to $U^2$.
	Furthermore, if $ (k,k')=(1,j) $, for some $ j\in\{2,3\} $, then $ (s^1_{j}, t^{j},U^2) $ is a quasi-triad and if $ (k,k')=(2,1) $, then $ (s_1^2,r_2,X) $ is a quasi-triad. Therefore, by \pref{lem:T,R Facts and Tools}~(i), $ X $ is anticomplete to $ U^2 $. Now, if either $ (k,k')=(2,1) $, or $ X=\emptyset $, or $ U^1=\emptyset $, or $ U^2\neq \emptyset $, then $ X $ is complete to $ U^1 $ and thus, $G$ is ISA, which contradicts \pref{lem:ccschlafli}. Hence, $ S^2_1=\emptyset $, $ S^1_2\cup S^1_3\neq \emptyset $, $ X\neq \emptyset $, $ U^1\neq \emptyset $ and $ U^2=\emptyset $. 
	 Assume that $\{X,U^1\}=\{A,B\}$, where $1\leq |A|\leq |B|$. In the sequel, we are going to give a clique covering $\mathscr{C}$ for $G$ of size at most $n-|B|+1$.
	 
	 To see this, consider the family of cliques $\mathscr{O}(G)\cup \{T^1\cup R^3,(T\setminus \{t^1\})\cup R^3,T_3\cup R_2,T_3\cup (R\setminus \{r_2\})\}$, replace the cliques $X_1^2,X_1^3$ and $Z_1^2$ with the cliques $X_1^2\cup U^1,X_1^3\cup U^1$ and $Z_1^2\cup X$ and add the clique $\{t^1,r_2\}$ as well as the cliques in $\mathscr{N}[A;B]$, and call the resulting collection $\mathscr{O}'(G)$. The cliques in $ \mathscr{O}'(G) $ cover all the edges in $E(G)\setminus \{t^2r_1,t^3r_1\}$. In order to cover $t^2r_1$ and $t^3r_1$, in $\mathscr{O}'(G)$, 
	if $S_1= \emptyset$, then replace the cliques $X_1^2\cup U^1$, $X_1^3\cup U^1, X_3^2$ and $X_2^3$ with the cliques  $X_1^2\cup U^1\cup S_3^1$, $X_1^3\cup U^1\cup S_2^1, (X_3^2\cup R_1)\setminus S_3^1$ and $(X_2^3\cup R_1)\setminus S_2^1$, if $S_1\neq \emptyset $ and $S_3^2\cup S_2^3=\emptyset$, then replace the cliques $Z_1^2\cup X$ and $Z_1^3$ with the cliques $Z_1^2\cup X\cup \{t^3\}$ and $Z_1^3\cup \{t^2\}$, 
	and finally, if $S_1\neq \emptyset$ and $S_3^2\cup S_2^3\neq \emptyset$, then add the clique $\{t^2,t^3,r_1\}$ and call the resulting clique covering $\mathscr{C}$. In the latter case, $|S|\geq 3$ and $|\mathscr{C}|=18+|A|\leq n-|B|+1$, as required. Also, in the first two cases, $|\mathscr{C}|=17+|A|=n-|S|-|B|+3$, which implies the claim, when $|S|\geq 2$. Now, assume that $|S|=1$, i.e. $|S|=|S^1_p|=1$. In this case, take $\mathscr{C}$ (which is obtained by the first modification on $\mathscr{O}'(G)$) and remove the clique $T_3\cup (R\setminus \{r_2\})$, and to cover the edges in $E(R_1, T_3\cup U^1)$, replace the cliques $(X_3^2\cup R_1)\setminus S_3^1, Z_1^2\cup X$ and $ Z_3^2$ with the cliques $(X_3^2\cup R_1\cup U^1)\setminus S_3^1, \{r_1,s_2^2,t_3^1\}\cup X$ and $\{r^3_1,r^3_3, s_2^2\}$.  This yields a clique covering of size $16+|A|=n-|B|+1$, as desired. Therefore, $|B|=1$ and thus $|X|=|U^1|=1$. Now, if $X$ is complete to $U^1$, then $G$ is ISA, which contradicts \pref{lem:ccschlafli}, and if $X$ is anticomplete to $U^1$, then removing the single clique in $\mathscr{N}[A;B]$ from $\mathscr{C}$ yields a clique covering for $G$ of size $n-1$, a contradiction. This proves \pref{lem:reduce-to-schl3,2-1}.
\end{proof}

\begin{lem}\label{lem:reduce-to-schl3,2-2}
	Let $G$ be a counterexample to \pref{thm:noap} with $|I_T(\rho)|=3$, $I_R(\rho)=\{1,2\}$, for some rotator $\rho$ of $\overline{G}$. If $G[T(\rho)]$ is of type $2$ at $ q=3 $, then $ R(\rho) $ is a clique of $ G $. 
\end{lem}
\begin{proof}
Suppose that $ G[T] $ is of type $ 2 $ at $ 3 $ with $t^1, t^{2}$ and $t^{3}$ as in the definition and let $X=T^3\setminus \{t^3\}$. Also, by the contrary, assume that $ R $ is not a clique and so by \pref{lem:G[T] structure |I_T|=2}, $ G[R] $ is of type $ 1 $, where $ r_1\in R_1, r_2\in R_2$, $U^1\subseteq R_2$ and $U^2\subseteq R_1$ are as in the definition. By \pref{lem:T,R Facts and Tools}~(vi), either $ r_1$ is complete to $ t^3$ and anticomplete to $ \{t^1,t^2\} $, or $ r_1$ is anticomplete to $ t^3$ and complete to $ \{t^1,t^2\} $. By symmetry, we may assume that the former case occurs. Thus, by \pref{lem:T,R Facts and Tools}~(vi), $r_2$ is complete to $\{t^1,t^2\}$ and anticomplete to $ \{t^3\} $. Now, since for every $i\in\{1,2\}$ and $ r\in U^i $, $\{s_1^1,t^1,r_i,r\}$ is not a claw, $\{t^1,t^2\}$ is anticomplete to $U^1$ and complete to $U^2$ and thus, $t^3$ is complete to $U^1$ and  anticomplete to $U^2$. Also, since $\{x,s_1^1,t^1,r_1\}$, $x\in X$, is not a claw, $r_1$ is anticomplete to $X$ and thus $r_2$ is complete to $X$. Moreover, since $(s_2^2,t^2,U^1)$ is a quasi-triad of $G$ and both $X$ and $U^2$ are complete to $\{s_2^2,t^2\}$, by \pref{lem:T,R Facts and Tools}~(i), $X$ and $U^2$ are both anticomplete to $U^1$. Now, if $ s_3^1\in S $, then $(s_3^1,t^3,U^2)$ is a quasi-triad of $G$ and $ X $ is anticomplete to $\{s_3^1\}$, and thus $ X $ is complete to $ U^2 $. Also, if either $ X=\emptyset $, or $ U^2=\emptyset $, or $ U^1\neq \emptyset $, then again $ X $ is complete $ U^2 $. Thus, in all these cases, $ G $ is ISA, which is in contradiction with \pref{lem:ccschlafli}. Hence, $ S_3^1=\emptyset $, $ X\neq \emptyset $, $ U^2\neq \emptyset $, $ U^1= \emptyset $ and $ X $ is not complete to $ U^2 $.
Now, assume that $\{X,U^1\}=\{A,B\}$, where $1\leq |A|\leq |B|$. In the sequel, we are going to provide a clique covering for $ G $ of size at most $ n-|B|+1 $.

Consider the family of cliques $\mathscr{O}(G)\cup \{T^3\cup R^3,(T\setminus \{t^3\})\cup R^3,T_3\cup R_1,T_3\cup (R\setminus \{r_1\})\}$, replace the cliques $X_3^1,X_3^2$ and $Z_2^1$ with the cliques $X_3^1\cup U^2,X_3^2\cup U^2$ and $Z_2^1\cup X$, add the clique $\{t^3,r_1\}$ and the cliques in $\mathscr{N}[A;B]$, and call the resulting collection $\mathscr{O}'(G)$. The cliques in $\mathscr{O}'(G)$ cover all the edges in $E(G)\setminus \{t^1r_2,t^2r_2\}$. 
In order to cover $t^1r_2$ and $t^2r_2$, in $\mathscr{O}'(G)$,
if $S_2^1\cup S_1^3=\emptyset$, then replace the cliques $Z_2^1\cup X$ and $Z_2^3$ with the cliques $Z_2^1\cup X\cup \{t^2\}$ and $Z_2^3\cup \{t^1\}$, if $S_3^2=\emptyset$, then replace the cliques $X_3^1\cup U^2$ and $X_3^2\cup U^2$ with the cliques $X_3^1\cup U^2\cup \{r_2\}$ and $X_3^2\cup U^2\cup \{r_2\}$ and if $S_2^1\cup S_1^3\neq \emptyset$, $S_3^2\neq \emptyset $ and $|S|\leq 2$ (i.e. either $S=\{s_{2}^1,s_3^2\}$ or $S=\{s_{1}^3,s_3^2\}$), then replace the cliques $Z_2^1\cup X$ and $ Z_2^3$ with the cliques  $Z_2^1\cup X\cup S_1^3$ and $ (Z_2^3\cup \{t^1,t^2\}) \setminus S_1^3$ and finally if $|S|\geq 3$, then add the clique  $\{t^1,t^2,r_2\}$, thereby obtaining a clique covering $\mathscr{C}$ for $G$. If $ |S|\geq 2 $, then  $|\mathscr{C}|\leq n-|B|+1$, as desired.
Now, assume that $|S|=1$ and $S_3=\emptyset$. Then, in $\mathscr{C}$, remove the clique $T^3\cup R^3$ and replace the cliques $Z_3^1$ and $Z_3^2$ with the cliques $Z_3^1\cup T^3\cup \{r^3_2\}$ and $Z_3^2\cup T^3\cup \{r^3_1\}$ to obtain a clique covering of size $n-|B|+1$.
If $ |B|\geq 2 $, then we have a clique covering for $ G $ of size at most $ n-1 $, a contradiction. Thus,  $|B|=1$ and so $|X|=|U^2|=1$. Now, since $X$ is anticomplete to $U^2$, one may remove the single clique in $\mathscr{N}[A;B]$ to obtain a clique covering for $G$ of size at most $n-1$. 
 Finally, assume that $|S|=|S_3|\leq 1$. Then, 
%
%
%
in $\mathscr{P}(G)$, merge the pair $(\Delta _1^3,\Delta _2^3)$, replace clique  $\nabla _2^3$ with the clique  $\nabla _2^3\cup U^2$, and in order to cover the edges in $E(T,R)\cup E(T\setminus \{t^3\})$, for every $i\in\{1,2\}$, replace the cliques $\Omega_i$ with the cliques in $\mathscr{N}[R_i;(\Omega_i\cup T)\setminus R_i]$. These modifications lead to a clique covering for $G$ of size $14+|U^2|\leq n-|X|\leq n-1$, a contradiction. This proves \pref{lem:reduce-to-schl3,2-2}.
\end{proof}

\begin{lem}\label{lem:reduce-to-schl3,2-3}
	Let $G$ be a counterexample to \pref{thm:noap} with $|I_T(\rho)|=3$, $I_R(\rho)=\{1,2\}$, for some rotator $\rho$ of $\overline{G}$. If $G[T(\rho)]$ is of type $3$, then $ R(\rho) $ is a clique of $ G $.  
\end{lem}
\begin{proof}
Suppose that $ G[T] $ is of type $3$ at $q$ for some $q\in \{1,2,3\}$, where $T_1=\{t_1^q,t_1^{q+1}\}\cup T_1^{q+2}$ and $T_2=\{t_2^q,t_2^{q+2}\}\cup T_2^{q+1}$ are as in the definition. Also, by the contrary, suppose that $ R $ is not a clique and thus by \pref{lem:G[T] structure |I_T|=2}, $ G[R] $ is of type $ 1 $ with $r_1\in R_1$, $r_2\in R_2$, $U^1\subseteq R_2$ and $U^2\subseteq R_1$ as in the definition. Note that by \pref{lem:T,R Facts and Tools}~(vi), there is a unique $i_0\in \{1,2\}$ such that $r_1$ is complete to $T_{i_0}$ and anticomplete to $T_{3-i_0}$, and $r_2$ is complete to $T_{3-i_0}$ and anticomplete to $T_{i_0}$.
First, assume that $q\neq 3$ and by symmetry, let $q=1$. Since for every $i\in \{1,2\}$ and $r\in U^i$, $\{s_1^1,t_{3-i_0}^1,r_i,r\}$ is not a claw, $T_{3-i_0}$ is anticomplete to $U^1$ and complete to $U^2$, and thus, by \pref{lem:T,R Facts and Tools}~(vi), $T_{i_0}$ is complete to $U^1$ and anticomplete to $U^2$. Also, since $(s_1^1,t^1_{i_0},U^2)$ is a quasi-triad of $G$ and $U^1$ is complete to $\{s_1^1,t_{i_0}^1\}$, $U^1$ is anticomplete to $U^2$. Hence, $G$ is ISA, which contradicts \pref{lem:ccschlafli}.
	
Next, suppose that $q=3$. Note that if $|R|=2$ (i.e. $U^1=U^2=\emptyset$), then evidently $G$ is ISA, which contradicts \pref{lem:ccschlafli}. Thus, $|R|\geq 3$. 
Now, assume that there exists $i\in \{1,2\}$ such that $S^i\neq \emptyset$, say $ S^1\neq \emptyset $. 
If $s_3^1\in S$, then let $i'=i''=3$, and if $ s_2^1\in S $, then let $i'=3-i_0$ and $i''=i_0$. Since for every $i\in \{1,2\}$ and every $r\in U^i$, $ \{s_{i'}^1,t_{3-i_0}^{i'},r_i,r\} $ is not a claw, $T_{3-i_0}$ is anticomplete to $U^1$ and  complete to $U^{2}$ and thus by \pref{lem:T,R Facts and Tools}~(vi), $T_{i_0}$ is complete to $U^1$ and anticomplete to $U^2$. Also, since $(s_{i''}^1,t_{i_0}^{i''},U^{2})$ is a quasi-triad of $G$ and $U^1$ is complete to $\{s_{i''}^1,t_{i_0}^{i''}\}$, $U^1$ is anticomplete to $U^{2}$. Hence, $G$ is ISA, again a contradiction with \pref{lem:ccschlafli}. Hence, $ |R|\geq 3 $ and $ S^1\cup S^2=\emptyset $. In the sequel, we are going to provide a clique covering for $ G $ of size at most $ n-1 $.

	Note that by \pref{lem:T,R Facts and Tools}~(vi), $R$ can be partitioned into $R^1$ and $R^2$ such that for every $j\in\{1,2\}$, $R^j$ is complete to $T_j$ and anticomplete to $T_{3-j}$. Now, consider $\mathscr{P}(G)$, if $G[R]$ is of type $1.2$, then remove the cliques $\nabla _1^3$ and $\nabla _2^3$ and add the cliques in $\mathscr{N}[R_1;\nabla _2^3]$, and if $G[R]$ is of type $1.1$ at $i$, for some $i\in \{1,2\}$, then replace the clique $\nabla_{3-i}^3$ with the clique $\nabla _{3-i}^3\cup U^{3-i}$. Also, if $|T|=4$, then replace the cliques $\Delta _1^2$ and $\Delta _2^1$ with the cliques $\Delta _1^2\cup \{t_1^3\}$ and $\Delta _2^1\cup \{t_2^3\}$, and for every $j\in \{1,2,3\}$, replace the clique $\Omega^j$ with the cliques in $\mathscr{N}[T^j;(\Omega ^j\cup R)\setminus T^j]$, and if $|T|\geq 5$, then add the cliques $T_1\cup R^1$ and $T_2\cup R^2$, thereby covering the edges in $E(T,R)\cup E(T_1)\cup E(T_2)$. Now, if $G[R]$ is of type $1.2$, then the resulting collection is a clique covering for $G$ of size at most $10+|T|+|R_1|=n-|S|-|R_2|+1\leq n-1$. Thus, $G[R]$ is of type $1.1$ at $i$, for some $i\in \{1,2\}$, and the resulting collection is a clique covering for $G$ of size at most $12+|T|=n-|S|-|R|+3$. Hence, $|S|+|R|\leq 3$, and since $|R|\geq 3$, we have $S=\emptyset$ and $|U^{3-i}|=1$, say $U^{3-i}=\{u\}$. 
Now, let $r_i\in R^{j}$, $j\in\{1,2\}$. Thus, by \pref{lem:T,R Facts and Tools}~(vi), $r_{3-i}\in R^{3-j}$. If $U^{3-i}\subseteq R^{3-j}$, then $G$ is ISA, a contradiction with \pref{lem:ccschlafli}. Thus, since $|U^{3-i}|=1$, we have $R^{j}=R_i$ and $R^{3-j}=R_{3-i}$. 
Now, in $\mathscr{P}(G)$, replace the clique $\nabla_{3-i}^3$ with the clique $\nabla _{3-i}^3\cup U^{3-i}$ to cover the edges in $E(R_1,R_2)$ and also replace the cliques $\Omega_i$ and $\Omega _{3-i}$ with the cliques $\Omega _i\cup T_{j}$ and $\Omega _{3-i}\cup T_{3-j}$, thereby obtaining a clique covering for $G$ of size $15=n-|T|+3\leq n-1$, a contradiction. This proves \pref{lem:reduce-to-schl3,2-3}.
\end{proof}

Now, we are ready to prove \pref{thm:IT+IR>=5}.
\begin{proof}[{\rm \textbf{Proof of \pref{thm:IT+IR>=5}.}}]
If $ |I_T|+|I_R|=5 $, then the result follows from Lemmas~\ref{lem:reduce-to-schl3,2-1},  \ref{lem:reduce-to-schl3,2-2} and \ref{lem:reduce-to-schl3,2-3}. Thus, suppose that $ |I_T|=|_R|=3 $. Also, on the contrary, suppose that both $ T $ and $ R $ are not clique. \vsp

(1) \textit{Both $ G[T] $ and $ G[R] $ are of type $2$.} \vsp

On the contrary, by \pref{lem:G[T] structure |I_T|=3} and due to the symmetry, assume that $G[R]$ is of type $3$ with $R^1$ and $R^2$ as in the definition. Note that if $G[T]$ is also of type $3$, then by \pref{lem:T,R Facts and Tools}~(vi), $G$ is ISA, which is impossible by \pref{lem:ccschlafli}. Thus, assume that $G[T]$ is of type $2$ at $q$, for some $q\in \{1,2,3\}$, with $t^q,t^{q+1}$ and $t^{q+2}$ as in the definition. Note that since $G[R]$ is of type $3$, again by \pref{lem:T,R Facts and Tools}~(vi), for every $t\in T$, there exists a unique $\beta(t)\in \{1,2\}$ such that $N(t,R)=R^{\beta(t)}$. Thus, $\beta(t^{q+1})=\beta(t^{q+2})=3-\beta(t^q)$. We claim that for every $t\in T^q\setminus \{t^q\}$, $\beta(t)=\beta(t^{q+1})=\beta(t^{q+2})$. To see the claim, on the contrary assume that $\beta(t)=\beta(t^q)$, for some $t\in T^q\setminus \{t^q\}$. Thus, since $R^{\beta(t)}\cap (R\setminus R_q)\neq \emptyset$, there exists $j\in \{1,2\}$ and a vertex $r_j\in R_{q+j}$ adjacent to both $t$ and $t^q$. Consequently, $\{t,s_{q+j}^{q+j},t^{q+j},r_j\}$ induces a claw, which is impossible. This proves the claim, which immediately implies that $G$ is ISA, a contradiction with \pref{lem:ccschlafli}. This proves (1). \vsp

By (1), suppose that both $G[T(\rho)]$ and $G[R(\rho)]$ are of type $2$ at $q$ and $q'$ respectively, for some $q,q'\in \{1,2,3\}$, with $t^1, t^{2}, t^{3}, r_{1}, r_{2}$ and $ r_{3}$ as in the definition. Also, let $X=T^q(\rho)\setminus \{t^q\}$  and $X'=R_{q'}(\rho)\setminus \{r_{q'}\}$.
Define $\mathscr{O}'(G)=\mathscr{O}(G)\cup \{T^q\cup R^3, (T\setminus \{t^q\})\cup R^3, T_3\cup R_{q'}, T_3\cup (R\setminus \{r_{q'}\})\}$ whose cliques cover all the edges in $E(G)\setminus E(T,R)$. \vsp

(2) \textit{If $q=q'$, then $t^q$ is nonadjacent to $r_{q}$.}\vsp

By symmetry, let $q=q'=1$ and on the contrary, suppose that $ t^1 $ is adjacent to $ r_1 $ and thus by \pref{lem:T,R Facts and Tools}~(vi), $ t^1 $ is anticomplete to $ \{r_2,r_3\} $ and $\{t^2,t^3\}$ is anticomplete to $ \{r_1\} $ and complete to $\{r_2,r_3\}$. 
Note that if $X=X'=\emptyset$, then by \pref{lem:T,R Facts and Tools}~(vi), $G$ is ISA, which is impossible by \pref{lem:ccschlafli}. First, assume that $S^1\neq \emptyset$, say $s^1_{p}\in S$, for some $p\in \{2,3\}$. Thus, $r_1$ is anticomplete to $X$ (otherwise, for some $ x\in X $ adjacent to $ r_1 $, $\{x,s_p^1,t^p,r_1\}$ would be a claw) and thus by \pref{lem:T,R Facts and Tools}~(vi), $ X $ is complete to $\{r_2,r_3\}$. Also, 
$X'$ is complete to $\{t^2,t^3\}$ and anticomplete to $\{t^1\}$ (otherwise, for $ x'\in X' $ nonadjacent to $ t^p $, $\{ r_p,s_p^1,t^p,x' \}$ would be a claw). Further, since $(s_1^1,t^1,X')$ is a quasi-triad of $G$ and $s_1^1$ is anticomplete to $X$, by \pref{lem:T,R Facts and Tools}~(i), $X$ is complete to $X'$, and so $G$ is ISA, which is impossible by \pref{lem:ccschlafli}. Therefore, $S^1= \emptyset$ and by symmetry, $ S_1= \emptyset $ (indeed,  for the rotator $\rho' =(s_1^1,s_2^2,s_3^3; r_1^3,r_2^3,r_3^3; t_3^1,t_3^2,t_3^3)$, we have $S_1(\rho)=S^1(\rho')$). Also, w.l.o.g. assume that $X\neq \emptyset$. 

Now, consider $\mathscr{O}'(G)$, replace the clique $Z_1^2$ with the cliques in $\mathscr{N}[R_1;(Z_2^1\cup T^1)\setminus R_1]$ to cover the edges in $E(T^1,R_1)$,  for every $i\in \{2,3\}$, replace the cliques $X_{5-i}^i$ and $ Z_{5-i}^i$ with the cliques $X_{5-i}^i\cup N(t^i,X')$ and $Z_{5-i}^i\cup N(r_i,X)$ to cover the edges in $E(T^i,R_1)$ and $E(T^1,R_i)$, respectively, replace the cliques $X_1^2,X_1^3,Z_2^1$ and $Z_3^1$ with the cliques $\{s_1^1,t^2_3,t^3_3\},\{s_1^1,r_2^3,r_3^3\}$ and $\{s_1^1,t^2,t^3,r_2,r_3\}$,  to cover the edges in $E(T\setminus T^1,R\setminus R^1)$. This yields a clique covering for $G$ of size $15+|X'|=n-|S|-|X|\leq n-1$, a contradiction. This proves (2).\vsp

(3) \textit{We have $ q\neq q' $.} \vsp

Suppose not, let $q=q'=1$. By (2), $ t^1 $ is nonadjacent to $ r_1 $ and thus by \pref{lem:T,R Facts and Tools}~(vi), $t^1$ is complete to $\{r_2,r_3\}$ and $\{t^2,t^3\}$ is complete to $r_1$ and anticomplete to $\{r_2,r_3\}$. Then, $X$ is anticomplete to $\{r_2,r_3\}$ and $X'$ is anticomplete to $\{t^2,t^3\}$ (otherwise, for $x\in X\cup X'$ adjacent to $ t^2,r_2 $, $\{x,s_2^2,t^2,r_2\}$ would be a claw) and thus by \pref{lem:T,R Facts and Tools}~(vi), $X$ is complete to $\{r_1\}$ and $X'$ is complete to $\{t^1\}$.  
First, assume that $ S^1\neq \emptyset $, say  $s_p^1\in S$, for some $p\in\{2,3\}$.
Since $(s_p^1,t^p,X')$ is a quasi-triad of $G$ and $X$ is complete to $\{s_p^1,t^p\}$, by \pref{lem:T,R Facts and Tools}~(i), $X$ is anticomplete to $X'$, and so $G$ is ISA, which is impossible by \pref{lem:ccschlafli}. Therefore, $S^1=\emptyset$. Now, $\rho'=(s_1^1,s_2^2,s^3_3; r_1,t^2_3,t^3_3; t^1,r^3_2,r^3_3)$ is a rotator of $\overline{G}$ and it is easy to see that $T^2(\rho')\cup T^3(\rho') = S^1=\emptyset$. Thus, $|I_T(\rho')|\leq 1$, which contradicts \pref{lem:|I_T|=3 |I_R|=1}. This proves (3). \vsp

Now, by (3) and due to symmetry, we may assume that $q=1$ and $q'=2$. First, suppose that $t^1$ is nonadjacent to $r_2$. Then, $X$ is anticomplete to $\{r_1,r_3\}$ and $X'$ is anticomplete to $\{t^2,t^3\}$ (otherwise, for $x\in X\cup X'$ adjacent to $ t^3,r_3 $, $\{x,s_3^3,t^3,r_3\}$ would be a claw) and thus by \pref{lem:T,R Facts and Tools}~(vi), $X$ is complete to $r_2$ and $X'$ is complete to $\{t^1\}$. Also, since $(s_2^2,t^2,X')$ is a quasi-triad of $G$ and $X$ is complete to $\{s_2^2,t^2\}$, by \pref{lem:T,R Facts and Tools}~(i), $X$ is anticomplete to $X'$, and thus $G$ is ISA, which is impossible by \pref{lem:ccschlafli}. 

Next, assume that $t^1$ is adjacent to $r_2$. Then, $X$ is anticomplete to $\{r_2\}$ and complete to $\{r_1,r_3\}$ (otherwise, for $x\in X$ adjacent to $ r_2 $, $\{x,s_2^2,t^2,r_2\}$ would be a claw). Also, $X'$ is anticomplete to $\{t^1\}$ and complete to $\{t^2,t^3\}$ (otherwise, for $x'\in X'$ adjacent to $ t^1 $, $\{x',s_1^1,t^1,r_1\}$ would be a claw). Now, if either $X=\emptyset$, or $X'=\emptyset$, then $G$ is ISA, which contradicts \pref{lem:ccschlafli}. Also, if $s_1^2\in S$, then since $(s_1^2,t^1,X')$ is a quasi-triad of $G$ and $X$ is anticomplete to $\{s_1^2\}$, $X$ is complete to $X'$, and so $G$ is ISA, again a contradiction with \pref{lem:ccschlafli}. Therefore, $X\neq \emptyset$, $X'\neq \emptyset$ and $ S_1^2=\emptyset $.
Also, w.l.o.g. assume that $|X|\geq |X'|\geq 1 $. 
Now, $\rho'=(s_1^1,s_2^2,s^3_3; r_1,t^2_3,t^3_3; t^1,r^3_2,r^3_3)$ is a rotator of $\overline{G}$ and it is easy to see that  $T^2(\rho')\cup T^3(\rho') = S^1$. Thus, if $S^1=\emptyset$, then $|I_T(\rho')|\leq 1$ which contradicts \pref{lem:|I_T|=3 |I_R|=1}. Therefore, $|S|\geq |S^1|\geq 1$. First, suppose that $ |S|=|S^1|=1 $. Then, in $\mathscr{O}'(G)$, replace the cliques $ X_1^2$, $X_1^3$, $Z_1^2$, $Z_1^3$, $Z_3^1 $ and $ Z_3^2 $ with the cliques $ X_1^2\cup X' $, $ X_1^3\cup X' $, $ Z_1^2\cup X\cup T^3 $, $ Z_1^3\cup T^2 $, $ (Z_3^1\cup T^2)\setminus S_2^1 $ and $ Z_3^2\cup X\cup T^3\cup S_2^1 $. Also, remove the clique $ Z_2^3 $ and add the cliques in $ \mathscr{N}[R_2;T^1\cup \{s_3^3,r_2^3\}] $. This yields a clique covering for $ G $ of size $ 16+|X'|=n-|X|\leq n-1 $. Finally, suppose that $ |S|\geq 2 $. Now, in $\mathscr{O}'(G)$, replace the cliques $X_1^2,X_1^3,Z_1^2$ and $Z_3^2$ with the cliques $X_1^2\cup X',X_1^3\cup X',Z_1^2\cup X$ and $Z_3^2\cup X$, and add the cliques in $\mathscr{N}[X';X]$ as well as the cliques $\{t^1,r_2\}$ and $\{t^2,t^3,r_1,r_3\}$, thereby obtaining a clique covering $\mathscr{C}$ for $G$ of size at most $18+|X'|=n-|S|-|X|+3$. Thus, $|S|=2$ and  $|X|=|X'|=1$. Now, if $X$ is complete to $X'$, then $G$ is ISA, a contradiction with \pref{lem:ccschlafli}, and if $X$ is anticomplete to $X'$, then removing the single clique in $\mathscr{N}[X';X]$ from $\mathscr{C}$, yields a clique covering for $G$ of size $n-1$, a contradiction. This completes the proof of \pref{thm:IT+IR>=5}.
\end{proof}
%

\subsection{Obtaining $|I_T|=|I_R|=2$}\label{sub:max=3}

As the penultimate step of the proof of \pref{thm:noap}, the aim of this subsection is the following theorem, whose proof will be given at end of this subsection.
\begin{thm}\label{thm:3,3}
	Let $G$ be a counterexample to \pref{thm:noap}. Then for every rotator $\rho$ of $\overline{G}$, we have  $|I_T(\rho)|=|I_R(\rho)|=2$. 
\end{thm}
We need the following three lemmas to prove \pref{thm:3,3}. For simplicity, throughout the proofs, we omit the term $ \rho $.
\begin{lem}\label{lem:3,2type2clique}
	Let $G$ be a counterexample to \pref{thm:noap}, with $|I_T(\rho)|+|I_R(\rho)|=5$, say $|I_T(\rho)|=3$ and $|I_R(\rho)|=2$ for some rotator $\rho$ of $\overline{G}$. Also, let $R(\rho)$ be a clique of $G$. Then $G[T(\rho)]$ is of type $3$.
\end{lem}
\begin{proof}
On the contrary, by lemmas \ref{lem:both clique} and \ref{lem:G[T] structure |I_T|=3}, $G[T]$ is of type $2$ at $q$, for some $q\in \{1,2,3\}$. Also, w.l.o.g. assume that $I_R=\{1,2\}$. Let $t^q\in T^q, t^{q+1}\in T^{q+1}$ and $t^{q+2}\in T^{q+2}$ be as in the definition of type $2$ at $q$, and let $X=T^q\setminus \{t^q\}$. Note that the cliques in $\mathscr{O}'(G)=\mathscr{O}(G)\cup \{T^q\cup R^3,(T\setminus \{t^q\})\cup R^3, T_3\cup R\}$ cover all the edges in $E(G)\setminus E(T,R)$.\vsp
	
(1) \textit{If $q\in \{1,2\}$, then $s_{k'}^k\in S$, for some $(k,k')\in \{(1,2),(2,1),(q,3)\}$. Also, exactly one of $t^q$ or $\{t^{q+1},t^{q+2}\}$ is complete to $R$ (and the other one is anticomplete to $R$).}\vsp
	
The first assertion follows from \pref{lem:square-forcer}~(i) (applying to $(i,j,k)=(q,3-q,3)$). Consequently, the second assertion follows from \pref{lem:one is complete to R}~(i) (applying to $(i,j)=(k,k')$) and \pref{lem:T,R Facts and Tools}~(vi).\vsp
	
Henceforth, let $(k,k')$ be as in (1).\vsp
	
(2) \textit{If $q \in \{1,2\}$, then $\{t^{q+1},t^{q+2}\}$ is complete to $R$ and $t^q$ is anticomplete to $R$.}\vsp
	
By symmetry, let $q=1$. Suppose not, by (1), $t^1$ is complete to $R$ and $\{t^{2},t^{3}\}$ is anticomplete to $R$. We observe that if $(k,k')=(2,1)$, then $(s_1^2,x,R_2), x\in X,$ is a quasi-triad of $G$ and since $R_1$ is complete to $R_2\cup \{s_1^2\}$, by \pref{lem:T,R Facts and Tools}~(i), $X$ is anticomplete to $R_1$, and if $k=1$, then $(s_{k'}^1,t^{k'},R_1)$ is a quasi-triad of $G$ and $X$ is complete to $\{s_{k'}^1,t^{k'}\}$, again by \pref{lem:T,R Facts and Tools}~(i), $X$ is anticomplete to $R_1$. Also, $(s_2^2,t^2,R_2)$ is a quasi-triad of $G$ and $X$ is complete to $\{s_2^2,t^2\}$. Thus by \pref{lem:T,R Facts and Tools}~(i), $X$ is anticomplete to $R_2$. Thus, $X$ is anticomplete to $R$.
In the sequel, first we claim that $s_1^3\in S$. Suppose not, in $\mathscr{O}'(G)$,  for every $i\in \{1,2\}$, replace the clique $Z_i^3$ with the clique $Z_i^3\cup \{t^1\}$ to cover the edges in $E(\{t^1\},R)$, thereby producing a clique covering $\mathscr{C}$ for $G$ of size $15=n-|S|-|X|-|R|+3$. Thus, $|S|+|X|+|R|\leq 3$, i.e. $|S|=1$, $X=\emptyset$ and $|R|=2$. Now, remove the clique $T^1\cup R^3$ from $\mathscr{C}$ and replace the cliques in $Z_3^1$ and $Z_3^2$ with the cliques $(Z_3^1\cup S_1^2)\setminus S_2^1$ and $(Z_3^2\cup S_2^1\cup \{t^1\})\setminus S_1^2$ to cover the edge $t^1r_3^3$ (note that for $i=1,2$, the edge $t^1r^3_i$ has been already covered by the clique $Z^3_i\cup \{t^1\}$). This yields a clique covering for $G$ of size $14=n-1$, a contradiction. Thus, $s_1^3\in S$, and consequently by (1), $|S|\geq 2$. 
Now, the collection $\mathscr{O}'(G)\cup \{\{t^1\}\cup R\}$ is a clique covering for $G$ of size $16=n-|S|-|R|-|X|+4$. Therefore $|S|+|R|+|X|\leq 4$, i.e. $|S|=2$, $|R|=2$ and $X=\emptyset$. 
Consequently, choose $l\in \{2,3\}$ for which $S_l^1=\emptyset$, and replace the cliques $X_2^1$ and $X_3^1$ with the cliques $X_l^1\cup N(t^1,R_1\cup R_{5-l})$ and $X_{5-l}^1\cup R_{l}$ in the collection $\mathscr{O}'(G)$, thereby obtaining a clique covering for $G$ of size $15=n-1$, again a contradiction. This proves (2).\vsp
	
(3) \textit{If $q\in \{1,2\}$, then $S_3\neq \emptyset$ and $|S|\geq 2$.}\vsp
	
On the contrary, suppose that $S_3=\emptyset$. By symmetry, let $q=1$. Due to (2), $\{t^{2},t^{3}\}$ is complete to $R$ and $t^1$ is anticomplete $R$. Now, in $\mathscr{O}'(G)$, replace the clique $X_3^2$ with the clique $X_3^2\cup R$, for every $i\in \{1,2\}$, replace the clique $Z_i^{3-i}$  with the clique $Z_i^{3-i}\cup T^3$ and also add the cliques in $\mathscr{N}[X;R]$, thereby covering the edges in $E(T,R)$ and obtaining a clique covering $\mathscr{C}$ for $G$ of size $15+|X|=n-|S|-|R|+3$. Thus, $|S|+|R|\leq 3$, i.e. $|S|=1$ and $|R|=2$ and so by (1), $|S|=|S^{3-i}_i|=1$ for some $i\in \{1,2\}$. Now, remove the clique $T^1\cup R^3$ from $\mathscr{C}$, replace the cliques $Z_1^3$ and $Z_2^3$ with the cliques $R\cup \{s_3^3\}$ and $\{s_3^3,t^1,r_1^3,r_2^3\}$ to cover the edges $t^1r_1^3$ and $t^1r_2^3$, and replace the cliques $Z_3^1$ and $Z_3^2$ with the cliques $(Z_3^1\cup S_1^2)\setminus S_2^1$ and $(Z_3^2\cup S_2^1\cup \{t^1\})\setminus S_1^2$, to cover the edge $t^1r_3^3$. This yields a clique covering of size $14+|X|=n-1$, a contradiction. This proves that $S_3\neq \emptyset$. Now, assume that $|S|\leq 1$, and thus by (1), $|S|=|S_3^1|=1$. In this case, considering $\mathscr{O}'(G)\cup \mathscr{N}[X;R]$, remove the clique  $T^1\cup R^3$, replace the cliques $X_1^3,X_2^3,Z_1^3,Z_2^3$ and $Z_3^2$ with the cliques $X_1^3\cup R_2,X_2^3\cup R_1,R\cup \{s_3^3,t^2\},\{s_3^3,t^1,r_1^3,r_2^3\}$ and $Z_3^2\cup \{t^1\}$ to cover the edges in $E(T,R)\cup E(\{t^1\},R^3)$, thereby comprising a clique covering for $G$ of size $14+|X|\leq n-1$, which is impossible. This proves (3).\vsp
	
(4) \textit{We have $q=3$.}\vsp
	
Suppose not, w.l.o.g. let $q=1$. By (2), $\{t^{2},t^{3}\}$ is complete to $R$ and $t^1$ is anticomplete $R$ and by (3), $S_3\neq \emptyset$ and $|S|\geq 2$. Note that $\mathscr{O}'(G)\cup \mathscr{N}[X;R]\cup \{R\cup \{t^2,t^3\}\}$ is a clique covering for $G$ size $16+|X|=n-|S|-|R|+4$. Hence, $|S|=|R|=2$. Since $S_3\neq \emptyset$, either $S^3=\emptyset$ or $S= S^3\cup S_3$.
Now, in $\mathscr{O}'(G)\cup \mathscr{N}[X;R]$, replace the cliques $X_1^3$ and $X_2^3$ with the cliques $(X_1^3\cup S_2^1\cup R_2)\setminus S_1^2$ and $(X_2^3\cup S_1^2\cup R_1)\setminus S_2^1$ to cover the edges in $E(T^3,R)$, if $s^3_2\in S$ (i.e. $(k,k')=(1,3)$), then replace the cliques $X_1^2$ and $X_3^2$ with the cliques $(X_1^2\cup R_2\cup S_3^1)$ and $(X_3^2\cup R_1)\setminus S_3^1$ and if $s^3_2\notin S$, then replace the cliques $Z_1^3$ and $Z_2^3$ with the cliques $Z_1^3\cup T^2$ and $Z_2^3\cup T^2$ to cover the edges in $E(T^2,R)$. This yields a clique covering for $G$ of size $15+|X|=n-1$, a contradiction This proves (4).\vsp
	
(5) \textit{If $S_3\neq \emptyset$, then $\{t^{q+1},t^{q+2}\}$ is complete to $R$ and $t^q$ is anticomplete to $R$.}\vsp
	
Let $s_3^p\in S$ for some $p\in \{1,2\}$. Note by (4), we have $q=3$. By \pref{lem:one is complete to R}~(i) (applying to $(i,j)=(p,3)$), either $t^3$ or $\{t^1,t^2\}$ is complete to $R$ (and the other one is anticomplete to $R$). Thus if (5) does not hold, then $t^3$ is complete to $R$ and $\{t^1,t^2\}$ is anticomplete to $R$. Note that since for every $j\in \{1,2\}$, $(s_j^j,t^j,R_j)$ is a quasi-triad of $G$ and $X$ is complete to $\{s_j^j,t^j\}$, by \pref{lem:T,R Facts and Tools}~(i), $X$ is anticomplete to $R$. Thus, adding the clique $\{t^3\}\cup R$ to the collection $\mathscr{O}'(G)$ provides a clique covering for $G$ of size $16=n-|S|-|R|-|X|+4$. Thus, $|S|+|R|+|X|\leq 4$. 
Also, if $|S|=|S_3^p|=1$, i.e. $S=\{s_3^p\}$, then in $\mathscr{O}'(G)$, remove the clique $T^3\cup R^3$ and replace the cliques $Z_p^{3-p},Z_3^{3-p}, Z_{3-p}^p$ and $Z_{3-p}^3$  with the cliques $Z_p^{3-p}\cup \{t^3\},Z_3^{3-p}\cup \{t^3\},(Z_{3-p}^p\cup \{t^3\})\setminus S_3^p$ and $Z_{3-p}^3\cup S_3^p$ to cover the edges in $E(\{t^3\};R')$, thereby comprising a clique covering for $G$ of size $14\leq n-1$, a contradiction. 
Thus, $|S|=|R|=2$ and $X=\emptyset$. Now,  in $\mathscr{O}'(G)$, replace the cliques $X_1^3$ and $X_2^3$ with the cliques $(X_1^3\cup S_2^1\cup R_2)\setminus S_1^2 $ and $(X_2^3\cup S_1^2\cup R_1)\setminus S_2^1$ to cover the edges in $E(\{t^3\},R)$, thereby obtaining a clique covering for $G$ of size $15=n-1$, a contradiction. This proves (5).\vsp
	
(6) \textit{We have $S_3=\emptyset$.}\vsp
	
By (4), we have $q=3$. If (6) does not hold, say $s_3^p\in \{1,2\}$ then by (5), $\{t^1,t^2\}$ is complete to $R$ and thus $t^3$ is anticomplete to $R$. We claim that $S^3\neq \emptyset$. On the contrary, considering $\mathscr{O}'(G)$, for every $i\in \{1,2\}$, replace the clique $Z_i^3$ with the clique $Z_i^3\cup \{t^1,t^2\}$ to cover the edges in $E(\{t^1,t^2\},R)$ and add the cliques in $\mathscr{N}[X;R\cup \{t^1,t^2\}]$, thereby obtaining a clique covering $\mathscr{C}$ for $G$ of size $15+|X|=n-|S|-|R|+3$. Thus, $|S|=1$, say $|S|=|S_3^p|=1$, for some $p\in \{1,2\}$, and $|R|=2$. Now, remove the clique $(T\setminus\{t^3\}\cup R^3)$ from $\mathscr{C}$ and note that the only uncovered edges are $t^1r_3^3$ and $t^2r_3^3$. In order to cover them, for every $i\in \{1,2\}$, replace the clique $Z_3^{3-i}$ with the clique $Z_3^{3-i}\cup \{t^i\}$. This yields a clique covering for $G$ of size $14+|X|=n-1$. Therefore, $S^3\neq \emptyset$, and consequently, $|S|\geq 2$. Now, the family $\mathscr{O}'(G)\cup \mathscr{N}[X;R]\cup \{\{t^1,t^2\}\cup R\}$ is a clique covering for $G$ of size a most $16+|X|=n-|S|-|R|+4$. Hence, $|S|+|R|\leq 4$, i.e. $|S|=|R|=2$. In this case, in $\mathscr{O}'(G)$, if $S_3^1=\emptyset$ replace the cliques $X_2^1,X_3^1,X_1^2$ and $X_3^2$ with the cliques $(X_2^1\cup S_3^2\cup R_1)\setminus S_2^3,(X_3^1\cup S_2^3\cup R_2)\setminus S_3^2, X_1^2\cup R_2$ and $X_3^2\cup R_1$ and if $S_3^1\neq \emptyset$ replace the cliques $X_2^1,X_3^1,X_1^2$ and $X_3^2$ with the cliques $X_2^1\cup R_1,X_3^1\cup R_2, (X_1^2\cup S_3^1\cup R_2)\setminus S_1^3$ and $(X_3^2\cup S_1^3\cup R_1)\setminus S_3^1$, thereby covering the edges in $E(\{t^1,t^2\},R)$ and also add the cliques in $\mathscr{N}[X;R]$, thereby producing a clique covering for $G$ of size $15+|X|=n-1$, a contradiction. This proves (6).\vsp
	
Note that by (4), $q=3$, and by (6), $S_3$ is empty. Now, in the collection $\mathscr{O}'(G)$,  for every $i\in \{1,2\}$, replace the cliques $X_3^i$ with the clique $X_3^i\cup N(t^i,R)$ and replace the clique $Z_i^{3-i}$ with the cliques in $\mathscr{N}[R_i;(Z_i^{3-i}\cup T^3 )\setminus R_i]$, thereby covering the edges in $E(T,R)$ and comprising a clique covering $\mathscr{C}$ for $G$ of size $13+|R|=n-|S|-|X|+1$. Thus, $|S|+|X|\leq 1$. If $|S|+|X|=1$, then removing the clique $T^3\cup R^3$ from $\mathscr{C}$ and replacing the cliques $Z_3^1$ and $Z_3^2$ with the cliques $Z_3^1\cup \{t^3,r_2^3\}$ and $Z_3^2\cup \{t^3,r_1^3\}$ to cover the edges in $E(\{t^3\},R^3)$, provides a clique covering for $G$ of size $12+|R|=n-1$, a contradiction. Thus, $S=X=\emptyset$. Now, consider the family $\mathscr{P}(G)$, merge the pairs $(\Delta _1^3,\Delta _2^3)$ and $(\nabla _1^3,\nabla _2^3)$ to cover the edges in $E(T^1,T^2)$ and $E(R_1,R_2)$, respectively, and also for every $i\in \{1,2,3\}$, replace the clique $\Omega ^i$ with the cliques $\Omega ^i\cup N(t^i,R)$ to cover the edges in $E(T,R)$, thereby obtaining a clique covering for $G$ of size $13=n-|R|+1\leq n-1$, again a contradiction. This proves \pref{lem:3,2type2clique}.
\end{proof}

\begin{lem}\label{lem:3,2-2clique}
	Let $G$ be a counterexample to \pref{thm:noap}, with $|I_T(\rho)|+|I_R(\rho)|=5$, say $|I_T(\rho)|=3$ and $|I_R(\rho)|=2$ for some rotator $\rho$ of $\overline{G}$. Then $G[R(\rho)]$ is of type $1$.
\end{lem}
\begin{proof}
	Suppose not, by \pref{lem:G[T] structure |I_T|=2}, $R$ is a clique. By  \pref{lem:3,2type2clique}, $G[T]$ is of type $3$ at $q$, for some $q\in \{1,2,3\}$, with $T_1=\{t_1^q,t_1^{q+1}\}\cup T_1^{q+2}$ and $T_2=\{t_2^q,t_2^{q+2}\}\cup T_2^{q+1}$ as in the definition. Note that the cliques in $\mathscr{O}'(G)=\mathscr{O}(G)\cup \{T_1\cup R^3,T_2\cup R^3,T_3\cup R\}$ cover all the edges in $E(G)\setminus E(T,R)$.\vsp
	
	(1) \textit{We have $S\neq \emptyset$. Also, if $S=S^3\neq \emptyset$, then there exists a unique $i_0\in \{1,2\}$ such that $T_{i_0}$ is complete $R$ and $T_{3-i_0}$ is anticomplete to $R$.}\vsp
	
	Note that by \pref{lem:T,R Facts and Tools}~(vi), for every vertex $r\in R$, there exists $i\in\{1,2\}$ such that $r$ is complete to $T_i$ and anticomplete to $T_{3-i}$. Thus, if (1) does not hold, then either $S=\emptyset$, or $S=S^3\neq \emptyset$ and for every $i\in \{1,2\}$, there is a vertex in $R$ complete to $T_i$ (and anticomplete to $T_{3-i}$). Consider the collection $\mathscr{P}(G)$, merge the pair $(\nabla _1^3,\nabla _2^3)$ to cover the edges in $E(R_1,R_2)$, and for every $i\in \{1,2\}$, if $S_i=\emptyset$, then replace the clique $\Omega _i$ with the cliques in $\mathscr{N}[R_i;(\Omega _i\cup T)\setminus R_i\}]$, and if $S_i\neq \emptyset$, then add the cliques in $\mathscr{N}[R_i;T]$ to cover the edges in $E(T,R)$. This procedure yields a family $\mathscr{C}_1$ of cliques of $G$ covering all the edges in $E(G)\setminus (E(T_1)\cup E(T_2))$. If $S=S^3\neq \emptyset$, then by the assumption, the latter modification in producing $\mathscr{C}_1$ also covers the edges in $E(T_1)\cup E(T_2)$ and so $\cc(G)\leq 12+|S|+|R|\leq n-1$, a contradiction. Also, if $S=\emptyset$, then there exists at most one $j\in \{1,2\}$ such that the edges in $E(T_j)$ are not covered. Now, in $\mathscr{C}_1$, if $|T|=4$, then replace the clique $\Delta _{q+j}^{q-j+3}$ with the clique $\Delta _{q+j}^{q-j+3}\cup \{t^q_j\}$ and if $|T|\geq 5$, then add the clique $T_j$. This yields a clique covering for $G$ of size at most $n-1$. This proves (1).\vsp
	
	(2) \textit{There exists a unique $i_0\in \{1,2\}$ such that $T_{i_0}$ is complete $R$ and $T_{3-i_0}$ is anticomplete to $R$.}\vsp
	
	For if $S=S^3$, then (2) follows from (1). Otherwise, there exists $s_{p'}^{p}\in S$, for some $(p,p')\in J$ such that $p\neq 3$. Also, note that there is a vertex in $T^p$ nonadjacent to a vertex in $T^{p'}$. Thus, (2) follows from \pref{lem:one is complete to R}~(i) (applying to $(i,j)=(p,p')$).\vsp
	
	Finally, assuming $i_0$ to be as in (2), $\mathscr{C}=\mathscr{O}'(G)\cup \{T_{i_0}\cup R\}$ is a clique covering for $G$ of size $16=n-|S|-|T|-|R|+7$. Thus, $|S|+|T|+|R|\leq 7$, which by (1), yields $|S|=1$, $|T|=4$ and $|R|=2$, say $R_1=\{r_1\}$ and $R_2=\{r_2\}$. Now, apply the following modifications to $\mathscr{O}'(G)$ to cover the edges in $E(T,R)$. 
	If $S^3=S_3=\emptyset$, then replace the cliques $Z_1^2,Z_1^3,Z_2^1$ and $Z_2^3$  with the cliques $Z_1^2\cup N(r_1,T^3),Z_1^3\cup N(r_1,T^1\cup T^2),Z_2^1\cup N(r_2,T^3)$ and $Z_2^3\cup N(r_2,T^1\cup T^2)$, if $S_3\neq \emptyset$, then replace the cliques $Z_1^2,Z_1^3,Z_2^1$ and $Z_2^3$  with the cliques $(Z_1^2\cup N(r_1,T^1\cup T^3))\setminus S_3^2, Z_1^3\cup S_3^2\cup N(r_1,T^2) ,(Z_2^1\cup N(r_2,T^2\cup T^3))\setminus S_3^1$ and $Z_2^3\cup S_3^1\cup N(r_2,T^1)$ and if $S^3\neq \emptyset$, then replace the cliques $Z_1^2,Z_1^3,Z_2^1$ and $Z_2^3$  with the cliques $Z_1^2\cup S_2^3\cup N(r_1,T^1\cup T^3), (Z_1^3\cup N(r_1,T^2))\setminus S_2^3,Z_2^1\cup S_1^3\cup N(r_2,T^2\cup T^3)$ and $(Z_2^3\cup N(r_2,T^1))\setminus S_1^3$.
	This provides a clique covering for $G$ of size $15=n-1$, a contradiction.
	This proves \pref{lem:3,2-2clique}.
\end{proof}

\begin{lem}\label{lem:3,2-3clique}
	Let $G$ be a counterexample to \pref{thm:noap}. Then, for every  rotator $\rho$ of $\overline{G}$, $|I_T(\rho)|=|I_R(\rho)|\geq 2$.
\end{lem}
\begin{proof} 
By \pref{lem:|I_T|=3 |I_R|=1}, $|I_T|, |I_R|\geq 2$. On the contrary and w.l.o.g. assume that $|I_T|=3$ and $|I_R|=\{1,2\}$. By \pref{lem:3,2-2clique}, $G[R]$ is of type $1$, with $r_1\in R_1$, $r_2\in R_2$, $U^1\subseteq R_2$ and $U^2\subseteq R_1$ as in the definition. Thus, by \pref{thm:IT+IR>=5}, $ T $ is a clique.\vsp 
	
	(1) \textit{Either $s_{3-p}^p\in S$, for some $p\in \{1,2\}$, or $|S_3|=2$.}\vsp
	
	This follows from \pref{lem:square-forcer}~(i) (applying to $(i,j,k)=(1,2,3)$ and $(i,j,k)=(2,1,3)$).\vsp
	
	(2) \textit{There exists a unique $i_0\in \{1,2\}$, for which $r_{i_0}$ is complete to $T$ and $r_{3-i_0}$ is anticomplete to $T$. Also, $U^{3-i_0}$ is anticomplete to $T$}.\vsp
	
	By (1), either $s_{3-p}^p\in S$, for some $p\in \{1,2\}$ or $|S_3|=2$, and the first claim follows from \pref{lem:one is complete to R}~(i) or (ii) (applying to $(i,j)=(3-p,p)$ or $\{i,j\}=\{1,2\}$ and $l=3$, respectively). Now, note that by (1), there exists $k\in \{1,2,3\}$ such that both $s_k^1,s_k^2\in V(G)$. Since $(s_k^{3-i_0},r_{3-i_0},T^k)$ is a quasi-triad of $G$ and $U^{3-i_0}$ is complete to $\{s_k^{3-i_0},r_{3-i_0}\}$, by \pref{lem:T,R Facts and Tools}~(i), $U^{3-i_0}$ is anticomplete to $T^k$. Thus, for every $t\in T^k$, $(s_k^{i_0},t,U^{3-i_0})$ is a quasi-triad of $G$, and since $T\setminus T^k$ is complete to $\{s_k^{i_0},t\}$, by \pref{lem:T,R Facts and Tools}~(i), $U^{3-i_0}$ is anticomplete to $T\setminus T^k$, as desired. This proves (2).\vsp
	
	Henceforth, let $i_0$ be as in (2).\vsp
	
	(3) \textit{If $G[R]$ is of type $1.1$, then $|S|+|T|\leq 5$, i.e. $1\leq |S|\leq 2$.}\vsp
	
	By symmetry, we may assume that $G[R]$ is of type $1.1$ at $1$ (i.e. $U^1=\emptyset$). 
	Thus, $\mathscr{C}(G)=\mathscr{O}(G)\cup \mathscr{N}[R\setminus \{r_{3-i_0}\};T]\cup \{T\cup R^3,T_3\cup R_1,T_3\cup (R\setminus \{r_1\})\}$ is a clique covering for $G$ of size $14+|R|=n-|S|-|T|+5$. This, with (1),  implies (3).\vsp
	
	(4) \textit{If $G[R]$ is of type $1.1$, then $|S|=1$.}\vsp
	
	Again, by symmetry, we may assume that $G[R]$ is of type $1.1$ at $1$ (i.e. $U^1=\emptyset$). Also, let $\mathscr{C}$ be the family of cliques defined in (3). If (4), does not hold, then by (3), we have $|S|=2$ and $|T|=3$. Thus by (1), either $S^3$ or $S_3$ is empty. We claim that $R_{i_0}=\{r_{i_0}\}$. For this is trivial if $i_0=2$, and if $i_0=1$, then by (2), $U^2$ is anticomplete to $T$, and removing the cliques in $\mathscr{N}[U^2;T]$ from $\mathscr{C}$, yields a clique covering for $G$ of size $n-|U^2|$. Thus, $U^2=\emptyset$. This proves the claim. Now, remove the clique $N[r_{i_0};T]$ from $\mathscr{C}$, and in order to cover the edges in $E(\{r_{i_0}\},T)$, in the case that $S^3=\emptyset$, replace the cliques $X_1^3,X_2^3$ and $Z_{i_0}^3$ with the cliques $(X_{i_0}^3\cup S_{3-i_0}^{i_0})\setminus S_{i_0}^{3-i_0},(X_{3-i_0}^3\cup S_{i_0}^{3-i_0}\cup \{r_{i_0}\})\setminus S_{3-i_0}^{i_0}$ and $Z_{i_0}^3\cup T^1\cup T^2$ and in the case that $S_3=\emptyset$, replace the cliques $X_3^1,X_3^2$ and $Z_{i_0}^{3-i_0}$ with the cliques $X_3^1\cup \{r_{i_0}\},X_3^2\cup \{r_{i_0}\}$ and $Z_{i_0}^{3-i_0}\cup T^3$, thereby obtaining a clique covering for $G$ of size $n-1$, a contradiction. This proves (4).\vsp
	
	(5) \textit{$G[R]$ is of type $1.2.$} \vsp
	
	For if $G[R]$ is of type $1$ say at $1$, then by (4), $|S|=1$, and so by (1), $S=\{s_{3-p}^p\}$, for some $p\in \{1,2\}$. Consider the collection $\mathscr{P}(G)$, merge the pairs $\{\Delta _p^{3-p},\Delta _3^{3-p}\}$ and $\{\Delta _1^3,\Delta _2^3\}$, replace the cliques $\Delta _{3-p}^p$ and $\nabla _2^3$ with the cliques $\Delta _{3-p}^p\cup T_3$ and $\nabla _2^3\cup U^2$, thereby covering the edges in $E(T)\cup E(R)$, 
	Also, on order to cover the edges in $E(T,R)$, in the resulting family, if $i_0=p$, then replace the clique $\Omega _{i_0}$ with the cliques in $\mathscr{N}[R_{i_0};(\Omega_i\cup T)\setminus R_{i_0}]$, and if $i_0=3-p$, then replace the cliques $\Omega ^{i_0}$ and $\Omega _{i_0}$ with the cliques in $\mathscr{N}[T^{i_0};(\Omega ^{i_0}\cup R_{i_0})\setminus T^{i_0}]$ and $\mathscr{N}[R_{i_0};(\Omega _{i_0}\cup T)\setminus (T^{i_0}\cup R_{i_0})]$ (moreover, if $i_0=2$, then add the cliques in $\mathscr{N}[U^2;T]$), thereby obtaining a clique covering for $G$ of size at most $13+|T^{i_0}|+|U^2|=n-|T\setminus T^{i_0}|+1\leq n-1$, a contradiction. This proves (5). \vsp
	

	Henceforth, by (5), we assume that $G[R]$ is of type $1.2$ and w.l.o.g.  $2\leq |R_1|\leq |R_2|$. Note that the cliques in $\mathscr{O}'(G)=\mathscr{O}(G)\cup \mathscr{N}[R_1;T_3\cup R_2]\cup \{T\cup R^3\}$ cover all the edges in  $E(G)\setminus E(T,R)$. \vsp
	
	(6) \textit{We have $S_3\neq \emptyset$}\vsp
	
	Suppose not, in $\mathscr{O}'(G)$,  for every $i\in \{1,2\}$, replace the clique $X_3^i$ with the cliques in $\mathscr{N}[T^i;(X_3^i\cup R)\setminus T^i]$ and add the cliques in $\mathscr{N}[T^3;R]$, thereby obtaining a clique covering for $G$ of size $11+|T|+|R_1|=n-|S|-|R_2|+2$, which by (1) is at most $n-1$. This proves (6).\vsp
	
	(7) \textit{We have $|S|=|R_1|=|R_2|=2$. In particular, either $|S|=|S_3|=2$ or $S=\{s_{3-p}^p,s_3^{p'}\}$, for some $p,p'\in \{1,2\}$.}\vsp
	
	Note that $\mathscr{O}'(G)\cup \mathscr{N}[T;R]$ is a clique covering for $G$ of size $13+|T|+|R_1|=n-|S|-|R_2|+4$. Thus, $|S|+|R_2|\leq 4$, i.e. $|S|\leq 2$. Also, by (6), $|S_3|\geq 1$. This, together with (1), proves (7).\vsp
	
	Note that by (7), we have $n=15+|T|$ and $|\mathscr{O}'(G)|=15$.\vsp
	
	(8) \textit{We have $|T|\geq 5$.}\vsp
	
	First, we claim that $|T^3|\geq 2$. Suppose not, let $T^3=\{t^3\}$. In the family $\mathscr{O}'(G)\cup \mathscr{N}[T\setminus T^3;R]$, replace the cliques $X_1^3$ and $X_2^3$ with the cliques $(X_1^3\cup S_2^1\cup N(t^3,R_2))\setminus S_1^2$ and $(X_2^3\cup S_1^2\cup N(t^3,R_1))\setminus S_2^1$ (in fact, the truth of (7) ensures that these two are cliques of $G$) to cover the edges in $E(T^3,R)$. This yields a clique covering for $G$ of size $14+|T|=n-1$, a contradiction. This proves the claim. 
	Now, if $|T|\leq 4$, then for every $i\in \{1,2\}$, $|T^i|=1$, say $T^i=\{t^i\}$, and also $|T^3|=2$. By (7), either $|S|=|S_3|=2$ or $S=\{s_{3-p}^p,s_3^{p'}\}$, for some $p,p'\in \{1,2\}$. In the former case, consider $\mathscr{O}'(G)\cup \mathscr{N}[T\setminus T^1;R]$ and replace the cliques $X_{2}^1$ and $X_3^1$ with the cliques $X_{2}^1\cup S_3^{2}\cup N(t^1,R_1)$ and $(X_3^1\cup N(t^1,R_{2}))\setminus S_3^{2}$ to cover the edges in $E(T^1,R)$, thereby comprising a clique covering for $G$ of size $14+|T|=n-1$.
	In the latter case,
	in $\mathscr{O}'(G)\cup \mathscr{N}[T\setminus T^{3-p};R]$, replace the cliques $X_p^{3-p}$ and $X_3^{3-p}$ with the cliques  $X_p^{3-p}\cup S_3^p\cup N(t^{3-p},R_{3-p})$ and $(X_3^{3-p}\cup N(t^{3-p},R_p))\setminus S_3^p$. This provides a clique covering for $G$ of size $14+|T|=n-1$, a contradiction. This proves (8).\vsp
	
	Finally, the collection $\mathscr{O}'(G)\cup \mathscr{N}[R;T]$ is a clique covering for $G$ of size $19=n-|T|+4$, which by (8), is at most $n-1$, a contradiction. This proves \pref{lem:3,2-3clique}.
\end{proof}

Now, we are ready to prove \pref{thm:3,3}.

\begin{proof}[{\rm \textbf{Proof of \pref{thm:3,3}.}}]
By \pref{lem:3,2-3clique}, $|I_T|=|I_R|\geq 2$. Now, on the contrary, assume that $|I_T|=|I_R|=3$. By \pref{thm:IT+IR>=5} and w.l.o.g. assume that $R$ is a clique. \vsp
	
	(1) \textit{We have $S\neq \emptyset$.}\vsp
	
	Suppose not, let $S=\emptyset$. Consider the family $\mathscr{P}(G)$, merge the pairs $(\nabla _2^1,\nabla _3^1),(\nabla _1^2,\nabla _3^2)$ and $(\nabla _1^3,\nabla _2^3)$ to cover the edges in $E(R)$ and for every $i\in \{1,2,3\}$, replace the clique $\Omega _i$ with the cliques in $\mathscr{N}[R_i;(\Omega _i\cup T)\setminus R_i]$ to cover the edges in $E(T,R)$, and call the resulting collection $\mathscr{P}'(G)$, the clique in which cover all the edges in $E(G)\setminus (\cup_{i=1}^3E(T^i,T^{i+1}))$. By \pref{lem:both clique}, $T$ is not a clique. Also, if $G[T]$ is of type $2$ at $q$, for some $q\in \{1,2,3\}$ with $t^q,t^{q+1}$ and $t^{q+2}$, as in the definition, then in $\mathscr{P}'(G)$, merge the pair $(\Delta _{q+1}^q,\Delta _{q+2}^q)$ and replace the cliques $\Delta _{q+1}^{q+2}$ and $\Delta _{q+2}^{q+1}$ with the cliques $\Delta _{q+1}^{q+2}\cup (T^q\setminus \{t^q\})$ and $\Delta _{q+2}^{q+1}\cup (T^q\setminus \{t^q\})$, thereby obtaining a clique covering for $G$ of size $11+|R|\leq n-1$, a contradiction. Thus, by \pref{lem:G[T] structure |I_T|=3}, $G[T]$ is of type $3$ at $q$, for some $q\in \{1,2,3\}$, with $T_1=\{t_1^q,t_1^{q+1}\}\cup T_1^{q+2}$ and $T_2=\{t_2^q,t_2^{q+2}\}\cup T_2^{q+1}$ as in the definition. Now, in $\mathscr{P}'(G)$, if $|T|=4$, then 
	replace the cliques $\Delta _{q+1}^{q+2}$ and $\Delta _{q+2}^{q+1}$ with the cliques $\Delta _{q+1}^{q+2}\cup \{t_1^q\}$ and $\Delta _{q+2}^{q+1}\cup \{t^q_2\}$, if $|T|=5$, say $|T^{q-i+3}_{i}|=1$ for some $i\in \{1,2\}$, then add the clique $T_i$ and replace the cliques $\Delta _{q}^{q+i}$ and $\Delta _{q-i+3}^{q+i}$ with the cliques $(\Delta _{q}^{q+i}\cup \{t^{q-i+3}_1\})\setminus \{t^{q}_2\}$ and $(\Delta _{q-i+3}^{q+i}\cup \{t^{q}_2\})\setminus \{t^{q-i+3}_1\}$, and if $|T|=6$, then add the cliques $T_1$ and $T_2$, thereby obtaining a clique covering of size $12+|T_2^{q+1}|+|T_1^{q+2}|+|R|=n-1$, a contradiction. This proves (1).\vsp
	
	(2) \textit{$G[T]$ is of type $3$.}\vsp
	
	Suppose not, by Lemmas~\ref{lem:both clique} and \ref{lem:G[T] structure |I_T|=3}, $G[T]$ is of type $2$ at $q$, for some $q\in \{1,2,3\}$, and w.l.o.g. we may assume that $q=1$. Let $X=T^1\setminus \{t^1\}$. By (1), $|S|\geq 1$, say $s_{p'}^p\in S$, for some $(p,p')\in J$. Note that the cliques $\mathscr{O}'(G)=\mathscr{O}(G)\cup \{T^1\cup R^3,(T\setminus \{t^1\})\cup R^3,T_3\cup R\}$ cover all the edges in $E(G)\setminus E(T,R)$.
	If either $p=1$ or $p'=1$, then by \pref{lem:one is complete to R}~(i) (applying to $(i,j)=(p,p')$), either $t^1$ or $\{t^2,t^3\}$ is complete to $R$ (and the other one is anticomplete to $R$), and so  add the clique $\{t^1\}\cup R$ or $\{t^2,t^3\}\cup R$ as well as the cliques in $\mathscr{N}[X;R]$ to $\mathscr{O}'(G)$. Also, if $S_1=S^1=\emptyset$,  then in $\mathscr{O}'(G)$, for every $i\in \{2,3\}$, replace the clique $X_1^i$ and $X_{5-i}^i$ with the clique $X_1^i\cup N(t^i,R\setminus R_1)$ and $X_{5-i}^i\cup N(t^i,R_1)$ and also add the cliques in $\mathscr{N}[T^1;R]$. This yields a clique covering for $G$ of size $16+|X|=n-|S|-|R|+4$. Thus, $|S|=|S_{p'}^p|=1$ and $|R|=3$, say $R_i=\{r_i\}$, $i=1,2,3$. Let $\{p,p',p''\}=\{1,2,3\}$. Now, in $\mathscr{O}'(G)$, replace the cliques $Z_p^{p'},Z_p^{p''},Z_{p'}^{p},Z_{p'}^{p''},Z_{p''}^p$ and $Z_{p''}^{p'}$ with the cliques $Z_p^{p'}\cup N(r_p,T^{p''}),Z_p^{p''}\cup N(r_p,T^{p}\cup T^{p'}),Z_{p'}^{p}\cup N(r_p,T^{p''}),Z_{p'}^{p''}\cup N(r_p,T^{p}\cup T^{p'}),Z_{p''}^p\cup N(r_{p''},T\setminus T^p))\setminus S_{p'}^p$ and $Z_{p''}^{p'}\cup S_{p'}^p\cup N(r_{p''},T^p)$, thereby comprising a clique covering for $G$ of size $15\leq n-1$, a contradiction. This proves (2).\vsp

	By (2), $G[T]$ is of type $3$ at $q$, for some $q\in \{1,2,3\}$, with $T_1=\{t_1^q,t_1^{q+1}\}\cup T_1^{q+2}$ and $T_2=\{t_2^q,t_2^{q+2}\}\cup T_2^{q+1}$ as in the definition. By (1), $s_{p'}^p\in S$, for some $(p,p')\in J$, and since there is vertex in $T^p$ nonadjacent to a vertex in $T^{p'}$, by \pref{lem:one is complete to R}(i) (applying to $(i,j)=(p,p')$) and also \pref{lem:T,R Facts and Tools}~(vi), there exists $i_0\in \{1,2\}$ such that $T_{i_0}$ is complete to $R$ $T_{3-i_0}$ is anticomplete to $R$. Now, the collection $\mathscr{O}(G)\cup \{T_{i_0}\cup R,T_1\cup R^3,T_2\cup R^3,T_3\cup R\}$ is a clique covering for $G$ of size $16=n-|S|-|T|-|R|+7$, which by (1) is at most $n-1$,  a contradiction. This proves \pref{thm:3,3}.
\end{proof}

\subsection{Proof of \pref{thm:noap}}\label{sub:2,2}

In this subsection, using \pref{thm:3,3}, we complete the proof of \pref{thm:noap}. We need the following lemma. As usual, we omit the term $ \rho $ within the proofs.

\begin{lem}\label{lem:2,2with a clique}
	Let $G$ be a counterexample to \pref{thm:noap}. Then for every rotator $\rho $ of $\overline{G}$,  both $G[T(\rho)]$ and $G[R(\rho)]$ are of type $1$.
\end{lem}
\begin{proof}
	By \pref{thm:3,3}, we have $|I_T|=|I_R|=2$. On the contrary, by \pref{lem:G[T] structure |I_T|=2} and w.l.o.g. let $R$ be a clique.  By symmetry, suppose that $I_T=\{1,2\}$ and $I_R=\{1,p\}$, for some $p\in \{2,3\}$. By \pref{lem:both clique}, $T$ is not a clique, and thus by \pref{lem:G[T] structure |I_T|=2}, $G[T]$ is of type $1$, with $t^1\in T^1, t^2\in T^2$, $U_1\subseteq T^2$ and $U_2\subseteq T^1$  as in the definition. Note that by \pref{lem:square-forcer}~(i) (applying to $(i,j,k)=(1,p,5-p)$), either $s_2^1$ or $s_1^p$ belongs to $S$. Also, if $s_2^1\in S$ or $p=2$, then apply \pref{lem:one is complete to R}~(i) for $\{i,j\}=\{1,2\}$, and if $s_1^p\in S$ and $p=3$, then apply \pref{lem:one is complete to R}~(ii) for $(l,l')=(p,1)$ to deduce that there is a unique $i_0\in \{1,2\}$ such that $t^{i_0}$ is complete to $R$ and  $t^{3-i_0}$ is anticomplete to $R$.\vsp
	
	(1) \textit{If $i_0=2$, then $s_2^1\in S$, $s_3^1, s^p_1\notin S $ and $G[T]$ is of type $ 1.1 $ at $1$.} \vsp
	
	Since $i_0=2$, $t^1$ is anticomplete to $R_1$. Choose $r_1\in R_1$. It can be checked that $\overline{G}$ induces a rotator on  $\rho'=(s_1^1,s^2_2,s^3_3;r_1,t_3^2,t_3^3; t^1, r^3_2,r^3_3)$. Note that $r^3_1\in T^1(\rho')$, $t^1_3\in R_1(\rho')$, $R_p\subseteq R_p(\rho')$ and $S_1^p\subseteq R_{5-p}(\rho')$. Now, if $s_1^p\in S$, then $|I_R(\rho')|=3$ which contradicts \pref{thm:3,3}. Therefore,  $s_1^p\not\in S$ and thus $s_2^1\in S$. Also, $s_2^1\in T^3(\rho')$ and $S^1_3\cup U_1 \subseteq T^2(\rho')$. Now, if $S^1_3\cup U_1\neq \emptyset$, then $|I_T(\rho')|=3$ which again contradicts \pref{thm:3,3}. Hence, $S_3^1=U^1=\emptyset$ and so $G[T]$ is of type $1.1$ at $1$. This proves (1). \vsp
	
	(2) \textit{Let $G[T]$ be of type $1.1$ at $i$, for some $i\in \{1,2\}$. Then $|S|\geq 2$. Also, if $i_0=1$, then $S\neq \{s_1^2,s_1^{3}\}$ and if $i_0=2$, then $S\neq \{s_2^1,s_1^{5-p}\}$.}\vsp
	
	Suppose not, assume that either $|S|=|S_2^1|=1$, or $|S|=|S_1^p|=1$, or $i_0=1$ and $S\neq \{s_1^2,s_1^{3}\}$ or $i_0=2$ and $S\neq \{s_2^1,s_1^{5-p}\}$.

	and if the third assertion does not hold, then by (1), $i=1$. Considering $\mathscr{P}(G)$, if $s_i^3\notin S$, then replace the clique $\Delta _{3-i}^3$ with the clique $\Delta_{3-i}^3\cup U_{3-i}$. Otherwise, add the clique $T\setminus \{t^i\}$, there by covering all edges in $E(T^1,T^2)$. Also, if $S_{5-p}=\emptyset$ (i.e. the case that either $p=2$ or $p=3$ and $s_2^1\notin S$), then merge the pair $(\nabla _1^{5-p},\nabla _p^{5-p})$. Otherwise (i.e. the case that $p=3$ and $s_2^1\in S$), then replace the clique $\nabla _1^2$ with the clique $\nabla _1^2\cup R_p$, there by covering all edges in $E(R_1,R_p)$. Thus, it remains to cover the edges in $E(T,R)$. For this purpose, add the cliques in $\mathscr{N}[U^{i_0};R]$, if either $S^{i_0}=\emptyset$ or $i_0\notin I_R$, then replace the clique $\Omega ^{i_0}$ with the cliques in $\mathscr{N}[T^{i_0};R\cup \{t_3^{i_0}\}]$, and if either $S^{i_0}\neq \emptyset$ and $i_0\in I_R$, then replace the cliques $\Omega ^{i_0}$ and $\Omega _{i_0}$ with the cliques in $\mathscr{N}[T^{i_0};(\Omega ^{i_0}\cup R)\setminus (T^{i_0}\cup R_{i_0})]$ and $\mathscr{N}[R_{i_0};(\Omega _{i_0}\cup T^{i_0})\setminus R_{i_0}]$.
	We leave the reader to check that this procedure yields a clique covering for $G$ of size at most $n-1$, a contradiction. This proves (2).\vsp 
	
	If $G[T]$ is of type $1.1$ at $i$, for some $i\in \{1,2\}$, then define $\mathscr{O}'(G)=\mathscr{O}(G)\cup \{T^i\cup R^3, (T\setminus \{t^i\})\cup R^3,T_3\cup R\}$, which is a collection of cliques of $G$ covering all the edges in $E(G)\setminus E(T,R)$.
	Therefore, by (1), $\mathscr{O}'(G)\cup \mathscr{N}[T\setminus \{t^{3-i_0}\};R]$ is a clique covering of size $14+|T|=n-|S|-|R|+5$ which is at most $n-1$, whenever $|S|\geq 4$. Hence, by (2), the following statement holds. \vsp
	
	(3) \textit{If $G[T]$ is of type $1.1$ at $i$, for some $i\in \{1,2\}$, then $2\leq |S|\leq 3$.}\vsp
	
	Now, we deduce more information about $G$. \vsp
	
	(4) \textit{We have $i_0=1$, $p=3$ and $|S_2^1\cup S_1^2|\leq 1$.}\vsp
	
	Suppose not, we may assume that $i_0=2$ (indeed, if $i_0=1$ and either $p=2$ or both $s_1^2,s_2^1\in S$, then in order to assume that $i_0=2$,  we may consider the decompositions corresponding to the rotators $(s_2^2,s_1^1,s_3^3;t_3^2,t_3^1,t_3^3;r_2^3,r_1^3,r_3^3)$ or $(s_2^1,s_1^2,s_3^3;t_3^2,t_3^1,t_3^3;r_1^3,r_2^3,r_3^3)$, respectively). Thus by (1), $s_2^1\in S$, $s_3^1, s^p_1\notin S $ and $G[T]$ is of type $ 1.1 $ at $1$. Now, considering $\mathscr{O}'(G)$, replace the cliques $X_1^2$ and $X_3^2$ with the cliques $X_1^2\cup R_p$ and $X_3^2\cup R_1$ as well as adding the cliques in $\mathscr{N}[U_2;R]$ to cover the edges in $E(T,R)$, thereby obtaining a clique covering $\mathscr{C}$ for $G$ of size $15+|U_2|=n-|S|-|R|+4$. Thus by (3), $|S|=|R|=2$ and since $i_0=2$ and $s_2^1\in S$, by (2),  we have $s_1^{5-p}\notin S$ (i.e. $S_1=\emptyset$).  Now, remove the clique $T^1\cup R^3$ from $\mathscr{C}$ and replace the cliques $Z_1^{5-p},Z_p^{5-p}$ and $Z_{5-p}^p$ with the cliques $R\cup S^{5-p}\cup \{s_{5-p}^{5-p}\},\{s_{5-p}^{5-p},r_1^3,r_p^3\}\cup S^{5-p}\cup T^1$ and $Z_{5-p}^p\cup T^1$ to cover the edges in $E(T^1,R^3)$, thereby obtaining clique covering for $G$ of size $14+|U_2|=n-1$, a contradiction. This proves (4).\vsp
	
	(5) \textit{If $G[T]$ is of type $ 1.1 $ at $i$, for some $i\in\{1,2\}$, then we have $s_1^3\in S$.} \vsp
	
	By (4), $i_0=1$ and $p=3$. Note that  if (5) does not hold, then $s^1_2\in S$ and again by (4), $s_1^2\notin S $. Now, in the collection $\mathscr{O'}(G)$, if $i=1$ (resp. $i=2$), then remove the clique $T^{1}\cup R^3$ (resp. $(T\setminus \{t^{2}\})\cup R^3$), replace the cliques $Z_1^3$, $Z_2^3$ and $Z_3^2$ with the cliques 
	$Z_1^3\cup \{t^1\}$, $Z_2^3\cup \{t^1\}$ and $Z_3^2\cup \{t^1\}$ to cover the edges in $E(\{t^1\},R')$ and add the cliques $\mathscr{N}[U_{3-i};R\cup \{t^1\}]$ to cover the edges in $E(T\setminus \{t^1\},R)\cup E(T^1,T^2)$, thereby obtaining a clique covering for $G$ of size $14+|U_{3-i}|=n-|S|-|R|+3$, which by (3) is at most $n-1$. This proves (5). \vsp
	
	(6) \textit{If $G[T]$ is of type $1.1$ at $i$, for some $i\in \{1,2\}$, then $S_2^1\cup S_1^2\cup S_2^3\neq \emptyset$.}\vsp
	
	Suppose not, by (4), we have $i_0=1$ and $p=3$ and by (5), $s_1^3\in S$. Considering $\mathscr{O'}(G)$, if $i=1$ (resp. $i=2$), remove the clique $T^1\cup R^3$ (resp. $(T\setminus \{t^2\})\cup R^3$), replace the cliques $Z_1^2,Z_3^2,Z_2^1$ and $Z_2^3$ with the cliques $Z_1^2\cup \{t^1\}, Z_3^2\cup \{t^1\},(Z_2^1\cup S_1^3)\setminus S_3^1$ and $(Z_2^3\cup S_3^1\cup \{t^1\})\setminus S_1^3$ to cover the edges in $E(\{t^1\},R')$. Also, add the cliques in $\mathscr{N}[U_{3-i};R\cup \{t^1\}]$ to cover the edges in $E(T\setminus \{t^1\},R)\cup E(T^1,T^2)$, thereby obtaining a clique covering for $G$ of size $14+|U_{3-i}|=n-|S|-|R|+3$, which by (2) is at most $n-1$. This proves (6). \vsp
	
	(7) \textit{If $G[T]$ is of type $1.1$ at $i$, for some $i\in \{1,2\}$, then $\{s_2^j,s_1^3\}\subseteq S\subseteq \{s_2^j,s_1^2,s_1^3\}$, for some $j\in \{1,3\}$.}\vsp
	
	By (4), $i_0=1$ and $p=3$, and by (5), $s_1^3\in S$. If $S_2=\emptyset$, then by (6), $s_1^2\in S$. Considering $\mathscr{O}'(G)$, in the case that $i=1$ (resp. $i=2$), remove the clique $(T\setminus \{t^1\})\cup R^3$ (resp. $T^2\cup R^3$) and replace the cliques $Z_2^1$ and $Z_2^3$ with the cliques $Z_2^1\cup \{t^2,r_3^3\}$ and $Z_2^3\cup \{t^2,r_1^3\}$ to cover the edges in $E(\{t^2\},R^3)$. Also, let $U'_2\subseteq U_2$ be the set of vertices of $U_2$ complete to $R$. Then $U_2\setminus U'_2$ is anticomplete to $R$ (indeed, if $t\in U_2$ is nonadjacent to some $r\in R_j$, $j\in \{1,3\}$, then, $t$ is anticomplete to $R_{4-j}$, since otherwise, $\{r',s_j^1,t,r\}$ would be a claw, for some $r'\in R_{4-j}$). Thus adding the clique $\{t^1\}\cup U'_2\cup R$ as well as the cliques in $\mathscr{N}[U_1;R]\cup \mathscr{N}[U_2;\{t^2\}]$ to the resulting family to  cover the edges in $E(T,R)\cup E(T^1,T^2)$, thereby providing a clique covering for $G$ of size $13+|T|=n-|S|-|R|+4$. Thus, $|S|=2$, i.e. $S=\{s_1^2,s_1^3\}$, and since $i_0=1$, this contradicts (2). Thus, $S_2\neq \emptyset$ and so $\{s_2^j,s_1^3\}\subseteq S$, for some $j\in \{1,3\}$.\\
	Also, if $s_1^2\notin S$, then in $\mathscr{O'}(G)$, replacing the cliques $Z_1^2$ and $Z_3^2$ with the cliques $Z_1^2\cup \{t^1\}$ and $Z_3^2\cup \{t^1\}$ and adding the cliques in $\mathscr{N}[U_{3-i};R]$ to cover the edges in $E(T,R)$, yield a clique covering for $G$ of size $13+|T|=n-|S|-|R|+4$. Hence, $|S|=2$. This, together with (3), proves (7).\vsp
	
	(8) \textit{$G[T]$ is of type $1.2$. }\vsp
	
	Suppose not, let $G[T]$ be of type $1.1$ at $i$ for some $i\in \{1,2\}$. By (4), $i_0=1$ and $p=3$. Also by (7), either $S=\{s_2^j,s_1^3\}$ or $S=\{s_2^j,s_1^2,s_1^3\}$, for some $j\in \{1,3\}$.  If $S=\{s_2^j,s_1^3\}$, for some $j\in \{1,3\}$, then considering $\mathscr{O}'(G)$ in the case that $i=1$ (resp. $i=2$), remove the cliques $T^1\cup R$ (resp. $(T\setminus \{t^2\})\cup R$), replace the cliques $Z_1^2$ and $Z_3^2$ with the cliques $Z_1^2\cup \{t^1\}$ and $Z_3^2\cup \{t^1\}$ and add the cliques in $\mathscr{N}[U_{3-i};R\cup \{t^1\}]$ to cover the edges in $E(T,R)\cup E(\{t^1\},\{r_1^3,r_3^3\})\cup E(T^1,T^2)$. Also, in order to cover the edge $t^1r_2^3$, if $j=1$, then replace the cliques $X_3^1$, $X_3^2$ and $Z_2^3$ with the cliques $(X_3^1\cup \{r_2^3\})\setminus \{t_3^1\}$, $(X_3^2\cup \{t_3^1\})\setminus T^2$ and $Z_2^3\cup T^2$ and if $j=3$, then replace the cliques $Z_2^1$ and $Z_2^3$ with the cliques $Z_2^1\cup S_1^3$ and $(Z_2^3\cup \{t^1\})\setminus S_1^3$. This yields a clique covering for $G$ of size $12+|T|=n-|R|+1\leq n-1$, a contradiction. Therefore, assume that $S=\{s_2^j,s_1^2,s_1^3\}$, for some $j\in \{1,3\}$. In this case, note that $U_2$ is anticomplete to $R$ (indeed, $U_2$ is anticomplete to $R_j$, since otherwise $\{t,s_2^j,t^2,r\}$ would be a claw for some $t\in U_2$, $r\in R_j$, and consequently, $U_2$ is anticomplete to $R_{4-j}$, since otherwise $\{r'',s_1^j,t,r'\}$ would be a claw for every $r'\in R_j$ and some $t\in U_2$, $r''\in R_{4-j}$). Now, in $\mathscr{O}'(G)$, replace the clique $Z_3^1$ and $Z_3^2$ with the cliques $(Z_3^1\cup S_1^2)\setminus S_2^1$ and $(Z_3^2\cup S_2^1\cup \{t^1\})\setminus S_1^2$ to cover the edges in $E(\{t^1\},R_3)$, replace the clique $X_3^1$ with the clique $(X_3^1\cup R_1)\setminus U_2$ to cover the edges in $E(\{t^1\},R_1)$, and add the cliques in $\mathscr{N}[U_1;R]\cup \mathscr{N}[U_2;\{s_3^3\}]$, thereby obtaining a clique covering for $G$ of size $13+|T|=n-|R|+1\leq n-1$, again a contradiction. This proves (8).\vsp

	By (8), $G[T]$ is of type $1.2$, and $\mathscr{C}=\mathscr{O}(G)\cup \{T_3\cup R\} \cup  \mathscr{N}[T^1;T^2\cup R^3]\cup \mathscr{N}[R;T]$ is a clique covering for $G$ of size $13+|T^1|+|R|=n-|S|-|T^2|+4$. Thus, $1\leq |S|\leq 2$. Now, we complete the proof through the following three cases.\\
	If $|S|=2$ and $S^3=\emptyset$, then remove the clique $Z_1^3$ from $\mathscr{C}$ and replace the cliques in $\mathscr{N}[R_1;T]$ with the cliques in $\mathscr{N}[R_1;T\cup \{s_3^3,r^3_1\}]$, thereby obtaining a clique covering for $G$ of size $12+|T^1|+|R|=n-|S|-|T^2|+3\leq n-1$, a contradiction.\\
	Also, if $|S|=|S^3|=2$ and $p=3$, then every vertex in $T^1$ is either complete or anticomplete to $R$ (indeed, if a vertex $t\in T^1$ is nonadjacent to some vertex $r\in R_i$ for some $i\in I_R$, then $t$ is anticomplete to $R_{4-i}$, since otherwise $\{r',s_1^j,t,r\}$ would be a claw for every vertex $r'\in R_{4-j}$ adjacent to $t$). Now, let $T'^1\subseteq T^1$ be the set of vertices in $T^1$ which are complete to $R$. In $\mathscr{C}$, remove the cliques in $\mathscr{N}[R,T]$ and in order to cover the edges in $E(T,R)$, add the cliques $T'^1\cup R$, replace the clique $Z_1^2$ with the clique $Z_1^2\cup S_2^3$ and also replace the cliques $Z_3^1$ and $Z_1^3$ with the cliques in $\mathscr{N}[R_3;(Z_3^1\cup T^2)\setminus R_3]$ and the cliques in $\mathscr{N}[R_1;(Z_1^3\cup T^2)\setminus S_2^3\cup R_3]$, thereby obtaining a clique covering for $G$ of size $12+|T^1|+|R|=n-|S|-|T^2|+3\leq n-1$ which is impossible.\\
	Finally assume that either $|S|=1$, or $|S|=2$ and $|S^3|=1$, or $|S|=|S^3|=2$ and $p=2$. Then, considering $\mathscr{P}(G)$, in order to cover the edges $E(T^1,T^2)$, for every $i\in \{1,2\}$ such that $s^3_{3-i}\notin S$, remove the clique $\Delta _{i}^3$, if $|S^3|=1$, then choose $j\in \{1,2\}$ satisfying $s_j^3\notin S$ and add the cliques in $\mathscr{N}[T^j;\Delta _{3-j}^3]$ and if $|S^3|=2$, then add the cliques in $\mathscr{N}[T^1;T^2]$.  Also, in order to cover the edges in $E(R_1,R_p)$, if $S_{5-p}=\emptyset$, then merge the pair $(\nabla _1^{5-p},\nabla _{p}^{5-p})$, if $|S_{5-p}|=1$, then choose $i\in \{1,p\}$ satisfying $s_{5-p}^{p-i+1}\notin S$ and replace the clique $\nabla_ {i}^{5-p}$ with the clique $\nabla_{i}^{5-p}\cup R_{p-i+1}$, and if $|S_{5-p}|=2$, then add the clique $R$. Finally, for every $i\in I_R$, in order to cover the edges $E(T,R_i)$, if $S_i=\emptyset$ or $i=3$, then replace the clique $\Omega _i$ and add clique $\mathscr{N}[R_i;(T\cup \Omega _i)\setminus R_i]$ and if $S_i\neq \emptyset$ and $i\neq 3$, then add the cliques in $\mathscr{N}[R_i;T]$. We leave it to the reader to check that this procedure produces a clique covering for $G$ of size at most $n-1$, a contradiction. This proves \pref{lem:2,2with a clique}.
\end{proof}

Now, we are ready to establish \pref{thm:noap}.

\begin{proof}[{\rm \textbf{Proof of \pref{thm:noap}.}}]
		Let $G$ be a counterexample. Due to \pref{thm:Inc-rotator-Exc-sf}, $\overline{G}$ contains an induced rotator $\rho$ and no square-forcer. Hence, by \pref{thm:3,3}, we have $|I_T(\rho)|=|I_R(\rho)|=2$, say  $I_T(\rho)=\{1,2\}$, $I_R(\rho)=\{1,p\}$, $p\in \{2,3\}$. Also, by \pref{lem:2,2with a clique}, both $G[T(\rho)]$ and $G[R(\rho)]$ are of type $1$ with $t^1\in T^1$, $U_2\subseteq T^1$, $t^2\in T^2$, $U_1\subseteq T^2$, $r_1\in R_1$, $U^p\subseteq R_1$, $r_p\in R_p$ and $U^1\subseteq R_p$ as in the definition.
	By \pref{lem:square-forcer}~(i) (applying to $(i,j,k)=(1,p,5-p)$), $s_{k'}^k\in S$, for some $(k,k')\in \{(1,2),(p,1)\}$. For every $i\in \{1,2\}$ and every $t\in U_i$, let $A_i(t)=\{s_{k'}^k,t^i,r_k,t\}$ and $B_i(t)=\{s_1^1,t^i,r_1,t\}$. Also, for every $j\in \{1,p\}$ and every $r\in U^j$, let $C_j(r)=\{s_{k'}^k,t^{k'},r_j,r\}$ and $D_j(r)=\{s_1^1,t^1,r_j,r\}$. 
	
	First, assume that $t^1$ is nonadjacent to $r_p$. Then, by \pref{lem:T,R Facts and Tools}~(vi), $ t^1 $ is adjacent to $ r_1 $ and $t^2$ is adjacent to $r_p$ and nonadjacent to $ r_1 $. Since $A_1(t), t\in U_1$, is not a claw, $r_1$ is complete to $U_1$ and $r_p$ is anticomplete to $U_1$. Also, since $A_2(t), t\in U_2$, is not a claw,  $r_1$ is anticomplete to $U_2$ and $r_p$ is complete to $U_2$. Further, since $C_1(r), r\in U^1$, is not a claw, $t^2$ is anticomplete to $U^1$ and $t^1$ is complete to $U^1$ and since  $C_p(r), r\in U^p$, is not a claw, $t^1$ is anticomplete to $U^p$ and $t^2$ is complete to $U^p$. Moreover, since $(s_1^1,r_1,U_2)$ is a quasi-triad of $G$ and $U_1$ is complete to $\{s_1^1,r_1\}$, by \pref{lem:T,R Facts and Tools}~(i), we have $|U_2|\leq 1$ and $U_1$ is anticomplete to $U_2$. Similarly, since $(s_1^1,t^1,U^p)$ is a quasi-triad of $G$ and $U^1$ is complete to $\{s_1^1,t^1\}$, we have $|U^p|\leq 1$ and $U^1$ is anticomplete to $U^p$. Also, $U_2$ is anticomplete to $U^1$ (otherwise, for $ t\in U_2 $ adjacent to $ r\in U^1 $, $(B_2(t)\cup \{r\})\setminus \{t^2\}$ would be a claw) and complete to $U^p$, (otherwise, for $ t\in U_2 $ nonadjacent to $ r\in U^p $, $(B_2(t)\cup \{r\})\setminus \{r_1\}$ would be a claw). Finally, $U_1$ is anticomplete to $U^p$ (otherwise, for $ t\in U_1 $ adjacent to $r\in U^p$, $(D_p(r)\cup \{t\})\setminus \{r_p\}$ would be a claw). Now, if either $U^1=\emptyset$, or $U_1=\emptyset$, or $U^p\neq \emptyset$, or $U_2\neq \emptyset$, then by \pref{lem:T,R Facts and Tools}~(vi), $U^1$ is complete to $U_1$ and so $G$ is ISA, which contradicts \pref{lem:ccschlafli}. Thus, both $U^1$ and $U_1$ are nonempty, say $0<|U_1|\leq |U^1|$, and $U^p=U_2=\emptyset$. Also, if either $p=2$ or $p=3$ and $s_2^3\in S$, then $U^1$ is complete to $U_1$ (otherwise, for $ t\in U_1 $ nonadjacent to $ r\in U^1 $, $\{s_2^p,t^1,t,r\}$ would be a claw) and so $G$ is ISA, which contradicts \pref{lem:ccschlafli}. Hence, $p=3$ and $s_2^3\notin S$. Subsequently, in the collection $\mathscr{O}'(G)=\mathscr{O}(G)\cup \{ T^2\cup R^3, (T\setminus \{t^2\})\cup R^3,T_3\cup R_1, T_3\cup (R\setminus \{r_1\})\}$ (the cliques in which cover all the edges in $E(G)\setminus E(T,R)$), add the cliques in $\mathscr{N}[T^2;R]$, replace the clique $X_2^1$ with the clique $X_2^1\cup U^1$ to cover the edges in $E(T^1,R_3)$. Also, in order to cover the edge $t^1r_1$, if both $S^1$ and $S_1$ have cardinality two, then add the clique $\{t^1,r_1\}$, if $|S^1|\leq 1$ (resp. $|S_1|\leq 1$), say $s^1_i\notin S$ (resp. $s_1^i\notin S$), for some $i\in \{2,3\}$, then replace the clique $X_i^1$  or possibly $X_i^1\cup U^1$ (resp. $Z_1^i$) with clique $X_i^1\cup \{r_1\}$ or possibly $X_i^1\cup U^1\cup \{r_1\}$ (resp. $Z_1^i\cup \{t^1\}$). It is easy to check that in the case that either $|S^1|\leq 1$ or $|S_1|\leq 1$, the resulting family $\mathscr{C}_1$ is a clique covering for $G$ of size $16+|T^2|=n-|S|-|U^1|+4$, and in the case that both $S^1$ and $S_1$ have cardinality two, it is a clique covering for $G$ of size $17+|T^2|=n-|S|-|U^1|+5$. Hence, $|S|\leq 4$ and if $|S|\geq 3$, then $|U_1|=|U^1|=1$.
	First, assume that $|S|\geq 3$. If $U_1$ is complete to $U^1$, then $G$ is ISA, a contradiction with \pref{lem:ccschlafli}, and if $U_1$ is anticomplete to $U^1$, then remove the single clique in $\mathscr{N}[U_1;R]$ from the $\mathscr{C}_1$ and replace the clique $Z_1^3$ or possibly $Z_1^3\cup \{t^1\}$ with the clique $Z_1^3\cup U_1\cup \{t^1\}$ to obtain a clique covering for $G$ of size  
	$n-1$, a contradiction. Therefore, $|S|\leq 2$. Now, considering $\mathscr{P}(G)$, 
	replace the cliques $\Delta^3_1$ and $\nabla^2_1$ with the cliques $\Delta_1^3\cup U_1$ and $\nabla^2_1\cup U^1$ to cover the edges in $E(T^1,T^2)\cup E(R_1,R_3)$, replace the clique $\Omega ^1$ with the clique $\Omega ^1\cup U^1$ to cover the edges in $E(T^1,R_3)$, replace the clique $\Omega _1$ with the clique $\Omega _1\cup U_1$ to cover the edges in $E(T^2, R_1)$, replace the clique $\Omega ^2$ with the cliques in $\mathscr{N}[T^2;(R_3\cup \Omega^2)\setminus T^2]$ to cover the edges in $E(T^2,R_3)$. Finally, add the clique $T^1\cup R_1$ to cover the edges in $E(T^1,R_1)$, thereby obtaining a clique covering for $G$ of size at most $12+|S|+|T^2|=n-|U^1|\leq n-1$, again a contradiction.
	
	Now, assume that $t^1$ is adjacent to $r_p$ and so by \pref{lem:T,R Facts and Tools}~(vi), $ t^1 $ is nonadjacent to $ r_1 $ and $t^2$ is adjacent to $r_1$ and nonadjacent to $ r_p $. Since $B_1(t),t\in U_1$, is not a claw, $r_1$ is anticomplete to $U_1$ and thus $r_p$ is complete to $U_1$. Also, since $B_2(t), t\in U_2$, is not a claw, $r_1$ is complete to $U_2$ and thus $r_p$ is anticomplete to $U_2$. Also, since $D_1(r), r\in U^1$, is not a claw, $t^1$ is anticomplete to $U^1$ and thus $t^2$ is complete to $U^1$ and since $D_p(r), r\in U^p$, is not a claw, $t^1$ is complete to $U^p$ and so $t^2$ is anticomplete to $U^p$. Moreover, since $(s_{k'}^k,r_k,U_{3-k'})$ is a quasi-triad of $G$ and $U_{k'}$ is complete to $\{s_{k'}^k,r_k\}$, by \pref{lem:T,R Facts and Tools}~(i), we have $|U_{3-k'}|\leq 1$ and $U_1$ is anticomplete to $U_2$. Similarly, since $(s_{k'}^k,t^{k'},U^{p-k+1})$ is a quasi-triad of $G$ and $U^k$ is complete to $\{s_{k'}^k,t^{k'}\}$, by \pref{lem:T,R Facts and Tools}~(i), we have $|U^{p-k+1}|\leq 1$ and $U^1$ is anticomplete to $U^p$. Also, $U_1$ is anticomplete to $U^1$ (otherwise, for $ t\in U_1 $ adjacent to $r\in U^1$, if $(k,k')=(1,2)$, then $(C_1(r)\cup \{t\})\setminus \{t^{k'}\}$ and if $(k,k')=(p,1)$, then $(C_1(r)\cup \{t\})\setminus \{r_1\}$ would be a claw). Finally, $U_2$ is anticomplete to $U^p$ (otherwise, for $ t\in U_2 $ adjacent to $ r\in U^p $, if $ (k,k')=(1,2) $, then $(C_p(r)\cup \{t\})\setminus \{r_p\}$ is a claw, and if $ (k,k')=(p,1) $, then  $(C_p(r)\cup \{t\})\setminus \{t^1\}$ would be a claw). Now, if either both $U^1$ and $U_1$ are nonempty, or both $U^p$ and $U_2$ are nonempty, or $U_1=U_2=\emptyset$, or $U^1=U^p=\emptyset$, then by \pref{lem:T,R Facts and Tools}~(vi), $U^1$ is complete to $U_1$ and so $G$ is ISA, which contradicts \pref{lem:ccschlafli}. Thus, assume that either $U_1=U^p=\emptyset$ and both $U^1$ and $U_2$ are nonempty, or $U^1=U_2=\emptyset$ and both $U^p$ and $U_1$ are nonempty, say the former occurs (otherwise, we may consider the rotator $(s_1^1,s_p^p,s_{5-p}^{5-p};r^3_1,r^3_p,r^3_{5-p};t_3^1,t_3^p,t_3^{5-p})$ of $\overline{G}$ as the initial rotator), say $0<|U_2|\leq |U^1|$. Also, note that $\rho'=(s_1^1,s_2^2,s_3^3;r_1,t^2_3,t^3_3; t^1,r^3_2,r^3_3)$ is a rotator of $\overline{G}$ and we have  $t^1_3\in R_1(\rho')$, $U^1\subseteq R_p(\rho')$ and $S_1^p\subseteq R_{5-p}(\rho')$. Thus, if $s_1^p\in S$, then $|I_R(\rho')|=3$,  which contradicts \pref{thm:3,3}. Therefore,  $s_1^p\not\in S$ and so by \pref{lem:square-forcer}~(i) (applying to $(i,j,k)=(1,p,5-p)$), $s_2^1\in S$. Also, $r^3_1\in T^1(\rho')$, $s_2^1\in T^3(\rho')$ and $S^1_3\subseteq T^2(\rho')$. Now, if $s_3^1\in S$ is not empty, then $|I_T(\rho')|=3$ which again contradicts \pref{thm:3,3}. Therefore, $s^1_3\notin S$. 
	
	Now, in the collection $\mathscr{O}''(G)=\mathscr{O}(G)\cup \{ T^1\cup R^3, (T\setminus \{t^1\})\cup R^3,T_3\cup R_p, T_3\cup (R\setminus \{r_p\})\}$ (the cliques in which cover all the edges in $E(G)\setminus E(T,R)$), add the cliques in $\mathscr{N}[T^1;R_p]$, replace the cliques $X_1^2$, $X_3^2$ and $Z_1^p$ with the cliques $X_1^2\cup U^1$, $X_3^2\cup \{r_1\}$ and $Z_1^p\cup U_2$, thereby covering the edges in $E(T,R)$. 
	The resulting family $\mathscr{C}_2$ is a clique covering for $G$ of size $17+|U_2|=n-|S|-|U^1|+4$. Thus, $|S|+|U^1|\leq 4$. If $|S|=3$, then $|U_2|=|U^1|=1$, in the case that $U_2$ is complete to $U^1$, $G$ is ISA, again a contradiction with \pref{lem:ccschlafli}, and in the case that $U_2$ is anticomplete to $U^1$,  removing the clique $U_2\in \mathscr{N}[t^1,R_p]$ from $\mathscr{C}_2$ implies that $\cc(G)\leq n-1$, a contradiction. Therefore, $|S|\leq 2$. Now, considering $\mathscr{P}(G)$, remove the clique $\Omega ^1$, add the cliques in $\mathscr{N}[T^1;(\Omega^1\cup R_p)\setminus T^1]$ to cover the edges in $E(T^1,R_p)$. Also, in order to cover the edges in $E(\{t^2\},R)\cup E(R_1,R_p)$, if $S^2=\emptyset$ or $p=3$, then replace the clique $\Omega ^2$ with the clique $\Omega ^2\cup U^1\cup \{r_1\}$ and if $S^2\neq \emptyset$ and $p=2$, then add the clique $U^1\cup \{t^2,r_1\}$, and in order to cover the edges in $E(\{t^2,r_1\},U_2)$, if $S_1=\emptyset$, then replace the clique $\Omega _1$ with the clique $\Omega _1\cup U_2\cup \{t^2\}$ and if $S_1\neq \emptyset$, then add the clique $U_2\cup \{t^2,r_1\}$. It is easy to see that these modifications lead to a clique covering for $G$ of size at most $12+|S|+|T^1|=n-|U^1|\leq n-1$, again a contradiction. This completes the proof of \pref{thm:noap}.
\end{proof}
\bibliographystyle{plain}

\appendix
\vspace{5mm}
\section{Basic classes of three-cliqued claw-free graphs} \label{app:3clique}
The following five classes of three-cliqued claw-free graphs are the building blocks in the structure theorem of all three-cliqued claw-free graphs.
\begin{itemize}
	\item[$\mathcal{TC}_1$] \textit{Line graphs.} Let $v_1, v_2, v_3$ be distinct pairwise nonadjacent vertices of a graph $H$, such
	that every edge of H is incident with one of $v_1, v_2, v_3$. Let $v_1, v_2, v_3$ be of degree at least
	three, and let all other vertices of $H$ be of degree at least one. Moreover, for all distinct $i, j \in
	\{1, 2, 3\}$, let there be at most one vertex different from $v_1, v_2, v_3$ that is adjacent to $v_i$ and
	not to $v_j$ in $H$.  Let $G$ be the line graph of $H$ and let $A,B,C$ be the sets of edges of $H$ incident with $v_1, v_2, v_3$, respectively. Then $(G,A,B,C)$ is a three-cliqued claw-free graph. 
	Moreover, let $F'$ be the set of all pairs $\{e,f\}$, where $e,f\in E(H)$ have a common endpoint in $H$ with degree exactly two and let  $F$ be a valid subset of $ F' $. 
	Define $\mathcal{TC}_1$ to be the class of all such three-cliqued graphs with the additional property that every vertex of $G$ is in a triad of $G\setminus F$.
	
	\item[$\mathcal{TC}_2$] \textit{Long circular interval graphs.} Let $G$ be a long circular interval graph with $\Sigma$ and $\mathcal{I}=\{I_1,\ldots, I_k\}$ as in the definition of long circular interval graph. By a line we mean either a subset $X\subseteq V(G)$ with $|X|\leq 1$, or a subset of
	some $I_i$ homeomorphic to the interval $[0,1]$, with both endpoints in $V (G)$. Let
	$L_1,L_2,L_3$ be pairwise disjoint lines with $V (G) \subseteq L_1\cup L_2\cup L_3$ and define $A=V(G)\cap L_1, B=V(G)\cap L_2 $ and $C=V(G)\cap L_3$. Then $(G,A,B,C) $ is a three-cliqued claw-free graph.
	Moreover, let $F'$ be the set of all pairs $\{u,v\}$ such that $u,v \in V (G)$ are distinct endpoints of $I_i$  for some $i$,  there is no $j\neq i$ for which $u,v\in I_j$ and $u,v$ are not in the same set $A$, $B$ or $C$. Also, let  $F$ be a subset of $ F' $. 
	We denote by $\mathcal{TC}_2$ the class of all such three-cliqued graphs with the additional property that every vertex of $G$ is in a triad of $G\setminus F$.
	
	\item[$\mathcal{TC}_3$] \textit{Near-antiprismatic graphs.} 
	Let $m \geq 2$ and construct a graph $H$ as follows. Its vertex set
	is disjoint union of three sets $\tilde{A},\tilde{B},\tilde{C}$, where $|\tilde{A}| = |\tilde{B}| = m + 1$ and $|\tilde{C}| = m$, say
	$\tilde{A} = \{a_0, a_1, \ldots, a_m\}$, $\tilde{B} = \{b_0, b_1,\ldots , b_m\}$ and $\tilde{C} = \{c_1,\ldots, c_m\}$. Adjacency is as follows.
	$\tilde{A},\tilde{B},\tilde{C}$ are three cliques. For $i, j\in\{0,\ldots,m\}$ with $(i, j )\neq (0, 0)$, let $a_i, b_j$ be adjacent if  and only if $i = j$, and for $ i\in\{1,\ldots,m\}$ and $j\in\{0,\ldots,m\}$, let $c_i$ be adjacent to $a_j, b_j$ if and only if
	$i \neq j \neq 0$. All other pairs not specified so far are nonadjacent. 
	Now let $X\subseteq \tilde{A}\cup \tilde{B}\cup \tilde{C} \setminus \{a_0, b_0\}$ with $|\tilde{C} \setminus X| \geq 2$. 
	Also, let $G=H\setminus X$, $A=\tilde{A}\setminus X$, $B=\tilde{B}\setminus X$ and $C=\tilde{C}\setminus X$.
	Then $(G,A,B,C)$ is a three-cliqued claw-free graph called near-antiprismatic. 
	Moreover, let $F'=\{\{a_i,b_i\}: 0\leq i\leq m\}\cup \{\{a_i,c_i\}: 1\leq i\leq m\}\cup \{\{b_i,c_i\}: 1\leq i\leq m\}$ and $F\subseteq F'$  such that
	\begin{itemize}
		\item $\{a_i,c_i\}\in F$ for at most one value of $ i\in\{1,\ldots , m\}$, and if so then $b_i\in X$,
		\item $\{b_i,c_i\}\in F$ for at most one value of $ i\in\{1,\ldots , m\}$, and if so then $a_i\in X$,
		\item $\{a_i,b_i\}\in F$ for at most one value of $ i\in\{1,\ldots , m\}$, and if so then $c_i\in X$. Moreover, $\{a_0,b_0\}$ may belong to $F$ with no restriction.
	\end{itemize}
	We denote by $\mathcal{TC}_3$ the class of all such three-cliqued graphs with the additional property that every vertex of $G$ is in a triad of $G\setminus F$.
	
	\item[$\mathcal{TC}_4$] \textit{Antiprismatic graphs.}
	Let $G$ be an antiprismatic graph and let $A,B,C$ be a partition of $V(G)$ into three cliques. 
	We denote by $\mathcal{TC}_4$ the class of all such three-cliqued graphs. Note that in this case $G$ may have vertices that are in no triad. Also, let $F$ be a valid set of changeable pairs of $G$.
	
	\item[$\mathcal{TC}_5$] \textit{Sporadic exceptions.}
	\begin{itemize}
		\item 
		Let $H$ be the graph with vertex set $\{v_1,\ldots, v_8\}$ and adjacency as follows: $v_i, v_j$ are
		adjacent for $1\leq i < j \leq 6$ with $j-i \leq 2$; $v_1, v_6, v_7$ are pairwise adjacent, and $v_7, v_8$ are adjacent. All other pairs not specified so far are nonadjacent. Let $X\subseteq \{v_3, v_4\}$, $A = \{v_1, v_2, v_3\}\setminus X$, $B = \{v_4, v_5, v_6\}\setminus X$, ${C} = \{v_7, v_8\}$ and $G=H\setminus X$; then $(G,A,B,C)$ is a three-cliqued claw-free graph, and all its vertices are in triads. Moreover, let $F'=\{\{v_1,v_4\}, \{v_3,v_6\}, \{v_2,v_5\}\}$ and $F\subseteq F'$ such that $\{v_1,v_4\},\{v_3,v_6\}\in F$.
		
		\item Let $H $ be the graph with vertex set $\{v_1, \ldots , v_9\}$, and adjacency as follows: the sets $A = \{v_1, v_2\}$, $\tilde{B} = \{v_3, v_4, v_5, v_6, v_9\}$ and $C = \{v_7, v_8\}$ are cliques; $v_1,v_8,v_9$ are pairwise adjacent; $v_2$ is adjacent to $v_3$ and $v_6$ is adjacent to $v_7$; the adjacency between the pairs $v_2v_4$ and $v_5v_7$ is arbitrary. All other pairs not specified so far are nonadjacent. Let $X \subseteq \{v_3, v_4, v_5, v_6\}$. Also, let $G=H\setminus X$ and $B=\tilde{B}\setminus X$. Then $(G,A,B,C)$ is a three-cliqued claw-free graph.
		Moreover, let $F'=\{\{v_2,v_4\},\{v_5,v_7\},\{v_1,v_3\}, \{v_6,v_8\}\}$ and $F\subseteq F'$ such that $\{v_1,v_3\}, \{v_6,v_8\}\in F$ and
		\begin{itemize}
			\item if $v_3\in X$, then either $v_2,v_4$ are adjacent, or $\{v_2,v_4\}\in F$,
			\item if $v_6\in X$, then either $v_5,v_7$ are adjacent, or $\{v_5,v_7\}\in F$,
			\item if $v_4, v_5\not\in X$, then for every $\{x,y\}\in\{\{v_2,v_4\},\{v_5,v_7\}\}$, either $x$ is adjacent to $y$, or $\{x,y\}\in F$.
		\end{itemize}
	\end{itemize}
	We denote by $\mathcal{TC}_5$ the class of such three-cliqued graphs  (given by one of these two constructions) with the additional property that every vertex of $G$ is in a triad  of $G\setminus F$.
\end{itemize}
\section{Basic classes of orientable antiprismatic graphs} \label{app:oap}
The following are the basic classes of orientable antiprismatic graphs which are appeared in the characterization of antiprismatic graphs. Note that they are recalled from \cite{seymour1}, where the description of orientable prismatic graphs is given (the graphs whose complements are  antiprismatic). Thus, we reformulate all the definitions in terms of the complements. \\

\textbf{The complement of mantled $L(K_{3,3})$.}
Let $G$ be a graph whose vertex set is the union of seven disjoint sets 
$W=\{a_j^i : 1\leq i,j\leq 3\}, V^1, V^2, V^3, V_1,V_2,V_3,$
with adjacency as follows. For $i,j,i',j'\in\{1,2,3\}$, $a_j^i$ and $a_{j'}^{i'}$ are adjacent if and only if $i= i'$ or $j=j'$. For $i\in\{1,2,3\}$, $V^i$ and $V_i$ are cliques; $V^i$ is anticomplete to $\{a^i_1,a^i_2,a^i_3\}$ and complete to the remainder of $W$; and $V_i$ is anticomplete to $\{a^1_i,a^2_i,a^3_i\}$  and complete to the remainder of $W$. Moreover, $V^1 \cup V^2 \cup V^3$ is complete to $V_1 \cup V_2\cup  V_3$, and there is no triad included in $V^1 \cup V^2 \cup V^3$ or in $V_1 \cup V_2 \cup V_3$. We call such a graph $G$ \textit{the complement of a mantled $L(K_{3,3})$}. \\

\textbf{The complement of a ring of five.}
Let $G$ be a graph with $V(G)$ the union of the disjoint sets $W = \{a_1,\ldots , a_5, b_1,\ldots , b_5\}$ and 
$V_0,V_1, \ldots , V_5$. Let adjacency be as follows (reading subscripts modulo 5). For $i\in\{1,\ldots, 5\}$,
$\{a_i, b_{i+1}, a_{i+2}\}$ is a triangle, and $\{b_1,\ldots,b_5\}$ is a clique; $V_0$ is anticomplete to $\{b_1,\ldots , b_5\}$ and complete to $\{a_1,\ldots , a_5\}$; $V_0,V_1,\ldots , V_5$ are all cliques; for $i \in \{ 1,\ldots , 5\}$, $V_i$ is anticomplete to $\{a_{i-1}, b_i, a_{i+1}\}$ and complete to the remainder of $W$; $V_0$ is complete to $V_1\cup \ldots \cup V_5$;
for $i\in\{1,\ldots, 5\}$, $V_i$	 is complete to $V_{i+2}$; and the adjacency between $V_i,V_{i+1}$ is arbitrary. We call
such a graph $G$ \textit{the complement of a  ring of five}. \\

\textbf{The complement of a path of triangles graph.}
Let $ m\geq 1 $ be some integer and $ G $ be a graph where $ V(G) $ is the union of pairwise disjoint cliques $X_1, \ldots , X_{2m+1}$, satisfying the following conditions (P1)--(P7).
\begin{enumerate}
	\setlength{\parskip}{0mm}
	\setlength{\itemsep}{1mm}
	\item[(P1)] For $ 1 \leq i \leq m $, there is a nonempty subset $ \hat{X}_{2i}\subseteq X_{2i} $; $|\hat{X}_2| = | \hat{X}_{2m}| = 1$ and for $0 < i <m$, at least one of $\hat{X}_{2i}, \hat{X}_{2i+2}$ has cardinality $1$.
	\item[(P2)] For $1 \leq  i < j \leq 2m+1$
	\begin{enumerate}
		\setlength{\itemsep}{0mm}
		\item[(1)] if $ j-i = 2$ modulo $3 $ and there exist $ u \in X_i $ and $ v \in X_j $, adjacent, then either $ i, j $ are odd and $ j = i + 2 $, or $ i, j $ are even and $ u \not\in \hat{X}_i $ and $ v \not\in \hat{X}_j $;
		\item[(2)] if $ j - i\neq 2$ modulo $3 $, then either $ j = i + 1 $ or $ X_i $ is complete to $ X_j $.
	\end{enumerate}
	\item[(P3)] For $ 1 \leq i \leq m + 1 $, $ X_{2i-1} $ is the union of three pairwise disjoint sets $ L_{2i-1} $, $ M_{2i-1} $, $ R_{2i-1} $,
	where $ L_1 =M_1 =M_{2m+1} = R_{2m+1} =\emptyset $.
	\item[(P4)] If $ R_1 =\emptyset $ then $ m \geq 2 $ and $ | \hat{X}_4| > 1 $, and if $ L_{2m+1} =\emptyset $, then $ m \geq 2 $ and $ | \hat{X}_{2m-2}| > 1 $.
	\item[(P5)] For $ 1 \leq i \leq m $, $ X_{2i} $ is complete to $ L_{2i-1} \cup R_{2i+1} $; $ X_{2i} \setminus \hat{X}_{2i} $ is complete to $ M_{2i-1} \cup
	M_{2i+1} $; and every vertex in $ X_{2i} \setminus \hat{X}_{2i} $ is adjacent to exactly one end of every non-edge between $ R_{2i-1} $ and $ L_{2i+1} $.
	\item[(P6)] For $ 1 \leq i \leq m $, if $ | \hat{X}_{2i}| = 1 $, then
	\begin{enumerate}
		\setlength{\itemsep}{0mm}
		\item[(1)] $ R_{2i-1},L_{2i+1} $ are anti-matched, and every non-edge between $ M_{2i-1} \cup R_{2i-1} $ and $ L_{2i+1} \cup M_{2i+1} $ is between $ R_{2i-1} $ and $ L_{2i+1} $;
		\item[(2)] the vertex in $ \hat{X}_{2i} $ is anticomplete to $ R_{2i-1} \cup M_{2i-1} \cup L_{2i+1} \cup M_{2i+1} $;
		\item[(3)] $ L_{2i-1} $ is anticomplete to $ X_{2i+1} $ and $ X_{2i-1} $ is anticomplete to $ R_{2i+1} $;
		\item[(4)] if $ i > 1 $, then $ M_{2i-1} $, $ \hat{X}_{2i-2} $ are anti-matched, and if $ i < m $, then $ M_{2i+1}, \hat{X}_{2i+2}$ are anti-matched.
	\end{enumerate}
	\item[(P7)] For $ 1 < i <m $, if $ | \hat{X}_{2i} | > 1 $, then
	\begin{enumerate}
		\setlength{\itemsep}{0mm}
		\item[(1)] $ R_{2i-1} = L_{2i+1} =\emptyset $;
		\item[(2)] if $ u \in X_{2i-1} $ and $ v \in X_{2i+1} $, then $ u, v $ are  adjacent if and only if they have the same neighbour in $ \hat{X}_{2i} $.
	\end{enumerate}
\end{enumerate}
Such a graph $G$ is called \textit{the complement of a path of triangles graph}. For each $ k\in\{1,2,3\} $, define
\[A_k =\bigcup (X_i\ :\ 1 \leq i \leq 2m+1 \text{ and } i = k \text{ modulo } 3).\]
Then, $ (G,A_1,A_2,A_3) $ is a three-cliqued graph. A permutation of $ (G,A_1,A_2,A_3) $ is called \textit{the complement of a canonically-coloured path of triangles graph}. \\

\textbf{The complement of a cycle of triangles graph.}
Let $ m \geq 5 $ be some integer with $ m = 2 $ modulo $ 3 $ and $ G $ be a graph where $ V(G) $ is the union of pairwise disjoint cliques $ X_1, \ldots , X_{2m} $, satisfying the following conditions (C1)--(C6) (reading subscripts modulo $ 2m $):
\begin{enumerate}
	\setlength{\parskip}{0mm}
	\setlength{\itemsep}{1mm}
	\item[(C1)] For $ 1 \leq i \leq m $, there is a nonempty subset $ \hat{X}_{2i} \subseteq X_{2i} $, and at least one of $ \hat{X}_{2i} $, $ \hat{X}_{2i+2} $ has
	cardinality $ 1 $.
	\item[(C2)] For  $ 1 \leq i \leq 2m $ and all $ k $ with $ 2 \leq k \leq 2m-2 $, let $ j \in \{1, \ldots , 2m\} $ with $ j = i + k $ modulo $ 2m $:
	\begin{enumerate}
		\setlength{\parskip}{0mm}
		\item[(1)] if $ k = 2$ modulo $3 $ and there exist $ u \in X_i $ and $ v \in X_j $, adjacent, then either $ i, j $ are
		odd and $ k \in \{2, 2m-2\} $, or $ i, j $ are even and $ u \not\in \hat{X}_i $ and $ v \not\in \hat{X}_j $;
		\item[(2)] if $ k\neq 2 $ modulo $ 3 $, then $ X_i $ is complete to $ X_j $. 
	\end{enumerate}
	\item[(C3)] For $ 1 \leq i \leq m+ 1 $, $ X_{2i-1} $ is the union of three pairwise disjoint sets $ L_{2i-1},M_{2i-1},R_{2i-1} $.
	\item[(C4)] For $ 1 \leq i \leq m $, $ X_{2i} $ is complete to $ L_{2i-1} \cup  R_{2i+1} $; $ X_{2i} \setminus \hat{X}_{2i} $ is complete to $ M_{2i-1} \cup M_{2i+1} $; and every vertex in $ X_{2i} \setminus \hat{X}_{2i} $ is adjacent to exactly one end of every
	non-edge between $ R_{2i-1} $ and $ L_{2i+1} $.
	\item[(C5)] For $ 1 \leq i \leq m $, if $ | \hat{X}_{2i}| = 1 $, then
	\begin{enumerate}
		\setlength{\parskip}{0mm}
		\item[(1)] $ R_{2i-1},L_{2i+1} $ are anti-matched, and every non-edge between $ M_{2i-1} \cup R_{2i-1} $ and $ L_{2i+1} \cup M_{2i+1} $ is between $ R_{2i-1} $ and $ L_{2i+1} $;
		\item[(2)] the vertex in $ \hat{X}_{2i} $ is anticomplete to $ R_{2i-1} \cup M_{2i-1} \cup L_{2i+1} \cup M_{2i+1} $;
		\item[(3)] $ L_{2i-1} $ is anticomplete to $ X_{2i+1} $ and $ X_{2i-1} $ is anticomplete to $ R_{2i+1} $;
		\item[(4)] $ M_{2i-1}, \hat{X}_{2i-2} $ are anti-matched and $ M_{2i+1}, \hat{X}_{2i+2} $ are anti-matched.
	\end{enumerate}
	\item[(C6)] For $ 1 \leq i \leq m $, if $ | \hat{X}_{2i} |>1 $, then
	\begin{enumerate}
		\setlength{\parskip}{0mm}
		\item[(1)] $ R_{2i-1} = L_{2i+1} =\emptyset $;
		\item[(2)] if $ u \in X_{2i-1} $ and $ v \in X_{2i+1} $, then $ u, v $ are adjacent if and only if they have the same neighbour in $ \hat{X}_{2i} $.
	\end{enumerate}
\end{enumerate}
Such a graph $ G $ is called \textit{the complement of a cycle of triangles graph}.

\section{Basic classes of non-orientable antiprismatic graphs} \label{app:noap}
The following classes of graphs are the building blocks in the structure theorems of non-orientable antiprismatic graphs. They are recalled from \cite{seymour2}, where the description of non-orientable prismatic graphs is given (the graphs whose complements are antiprismatic). Thus, we reformulate all the definitions in terms of the complements, maintaining the titles of the classes for simplicity. 
\begin{itemize}
	\item\textbf{Graphs of parallel-square type}\\
	Let $A\subset \{a_1,a_2,a_3,\ldots\}$, $B\subset \{b_1,b_2,b_3,\ldots\}$, $C\subset \{c_1,c_2,c_3,\ldots\}$ and $D\subset \{d_1,d_2,d_3,\ldots\} $ be four nonempty finite sets and $Z$ is a set with $|Z|\leq 1$. Define $G$ as the graph on the vertex set $V(G)=A\cup B\cup C\cup D\cup Z\cup \{u,v,x,y\}$, with adjacency as follows. The sets $A\cup \{u,y\},B\cup\{u,v\},C\cup\{v,x\}$ and $D\cup\{x,y\}$ are cliques, $Z$ is complete to $A\cup B\cup C\cup D$ and anticomplete to $\{u,v,x,y\}$, for every $i,j$, the vertex $a_i$ is adjacent to $b_j$ and $d_j$ if and only if $i\neq j$ and is adjacent to $c_j$ if and only if $i=j$, the vertex $b_i$ is adjacent to $c_j$ if and only if $i\neq j$ and is adjacent to $d_j$ if and only if $i=j$, the vertex $c_i$ is adjacent to $d_j$ if and only if $i\neq j$ and there are no more edges in $G$. Any such graph $G$ is antiprismatic and is called a graph of \textit{parallel-square type}.
	\item\textbf{Graphs of skew-square type}\\
	Let $A\subset \{a_1,a_2,a_3,\ldots\}$, $B\subset \{b_1,b_2,b_3,\ldots\}$ and $C\subset \{c_1,c_2,c_3,\ldots\} $ be three nonempty finite sets and define $G$ as the graph on the vertex set $V(G)=A\cup B\cup C\cup \{s,t,d_1,d_2,d_3\}$, with adjacency as follows. The sets $A,B,C$ and $\{s,t,d_1,d_2,d_3\}$ are cliques, $s$ is complete to $B$ and $t$ is complete to $A$. For every $i,j$, $a_i\in A$ is adjacent to $b_j\in B$ if and only if $i=j$. Also, $a_i\in A$ and $b_i\in B$ are adjacent to $c_j\in C$ if and only if $i\neq j$. For $1\leq i\leq 3$ and every $j$, $d_i$ is adjacent to $a_j\in A$ and $b_j\in B$, if and only if either $i=j$ or $j\geq 4$. Also, $d_i$ is adjacent to $c_j$ if and only if $i\neq j$ and $1\leq j\leq 3$ and there are no more edges in $G$. Any such graph is antiprismatic and is called a graph of \textit{skew-square type}.
	
	\item\textbf{The Class } $\mathcal{F}_0$\\
	Take the Schl\"afli graph with the vertices numbered $r_j^i,s_j^i$ and $t_j^i$ as usual (for the definition of Schl\"{a}fli graph, see \pref{sub:Inc-rot}). Suppose that $H$ be the subgraph induced on
	\[\{r_j^i\ :\ (i,j)\in I_1\}\cup \{s_j^i\ :\ (i,j)\in I_2\}\cup \{t_j^i\ :\ (i,j)\in I_3\},\]
	where $I=\{(i,j): 1\leq i,j\leq 3\}$ and
	\begin{itemize}
		\item $I\setminus \{(1,2),(2,1),(3,3)\}\subseteq I_1\subseteq I\setminus \{(3,3)\}$
		\item $\{(1,1),(2,1),(3,2)\}\subseteq I_2\subseteq \{(1,1),(2,1),(3,1),(3,2),(3,3)\}$
		\item $\{(2,1),(2,2),(1,3)\}\subseteq I_3\subseteq \{(2,1),(2,2),(1,3),(2,3),(3,3)\}$.
	\end{itemize}
	Let $G$ be the graph obtained from $H$ by deleting the edges in $E(H)\cap \{ s_1^3t_3^2,s_1^3t_3^3, s_3^3t_3^2\}$. We define $\mathcal{F}_0$ to be the class of all such graphs $G$ (they are all antiprismatic).
	\item\textbf{The class} $\mathcal{F}_1$\\
	Let $G$ be a graph with vertex set the disjoint union of sets $\{s,t\},R,A,B$, where $|R|\leq 1$, and with edges as follows:
	\begin{itemize}
		\item $s$, $t$ are nonadjacent, and both are anticomplete to $R$;
		\item $s$ is anticomplete to $A$; $t$ is anticomplete to $B$;
		\item every vertex in $A$ has at most one non-neighbour in $A$, and every vertex in $B$ has at most one non-neighbour in $B$;
		\item if $a,a'\in A$ are nonadjacent and $b,b'\in B$ are nonadjacent, then the in the subgraph induced on $\{a,a',b,b'\}$ is a matching;
		\item if $a,a'\in A$ are nonadjacent, and $b\in B$ is adjacent to all other vertices of $B$, then $b$ is adjacent to exactly one of $a,a'$.
		\item if $b,b'\in B$ are nonadjacent, and $a\in A$ is adjacent to all other vertices of $A$, then $a$ is adjacent to exactly one of $b,b'$.
		\item if $a\in A$ is adjacent to all other vertices of $A$, and $b\in B$ is adjacent to all other vertices of $B$, then $a,b$ are nonadjacent.
	\end{itemize}
	We define $\mathcal{F}_1$ to be the class of all such graphs $G$ (they are all antiprismatic).
	\item\textbf{The class} $\mathcal{F}_2$\\
	Let $H$ be a graph of parallel-square type and $A,B,C,D,Z,u,v,x,y$ be as in its definition. Assume  that $Z$ is nonempty and $a_1\in A, b_1\in B, c_1\not\in C, d_1\not\in D$. Also, for every $c_i\in C$ and $d_j\in D$, we have $i\neq j$.
	Let $G$ be obtained from $H$ by exponentiating the leaf triad $\{a_1,b_1,x\}$ (see the definition of exponentiating in \pref{sub:Inc-rotator-Exc-sf}). We define $\mathcal{F}_2$ to be the class of all such graphs $G$ (they are all antiprismatic).
	\item\textbf{The class} $\mathcal{F}_3$\\
	Let $H$ be obtained from a graph of parallel-square type with $A,B,C,D,Z,u,v,x,y$ as in its definition by removing the vertex $u$. Also, assume that $a_1\not\in A, a_2\in A, b_1\in B, b_2\not\in B, c_1\in C, c_2\not\in C, d_1\not\in D, d_2\in D$.
	Let $G$ be obtained from $H$ by exponentiating the leaf triads $\{b_1,c_1,y\}$ and $\{a_2, d_2,v\}$. We define $\mathcal{F}_3$ to be the class of all such graphs $G$ (they are all antiprismatic).
	\item\textbf{The class} $\mathcal{F}_4$\\
	Take the Schl\"{a}fli graph, with vertices numbered $r_j^i$, $s_j^i$, $t_j^i$ as usual (for the definition of Schl\"{a}fli graph, see \pref{sub:Inc-rot}). Let $H$ be the subgraph induced on
	\[Y\cup \{s_j^i: (i,j)\in I\}\cup \{t_1^1,t_2^2,t_3^3\}\]
	where $\emptyset \neq Y\subseteq \{r_1^3,r_2^3,r_3^3\}$ and $I\subseteq \{(i,j): 1\leq i,j \leq 3\}$ with $|I|\geq 8$, where
	\[\{(i,j): 1\leq i \leq 3\text{ and }1\leq j \leq 2\}\subseteq I.\]
	Let $G$ be obtained from $H$ by exponentiating the leaf triad $\{t_1^1,t_2^2,t_3^3\}$. We define $\mathcal{F}_4$ to be the class of all such graphs $G$ (they are all antiprismatic).
\end{itemize}

\section{Proof of \pref{lem:ccschlafli}} \label{app:schlafli}
In this section, we give a proof for \pref{lem:ccschlafli}. Let $G$ be Schl\"{a}fli-antiprismatic and let $\mathscr{S}(G)$ be the collection obtained from $\mathscr{O}(G)$ by adding the cliques $Y_j^i=V(G)\cap \{t_j^1,t_j^2,t_j^3,r_1^i,r_2^i,r_3^i\}$, $(i,j)\in J$. Note that $ \mathscr{S}(G) $ is a clique covering for $ G $ of size  $18$. Also, let $\mu$ be an automorphism of $\Gamma$ defined as follows (we leave to the reader to check that $\mu$ is in fact an automorphism).
\begin{align*}
r_1^1\mapsto s_1^1, \ s_2^2\mapsto s_2^1, \ t_3^3\mapsto s_3^1, \quad  r_3^2\mapsto t_1^1, \ s_2^1\mapsto t_2^1, \ t_1^3\mapsto t_3^1, \quad  r_2^3\mapsto r_1^1, \ t_2^3\mapsto r_2^1, \ s_2^3\mapsto r_3^1,\\
s_3^3\mapsto s_1^2, \ t_1^1\mapsto s_2^2, \ r_2^2\mapsto s_3^2, \quad  t_3^2\mapsto t_1^2, \ r_2^1\mapsto t_2^2, \ s_1^3\mapsto t_3^2, \quad  t_1^2\mapsto r_1^2, \ s_1^2\mapsto r_2^2, \ r_1^2\mapsto r_3^2,\\
t_2^2\mapsto s_1^3, \ r_3^3\mapsto s_2^3, \ s_1^1\mapsto s_3^3, \quad  s_3^2\mapsto t_1^3, \ t_2^1\mapsto t_2^3,  \ r_1^3\mapsto t_3^3, \quad  s_3^1\mapsto r_1^3, \ r_3^1\mapsto r_2^3, \ t_3^1\mapsto r_3^3.
\end{align*}

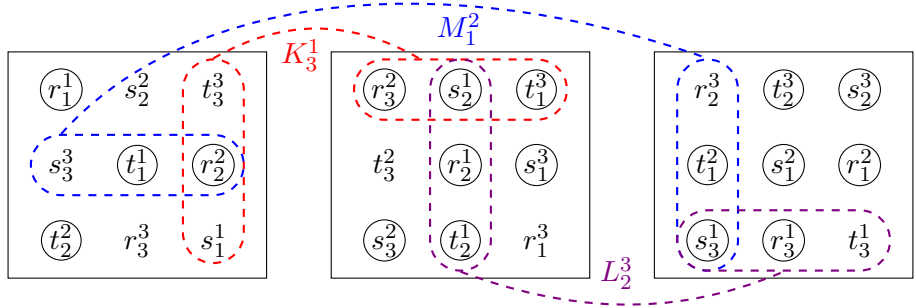
\begin{figure}
	\begin{center}
		\begin{tikzpicture}
		\begin{scope}[scale=.5]
		\SetVertexNoLabel
		\tikzset{VertexStyle/.style ={shape=circle,inner sep=1mm,minimum size=15pt,draw}}

\node at (-4.5,4) [circle,draw,inner sep=0pt] (r11) {$ r_1^1 $};
\node at (-4.5,2)    (s33) {$ s_3^3 $};
\node at (-4.5,0) [circle,draw,inner sep=0pt] (t22) {$t_2^2 $};
\node at (-2.5,4)  (s22) {$s_2^2$};
\node at (-2.5,2) [circle,draw,inner sep=0pt] (t11) {$t_1^1$};
\node at (-2.5,0)  (r33) {$ r_3^3$};
\node at (-.5,4)  (t33) {$ t_3^3$};
\node at (-.5,2) [circle,draw,inner sep=0pt] (r22) {$r_2^2$};
\node at (-.5,0)  (s11) {$s_1^1$};
		\node at (-2.5,2) [rectangle,draw,inner xsep=1.7cm,inner ysep=1.5cm] (S) {};
\node at (6,2) [rectangle,draw,inner xsep=1.7cm,inner ysep=1.5cm] (T) {};
\node at (4,4) [circle,draw,inner sep=0pt] (r23) {$ r_3^2 $};
\node at (4,2)  (t23) {$ t_3^2 $};
\node at (4,0) [circle,draw,inner sep=0pt] (s23) {$ s_3^2 $};
\node at (6,4) [circle,draw,inner sep=0pt] (s12) {$ s_2^1 $};
\node at (6,2) [circle,draw,inner sep=0pt] (r12) {$ r_2^1 $};
\node at (6,0) [circle,draw,inner sep=0pt] (t12) {$ t_2^1 $};
\node at (8,4) [circle,draw,inner sep=0pt](t31) {$ t_1^3 $};
\node at (8,2) [circle,draw,inner sep=0pt](s31) {$ s_1^3 $};
\node at (8,0) (r31) {$ r_1^3 $};
\node at (14.5,2) [rectangle,draw,inner xsep=1.7cm,inner ysep=1.5cm] (R) {};
\node at (12.5,4)  (r32) {$ r_2^3 $};
\node at (12.5,2) [circle,draw,inner sep=0pt] (t21) {$ t_1^2 $};
\node at (12.5,0) [circle,draw,inner sep=0pt] (s13) {$ s_3^1 $};
\node at (14.5,4) [circle,draw,inner sep=0pt] (t32) {$ t_2^3 $};
\node at (14.5,2) [circle,draw,inner sep=0pt] (s21) {$ s_1^2 $};
\node at (14.5,0) [circle,draw,inner sep=0pt] (r13) {$ r_3^1 $};
\node at (16.5,4) [circle,draw,inner sep=0pt] (s32) {$ s_2^3 $};
\node at (16.5,2) [circle,draw,inner sep=0pt] (r21) {$ r_1^2 $};
\node at (16.5,0)  (t13) {$ t_3^1 $};
\node at (-.5,2.1) [rotate=-90,rounded rectangle,dashed,thick,red,draw,inner xsep=1.35cm,inner ysep=.4cm] (k1) {};
\node at (6,4) [rounded rectangle,dashed,red,thick,draw,inner xsep=1.4cm,inner ysep=.4cm] (k2) {};
\node at (-2.5,2) [rounded rectangle,dashed,thick,blue,draw,inner xsep=1.4cm,inner ysep=.4cm] (m1) {};
\node at (12.5,2) [rotate=90,rounded rectangle,dashed,thick,blue,draw,inner xsep=1.4cm,inner ysep=.4cm] (m2) {};
\node at (14.5,0) [rounded rectangle,dashed,thick,violet,draw,inner xsep=1.4cm,inner ysep=.4cm] (l1) {};
\node at (6,2) [rotate=90,rounded rectangle,dashed,thick,violet,draw,inner xsep=1.4cm,inner ysep=.4cm] (l2) {};
\draw [bend left,red,thick,dashed] (k1.west) to node [below,xshift=-2mm] {$ K_3^1 $} (k2);
\draw [bend right,violet,thick,dashed] (l2.west) [yshift=1cm] to node [above] {$ L_2^3 $} (l1.south);
\draw [bend left,blue,thick,dashed] (m1.north west) [yshift=-1.7cm,xshift=-3cm] to node [below,xshift=2cm,yshift=.2cm] {$ M_1^2 $} (m2.east);
\end{scope}
\end{tikzpicture}
	\end{center}
	\vspace{-25pt}
	\caption{The automorphism $ \mu $. \label{fig:automorphism}}
\end{figure}

Also for every $(i,j)\in J$, define
\begin{align*}
K_j^i=\mu^{-1}(X_j^i), \quad 
L_j^i=\mu^{-1}(Y_j^i), \quad 
M_j^i=\mu^{-1}(Z_j^i). 
\end{align*}
Therefore, $\mathscr{M}(G)=\{K_j^i,L_j^i,M_j^i: (i,j)\in J\}$ is a clique covering for $G$ of size $18$ (see \pref{fig:automorphism}), and we use its modified versions to provide suitable clique coverings for some Schl\"{a}fli-antiprismatic graphs. We need a lemma beforehand which enables us to handle Schl\"{a}fli-antiprismatic graphs with at least 17 vertices. First, we need a definition. Let $k$ be a positive integer. A graph $H$ is said to be \textit{$k$-ultrahomogeneous}, if every isomorphism between two of its induced subgraphs on at most $k$ vertices can be extended to an automorphism of the whole graph. 
It is proved in \cite{cameron} that every $5$-ultrahomogeneous graph is $k$-ultrahomogeneous for every $k$ (nevertheless, there are only a few finite graphs that are $k$-ultrahomogeneous for every $k$).
  Also, the Schl\"{a}fli graph $\Gamma$ and its complement $\overline{\Gamma}$ are the only connected finite graphs that are $4$-ultrahomogeneous, but not $5$-ultrahomogeneous. Thus, every isomorphism between two induced subgraphs of $\Gamma$ on at most $4$ vertices can be extended to an automorphism of $\Gamma$, and this is what we are going to use in the proof of the following lemma. To learn more about this concept, the reader may refer to \cite{devillers,cameron}. For simplicity, we omit the term $ \rho $ within the proofs.
\begin{lem}\label{lem:schlafli Ultra}
	Let $G$ be a counterexample to \pref{thm:noap}, which is Schl\"{a}fli-antiprismatic with respect to some rotator $\rho$ of $\overline{G}$. Then, $|S(\rho)|+|T(\rho)|+|R(\rho)|\leq 7$.
\end{lem}
\begin{proof}
	First we observe that,\vsp
	
	(1) \textit{We have $n\leq 18$. Also, if there exists a triangle $\{x,y,z\}$ in $\Gamma$ disjoint from $V(G)$, then $n\leq 16$.}\vsp
	
	The first assertion follows from the fact that $\mathscr{S}(G)$ is a clique covering for $ G $ of size $ 18 $. For the second, note that since $\Gamma$ is $4$-ultrahomogeneous, there exists an automorphism $\sigma$ of $\Gamma$, mapping the vertices $x,y$ and $z$ to the vertices $s_1^1,s_1^2$ and $s_1^3$, respectively. Now, for every $(i,j)\in J$, define 
	\begin{align*}
	A_j^i=\sigma^{-1}(X_j^i), \quad 
	B_j^i=\sigma^{-1}(Y_j^i), \quad 
	C_j^i=\sigma^{-1}(Z_j^i).
	\end{align*}
	It is easy to see that the family of cliques $\{A_j^i,B_j^i,C_j^i: (i,j)\in J\}\setminus \{A_1^2,A_1^3\}$ is a clique covering for $G$ of size $16$, and thus $n\leq 16$. This proves (1).\vsp
	
	Now, on the contrary, assume that $|S|+|T|+|R|\geq 8$, and so $n\geq 17$. If $|T|+|R|\leq 7$, then there exists $i\in \{1,2\}$ such that $|Y_i^{3-i}|\leq 3$, and so we may choose  $\{x,y,z\}\subseteq\{t_i^1,t_i^2,t_i^3,r^{3-i}_1,r^{3-i}_2,r^{3-i}_3\}\setminus V(G) $ as a triangle of $\Gamma$ disjoint from $V(G)$, which contradicts (1). Thus, $|T|+|R|\geq 8$ and w.l.o.g. we may assume that $ |T|\leq |R| $. Also, if either $|S|\geq 2$ or $|T|+|R|\geq 10$, then $n\geq 19$, which is impossible by (1). Thus, $|S|\leq 1$ and $ |T|\leq 4 $. Now, choose  $t_j^i\notin T$, for some $i\in \{1,2,3\}$ and $j\in \{1,2\}$ and choose $i'\in \{1,2,3\}\setminus \{i\}$ such that $S_{i'}=\emptyset$. Therefore,  $\{s_{i'}^{i'+1},s_{i'}^{i'+2},t_j^i\}$ is a triangle of $\Gamma$ disjoint from $V(G)$, again a contradiction with (1). This proves \pref{lem:schlafli Ultra}.
\end{proof}

Now, we are ready to prove \pref{lem:ccschlafli}.

\begin{proof}[{\rm \textbf{Proof of \pref{lem:ccschlafli}.}}]
Suppose not, assume that $G$ is ISA. By \pref{lem:|I_T|=3 |I_R|=1}, $|I_T|,|I_R|\geq 2$. We will provide a clique covering for $G$ of size at most $n-1$ through the following three cases. Since a clique covering for a graph can be trivially extended to a clique covering of the same size for its replication, we may assume that $G$ is Schl\"{a}fli-antiprismatic. \vsp
	
	\textbf{Case 1.} $|I_T|=|I_R|=2$.\\
	By symmetry, we may assume that $I_T=\{1,2\}$, $I_R=\{1,p\}$, $p\in \{2,3\}$ and $|T|\geq |R|$. By \pref{lem:square-forcer}~(ii) (applying to $(i,j,k)=(1,2,3)$), either $S_2^1$ or $S_1^p$ is nonempty. \vsp
	
	(1.1) \textit{We have $|T|\geq 3$.}\vsp
	
	On the contrary, assume that $|T|=|R|=2$, say $T=\{t_1^1,t_2^2\}$ and $R=\{r_i^1,r_{p-i+1}^2\}$, $i\in \{1,p\}$. By \pref{lem:schlafli Ultra}, $1\leq |S|\leq 3$. Through the following four cases, we apply suitable modifications to $\mathscr{M}(G)$ which yields a clique covering for $G$ of size at most $n-1$, a contradiction.
	\begin{itemize}
		\item $(p,i)=(2,1)$. If $ S^1_2= \emptyset $, then remove the cliques $ K_2^1, K_3^1, L_2^1$, $L_2^3 $ and $ M_2^1 $ and replace the cliques $ L_1^2$ and $ L_3^2 $ with the cliques $(L_1^2\cup  S^3_1\cup \{t^3_3,r_1^1\})\setminus S_3^2$ and $ L_3^2\cup S^2_3\cup \{s^2_2\})\setminus S_1^3$.
		Also, if $ S^2_1= \emptyset $, then remove the cliques $ L_1^2, L_3^2, M_2^1, M_2^3 $ and $ K_2^1 $ and replace the cliques $ L_2^1$ and $ L_2^3 $ with the cliques $(L_2^1\cup  S^1_3\cup \{t^1_1,r_3^3\})\setminus S_2^3$ and $ L_2^3\cup S^3_2\cup \{s^2_2\})\setminus S_3^1$.
		Now, assume that $ |S^1_2|=|S^2_1|=1 $. Since $ |S|\leq 3 $, either $ S^1_3\cup S^3_2=\emptyset $, or $ S^2_3\cup S^3_1 =\emptyset $. In the former case, remove the cliques $ K_2^1, K_3^1, M_3^1, M_1^2$ and $ M_2^3 $, add the clique $ \{s_1^1,s^1_2, r^3_3\} $ and replace the cliques $ L_2^1, L_3^1, L_1^2, L_3^2, L_1^3 $ and $ L_2^3 $ with the cliques $ L_2^1\cup \{t^1_1,r_2^2\}, L_3^1\cup \{s_3^3\}, (L_1^2\cup S^3_1\cup \{s_1^1,t^2_2\})\setminus S^2_3, (L_3^2\cup S^2_3\cup \{r_3^3\})\setminus S^3_1, L_1^3\cup \{r_1^1\} $ and $ L_2^3\cup \{s_2^2, t^3_3\} $. In the latter case, remove the cliques $ K_2^1, K_3^2, K_1^3, M_2^1$ and $M_2^3 $, add the clique $ \{s^2_1,s^2_2, r_3^3\} $ and replace the cliques $ L_1^2, L_1^3, L_2^1, L_2^3, L_3^1 $ and $ L_3^2 $ with the cliques $ L_1^2\cup \{s^1_1,t_3^3\}, L_1^3\cup \{r_2^2\}, (L_2^1\cup S^1_3\cup \{t_1^1,r^3_3\})\setminus S^3_2, (L_2^3\cup S^3_2\cup \{s_2^2\})\setminus S^1_3, L_3^1\cup \{s_3^3\} $ and $ L_3^2\cup \{t_2^2, r^1_1\} $.
		\item $(p,i)=(2,2)$. If $S_2^1=\emptyset$, then remove the cliques $K_2^1$ and $K_3^1$ and if $S_2^1\neq \emptyset$, then remove the clique $K_3^1$ and replace the cliques $L_2^1$ and $L_2^3$ with the cliques $(L_2^1\cup S_3^1\cup \{s_1^1\})\setminus S_2^3$ and $(L_2^3\cup S_2^3\cup \{t_3^3\})\setminus S_3^1$. If $S_1^2=\emptyset$, then remove the cliques $M_2^1$ and $M_2^3$ and if $S_1^2\neq \emptyset$, then remove the clique $M_2^1$ and replace the cliques $L_1^2$ and $L_3^2$ with the cliques $(L_1^2\cup S_1^3\cup \{t_3^3\})\setminus S_3^2$ and $(L_3^2\cup S_3^2\cup \{s_2^2\})\setminus S_1^3$. Also, if $S_3^1\cup S_1^3=\emptyset$, then merge the pair $(L_3^1,M_1^2)$ and if $S_3^1\cup S_1^3\neq \emptyset$ and $S_3^2\cup S_2^3=\emptyset$, then merge the pair $(K_1^3,L_3^1)$, and call the resulting family $\mathscr{C}$. Now, if $S\neq S_2^1\cup S_1^2$, then $|\mathscr{C}|\leq n-1$, as required. Thus, we may assume that $S=S_2^1\cup S_1^2$. Now, if $|S|=|S_2^1\cup S_1^2|=2$, then $|\mathscr{C}|=n$, so remove the clique $K_3^2$ from $\mathscr{C}$ and replace the cliques $(L_1^2\cup S_1^3\cup \{t_3^3\})\setminus S_3^2$ and $L_1^3$ with the cliques $(L_1^2\cup S_1^3\cup \{s_1^1,t_3^3\})\setminus (S_3^2\cup \{r_1^2\})$ and $L_1^3\cup \{r_1^2\}$, if $|S|=|S_2^1|=1$, then merge the pair $(L_1^2,L_1^3)$ and if $|S|=|S_1^2|=1$, then merge the pair $(L_1^3,L_2^3)$.
		\item $(p,i)=(3,1)$. If $ S^2_1=\emptyset $, then remove the cliques $ L_1^2, L_3^2, M_2^1 $ and $ M_2^3 $ and in the resulting family, if $|S|=|S_2^1|=1$, then merge the pair $(L_3^1,M_1^2)$ and if $|S|=|S_1^3|=1$, then remove the cliques $L_2^1$ and $L_2^3$.
		Also, if $ S^2_1\neq \emptyset $ (i.e. $ |S|\geq 2 $) and $ S^1_2=\emptyset $, then remove the cliques $K_1^2, K_1^3, L_2^1, L_2^3 $, add the clique $ S^2_3\cup \{s_3^3, t^2_2, t^2_3, r_1^1 \} $, replace the cliques $ L_3^1, L_3^2 $ with the cliques  $ L_3^1\cup \{s_3^3\}, L_3^2\cup \{t^2_2, r_1^1\} $ and in the resulting family, if $ |S|=|S_1|=2$, then merge the pair $ (K_3^1,L_1^2) $. Finally, assume that both $ S^2_1, S^1_2 \neq \emptyset$. 
		In this case, remove the cliques $K_1^2, K_1^3$, add the clique $ S^2_3\cup \{s_3^3, t^2_2, t^2_3, r_1^1 \} $ and replace the cliques $ L_3^1, L_3^2 $ with the cliques  $ L_3^1\cup \{s_3^3\}, L_3^2\cup \{t^2_2, r_1^1\} $, and call the resulting family $\mathscr{C}$. Since $ |S|\leq 3 $, either $ S^1_3\cup S^3_2=\emptyset $, or $ S^2_3\cup S^3_1=\emptyset $. In the former case, remove the cliques $ M_1^3 $ and $ M_3^1$ from $\mathscr{C}$, replace the cliques $ L_2^1, L_3^1\cup \{s_3^3\}, L_1^3, L_2^3 $ with the cliques $ L_2^1\cup \{s_1^1,r_3^3\}, L_3^1\cup \{s_3^3,t^2_2\}, (L_1^3\cup \{r_1^1\})\setminus \{r^2_3\}, L_2^3\cup \{s^2_2,t_3^3,r^2_3\} $ and in the resulting family, if $ |S|=2 $, then remove the clique $ K_3^1 $ and replace the clique $ K_3^2 $ with the clique $ K_3^2\cup \{r^2_3\} $ (the edges $s_1^1s_2^1$ and $s_2^1t_3^3$ are already covered by the cliques $L_2^1\cup \{s_1^1,r_3^3\}$ and $L_2^3\cup \{s^2_2,t_3^3,r^2_3\}$). In the latter case, assuming $ S^1_3\cup S^3_2\neq \emptyset $ (i.e.n $|S|=3$), remove the cliques $ M_1^2,M_3^2 $ and replace the cliques $L_3^1, L_1^3, L_2^3 $ with the cliques $ L_3^1\cup \{s^3_3,t^1_1\}, (L_1^3\cup \{s_3^3\})\setminus \{r^2_3\}, L_2^3\cup \{t^1_1,r^2_3\}$.  
		\item $(p,i)=(3,3)$. If $S_2^1=\emptyset$, then remove the cliques $K_2^1,K_3^1,L_2^1$ and $L_2^3$, and in the resulting family, if $|S|=|S_1^3|=1$, then merge the pair $(K_1^3,L_3^1)$. Now, assume that $S_2^1\neq \emptyset$. First, suppose that $S_2^3=\emptyset$. Remove the cliques $K_2^1,K_1^2,K_1^3,M_2^1$ and $M_3^1$, add the clique $S_3^2\cup \{s_3^3,t_2^2,t_3^2\}$, replace the cliques $L_2^1,L_3^1, L_1^2,L_3^2$ and $L_2^3$ with the cliques $L_2^1\cup S_3^1\cup \{t_1^1,r_3^3\},L_3^1\cup \{t_2^2,s_3^3\},(L_1^2\cup S_1^3\cup \{t_3^3\})\setminus S_3^2,(L_3^2\cup S_3^2\cup \{s_2^2\})\setminus S_1^3$ and $(L_2^3\cup \{s_2^2,t_3^3\})\setminus S_3^1$ and in the resulting family, if $|S|=|S_2^1|=1$, then remove the clique $M_1^2$ (the edges in $E(M_1^2)$ are already covered by the cliques $L_2^1\cup S_3^1\cup \{t_1^1,r_3^3\},L_3^1\cup \{t_2^2,s_3^3\}$ and $M_3^2$).
		
		Next, assume that $|S_2^3|=1$. Remove the cliques $K_2^1,K_3^1,K_1^2$ and $K_1^3$, add the clique $\{s_1^1,s_2^1,r_3^3\}$ and replace the cliques $L_1^2,L_1^3$ and $L_2^1$ with the cliques $(L_1^2\cup \{t_3^1,s_3^3\})\setminus S_1^2,(L_1^3\cup S_1^2\cup \{t_2^2\})\setminus \{t_3^1\}$ and $L_2^1\cup \{t_1^1\}$, if $S_3^1=\emptyset$, then replace the clique $L_2^3$ with the clique $L_2^3\cup \{s_2^2,t_3^3\}$, if $S_3^1\neq \emptyset$, then add the cliques $\{s_2^1,s_2^2,t_3^3\}$, if $S_1^2=\emptyset$, then replace the cliques $L_3^1$ and $L_3^2$ with the cliques $(L_3^1\cup \{t_2^2,s_3^3\})\setminus S_2^3$ and $L_3^2\cup S_2^3$ and if $S_1^2\neq \emptyset$, then add the clique $S_1^3\cup \{s_3^3,t_2^2,r_1^3\}$, and finally if $|S|=|S_2|=2$, then remove the clique $M_1^2$ (the edges in $E(M_1^2)$ are already covered by the cliques are already covered by other cliques $L_2^1\{t_1^1\},(L_3^1\cup \{t_2^2,s_3^3\})\setminus S_2^3$ and $M_3^2$).
	\end{itemize}
	This proves (1.1). \vsp
	
	(1.2) \textit{If $|R|=2$, then $|T|=4$.}\vsp
	
	Suppose not, by (1.1), $|T|=3$, and w.l.o.g. we may assume that $T=\{t_1^1,t_{3-i}^i,t_2^2\}$ and $R=\{r_j^1,r_{p-j+1}^2\}$, for some $i\in \{1,2\}$ and $j\in \{1,p\}$. By \pref{lem:schlafli Ultra} $1\leq |S|\leq 2$. Through the following six cases, we apply suitable modifications to $\mathscr{M}(G)$ which yields a clique covering for $G$ of size at most $n-1$, a contradiction.
	\begin{itemize}
		\item $(p,j)=(2,1)$. By symmetry, we may assume that $i=1$. If $S_2^1=\emptyset$, then remove the cliques $K_2^1,K_3^1,K_1^2$ and $K_1^3$ and replace the cliques $L_1^2,L_1^3,L_2^1,L_2^3,L_3^1$ and $L_3^1$ with the cliques $L_1^2\cup \{t_2^2,r_1^1\},L_1^3\cup \{s_3^3\},(L_2^1\cup S_3^1\cup \{s_3^3,t_2^2\})\setminus S_2^3,(L_2^3\cup S_2^3\cup \{s_3^3,r_1^1\})\setminus S_3^1,L_3^1\cup \{s_3^3\}$ and $L_3^1\cup \{t_2^2,r_1^1\}$. Also, if $S_1^2=\emptyset$, then remove the cliques $L_1^2,L_3^2,M_2^1$ and $L_2^3$. Finally, if $|S|=|S_2^1\cup S_1^2|=2$, then remove the cliques $K_2^1,K_1^2$ and $K_2^3$ and replace the cliques $L_1^2,L_1^3,L_2^1,L_2^3,L_3^1$ and $L_3^2$ with the cliques $L_1^2\cup \{t_2^2,r_1^1\},L_1^3\cup \{s_3^3\},L_2^1\cup \{t_1^1,r_3^3\},L_2^3\cup \{s_2^2,t_1^1\},L_3^1\cup \{t_1^1\}$ and $L_3^1\cup \{s_2^2,r_3^3\}$.
		\item $(p,j)=(2,2)$. By symmetry, we may assume that $i=1$. First, suppose that $S_2^1=\emptyset$, and remove the cliques $K_2^1$ and $K_3^1$. Now, in the resulting family, if $S_3^1\cup S_2^3=\emptyset$, then remove the cliques $K_1^2$ and $K_1^3$ and replace the cliques $L_2^1,L_3^1,L_1^2$ and $L_1^3$ with the cliques $L_2^1\cup \{s_3^3,t_2^2\},L_3^1\cup \{s_3^3,t_2^2\},(L_1^2\cup \{t_2^2\})\setminus \{r_1^2\}$ and $L_1^3\cup \{s_3^3,r_1^2\}$, and if $S_3^2\cup S_1^3=\emptyset$, then remove the cliques $M_2^1$ and $M_1^2$ and replace the cliques $L_1^2,L_1^3,L_2^3,L_3^1$ and $L_3^2$ with the cliques $L_1^2\cup \{t_3^3\},L_1^3\cup \{r_2^1\},(L_2^3\cup \{s_3^3,t_1^1\})\setminus \{r_2^1\},L_3^1\cup \{s_3^3,t_1^1\}$ and $L_3^2\cup \{s_2^2\}$.
		
		Next, assume that $S_1^2=\emptyset$, and remove the cliques $M_2^1$ and $M_2^3$. Now, in the resulting family, if $S_3^1\cup S_1^3=\emptyset$, then merge the pairs $(L_1^2,L_1^3)$ and $(L_3^1,M_1^2)$, if $|S|=|S^1|=2$, then remove the clique $M_1^2$ and replace the cliques $L_1^3,L_2^3$ and $L_3^1$ with the cliques $L_1^3\cup \{s_3^3,r_2^1\},(L_2^3\cup \{t_1^1\})\setminus \{r_2^1\}$ and $L_3^1\cup \{s_3^3,t_1^1\}$, and if $|S|=|S_2^1\cup S_1^3|=2$, then merge the pair $(L_1^2,L_1^3)$.
		
		Finally, assume that $|S|=|S_2^1\cup S_1^2|=2$. Merge the pair $(K_1^3,L_3^1)$, remove the cliques $K_3^2$ and $M_2^1$ and replace the cliques $K_3^1,L_1^2,L_1^3$ and $L_3^2$ with the cliques $K_3^1\cup \{r_2^1\},(L_1^2\cup \{s_1^1,t_3^3\})\setminus \{r_1^2\},L_1^3\cup \{r_1^2\}$ and $L_3^2\cup \{s_2^2\}$.
		\item $(p,j)=(3,1)$ and $i=1$. If $S_1^2=\emptyset$, then remove cliques $L_1^2,L_3^2,M_2^1$ and $M_2^3$. Now, assume that $S_1^2\neq \emptyset$, i.e. either $S=\{s_2^1,s_1^2\}$ and $S=\{s_1^2,s_1^3\}$. In the former case, merge the pair $(L_3^1,M_1^2)$, remove the cliques $M_3^1,M_3^2$ and $M_2^3$, add the clique $\{s_3^3,t_3^1,r_1^1\}$ and replace the cliques $L_1^2,L_3^2,L_1^3$ and $L_2^3$ with the cliques $L_1^2\cup \{s_1^1\},L_3^2\cup \{t_2^2,r_3^3\},L_1^3\cup \{t_3^3\}$ and $L_2^3\cup \{s_2^2,t_1^1\}$. In the latter case, merge the pairs $(L_2^1,M_1^2)$ and $(L_2^3,M_3^2)$, remove the cliques $M_2^1$ and $M_2^3$, add the clique $S_1^2\cup \{s_2^2,r_3^3\}$ and replace the cliques $L_1^2$ and $L_3^2$ with the cliques $L_1^2\cup \{s_1^1,t_3^3\}$ and $L_3^2\cup \{t_2^2,r_1^1\}$.
		\item $(p,j)=(3,1)$ and  $i=2$. First, suppose that either $S_2^1$ or $S_1^2$ is empty.  If $S_2^1=\emptyset$, then remove the cliques $L_2^1$ and $L_2^3$, if $S_1^2=\emptyset$, then remove the cliques $M_2^1$ and $M_2^3$, if $S_2^1\cup S_1^3=\emptyset$, then merge the pair $(K_3^1,K_3^2)$, if $S_1^2\cup S_2^3=\emptyset$, then merge the pair $(L_3^1,L_3^2)$, if $S_2^3\neq \emptyset$ and $S_3^1\cup S_1^2\cup S_1^3=\emptyset$, then merge the pair $(L_3^2,M_1^2)$, if $S_3^1\cup S_1^3=\emptyset$, then remove the clique  $K_3^2$ and replace the cliques $L_1^2$ and $L_1^3$ with the cliques $L_1^2\cup \{s_1^1\}$ and $L_1^3\cup \{t_3^3\}$ and if $S_1^3\neq \emptyset$ and $S_2^1\cup S_3^1\cup S_3^2=\emptyset$, then remove the clique  $K_3^1$ and replace the cliques $L_1^2$ and $L_1^3$ with the cliques $L_1^2\cup \{s_1^1\}$ and $L_1^3\cup \{t_3^3\}$.
		
		Next, assume that $|S|=|S_2^1\cup S_1^2|=1$. Merge the pair $(L_3^1,M_1^2)$, remove the cliques $K_3^1$ and $M_2^3$ and replace the cliques $K_3^2,L_2^1,L_2^3,L_1^2$ and $L_3^2$ with the cliques $K_3^2\cup \{r_3^2\},L_2^1\cup \{s_1^1\},L_2^3\cup \{t_3^3\},L_1^2\cup \{s_1^1\}$ and $L_3^2\cup \{t_2^2,r_3^3\}$.
		\item $(p,j)=(3,3)$ and $i=1$. First, suppose that $S_2^1=\emptyset$. Remove the cliques $K_2^1,K_3^1,K_1^2$ and $K_1^3$, replace the cliques $L_1^2,L_1^3,L_3^1$ and $L_3^2$ with the cliques $(L_1^2\cup \{s_3^3,t_3^1\})\setminus S_1^2,(L_1^3\cup S_1^2\cup \{t_2^2\})\setminus \{t_3^1\},(L_3^1\cup \{s_3^3,t_2^2\})\setminus S_2^3$ and $L_3^2\cup S_2^3$, and in th resulting family, if $S_2^3=\emptyset$, then replace the clique $L_2^1$ with the clique $L_2^1\cup \{s_3^3,t_2^2\}$ and if $S_2^3\neq \emptyset$, then add the clique $\{s_3^3,t_2^1,t_2^2\}$.
		Next, assume that  $S_2^1\neq \emptyset$. In the case that $S_2^3=\emptyset$, remove the cliques $K_2^1,K_1^2,K_1^3$ and $M_1^2$ and replace the cliques $L_1^2,L_1^3,L_2^1,L_2^3$ and $L_3^1$ with the cliques $(L_1^2\cup \{s_3^3,t_3^1\})\setminus S_1^2,(L_1^3\cup S_1^2\cup \{t_2^2\})\setminus \{t_3^1\},L_2^1\cup S_3^1\cup \{t_1^1,r_3^3\},(L_2^3\cup \{s_2^2\})\setminus S_3^1$ and $L_3^1\cup \{s_3^3,t_2^2\}$, and in the resulting family, if $S_1^3=\emptyset$, then replace the clique $L_3^1\cup \{s_3^3,t_2^2\}$ with the clique $L_3^1\cup \{s_3^3,t_2^1,t_2^2\}$, if $S_1^3\neq \emptyset$, then add the clique $\{s_3^3,t_2^1,t_2^2\}$ and $S_3^1\neq \emptyset$, then remove the clique $M_1^2$ (the edges in $E(M_1^2)$ are already covered by the cliques $L_2^1\cup S_3^1\cup \{t_1^1,r_3^3\},L_3^1\cup \{s_3^3,t_2^2\}$ and $M_3^2$).  Also, in the case that $|S|=|S_2|=2$, merge the pair $(M_1^2,L_3^1)$, remove the cliques $K_3^2$ and $M_2^1$ and replace the cliques $L_1^2,L_1^3$ and $L_2^3$ with the cliques $L_1^2\cup \{t_3^1\},(L_1^3\cup \{s_1^1,t_3^3\})\setminus \{t_3^1\}$ and $L_2^3\cup \{s_2^2\}$.
		\item $(p,j)=(3,3)$ and $i=2$. If $S_2^1=\emptyset$, remove the cliques $K_2^1,K_3^1,L_2^1$ and $L_2^3$. Now, assume that $S_2^1\neq \emptyset$. In the case that $S_3^1\cup S_1^2\cup S_2^3=\emptyset$, merge the pair $(L_2^3,M_2^1)$, remove the cliques $K_2^1,K_1^2$ and $K_1^3$, add the clique $S_3^2\cup \{s_3^3,t_2^2,t_3^2\}$, replace the cliques $L_2^1,L_2^3$ and $L_3^1$ with the cliques $L_2^1\cup \{t_1^1,r_3^3\},L_2^3\cup \{s_2^2\}$ and $L_3^1\cup \{s_3^3,t_2^2\}$ and in the resulting family, if $|S|=|S_2^1|=1$, then remove the clique $M_1^2$ and replace the clique $L_3^2$ with the clique $L_3^2\cup \{s_3^3,t_1^1\}$ (the edges $s_3^3r_2^3$ and $t_1^1r_2^3$ are already covered by the cliques $L_2^1\cup \{t_1^1,r_3^3\}$ and $L_3^1\cup \{s^3_3,t_2^2\}$). Also, in the case that $S_3^1\cup S_1^2\cup S_2^3\neq \emptyset$, if $S_3^1\neq \emptyset$, then remove the cliques $K_3^1,M_1^2$ and $M_3^2$, add the clique $\{s_3^1,s_3^3,t_1^1,t_3^1\}$, merge the pair $(K_1^3,L_3^1)$ and replace the cliques $L_2^1,L_2^3$ and $L_3^2$ with the cliques $L_2^1\cup S_3^1\cup \{t_1^1,r_3^3\},(L_2^3\cup \{s_2^2\})\setminus S_3^1$ and $L_3^2\cup \{s_3^3,t_1^1\}$, if $S_1^2\neq \emptyset$, then remove the cliques $K_3^1,M_2^1$ and $M_3^1$, add the clique $\{s_1^2,s_2^2,t_3^3,r_1^2\}$, merge the pair $(L_3^1,M_1^2)$ and replace the cliques $L_2^1,L_2^3$ and $L_2^3$ with the cliques $L_2^1\cup  \{s_1^1\}$ and $L_2^3\cup \{s_2^2,t_3^3\}$, and if $S_2^3\neq \emptyset$, then remove the cliques $K_1^2,K_3^2$ and $M_1^2$, add the clique $\{s_1^1,t_2^2,t_3^2\}$, merge the pair $(L_2^3,M_2^1)$ and replace the cliques $L_1^2,L_1^3,L_3^1$ and $L_3^2$ with the cliques $L_1^2\cup \{s_3^3\},L_1^3\cup \{t_3^3\},L_3^1\cup \{s_3^3,t_2^2\}$ and $L_3^2\cup \{s_3^3,t_2^2\}$.
	\end{itemize}
	This proves (1.2).\vsp
	
	(1.3) \textit{We have $|R|\geq 3$.}\vsp
	
	Suppose not, assume that $|R|=2$. By (1.2), $|T|=4$ and w.l.o.g. we may assume that $T=\{t_1^1,t_2^2,t_1^2,t_2^2\}$ and $R=\{r_1^1,r_p^2\}$. By \pref{lem:schlafli Ultra}, $|S|=1$. Now, in $\mathscr{M}(G)$, if either $p=2$ and $|S|=|S_2^1|=1$, or $p=3$, then remove the cliques $M_2^1$ and $M_2^3$ and merge the pair $(L_3^1,L_3^2)$, and if $p=2$ and $|S|=|S_1^2|=1$, then remove the cliques $K_2^1$ and $K_3^1$ and merge the pair $(L_2^1,L_3^1)$, thereby obtaining a clique covering for $G$ of size $15=n-1$, a contradiction. This proves (1.3).\vsp
	
	(1.4) \textit{We have $|T|=4$.}\vsp
	
	For if this is not the case, then by (1.3) and the assumption $|T|\geq |R|$, we have $|T|=|R|=3$. Due to symmetry, let $T=\{t_1^1,t_{3-i}^i,t_2^2\}$, for some $i\in \{1,2\}$. Also, for every $(j,k)\in \{(1,1),(2,1),(1,p),(2,p)\}$, let $R_k^j=V(G)\cap \{r_k^j\}$, and also let $(j_0,k_0)$ be the unique element of $\{(1,1),(2,1),(1,p),(2,p)\}$ such that $R_{k_0}^{j_0}=\emptyset$. By \pref{lem:schlafli Ultra}, $|S|=1$. Through the following three cases, we apply suitable modifications to $\mathscr{M}(G)$  that yields a clique covering for $G$ of size at most $n-1$, a contradiction.
	\begin{itemize}
		\item $p=2$. By symmetry, we may assume that $i=1$. In the case that $|S|=|S_2^1|=1$, remove the cliques $M_2^1$ and $M_2^3$ and merge the pair $(L_1^2,L_1^3)$, and in the case that $|S|=|S_1^2|=1$, remove the cliques $K_2^1,K_3^1$ and $M_2^1$ and replace the cliques $L_1^2$ and $L_3^2$ with the cliques $L_1^2\cup R_1^1\cup \{t_3^3\}$ and $L_3^2\cup \{s_2^2\}$. 
		\item $p=3$ and $i=1$. In the case that $(j_0,k_0)=(1,1)$, remove the cliques $M_2^1$ and $M_1^2$, replace the cliques $L_2^1,L_3^1,L_1^3$ and $L_2^3$ with the cliques $L_2^1\cup \{t^1_1\},L_3^1\cup \{s_3^3\},L_1^3\cup \{t_3^3\}$ and $L_2^3\cup \{s_2^2\}$ and in the resulting family, if $|S|=|S_2^1|=1$, then merge the pair $(K_1^3,L_3^1\cup \{s_3^3\})$, and if $|S|=|S_1^3|=1$, then remove the clique $K_3^1$ and replace the cliques $L_1^2$ and $L_1^3\cup \{t_3^3\}$ with the cliques $L_1^2\cup \{t_3^1\}$ and $(L_1^3\cup \{s_1^1,t_3^3\})\setminus \{t_3^1\}$. Also, in the case that $(j_0,k_0)\in \{(2,1),(1,3)\}$, remove the clique $M_1^2$, replace the cliques $L_2^1$ and $L_3^1$ with the cliques $L_2^1\cup \{t_1^1\}$ and $L_3^1\cup \{s_3^3\}$ and in the resulting family, if $(j_0,k_0)=(2,1)$, then remove the cliques $L_1^2$ and $L_3^2$ and if $(j_0,k_0)=(1,3)$, then remove the cliques $M_2^1$ and $M_2^3$. Moreover, in the case that $(j_0,k_0)=(2,3)$, remove the clique $M_2^1$, replace the cliques $L_1^3$ and $L_2^3$ with the cliques $L_1^3\cup \{t_3^3,r_1^1\}$ and $L_2^3\cup \{s_2^2\}$, and in the resulting family, if $|S|=|S_2^1|=1$, then remove the clique $K_2^1$, merge the pair $(L_3^1,M_1^2)$ and replace the clique $L_2^1$ with the clique $L_2^1\cup \{t_1^1,r_3^3\}$, and if $|S|=|S_1^3|=1$, then remove the cliques $K_2^1$ and $K_3^1$.
		\item $p=3$ and $i=2$. First assume that $(j_0,k_0)=(2,3)$. If $|S|=|S_2^1|=1$, then merge the pair $(L_3^1,M_1^2)$, remove the cliques $K_3^1$ and $M_2^1$ and replace the cliques $L_2^1,L_1^3$ and $L_2^3$ and the cliques $L_2^1\cup \{s_1^1\},L_1^3\cup \{r_1^1\}$ and $L_2^3\cup \{s_2^2,t_3^3\}$, and if $|S|=|S_1^3|=1$, then remove the cliques $K_2^1,K_3^1,L_2^1$ and $L_2^3$. Now, suppose $(j_0,k_0)\neq (2,3)$. In the case that $|S|=|S_2^1|=1$, remove the cliques $K_3^1$ and $K_2^3$, replace the cliques $K_3^2,L_2^1,L_2^3,L_3^1$ and $L_3^2$ with the cliques $K_3^2\cup \{r_3^2\},L_2^1\cup \{s_1^1\},L_2^3\cup \{t_3^3\},L_3^1\cup \{t_1^1,t_1^2,r_3^3\}$ and $(L_3^2\cup \{s_2^2\})\setminus \{t_1^2\}$ and in the resulting family, if $(j_0,k_0)=(1,1)$, then remove the clique $M_2^1$ and replace the clique $L_2^3\cup \{t_3^3\}$ with $L_2^3\cup \{s_2^2,t_3^3\}$, if $(j_0,k_0)=(2,1)$, then merge the pair $(L_3^1,L_3^2)$ and if $(j_0,k_0)=(1,3)$, then remove the cliques $M_2^1$ and $M_2^3$. Also, in the case that $|S|=|S_1^3|=1$, remove the cliques $L_2^1$ and $L_2^3$ and in the resulting family, if $(j_0,k_0)=(1,1)$, then merge the pair $(K_1^3,L_3^1)$, if $(j_0,k_0)=(2,1)$, then merge the pair $(L_3^1,L_3^2)$ and if $(j_0,k_0)=(1,3)$, then remove the cliques $M_2^1$ and $M_2^3$.
	\end{itemize}
	This proves (1.4).\vsp
	
	Now (1.3) and (1.4) together imply that $|S|+|T|+|R|\geq 8$, which contradicts to \pref{lem:schlafli Ultra}.\\
	
	\textbf{Case 2.} $|I_T|+|I_R|=5$.\\
	By symmetry, assume that $|I_T|=3$ and $I_R=\{1,2\}$, and also $\{r_1^1,r_2^2\}\subseteq R\subseteq \{r_1^1,r_2^1,r_1^2,r_2^2\}$. Let $R_{3-i}^i=V(G)\cap \{r_{3-i}^i\}$, $i\in \{1,2\}$.\vsp
	
	(2.1) \textit{We have $|T|\geq 4$.}\vsp 
	
	Suppose not, assume that $|T|=3$, say $T=\{t_j^i,t_{3-j}^{i+1},t_{3-j}^{i+2}\}$. By symmetry, we may assume that $(i,j)\in \{(1,1),(1,2),(3,1)\}$. Through the following three cases, we apply some modifications to $\mathscr{M}(G)$  that yields a clique covering for $G$ of size at most $n-1$, a contradiction.
	\begin{itemize}
		\item $(i,j)=(1,1)$. Note that by \pref{lem:square-forcer} (applying to $(i,j,k)=(1,2,3)$), $S_2^1\cup S_3^1\cup S_1^2\neq \emptyset$. By \pref{lem:schlafli Ultra}, $|S|+|R|\leq 4$, i.e. $1\leq |S|\leq 2$ and $2\leq |R|\leq 3$. First, assume that $|R|=2$. If $S_2^1=\emptyset$, then remove the cliques $K_2^1,K_3^1,L_2^1$ and $L_2^3$. Now, suppose that $S_2^1\neq \emptyset$ and $S_1^2=\emptyset$. Remove the cliques $L_1^2$ and $L_2^3$ and in the resulting family, if $S_3^1=\emptyset$, then merge the pair $(M_1^3,M_2^3)$, remove the clique $M_3^1$ and replace the cliques $L_1^3$ and $L_2^3$ with the cliques $L_1^3\cup \{r_1^1\}$ and $L_2^3\cup \{s_2^2,t_3^3\}$, and if $S_3^1\neq \emptyset$, then remove the clique $M_2^3$ and replace the cliques $L_2^1$ and $L_3^1$ with the cliques $L_2^1\cup \{s_1^1, r_3^3\}$ and $L_3^1\cup \{t_2^2\}$. Thus, we may assume that $|S|=|S_2^1\cup S_1^2|=2$, and in this case, remove the cliques $K_2^1,M_3^2$ and $M_2^3$ and replace the cliques $L_2^1,L_1^2,L_3^2,L_1^3,L_2^3$ and $M_1^3$  with the cliques $L_2^1\cup \{r_3^3\},L_1^2\cup \{s_1^1,t_2^2\},L_3^2\cup \{r_3^3\},L_1^3\cup \{s_3^3,r_2^2\},L_2^3\cup \{s_2^2,t_1^1\}$ and $M_1^3\cup \{t_2^3\}$.
		
		Next, assume that $|R|=3$. Thus, $|S|=1$. If $S_2^1=\emptyset$, then remove the cliques $K_2^1$ and $K_3^1$ and merge the pair $(L_1^3,L_2^3)$. Also, if $|S|=|S_2^1|=1$, then merge the pairs $(L_1^2,L_1^3)$ and $(M_1^3,M_2^3)$, remove the clique $K_1^3$ and replace the cliques $L_3^1$ and $L_3^2$ with the cliques $L_3^1\cup \{t_2^2\}$ and $L_3^2\cup \{s_1^1, r_3^3\}$.
		\item $(i,j)=(1,2)$. Again by \pref{lem:square-forcer} (applying to $(i,j,k)=(1,2,3)$), $S_2^1\cup S_3^1\cup S_1^2\neq \emptyset$. By \pref{lem:schlafli Ultra}, $|S|+|R|\leq 4$, i.e. $1\leq |S|\leq 2$ and $2\leq |R|\leq 3$. First, assume that $|R|=2$. In the case that  $S_1^2=\emptyset$, remove the cliques $M_2^1$ and $M_2^3$ and in the resulting family, if $S_2^3=\emptyset$, then merge the pair $(L_2^1,L_3^1)$, if $S_3^2\cup S_2^3=\emptyset$, then merge the pair $(L_1^3,M_3^2)$ and if $S_2^3\neq \emptyset$ (and so exactly one of $S_2^1,S_3^1=\emptyset$), then remove the clique $K_1^2$ and replace the cliques $L_1^2$ and $L_1^3$ with the cliques $L_1^2\cup S_3^1\cup \{s_3^3\}$ and $(L_1^3\cup \{r_1^1\})\setminus S_3^1$. Now, assume that $S_1^2\neq \emptyset$. If $|S|=2$ and $S_3^2\cup S_2^3=\emptyset$, then merge the pair $(L_1^3,M_3^2)$, remove the cliques $K_2^1$ and $M_2^3$ and replace the cliques $K_2^3,L_2^1,L_2^3,L_1^2$ and $L_3^2$ with the cliques $(K_2^3\cup \{t_1^3\})\setminus \{t_2^1\},L_2^1\cup S_3^1\cup \{r_3^3\},(L_2^3\cup \{s_2^2\})\setminus S_3^1,(L_1^2\cup S_1^3\cup \{s_1^1\}$ and $(L_3^2\cup \{r_3^3\})\setminus S_1^3$. Also, if $|S|=2$ and $S_3^2\cup S_2^3\neq \emptyset$, then merge the pair $(K_1^2,L_1^3)$ and in the resulting family, in the case that $|S|=|S^2|=2$, remove the cliques $K_2^1$ and $M_1^3$ and replace the cliques $L_2^1,L_3^1,L_3^2$ and $M_2^3$ with the cliques $L_2^1\cup \{r_1^3\},(L_3^1\cup \{s_1^1,t_1^2,r_3^3\})\setminus \{r_1^3\},(L_3^2\cup \{s_2^2\})\setminus \{t_1^2\}$ and $M_2^3\cup \{t_1^2\}$ and in the case that $|S|=|S_1^2\cup S_2^3|=1$, remove the cliques $K_2^3$ and $M_2^3$ and replace the cliques $K_2^1,L_2^1,L_2^3,L_1^2$ and $L_3^2$ with the cliques $K_2^1\cup \{r_1^3\},(L_2^1\cup \{r_3^3\})\setminus S_2^3,L_2^3\cup S_2^3\cup \{s_2^2\},L_1^2\cup \{s_1^1\}$ and $L_3^2\cup \{r_3^3\}$. Finally, if $|S|=|S_1^2|=1$, then merge the pair $(L_1^3,M_3^2)$, remove the cliques $K_2^3,M_1^3$ and $M_2^3$ and replace the cliques $K_2^1,L_2^1,L_3^1,L_1^2,L_3^2$ and $L_2^3$ with the cliques $K_2^1\cup \{r_1^3\},L_2^1\cup \{r_1^3,r_3^3\},(L_3^1\cup \{s_1^1,r_3^3\})\setminus \{r_1^3\},L_1^2\cup \{s_1^1\},L_3^2\cup \{r_3^3\}$ and $L_2^3\cup \{s_2^2\}$.
		
		Next, assume that $|R|=3$. Thus, $|S|=1$. If $S_1^2=\emptyset$, then remove the cliques $M_2^1,M_2^3$ and $M_3^2$ and replace the cliques $L_1^2$ and $L_1^3$ with the cliques $L_1^2\cup \{s_3^3,r_2^2\}$ and $L_1^3\cup \{s_3^3,r_2^2\}$. Also, if $|S|=|S_1^2|=1$, then merge the pair $(K_1^2,L_1^3)$ and in the resulting family, if $R_2^1=\emptyset$, then remove the cliques $K_2^3$ and $M_1^3$ and replace the cliques $K_2^1,L_2^1,L_2^3$ and $L_3^1$ with the cliques $K_2^1\cup \{r_1^3\},L_2^1\cup \{r_1^3,r_3^3\},(L_3^1\cup \{s_1^1,r_3^3\})\setminus \{r_1^3\}$ and $L_2^3\cup \{s_2^2\}$, and if $R_1^2=\emptyset$, then remove the cliques $K_2^1$ and $M_2^3$ and replace the cliques $L_1^2,L_3^1$ and $L_3^2$ with the cliques $L_1^2\cup \{s_1^1\}, L_3^1\cup \{t_1^2\}$ and $(L_3^2\cup \{s_2^2,r_3^3\})\setminus \{t_1^2\}$.
		\item $(i,j)=(3,1)$. By \pref{lem:schlafli Ultra}, $|S|+|R|\leq 4$, i.e. $|S|\leq 2$ and $2\leq |R|\leq 4$. First, assume that $|R|=2$. If $S_1^2\neq \emptyset$, then remove the cliques $L_1^2,L_3^2,M_2^1$ and $M_2^3$ and if $S=\emptyset$, then merge the pair $(L_1^3,M_3^2)$. Now, assume that $S_1^2\neq \emptyset$. If $|S|=2$, then remove the cliques $M_1^2,M_2^1,M_2^3$ and $M_3^2$, add the clique $\{s_1^1,s_1^2,t_3^3\}$ and replace the cliques $L_2^1,L_3^1,L_1^2,L_3^2,L_1^3$ and $L_2^3$ with the cliques $(L_2^1\cup \{s_3^3,r_1^3\})\setminus S_2^1,(L_3^1\cup S_2^1\cup \{r_2^2\})\setminus \{r_1^3\},L_1^2\cup S_1^3\cup \{t_2^2,r_1^1\},(L_3^2\cup \{s_2^2,r_3^3\})\setminus S_1^3,(L_1^3\cup \{s_3^3,r_2^2\})\setminus S_3^2$ and $L_2^3\cup S_3^2$. Also, if $|S|=|S_1^2|=1$, then merge the pair $(L_1^3,M_3^2)$, remove the cliques $K_2^3,M_1^3,M_2^3$  and replace the cliques $K_2^1,L_2^1,L_3^1,L_1^2,L_3^2$ and $L_2^3$ with the cliques $K_2^1\cup \{r_1^3\},L_2^1\cup \{t_2^2,r_1^3,r_3^3\},(L_3^1\cup \{s_1^1,r_3^3\})\setminus \{r_1^3\},L_1^2\cup \{s_1^1,t_2^2\},L_3^2\cup \{r_3^3\}$ and $L_2^3\cup \{s_2^2\}$.
		Next, assume that $3\leq |R|\leq 4$. Thus, $|S|\leq 1$. If $|S|=1$ and $S_1^2=\emptyset$, then remove the cliques $M_2^1$ and $M_2^3$ and in the resulting family, if $S_3^1\neq \emptyset$, then merge the pair $(L_1^2,L_1^3)$ and if $S_3^1\neq \emptyset$, then remove the clique $K_2^1$ and replace the cliques $L_3^1$ and $L_3^2$ with the cliques $L_3^1\cup \{r_3^3\}$ and $L_3^2\cup \{s_2^2\}$. Also, if $|S|=|S_1^2|=1$, then merge the pair $(L_1^3,M_3^2)$, remove the cliques $K_2^1$ and $M_2^1$ and replace the cliques $L_3^1,L_1^2$ and $L_3^2$ with the cliques $L_3^1\cup \{r_3^3\},L_1^2\cup \{t_3^3,r_1^1\}$ and $L_3^2\cup \{s_2^2\}$. Finally, if $S=\emptyset$, then remove the cliques $M_2^1$ and $M_2^3$ and merge the pairs $(L_1^2,L_1^3)$ and $(L_1^3,M_3^2)$.
	\end{itemize}
	This proves (2.1).\vsp
	
	(2.2) \textit{We have $|T|=5$.}\vsp 
	
	Note that if $|T|=6$, then since $|R|\geq 2$, we have $|S|+|T|+|R|\geq 8$, which contradicts \pref{lem:schlafli Ultra}. Thus, if (2.2) does not hold, then by (1.2), $|T|=4$, and so by \pref{lem:schlafli Ultra}, $|S|+|R|\leq 3$, i.e. $|S|\leq 1$ and $2\leq |R|\leq 3$. In the sequel, first assume that $T=\{t_1^i,t_2^i,t_j^{i+1},t_j^{i+2}\}$, for some $i\in \{1,2,3\}$ and $j\in \{1,2\}$. By symmetry, we may assume that $(i,j)\in \{(1,1),(1,2),(3,2)\}$. Through the following three cases, we apply some modifications to $\mathscr{M}(G)$  that yields a clique covering for $G$ of size at most $n-1$, a contradiction.
	\begin{itemize}
		\item $(i,j)=(1,1)$. Note that by \pref{lem:square-forcer} (applying to $(i,j,k)=(1,2,3)$), $S_1^2\cup S_2^1\cup S_3^1\neq \emptyset$. Thus, $|S|=|S_1^2\cup S_2^1\cup S_3^1|=1$ and $|R|=2$. If $S_2^1=\emptyset$, then remove the cliques $K_2^1$ and $K_3^1$ and merge the pair $(L_2^1,L_3^1)$, and if $S_2^1\neq \emptyset$, then remove the cliques $L_1^2$ and $L_3^2$ and merge the pair $(M_1^3,M_2^3)$.
		\item $(i,j)=(1,2)$. Again by \pref{lem:square-forcer} (applying to $(i,j,k)=(1,2,3)$), we have $|S|=|S_1^2\cup S_2^1\cup S_3^1|=1$ and $|R|=2$. If $S_1^2=\emptyset$, then remove the cliques $M_2^1$ and $M_2^3$ and merge the pair $(L_3^1,L_3^2)$, and if $S_1^2\neq \emptyset$, then remove the cliques $M_1^3,M_2^3$ and $M_3^1$ and replace the cliques $L_2^1,L_3^1,L_1^2,L_3^2,L_1^3$ and $L_2^3$ with the cliques $L_2^1\cup \{r_1^3\},(L_3^1\cup \{s_1^1,r_3^3\})\setminus \{r_1^3\},L_1^2\cup \{s_1^1\},L_3^2\cup \{r_3^3\},L_1^3\cup \{t_3^3,r_1^1\}$ and $L_2^3\cup \{s_2^2\}$.
		\item $(i,j)=(3,1)$. First, assume that $|R|=2$. Thus, $|S|\leq 1$. In the case that $S_1^2=\emptyset$, remove the cliques $L_1^2$ and $L_3^2$ and in the resulting family, if $S_3^1=\emptyset$, then merge the pair $(M_1^3,M_2^3)$, remove the clique $M_3^1$ and replace the cliques $L_1^3,L_2^3$ and $M_2^1$ with the cliques $(L_1^3\cup \{t_3^3\})\setminus S_3^2,L_2^3\cup S_3^2\cup \{s_2^2,r_1^1\}$ and $M_2^1\cup S_2^3$, and if $S_3^1\neq \emptyset$, then remove the clique $M_2^3$ and replace the cliques $L_2^1$ and $L_3^1$ with the cliques $L_2^1\cup \{t_2^2,r_1^3,r_3^3\}$ and $(L_3^1\cup \{s_1^1\})\setminus \{r_1^3\}$. Also, in the case that $S_1^2\neq \emptyset$, merge the pair $(K_2^1,L_3^2)$, remove the cliques $K_1^2$ and $K_2^3$ and replace the cliques $K_2^1,L_1^2,L_1^3,L_2^1$ and $L_2^3$ with the cliques $K_2^1\cup \{r_1^3\}, L_1^2\cup \{t_2^2,r_1^1\},L_1^3\cup \{s_3^3\},L_2^1\cup \{r_3^3\}$ and $L_2^3\cup \{s_2^2\}$.
		
		Next, assume that $|R|=3$ and thus $S=\emptyset$, and in this case, merge the triple $(L_1^2,L_1^3,M_3^2)$ and the pair $(M_1^3,M_2^3)$.
	\end{itemize}
	
	Next, suppose that $T=\{t_1^i,t_2^i,t_j^{i+1},t_{3-j}^{i+2}\}$, for some $i\in \{1,2,3\}$ and $j\in \{1,2\}$. By symmetry, we may assume that $(i,j)\in \{(1,1),(1,2),(3,1),(3,2)\}$. Through the following possible cases, we apply some modifications to $\mathscr{M}(G)$  that yields a clique covering for $G$ of size at most $n-1$, a contradiction.
	\begin{itemize}
		\item $(i,j)=(1,1)$. First, assume that $|R|=2$. Thus, $|S|\leq 1$. If $S_2^1=\emptyset$, then remove the cliques $K_2^1$ and $K_3^1$ and in the resulting family, in the case that $S_1^3=\emptyset$, merge the pair $(L_2^1,L_3^1)$ and in the case that $S\setminus S_1^3=\emptyset$, remove the cliques $M_3^1$ and replace the cliques $L_1^3$ and $L_2^3$ with the cliques $L_1^3\cup \{t_3^3,r_1^1\}$ and $L_2^3\cup \{s_2^2\}$. Also, if $S_2^1\neq \emptyset$, then merge the pair $(K_1^2,L_1^3)$, remove the cliques $K_2^1$ and $M_3^1$ and replace the cliques $L_2^1,L_1^3$ and $L_2^3$ with the cliques $L_2^1\cup \{r_3^3\},L_1^3\cup \{t_3^3,r_1^1\}$ and $L_2^3\cup \{s_2^2,t_1^1\}$.
		
		Next, assume that $|R|=3$ and thus $S=\emptyset$, and in this case, remove the cliques $K_2^1,K_3^1$ and $M_3^2$ and replace the cliques $L_1^2,L_3^2,L_1^3$ and $L_2^3$ with the cliques $L_1^2\cup \{s_3^3,r_2^2\},L_3^2\cup \{t_1^1\},L_1^3\cup R_2^1\cup \{s_3^3,r_2^2\}$ and $(L_2^3\cup \{t_1^1\})\setminus R_2^1$.
		\item $(i,j)=(1,2)$. First, assume that $|R|=2$. Thus, $|S|\leq 1$. If $S_1^2=\emptyset$, then remove the cliques $L_1^2,L_1^3,M_2^1$ and $M_2^3$. Also, if $S_1^2\neq \emptyset$, then remove the cliques $M_1^2,M_2^3$ and $M_3^1$ and replace the cliques $L_2^1,L_3^1,L_1^2,L_3^2,L_1^3$ and $L_2^3$ with the cliques $L_2^1\cup \{s_3^3,r_1^3\},(L_3^1\cup \{t_1^1,r_2^2\})\setminus \{r_1^3\},L_1^2\cup \{s_1^1,t_2^2\},L_3^2\cup \{r_3^3\},L_1^3\cup \{t_3^3,r_1^1\}$ and $L_2^3\cup \{s_2^2\}$.
		
		Next, assume that $|R|=3$ and thus $S=\emptyset$, and in this case, remove the cliques $M_2^1$ and $M_2^3$ and merge the pair $(L_1^2,L_1^3)$.
		\item $(i,j)=(3,1)$. First, assume that $|R|=2$. Thus, $|S|\leq 1$. In the case that $S_2^1\cup S_3^2=\emptyset$, remove the cliques $L_2^1$ and $L_2^3$ and merge the pair $(K_2^1,K_2^3)$ and if $S_3^1\cup S_1^2=\emptyset$, then remove the cliques $L_1^2$ and $L_3^2$ and merge the pair $(M_1^3,M_2^3)$.
		
		Next, assume that $|R|=3$ and thus $S=\emptyset$, and in this case, merge the pairs $(K_2^1,K_2^3),(L_1^3,L_2^3)$ and $(M_1^3,M_2^3)$.
		\item $(i,j)=(3,2)$ and $|R|=2$. First, assume that $S_1^2\cup S_2\neq \emptyset$. Merge the pair $(K_1^2,L_1^3)$, if $S_2^3=\emptyset$, then remove the cliques $K_2^1$ and $K_2^3$ and replace the cliques $L_2^1,L_2^3,L_3^1$ and $L_3^2$ with the cliques $(L_2^1\cup \{r_3^3\})\setminus S_2^3,L_2^3\cup S_2^3\cup \{s_2^2\},(L_3^1\cup S_1^2 \cup \{s_2^2,r_3^3\})\setminus \{r_2^3\}$ and $(L_3^2\cup \{r_2^3\})\setminus S_1^2$ and if $S_2^3\neq \emptyset$, then remove the cliques $K_2^3$ and $M_1^3$ and replace the cliques $K_2^1,L_2^1,L_2^3,L_1^2,L_3^2$ and $M_2^3$ with the cliques $K_2^1\cup \{r_1^3\},(L_2^1\cup \{r_3^3\})\setminus S_2^3,L_2^3\cup S_2^3\cup \{s_2^2\}, L_1^2\cup \{s_1^1\}, L_3^2\cup \{r_3^3\}$ and $M_2^3\cup \{r_2^3\}$. 
		
		Next, assume that $S_1^2\cup S_2\neq \emptyset$. In this case, if $|S|=|S_3^1|=1$, then remove the cliques $K_2^1,K_3^2$ and $M_3^2$ and replace the cliques $L_3^1,L_3^2,L_1^2,L_1^3$ and $L_2^3$ with the cliques $(L_3^1\cup \{s_2^2,r_3^3\})\setminus \{r_2^3\},L_3^2\cup \{r_2^3\},L_1^2\cup S_3^1\cup \{s_1^1\},(L_1^3\cup \{t_3^3,r_2^2\})\setminus S_3^1$ and $L_2^3\cup \{s_3^3\}$, if $|S|=|S_1^3|=1$, then remove the cliques $K_3^2,M_3^1$ and $M_2^3$ and replace the cliques $L_2^1,L_3^1,L_1^2,L_1^3$ and $L_2^3$ with the cliques $L_2^1\cup \{r_1^3,r_3^3\},(L_3^1\cup \{s_1^1\})\setminus \{r_1^3\},L_1^2\cup \{s_1^1,r_2^2\},L_1^3\cup \{t_3^3\}$ and $L_2^3\cup \{s_2^2,r_1^1\}$, and if $S=S_3^2$, then merge the pair $(K_1^2,L_1^3)$, remove the cliques $K_2^1,M_3^1$ and $M_3^2$, add the clique $S_2^3\{t_3^1,t_3^3,r_2^2\}$, replace the cliques $L_3^1,L_3^2,L_1^3$ and $L_2^3$ with the cliques $(L_3^1\cup \{s_2^2,r_3^3\})\setminus \{r_2^3\},L_3^2\cup \{r_2^3\},L_1^3\cup \{s_3^3,r_1^1\}$ and $L_2^3\cup \{s_2^2\}$ and in the resulting family, if $S=\emptyset$, then remove the clique $K_2^3$ and replace the clique $L_2^1$ with the clique $L_2^1\cup \{r_3^3\}$.
		\item $(i,j)=(3,2)$ and $|R|=3$. Thus, $S=\emptyset$. Merge the pair $(L_1^3,M_3^2)$ and in the resulting family, if $R_2^1=\emptyset$, then remove the cliques $K_2^3$ and $M_2^3$ and replace the cliques $K_2^1,L_2^1,L_3^1$ and $L_2^3$ with the cliques $K_2^1\cup \{r_1^3\},L_2^1\cup \{r_1^3,r_3^3\},(L_3^1\cup \{s_1^1\})\setminus \{r_1^3\}$ and $L_2^3\cup \{s_2^2\}$, and if $R_1^2=\emptyset$, then remove the cliques $K_2^1$ and $M_1^3$ and replace the cliques $L_3^1,L_1^2,L_3^2$ and $M_2^3$ with the cliques $(L_3^1\cup \{s_2^2\})\setminus \{r_2^3\},L_1^2\cup \{s_1^1\},L_3^2\cup \{r_2^3,r_3^3\}$ and $M_2^3\cup \{r_2^3\}$.
	\end{itemize}
	This proves (2.2).\vsp
	
	By (2.2), $|T|=5$, say $t_j^i\notin T$, for some $i\in \{1,2,3\}$ and $j\in \{1,2\}$, and thus by \pref{lem:schlafli Ultra}, $S=\emptyset$ and $|R|=2$. By symmetry, we may assume that $(i,j)\in \{(1,1),(1,2),(3,1)\}$. Through the following three cases, we apply some modifications to $\mathscr{M}(G)$  that yields a clique covering for $G$ of size at most $n-1$, a contradiction.
	\begin{itemize}
		\item $(i,j)=(1,1)$. Remove the cliques $K_1^2,K_2^1$ and $K_2^3$ and replace the cliques $L_1^2,L_1^3,L_2^1,L_2^3,L_3^1$ and $L_3^2$ with the cliques $L_1^2\cup \{s_3^3,t_2^2\},L_1^3\cup \{r_1^1\},L_2^1\cup \{r_3^3\},L_2^3\cup \{s_2^2\},(L_3^1\cup \{s_2^2,r_3^3\})\setminus \{r_2^3\}$ and $L_3^2\cup \{r_2^3\}$.
		\item $(i,j)=(1,2)$. Remove the cliques $L_2^1$ and $L_2^3$ and merge the pair $(K_2^1,K_2^3)$.
		\item $(i,j)=(3,1)$. Remove the cliques $K_2^1$ and $K_3^1$ and merge the pair $(L_2^1,L_3^1)$.
	\end{itemize}

	\textbf{Case 3.} $|I_T|=|I_R|=3$.\\
	Due to symmetry, assume that $|T|\leq |R|$.\vsp
	
	(3.1) \textit{We have $|R|\geq 4$.}\vsp
	
	On the contrary, assume that $|T|=|R|=3$, and w.l.o.g, let $T=\{t_1^1,t_2^2,t_2^3\}$ and $R=\{r^i_j,r^{3-i}_{j+1},r^{3-i}_{j+2}\}$, for some $i\in \{1,2\}$ and $j\in \{1,2,3\}$. By \pref{lem:schlafli Ultra}, $|S|\leq 1$. Through the following six cases, we apply suitable modifications to $\mathscr{M}(G)$,  which yields a clique covering for $G$ of size at most $n-1$, a contradiction.
	\begin{itemize}
		\item $(i,j)=(1,1)$. In the case $S_2^1\cup S_1^2=\emptyset$, remove the cliques $L_2^1,L_1^2,L_3^2$ and $L_2^3$, in the case $|S|=|S_2^1|=1$, remove the cliques $L_1^2,L_3^2$ and merge the pair $(M_1^3,M_2^3)$ and in the case $|S|=|S_1^2|=1$, remove the cliques $L_2^1,L_2^3$ and merge the pair $(K_3^1,K_3^2)$.
		\item $(i,j)=(1,2)$. If $|S|=|S_3^1|=1$, then remove the cliques $K_3^1,K_1^3$ and $M_2^3$ and replace the cliques $L_2^1,L_3^1,L_1^2,L_3^2$ and $L_1^3$ with the cliques $L_2^1\cup \{s_1^1\},L_3^1\cup \{t_2^2,r_3^3\},L_1^2\cup \{t_3^1,t_3^3\},L_3^2\cup \{s_3^3\}$ and $(L_1^3\cup \{s_1^1\})\setminus \{t_3^1\}$. If $|S|=|S_1^2|=1$, then $K_3^1,K_3^2$ and $M_1^3$ and replace the cliques $L_2^1,L_3^1,L_1^2,L_1^3$ and $L_2^3$ with the cliques $L_2^1\cup \{s_1^1\},L_3^1\cup \{t_2^2,r_3^3\},(L_1^2\cup \{s_1^1,t_3^3\})\setminus \{r_1^2\},L_1^3\cup \{r_1^2\}$ and $L_2^3\cup \{t_3^3\}$. If $|S|=|S_2^3|=1$, then merge the pairs $(L_1^2,L_1^3)$ and $(M_1^3,M_2^3)$, remove the clique $K_3^2$ and replace the cliques $K_3^1,L_2^1$ and $L_2^3$ with the cliques $K_3^1\cup \{t_3^2\},L_2^1\cup \{s_1^1\}$ and $L_2^3\cup \{t_3^3\}$. Finally, if $S_3^1\cup S_1^2\cup S_2^3=\emptyset$, then merge the pairs  $(L_1^2,L_1^3)$ and $(M_1^3, M_2^3)$, and in the resulting family, if $ S^1_2\cup S^3_1=\emptyset $, then remove the clique $ K_3^2 $ and replace the cliques $K_3^1$, $L_2^1$ and $ L_2^3 $ with the cliques  $K_3^1\cup \{t^2_3\}$, $L_2^1\cup \{s_1^1\}$ and $ L_2^3\cup \{t^3_3\} $, and if $ S^2_3=\emptyset $, then remove the clique $ K_1^3 $ and replace the cliques $ L_3^1 $ and $ L_3^2 $ with the cliques $ L_3^1\cup \{t^2_2\} $ and $ L_3^2\cup \{s_3^3\} $.
		\item $(i,j)=(1,3)$. In the case that $S_2^1=\emptyset$, remove the cliques $K_2^1,K_3^1,L_2^1$ and $L_2^3$ and in the case $|S|=|S_2^1|=1$, remove the cliques $K_2^1, K_3^1, K_1^3$ and $M_1^3$, add the clique $\{s_2^1,t_1^1,r_2^2\}$ and replace the cliques $L_2^1,L_3^1,L_3^2$ and $ L_2^3$ with the cliques $L_2^1\cup \{s_1^1,r_3^3\},L_3^1\cup \{t_2^2\},L_3^2\cup \{s_3^3\}$ and $L_2^3\cup\{s_2^2,t^3_3\}$.
		\item $(i,j)=(2,1)$. In the case that $S_2^1\cup S_3^2=\emptyset$, remove the cliques $K_2^1,K_3^1$ and merge the pair $(L_1^3,L_2^3)$, and in the resulting family, if $S=\emptyset$, then remove the clique $K_1^3$ and replace the cliques $L_3^1$ and $L_3^2$ with the cliques $L_3^1\cup \{t_2^2\}$ and $L_3^2\cup \{s_3^3\}$. Also, in the case $|S|=|S_2^1|=1$, remove the cliques $K_1^3,K_3^1$ and $M_1^3$ and replace the cliques $L_2^1,L_2^3,L_3^1$ and $L_3^2$ with the cliques $L_2^1\cup \{s_1^1\},L_2^3\cup \{t_3^3\},L_3^1\cup \{t_2^2,r_3^3\}$ and $L_3^2\cup \{s_3^3\}$ and in the case $|S|=|S_3^2|=1$, remove the cliques $K_2^1,K_3^1$ and $M_1^3$ and replace the cliques $L_2^1$ and $L_3^1$ with the cliques $L_2^1\cup \{s_1^1\}$ and $L_3^1\cup \{t_2^2,r_3^3\}$.
		\item $(i,j)=(2,2)$. In the case $S_2^1=\emptyset$, remove the cliques $K_2^1,K_3^1,L_2^1$ and $L_2^3$ and in the case $|S|=|S_2^1|=1$, remove the cliques $L_1^2, L_3^2$ and $M_1^3$ and replace the cliques $L_2^1$ and $L_3^1$ with the cliques $L_2^1\cup \{s_1^1,r_3^3\}$ and $L_3^1\cup \{t_2^2\}$.
		\item $(i,j)=(2,3)$. In the case $S_1^2\cup S_3^1=\emptyset$, remove the cliques $L_1^2$ and $L_3^2$ and merge the pair $(M_1^3,M_2^3)$, and if $S=\emptyset$, then remove the clique $K_3^2$ and replace the cliques $K_3^1,L_2^1,L_2^3$ with the cliques $K_3^1\cup \{t_3^2\},L_2^1\cup \{s_1^1\},L_2^3\cup \{t_3^3\}$. Also, in the case $|S|=|S_1^2|=1$, remove the clique $K_3^2$, replace the cliques $L_2^1,L_2^3,L_3^1$ and $L_3^2$ with the cliques $L_2^1\cup \{s_1^1\},L_2^3\cup \{t_3^3\},L_3^1\setminus \{t_2^3\}$ and $L_3^2\cup \{t_2^3\}$ and merge the pairs $(K_3^1,L_1^2)$ and $(L_3^1\setminus \{t_2^3\},M_1^2)$ in the resulting family. Finally, in the case $|S|=|S_3^1|=1$, remove the cliques $L_1^2,L_3^2$ and $M_2^3$ and replace the cliques $L_2^1$ and $L_3^1$ with the cliques $L_2^1\cup \{s_1^1\}$ and $L_3^1\cup \{t_2^2,r_3^3\}$.
	\end{itemize}
	This proves (3.1).\vsp
	
	(3.2) \textit{If $|T|=3$, then $|R|\geq 5$.}\vsp
	
	On the contrary, let $|T|=3$ and $|R|\leq 4$. By (3.1), $|R|=4$. Also, w.l.o.g. we may assume that $T=\{t_1^1,t_2^2,t_2^3\}$. By \pref{lem:schlafli Ultra}, $S=\emptyset$. In the sequel, first suppose that $R=\{r_j^1,r_j^2,r_{j+1}^i,r_{j+2}^i\}$ for some $i\in \{1,2\}$ and some $j\in \{1,2,3\}$. Through the following four cases, we apply suitable modifications to $\mathscr{M}(G)$, which yields a clique covering for $G$ of size at most $n-1$, a contradiction.
	\begin{itemize}
		\item $(i,j)\in \{(2,1),(2,3)\}$. Remove the cliques $L_2^1$ and $L_2^3$ and merge the pair $(K_3^1,K_3^2)$.
		\item $(i,j)\in \{(1,1),(1,2)\}$. Remove the cliques $K_2^1$ and $K_3^1$ and merge the pair $(L_1^3,L_2^3)$.
		\item $(i,j)=(2,2)$. Merge the pairs $(L_1^2,L_1^3)$ and $ (M_1^3,M_2^3) $, remove the clique $K_3^2$ and replace the cliques
		$ K_3^1 $, $ L_2^1 $ and $ L_2^3 $ with the cliques $ K_3^1\cup \{t^2_3\} $, $ L_2^1\cup \{s_1^1\} $ and $ L_2^3\cup \{t^3_3,r_2^2\} $.
		\item $(i,j)=(1,3)$. Remove the cliques $L_1^2,L_3^2$ and $M_1^3$, and replace the cliques $L_2^1$ and $L_3^1$ with the cliques $L_2^1\cup \{s_1^1\}$ and $L_3^1\cup \{t_2^2,r_3^3\}$.
	\end{itemize} 
	
	Next, suppose that $R=\{r_j^1,r_j^2,r_{j+1}^i,r_{j+2}^{3-i}\}$ for some $i\in \{1,2\}$ and some $j\in \{1,2,3\}$. Through the four cases, we apply some modifications to $\mathscr{M}(G)$, which yields a clique covering for $G$ of size at most $n-1$, a contradiction.
	\begin{itemize}
		\item $(i,j)\in \{(1,2),(2,1)\}$. Remove the cliques $K_2^1$ and $K_3^1$ and merge the pair $(L_1^3,L_2^3)$.
		\item $(i,j)\in \{(1,3),(2,2)\}$. Remove the cliques $L_1^2$ and $L_3^2$, if $(i,j)=(1,3)$, then merge the pair $(K_3^1,K_3^2)$ and if $(i,j)=(2,2)$, then merge the pair $(M_1^3,M_2^3)$.
		\item $(i,j)=(1,1)$. Merge the pairs $(L_1^2,L_1^3)$ and $ (M_1^3,M_2^3) $, remove the clique $K_3^2$ and replace the cliques
		$ K_3^1 $, $ L_2^1 $ and $ L_2^3 $ with the cliques $ K_3^1\cup \{t^2_3\} $, $ L_2^1\cup \{s_1^1\} $ and $ L_2^3\cup \{t^3_3\} $.
		\item $(i,j)=(2,3)$. Remove the cliques $K_3^2,K_2^3$ and $M_1^3$ and replace the cliques $K_3^1,L_2^1,L_3^1,L_3^2$ and $L_2^3$ with the cliques $K_3^1\cup \{t_3^2\},L_2^1\cup \{s_1^1\},L_3^1\cup \{t_2^2,r_3^3\},L_3^2\cup \{s_2^2,t_1^1\}$ and $L_2^3\cup \{t_3^3\}$.
	\end{itemize}
	This proves (3.2).\vsp
	
	Now, (3.2) and the assumption $|T|\leq |R|$ implies that $|T|+|R|\geq 8$, which is impossible by \pref{lem:schlafli Ultra}.\\
	
	This completes the proof of \pref{lem:ccschlafli}.
\end{proof}
\end{document}